\documentclass[12pt]{amsbook}
\usepackage{amsmath,amssymb,amsfonts,amsthm,latexsym,graphicx,multirow}
\usepackage{hyperref}

\def\A{\mathrm{A}} \def\AGL{\mathrm{AGL}} \def\AGaL{\mathrm{A\Gamma L}} \def\ASL{\mathrm{ASL}} \def\Aut{\mathrm{Aut}}
\def\B{\mathrm{B}} \def\bfO{\mathbf{O}}
\def\C{\mathbf{C}}     \def\calO{\mathcal{O}}  \def\calT{\mathcal{T}} \def\Cay{\mathrm{Cay}} \def\Cen{\mathbf{Z}} \def\Co{\mathrm{Co}} \def\Cos{\mathrm{Cos}}
\def\D{\mathrm{D}} \def\di{\,\big|\,} 
\def\E{\mathrm{E}}
\def\F{\mathrm{F}} 
\def\G{\mathrm{G}} \def\Ga{{\it\Gamma}} \def\GaL{\mathrm{\Gamma L}} \def\GF{\mathrm{GF}} \def\GL{\mathrm{GL}} \def\GO{\mathrm{O}} \def\GU{\mathrm{GU}}
\def\He{\mathrm{He}} \def\HS{\mathrm{HS}}
\def\J{\mathrm{J}}
\def\K{\mathbf{K}}
\def\M{\mathrm{M}} \def\magma{{\sc Magma}}
\def\N{\mathrm{N}}
\def\Nor{\mathbf{N}}
\def\Out{\mathrm{Out}}
\def\Pa{\mathrm{P}}  \def\PGL{\mathrm{PGL}} \def\PGaL{\mathrm{P\Gamma L}} \def\PGU{\mathrm{PGU}}  \def\POm{\mathrm{P\Omega}} \def\PSL{\mathrm{PSL}}  \def\PSiU{\mathrm{P\Sigma U}} \def\PSO{\mathrm{PSO}} \def\PSp{\mathrm{PSp}} \def\PSU{\mathrm{PSU}}
\def\Q{\mathrm{Q}}
 \def\Rad{\mathrm{rad}}  
\def\Si{{\it\Sigma}} \def\SL{\mathrm{SL}} \def\SO{\mathrm{SO}} \def\Soc{\mathrm{soc}} \def\soc{\mathrm{soc}} \def\Sp{\mathrm{Sp}} \def\Stab{\mathrm{Stab}} \def\SU{\mathrm{SU}} \def\Suz{\mathrm{Suz}} \def\Sy{\mathrm{S}} \def\Sz{\mathrm{Sz}}
\def\Z{\mathrm{C}} \def\ZZ{\mathrm{C}}

\newtheorem{theorem}{Theorem}[chapter]
\newtheorem{lemma}[theorem]{Lemma}
\newtheorem{proposition}[theorem]{Proposition}
\newtheorem{corollary}[theorem]{Corollary}

\theoremstyle{definition}

\newtheorem{construction}[theorem]{Construction}
\newtheorem{example}[theorem]{Example}
\newtheorem{remark}[theorem]{Remark}
\newtheorem{hypothesis}[theorem]{Hypothesis}

\numberwithin{equation}{chapter}
\numberwithin{table}{chapter}
\numberwithin{section}{chapter}

\begin{document}

\title{Factorizations of almost simple groups with a solvable factor, and Cayley graphs of solvable groups}
\author{Cai Heng Li, Binzhou Xia}
\address[Li]
{The University of Western Australia\\ Crawley 6009, WA\\ Australia.\newline Email: {\tt cai.heng.li@uwa.edu.au}}
\address[Xia]
{Peking University\\ Beijing 100871, China.\newline Email: {\tt binzhouxia@pku.edu.cn}}

\date{\today}
\maketitle

\begin{center}
\textbf{\large Abstract}
\end{center}
~\\
\noindent A classification is given for factorizations of almost simple groups with at least one factor solvable, and it is then applied to characterize $s$-arc-transitive Cayley graphs of solvable groups, leading to a striking corollary: Except the cycles, every non-bipartite connected $3$-arc-transitive Cayley graph of a solvable group is a cover of the Petersen graph or the Hoffman-Singleton graph.
\vspace{5\baselineskip}

\noindent Key words and phrases:

\noindent factorizations; almost simple groups; solvable groups; $s$-arc-transitive graphs
\vspace{5\baselineskip}

\vspace{5\baselineskip}

\noindent\textsc{Acknowledgements.}
We would like to thank Cheryl Praeger for valuable comments. We also thank Stephen Glasby and Derek Holt for their help with some of the computation in \magma. The first author acknowledges the support of a NSFC grant and an ARC Discovery Project Grant. The second author acknowledges the support of NSFC grant 11501011.

\setcounter{page}{2}
\tableofcontents


\chapter{Introduction}\label{Intro}


A classification is given for factorizations of almost simple groups with at least one factor solvable (Theorem~\ref{SolvableFactor}), and it is then applied to characterize $s$-arc-transitive Cayley graphs of solvable groups (Theorem~\ref{CayleyGraph}), leading to a striking corollary:
\begin{quote}
{\it Except the cycles, every non-bipartite connected $3$-arc-transitive Cayley graph of a solvable group is a cover of the Petersen graph or the Hoffman-Singleton graph.}
\end{quote}

\section{Factorizations of almost simple groups}\label{factorization}

For a group $G$, the expression $G=HK$ with $H,K$ proper subgroups of $G$ is called a \emph{factorization} of $G$,
and $H,K$ are called its \emph{factors}. A finite group $G$ is said to be \emph{almost simple} if it lies between a nonabelian simple group $S$ and its automorphism group $\Aut(S)$, or equivalently, if $G$ has a unique minimal normal subgroup, which is nonabelian simple.

The \emph{socle} of a group $G$ is the product of all minimal normal subgroups, denoted by $\Soc(G)$. For a finite almost simple group, the factorization with a factor containing the socle is trivial in some sense. Thus we only consider the factorizations of which the factors are core-free. A factorization of an almost simple group is said to be \emph{nontrivial} if both its factors are core-free. We also say the factors in nontrivial factorizations of an almost simple group \emph{nontrivial}. When speaking of factorizations of almost simple groups, we always refer to nontrivial ones. It is at the heart of studying factorizations of general groups to understand the factorizations of almost simple groups.

\vskip0.1in
{\sc Problem A.} Classify factorizations of finite almost simple groups.
\vskip0.1in

This problem has been investigated for a long time. In the early stage, there were some partial results for certain simple groups. For instance, It\^{o} \cite{ito1953} determined the factorizations of $\PSL_2(q)$, and the factorizations of certain finite simple groups into two simple groups are classified \cite{Gentchev-1,Gentchev-2,TchGen}.
A factorization is said to be \emph{exact} if the intersection of the two factors is trivial. The exact factorizations of alternating and symmetric groups are determined by Wiegold and Williamson \cite{WiWi}. Fisman and Arad \cite{FA} in 1987 proved a conjecture of Sz\'{e}p \cite{szep1963,szep1968} saying that a nonabelian simple group $G$ does not have a factorization $G=HK$ with the centers of $H$ and $K$ both nontrivial.

During 1980s' and 1990s', the status of research on Problem~A dramatically changed
since some significant classification results have been achieved.
In 1987, Hering, Liebeck and Saxl \cite{hering1987factorizations} classified
factorizations of exceptional groups of Lie type.
A factorization $G=HK$ is called a \emph{maximal factorization} of $G$
if both $H,K$ are maximal subgroups of $G$.
In 1990, Liebeck, Praeger and Saxl published the landmark work \cite{liebeck1990maximal}
classifying maximal factorizations of finite almost simple groups,
which is the foundation of many further works on Problem~A.
When the socle is an alternating group,
in fact all the factorizations of an almost simple group were determined in \cite[THEOREM~D]{liebeck1990maximal}.
Based on the maximal factorizations in \cite[THEOREM~C]{liebeck1990maximal}, Giudici \cite{giudici2006factorisations} in 2006 determined the factorizations of almost simple groups with sporadic simple group socle.

Nevertheless, Problem~A for classical groups of Lie type is widely open. Our approach to this problem is to divide it into three major cases:
\begin{itemize}
\item at least one of the two factors is solvable;
\item at least one of the two factors has at least two unsolvable composition factors;
\item both factors have a unique unsolvable composition factor.
\end{itemize}
In this paper, we solve Problem~A for the first case (Theorem~\ref{SolvableFactor}). In subsequent work \cite{LiXia-2}, we solve the problem for the second case.

For the notation in the following theorem, refer to Section \ref{sec2}.

\begin{theorem}\label{SolvableFactor}
Suppose that $G$ is an almost simple group with socle $L$ and $G=HK$ for solvable subgroup $H$ and core-free subgroup $K$ of $G$. Then one of the following statements holds.
\begin{itemize}
\item[(a)] Both factors $H,K$ are solvable, and the triple $(G,H,K)$ is described in \emph{Proposition~\ref{BothSolvable}}.
\item[(b)] $L=\A_n$, and the triple $(G,H,K)$ is described in \emph{Proposition~\ref{Alternating}}.
\item[(c)] $L$ is a sporadic simple group, and the triple $(G,H,K)$ is described in \emph{Proposition~\ref{Sporadic}}.
\item[(d)] $L$ is a classical group of Lie type, and the triple $(L,H\cap L,K\cap L)$ lies in \emph{Table \ref{tab7}} or \emph{Table~\ref{tab1}}.
\end{itemize}
Conversely, for each simple group $L$ in \emph{Table \ref{tab7}} and \emph{Table~\ref{tab1}}, there exist group $G$ and subgroups $H,K$ as described such that $\Soc(G)=L$ and $G=HK$.
\end{theorem}


Here are some remarks on Theorem~\ref{SolvableFactor}.
\begin{itemize}
\item[(i)] For the triple $(G,H,K)$ in part (d) of Theorem \ref{SolvableFactor}, $\Nor_L(K\cap L)$ is maximal in $L$ except when $K\cap L=\PSL_3(3)$ in row~6 of Table~\ref{tab1} or $K\cap L=\A_5$, $\Sy_5$ in row~11 of Table~\ref{tab1}, respectively.
\item[(ii)] By the benefit of isomorphisms
$$
\qquad\quad\ \ \PSL_4(q)\cong\POm_6^+(q),\quad\PSp_{2m}(2^f)\cong\POm_{2m+1}(2^f)\quad\mbox{and}\quad\PSp_4(q)\cong\POm_5(q),
$$
there is a uniform description for rows~2--9 of Table \ref{tab7}: $L=\PSU_{2m}(q^{1/2})$ with $m\geqslant2$ or $\POm_{2m+1}(q)$ with $m\geqslant2$ or $\POm_{2m}^+(q)$ with $m\geqslant3$, and up to graph automorphisms of $L$, we have
\begin{equation*}
H\cap L\leqslant\mathbf{O}_p(\Pa_m){:}\hat{~}\GL_1(q^m).m<\Pa_m\quad\mbox{and}\quad\Nor_L(K\cap L)=\N_1^\varepsilon,
\end{equation*}
where $\varepsilon=2m-n\in\{0,-\}$ with $n$ being the dimension of $L$. To illustrate this, take row~4 of Table \ref{tab7} as an instance. Denote by $\tau$ the graph automorphism of $L$ of order $2$, and consider the factorization $G^\tau=H^\tau K^\tau$ deriving from $G=HK$. Since $\tau$ maps $\Pa_1$ to $\Pa_2$ and $\Sp_2(q^2).2$ to $\GO_4^-(q)$, we deduce that $H^\tau\cap L\leqslant q^3{:}(q^2-1).2<\Pa_2$ and $K^\tau\cap L\trianglerighteq\Omega_4^-(q)$. This means that the triple $(G^\tau,H^\tau,K^\tau)$ is in our uniform description with $L=\Soc(G^\tau)=\POm_5(q)$ and $q$ even.
\item[(iii)] Although Table \ref{tab7} presents $\Pa_k$ as a maximal subgroup of $L$ containing $H\cap L$ for each of the rows~2--9, it does not assert that $\Pa_k$ is the only such maximal subgroup. In fact, $H\cap L$ may be contained in the intersection of two different maximal subgroups of $L$, one of which is $\Pa_k$. For example, $G=\Sp_4(4)$ has a factorization $G=HK$ with $H=2^4{:}15<\Pa_2\cap\GO_4^-(4)$ and $K=\Sp_2(16){:}2$.
\end{itemize}

\begin{table}[htbp]
\caption{}\label{tab7}
\centering
\begin{tabular}{|l|l|l|l|l|}
\hline
row & $L$ & $H\cap L\leqslant$ & $K\cap L\trianglerighteq$ & remark\\
\hline
\multirow{2}*{1} & \multirow{2}*{$\PSL_n(q)$} & \multirow{2}*{$\hat{~}\GL_1(q^n){:}n=\frac{q^n-1}{(q-1)d}{:}n$} & \multirow{2}*{$q^{n-1}{:}\SL_{n-1}(q)$} & \multirow{2}*{$d=(n,q-1)$}\\
 & & & & \\
\hline
\multirow{2}*{2} & \multirow{2}*{$\PSL_4(q)$} & \multirow{2}*{$q^3{:}\frac{q^3-1}{d}.3<\Pa_k$} & \multirow{2}*{$\PSp_4(q)$} & $d=(4,q-1)$,\\
 & & & & $k\in\{1,3\}$\\
\hline
\multirow{2}*{3} & \multirow{2}*{$\PSp_{2m}(q)$}  & \multirow{2}*{$q^{m(m+1)/2}{:}(q^m-1).m<\Pa_m$} & \multirow{2}*{$\Omega_{2m}^-(q)$} & \multirow{2}*{$m\geqslant2$, $q$ even}\\
 & & & & \\
\hline
\multirow{2}*{4} & \multirow{2}*{$\PSp_4(q)$} & \multirow{2}*{$q^3{:}(q^2-1).2<\Pa_1$} & \multirow{2}*{$\Sp_2(q^2)$} & \multirow{2}*{$q$ even}\\
 & & & & \\
\hline
\multirow{2}*{5} & \multirow{2}*{$\PSp_4(q)$} & \multirow{2}*{$q^{1+2}{:}\frac{q^2-1}{2}.2<\Pa_1$} & \multirow{2}*{$\PSp_2(q^2)$} & \multirow{2}*{$q$ odd}\\
 & & & & \\
\hline
\multirow{2}*{6} & \multirow{2}*{$\PSU_{2m}(q)$} & \multirow{2}*{$q^{m^2}{:}\frac{q^{2m}-1}{(q+1)d}.m<\Pa_m$} & \multirow{2}*{$\SU_{2m-1}(q)$} & $m\geqslant2$,\\
 & & & & $d=(2m,q+1)$\\
\hline
\multirow{2}*{7} & \multirow{2}*{$\Omega_{2m+1}(q)$} & \multirow{2}*{$(q^{m(m-1)/2}.q^m){:}{q^m-1\over2}.m<\Pa_m$} & \multirow{2}*{$\Omega_{2m}^-(q)$} & \multirow{2}*{$m\geqslant3$, $q$ odd}\\
 & & & & \\
\hline
\multirow{3}*{8} & \multirow{3}*{$\POm_{2m}^+(q)$} & \multirow{3}*{$q^{m(m-1)/2}{:}\frac{q^m-1}{d}.m<\Pa_k$} & \multirow{3}*{$\Omega_{2m-1}(q)$} & $m\geqslant5$,\\
 & & & & $d=(4,q^m-1)$,\\
 & & & & $k\in\{m,m-1\}$\\
\hline
\multirow{2}*{9} & \multirow{2}*{$\POm_8^+(q)$} & \multirow{2}*{$q^6{:}\frac{q^4-1}{d}.4<\Pa_k$} & \multirow{2}*{$\Omega_7(q)$} & $d=(4,q^4-1)$,\\
 & & & & $k\in\{1,3,4\}$\\
\hline
\end{tabular}
\end{table}

\begin{table}[htbp]
\caption{}\label{tab1}
\centering
\begin{tabular}{|l|l|l|l|}
\hline
row & $L$ & $H\cap L\leqslant$ & $K\cap L$\\
\hline
1 & $\PSL_2(11)$ & $11{:}5$ & $\A_5$\\
2 & $\PSL_2(16)$ & $\D_{34}$ & $\A_5$\\
3 & $\PSL_2(19)$ & $19{:}9$ & $\A_5$\\
4 & $\PSL_2(29)$ & $29{:}14$ & $\A_5$\\
5 & $\PSL_2(59)$ & $59{:}29$ & $\A_5$\\
\hline
6 & $\PSL_4(3)$ & $2^4{:}5{:}4$ & $\PSL_3(3)$, $3^3{:}\PSL_3(3)$\\
7 & $\PSL_4(3)$ & $3^3{:}13{:}3$ & $(4\times\PSL_2(9)){:}2$\\
8 & $\PSL_4(4)$ & $2^6{:}63{:}3$ & $(5\times\PSL_2(16)){:}2$\\
\hline
9 & $\PSL_5(2)$ & $31{:}5$ & $2^6{:}(\Sy_3\times\PSL_3(2))$\\
\hline
10 & $\PSp_4(3)$ & $3^3{:}\Sy_4$ & $2^4{:}\A_5$\\
11 & $\PSp_4(3)$ & $3_+^{1+2}{:}2.\A_4$ & $\A_5$, $2^4{:}\A_5$, $\Sy_5$, $\A_6$, $\Sy_6$\\
12 & $\PSp_4(5)$ & $5_+^{1+2}{:}4.\A_4$ & $\PSL_2(5^2)$, $\PSL_2(5^2){:}2$\\
13 & $\PSp_4(7)$ & $7_+^{1+2}{:}6.\Sy_4$ & $\PSL_2(7^2)$, $\PSL_2(7^2){:}2$\\
14 & $\PSp_4(11)$ & $11_+^{1+2}{:}10.\A_4$ & $\PSL_2(11^2)$, $\PSL_2(11^2){:}2$\\
15 & $\PSp_4(23)$ & $23_+^{1+2}{:}22.\Sy_4$ & $\PSL_2(23^2)$, $\PSL_2(23^2){:}2$\\
\hline
16 & $\Sp_6(2)$ & $3_+^{1+2}{:}2.\Sy_4$ & $\A_8$, $\Sy_8$\\
17 & $\PSp_6(3)$ & $3_+^{1+4}{:}2^{1+4}.\D_{10}$ & $\PSL_2(27){:}3$\\
\hline
18 & $\PSU_3(3)$ & $3_+^{1+2}{:}8$ & $\PSL_2(7)$\\
19 & $\PSU_3(5)$ & $5_+^{1+2}{:}8$ & $\A_7$\\
\hline
20 & $\PSU_4(3)$ & $3^4{:}\D_{10}$, $3^4{:}\Sy_4$, & $\PSL_3(4)$\\
 & & $3^4{:}3^2{:}4$, $3_+^{1+4}.2.\Sy_4$ & \\
\hline
21 & $\PSU_4(8)$ & $513{:}3$ & $2^{12}{:}\SL_2(64).7$\\
\hline
22 & $\Omega_7(3)$ & $3^5{:}2^4.\AGL_1(5)$ & $\G_2(3)$\\
23 & $\Omega_7(3)$ & $3^{3+3}{:}13{:}3$ & $\Sp_6(2)$\\
\hline
24 & $\Omega_9(3)$ & $3^{6+4}{:}2^{1+4}.\AGL_1(5)$ & $\Omega^-_8(3)$, $\Omega^-_8(3).2$\\
\hline
25 & $\Omega_8^+(2)$ & $2^2{:}15.4<\A_9$ & $\Sp_6(2)$\\
26 & $\Omega_8^+(2)$ & $2^6{:}15.4$ & $\A_9$\\
27 & $\POm_8^+(3)$ & $3^6{:}2^4.\AGL_1(5)$& $\Omega_7(3)$\\
28 & $\POm_8^+(3)$ & $3^6{:}(3^3{:}13{:}3)$, $3^{3+6}{:}13.3$ & $\Omega_8^+(2)$\\
\hline
\end{tabular}

~\\
\vskip0.1in
Note the isomorphism $\PSp_4(3)\cong\PSU_4(2)$ for rows 10 and 11.
\end{table}

\vfill
\eject

\section{$s$-Arc transitive Cayley graphs}

The second part of this paper gives a characterization of non-bipartite connected $s$-arc-transitive Cayley graphs of solvable groups where $s\geqslant2$, as an application of Theorem~\ref{SolvableFactor}. Before stating the result, we introduce some definitions.

Let $\Ga=(V,E)$ be a simple graph with vertex set $V$ and edge set $E$. An automorphism of $\Ga$ is a permutation on $V$ that preserves the edge set $E$. The group consisting of all the automorphisms of $\Ga$ is called the \emph{(full) automorphism group} of $\Ga$ and denoted by $\Aut(\Ga)$. For a positive integer $s$, an \emph{$s$-arc} in $\Ga$ is a sequence of vertices $(\alpha_0,\alpha_1,\dots,\alpha_s)$ such that $\{\alpha_{i-1},\alpha_i\}\in E$ and $\alpha_{i-1}\neq\alpha_{i+1}$ for all admissible $i$. A $1$-arc is simply called an \emph{arc}. For some $G\leqslant\Aut(\Ga)$, the graph $\Ga$ is said to be \emph{$(G,s)$-arc-transitive} if $\Ga$ is regular (means that each vertex has the same number of neighbors) and $G$ acts transitively on the $s$-arcs of $\Ga$; $\Ga$ is said to be \emph{$(G,s)$-transitive} if $\Ga$ is $(G,s)$-arc-transitive but not $(G,s+1)$-arc-transitive. If $\Ga$ is $(\Aut(\Ga),s)$-arc-transitive or $(\Aut(\Ga),s)$-transitive, then we say that $\Ga$ is \emph{$s$-arc-transitive} or \emph{$s$-transitive}, respectively. Note that the cycles can be $s$-arc-transitive for any positive integer $s$.

Given a group $R$ and a subset $S\subset R$ which does not contain the identity of $R$ such that $S=S^{-1}:=\{g^{-1}\mid g\in S\}$,
the \emph{Cayley graph} $\Cay(R,S)$ of $R$ is the graph with vertex set $R$ such that two vertices $x,y$ are adjacent
if and only if $yx^{-1}\in S$. It is easy to see that $\Cay(R,S)$ is connected if and only if $S$ generates the group $R$, and a graph $\Ga$ is (isomorphic to) a Cayley graph of $R$ if and only if $\Aut(\Ga)$ has a subgroup isomorphic to $R$ acting regularly on the vertices of $\Ga$.

There have been classification results for certain classes of $s$-arc-transitive Cayley graphs in the literature. For instance, see \cite{ACMX} for $2$-arc-transitive circulants (Cayley graphs of cyclic groups), \cite{DMM,Marusic1,Marusic2} for $2$-arc-transitive dihedrants (Cayley graphs of dihedral groups) and \cite{IP,LP} for $2$-arc-transitive Cayley graphs of abelian groups. Cubic $s$-arc-transitive Cayley graphs are characterized in \cite{Conder,LL,XFWX05,XFWX07}, and tetravalent case is studied in \cite{LLZ}.

A graph $\Ga$ is said to be a \emph{cover} of a graph $\Si$, if there exists a surjection $\phi$ from the vertex set of $\Ga$ to the vertex set of $\Si$ which preserves adjacency and is a bijection from the neighbors of $v$ to the neighbors of $v^\phi$ for any vertex $v$ of $\Ga$. Suppose that $\Ga=(V,E)$ is a $(G,s)$-arc-transitive graph with $s\geqslant2$, and $G$ has a normal subgroup $N$ which has at least three orbits on the vertex set $V$. Then $N$ induces a graph $\Ga_N=(V_N,E_N)$, where $V_N$ is the set of $N$-orbits on $V$ and $E_N$ is the set of $N$-orbits on $E$, called the \emph{normal quotient} of $\Ga$ induced by $N$. In this context, we call $\Ga$ a \emph{normal cover} of $\Ga_N$, and call $\Ga$ a \emph{proper normal cover} of $\Ga_N$ if in addition $N\neq1$.

A transitive permutation group $G$ is called \emph{primitive} if $G$ does not preserve any nontrivial partition of the points, and is called \emph{quasiprimitive} if every nontrivial normal subgroup of $G$ is transitive. Given a non-bipartite $(G,s)$-arc-transitive graph $\Si$ with $s\geqslant2$, one can take a maximal intransitive normal subgroup $N$ of $G$, so that $\Si_N$ is $(G/N,s)$-arc-transitive with $G/N$ vertex-quasiprimitive and $\Si$ is a cover of $\Si_N$, see~\cite{Praeger}. If in addition $\Si$ is a Cayley graph of a solvable group, then the normal quotient $\Si_N$ admits a solvable vertex-transitive group of automorphisms. Thus to characterize $s$-arc-transitive Cayley graphs of solvable groups, the first step is to classify $(G,s)$-arc-transitive graphs with $G$ vertex-quasiprimitive containing a solvable vertex-transitive subgroup. This is essentially the result in Theorem~\ref{CayleyGraph}, where $\Si$ is extended to the class of graphs admitting a solvable vertex-transitive group of automorphisms although our original purpose is those Cayley graphs of solvable groups. A group $G$ is called \emph{affine} if $\Z_p^d\trianglelefteq G\leqslant\AGL_d(p)$ for some prime $p$ and positive integer $d$. The $(G,2)$-arc-transitive graphs for primitive affine groups $G$ are classified in~\cite{IP}. The $(G,2)$-arc-transitive graphs for almost simple groups $G$ with socle $\PSL_2(q)$ are classified in~\cite{HNP}.

\begin{theorem}\label{CayleyGraph}
Let $\Si$ be a non-bipartite connected $(X,s)$-transitive graph admitting a solvable vertex-transitive subgroup of $X$, where $s\geqslant2$ and the valency of $\Si$ is at least three. Then $s=2$ or $3$, and $X$ has a normal subgroup $N$ with the $(G,s)$-transitive normal quotient $\Ga$ described below, where $G=X/N$ and $\Ga=\Si_N$.
\begin{itemize}
\item[(a)] $G$ is a $3$-transitive permutation group of degree $n$, $\Ga=\K_n$, and $s=2$.
\item[(b)] $G$ is a primitive affine group, so that $\Ga$ is classified in~\cite{IP}, and $s=2$.
\item[(c)] $\PSL_2(q)\leqslant G\leqslant\PGaL_2(q)$ for some prime power $q\geqslant4$, so that $\Ga$ is classified in~\cite{HNP}.
\item[(d)] $G=\PSU_3(5)$ or $\PSiU_3(5)$, $\Ga$ is the Hoffman-Singlton graph, and $s=3$.
\item[(e)] $G=\HS$ or $\HS.2$, $\Ga$ is the Higman-Sims graph, and $s=2$.
\end{itemize}
In particular, $s=3$ if and only if $\Ga$ is the Petersen graph or the Hoffman-Singleton graph.
\end{theorem}


A Cayley graph $\Ga$ of a group $R$ is called a \emph{normal Cayley graph} of $R$ if the right regular representation of $R$ is normal in $\Aut(\Ga)$, and a Cayley graph of a solvable group is called a \emph{solvable Cayley graph}. An instant consequence of Theorem~\ref{CayleyGraph} is the following.

\begin{corollary}
A connected non-bipartite $2$-arc-transitive solvable Cayley graph of valency at least three is a normal cover of
\begin{itemize}
\item[(a)] a complete graph, or
\item[(b)] a normal Cayley graph of $\ZZ_2^d$, or
\item[(c)] a $(G,2)$-arc-transitive graph with  $\soc(G)=\PSL_2(q)$, or
\item[(d)] the Hoffman-Singlton graph, or
\item[(e)] the Higman-Sims graph.
\end{itemize}
In particular, a connected non-bipartite $3$-arc-transitive solvable Cayley graph of
valency at least three is a normal cover of the Petersen graph or the Hoffman-Singleton graph.
\end{corollary}

We remark that neither the Petersen graph nor the Hoffman-Singleton graph is a Cayley graph. Thus a non-bipartite $3$-arc-transitive solvable Cayley graph (if any) is a proper normal cover of the Petersen graph or the Hoffman-Singleton graph.


\chapter{Preliminaries}


We collect in this chapter the notation and elementary facts as well as some technical lemmas. Some basic facts will be used in the sequel without further reference.

\section{Notation}\label{sec2}

In this paper, all the groups are supposed to be finite and all graphs are supposed to be finite and simple if there are no further instructions. We set up the notation below, where $G,H,K$ are groups, $M$ is a subgroup of $G$, $n$ and $m$ are positive integers, $p$ is a prime number, and $q$ is a power of $p$ (so that $q$ is called a $p$-power).

\vskip0.1in
\begin{tabular}{ll}
$\Cen(G)$ & center of $G$\\
$\Rad(G)$ & solvable radical (the largest solvable normal subgroup) of $G$\\
$\Soc(G)$ & socle (product of the minimal normal subgroups) of $G$\\
$G^{(\infty)}$ & $=\bigcap_{i=1}^\infty G^{(i)}$\\
$\C_G(M)$ & centralizer of $M$ in $G$\\
$\Nor_G(M)$ & normalizer of $M$ in $G$\\
$\bfO_p(G)$ & largest normal $p$-subgroup of $G$\\
$[G{:}M]$ & set of right cosets of $M$ in $G$\\
$G\circ H$ & a central product of $G$ and $H$\\
$G{:}H$ & a split extension of $G$ by $H$\\
$G.H$ & an extension of $G$ by $H$\\
$G{:}H{:}K$ & a group of both type $G{:}(H{:}K)$ and $(G{:}H){:}K$\\
$G{:}H.K$ & a group of type $G{:}(H.K)$ (this is automatically of type $(G{:}H).K$)\\
$G.H.K$ & a group of type $G.(H.K)$ (this is automatically of type $(G.H).K$)\\
$\Sy_n$ & symmetric group of degree $n$ (naturally acting on $\{1,2,\dots,n\}$)\\
$\A_n$ & alternating group of degree $n$ (naturally acting on $\{1,2,\dots,n\}$)\\
$\Z_n$ & cyclic group of order $n$ (sometimes just denoted by $n$)\\
$\D_{2n}$ & dihedral group of order $2n$\\
$[n]$ & an unspecified group of order $n$\\
$n_p$ & $p$-part of $n$ (the largest $p$-power that divides $n$)\\
$n_{p'}$ & $=n/n_p$\\
$p^n$ & elementary abelian group of order $p^n$\\
$p^{n+m}$ & $=p^n.p^m$\\
$p^{1+2n}_+$ & extraspecial group of order $p^{1+2n}$ with exponent $p$ when $p$ is odd\\
$p^{1+2n}_-$ & extraspecial group of order $p^{1+2n}$ with exponent $p^2$ when $p$ is odd\\
$\GF(q)$ & finite field of $q$ elements\\
$\GaL_n(q)$ & group of semilinear bijections on $\GF(q)^n$\\
$\GL_n(q)$ & general linear group on $\GF(q)^n$\\
$\SL_n(q)$ & special linear group on $\GF(q)^n$\\
$\PGaL_n(q)$ & $=\GaL_n(q)/\Cen(\GL_n(q))$\\
\end{tabular}

\begin{tabular}{ll}
$\PGL_n(q)$ & projective general linear group on $\GF(q)^n$\\
$\PSL_n(q)$ & projective special linear group on $\GF(q)^n$\\
$\Sp_{2n}(q)$ & symplectic group on $\GF(q)^{2n}$\\
$\PSp_{2n}(q)$ & projective symplectic group on $\GF(q)^{2n}$\\
$\GU_n(q)$ & general unitary group on $\GF(q^2)^n$\\
$\SU_n(q)$ & special unitary group on $\GF(q^2)^n$\\
$\PGU_n(q)$ & projective general unitary group on $\GF(q^2)^n$\\
$\PSU_n(q)$ & projective special unitary group on $\GF(q^2)^n$\\
$\GO_{2n+1}(q)$ & general orthogonal group on $\GF(q)^{2n+1}$\\
$\GO_{2n}^+(q)$ & general orthogonal group on $\GF(q)^{2n}$ with Witt index $n$\\
$\GO_{2n}^-(q)$ & general orthogonal group on $\GF(q)^{2n}$ with Witt index $n-1$\\
$\SO_{2n+1}(q)$ & special orthogonal group on $\GF(q)^{2n}$\\
$\SO_{2n}^+(q)$ & special orthogonal group on $\GF(q)^{2n}$ with Witt index $n$\\
$\SO_{2n}^-(q)$ & special orthogonal group on $\GF(q)^{2n}$ with Witt index $n-1$\\
$\Omega_{2n+1}(q)$ & derived subgroup of $\SO_{2n+1}(q)$\\
$\Omega_{2n}^+(q)$ & derived subgroup of $\SO_{2n}^+(q)$\\
$\Omega_{2n}^-(q)$ & derived subgroup of $\SO_{2n}^-(q)$\\
$\PSO_{2n+1}(q)$ & projective special orthogonal group on $\GF(q)^{2n}$\\
$\PSO_{2n}^+(q)$ & projective special orthogonal group on $\GF(q)^{2n}$ with Witt\\
 & index $n$\\
$\PSO_{2n}^-(q)$ & projective special orthogonal group on $\GF(q)^{2n}$ with Witt\\
 & index $n-1$\\
$\POm_{2n+1}(q)$ & derived subgroup of $\PSO_{2n+1}(q)$\\
$\POm_{2n}^+(q)$ & derived subgroup of $\PSO_{2n}^+(q)$\\
$\POm_{2n}^-(q)$ & derived subgroup of $\PSO_{2n}^-(q)$\\
$\K_n$ & complete graph on $n$ vertices\\
$\K_{n,m}$ & complete bipartite graph with bipart sizes $n$ and $m$\\
\end{tabular}

\vskip0.1in
Let $T$ be a classical linear group on $V$ with center $Z$ such that $T/Z$ is a classical simple group, and $X$ be a subgroup of $\GL(V)$ containing $T$ as a normal subgroup. Then for any subgroup $Y$ of $X$, denote by $\hat{~}Y$ the subgroup $(Y\cap T)Z/Z$ of $T/Z$. For example, if $X=\GL_n(q)$ and $Y=\GL_1(q^n){:}n\leqslant X$ (see Section \ref{sec4}), then $\hat{~}Y=((q^n-1)/((q-1)(n,q-1))){:}n$ as the third column of row 1 in Table \ref{tab7}.

If $G$ is a linear group, define $\Pa_k[G]$ to be the stabilizer of a $k$-space in $G$. For the rest of this section, let $G$ be a symplectic, unitary or orthogonal group. If $G$ is transitive on the totally singular $k$-spaces, then define $\Pa_k[G]$ to be the stabilizer of a totally singular $k$-space in $G$. If $G$ is not transitive on the totally singular $k$-spaces (so that $G$ is an orthogonal group of dimension $2k$ with Witt index $k$, see \cite[2.2.4]{liebeck1990maximal}), then $G$ has precisely two orbits on them and
\begin{itemize}
\item[(i)] define $\Pa_{k-1}[G],\Pa_k[G]$ to be the stabilizers of totally singular $k$-spaces in the two different orbits of $G$;
\item[(ii)] define $\Pa_{k-1,k}[G]$ to be the stabilizer of a totally singular $(k-1)$-space in $G$;
\item[(iii)] define $\Pa_{1,k-1}[G]$ to be the intersection $\Pa_1[G]\cap\Pa_{k-1}[G]$, where the $1$-space stabilized by $\Pa_1[G]$ lies in the $k$-space stabilized by $\Pa_{k-1}[G]$;
\item[(iv)] define $\Pa_{1,k}[G]$ to be the intersection $\Pa_1[G]\cap\Pa_k[G]$, where the $1$-space stabilized by $\Pa_1[G]$ lies in the $k$-space stabilized by $\Pa_k[G]$.
\end{itemize}
Let $W$ be a non-degenerate $k$-space. We define
\begin{itemize}
\item[(i)] $\N_k[G]=G_W$ if either $G$ is symplectic or unitary, or $G$ is orthogonal of even dimension with $k$ odd;
\item[(ii)] $\N_k^\varepsilon[G]=G_W$ for $\varepsilon=\pm$ if $G$ is orthogonal and $W$ has type $\varepsilon$;
\item[(iii)] $\N_k^\varepsilon[G]=G_W$ for $\varepsilon=\pm$ if $G$ is orthogonal of odd dimension and $W^\bot$ has type $\varepsilon$.
\end{itemize}
For the above defined groups $\Pa_k[G]$, $\Pa_{i,j}[G]$, $\N_k[G]$, $\N_k^-[G]$ and $\N_k^+[G]$, we will simply write $\Pa_k$, $\Pa_{i,j}$, $\N_k$, $\N_k^-$ and $\N_k^+$, respectively, if the classical group $G$ is clear from the context.


We shall employ the families $\mathcal{C}_1$--$\mathcal{C}_8$ of subgroups of classical groups defined in \cite{aschbacher1984maximal}, see \cite{kleidman1990subgroup,kleidman1987maximal} for their group structure and \cite{kleidman1990subgroup,BHR-book} for determination of their maximality in classical simple groups. The subgroups $\Pa_k$, $\N_k$, $\N_k^+$ and $\N_k^-$ defined in the previous paragraph are actually $\mathcal{C}_1$ subgroups of classical groups.

\section{Results on finite simple groups}

This section contains some information about the finite simple groups which is needed in the sequel.

\subsection{Classification of the finite simple groups}

Let $q$ be a prime power. By the \emph{classification of finite simple groups} (CFSG), the finite nonabelian simple groups are:
\begin{itemize}
\item[(i)] $\A_n$ with $n\geqslant5$;
\item[(ii)] classical groups of Lie type:
\begin{eqnarray*}
&&\mbox{$\PSL_n(q)$ with $n\geqslant2$ and $(n,q)\neq(2,2)$ or $(2,3)$,}\\
&&\mbox{$\PSp_{2m}(q)$ with $m\geqslant2$ and $(m,q)\neq(2,2)$,}\\
&&\mbox{$\PSU_n(q)$ with $n\geqslant3$ and $(n,q)\neq(3,2)$,}\\
&&\mbox{$\Omega_{2m+1}(q)$ with $m\geqslant3$ and $q$ odd,}\\
&&\mbox{$\POm^+_{2m}(q)$ with $m\geqslant4$,\quad$\POm^-_{2m}(q)$ with $m\geqslant4$;}
\end{eqnarray*}
\item[(iii)] exceptional groups of Lie type:
\begin{eqnarray*}
&&\mbox{$\G_2(q)$ with $q\geqslant3$,\quad$\F_4(q)$,\quad$\E_6(q)$,\quad$\E_7(q)$,\quad$\E_8(q)$,}\\
&&\mbox{$\,^2\B_2(2^{2c+1})=\Sz(2^{2c+1})$ with $c\geqslant1$,\quad$\,^2\G_2(3^{2c+1})$ with $c\geqslant1$,}\\
&&\mbox{$\,^2\F_4(2^{2c+1})$ with $c\geqslant1$,\quad$\,^2\F_4(2)'$,\quad$\,^3\D_4(q)$,\quad$\,^2\E_6(q)$;}
\end{eqnarray*}
\item[(iv)] 26 sporadic simple groups.
\end{itemize}
Furthermore, the only repetitions among (i)--(iv) are:
\begin{eqnarray*}
&&\mbox{$\PSL_2(4)\cong\PSL_2(5)\cong\A_5$,\quad$\PSL_3(2)\cong\PSL_2(7)$,}\\
&&\mbox{$\PSL_2(9)\cong\A_6$,\quad$\PSL_4(2)\cong\A_8$,\quad$\PSU_4(2)\cong\PSp_4(3)$.}
\end{eqnarray*}
For the orders of the groups listed in (ii)--(iv), see \cite[TABLE 2.1]{liebeck1990maximal}.
\begin{remark}
We note that
\[\mbox{$\PSp_4(2)\cong\Sy_6$,\quad$\G_2(2)\cong\PSU_3(3){:}2$,\quad$\,^2\G_2(3)\cong\PGaL_2(8)$,}\]
and the groups $\PSL_2(2)$, $\PSL_2(3)$, $\PSU_3(2)$, $\,^2\B_2(2)$ are all solvable.
\end{remark}

\subsection{Outer automorphism group}

As a consequence of CFSG, the outer automorphism groups of finite simple groups are explicitly known. In particular, the \emph{Schreier conjecture} is true, that is, the outer automorphism group of every finite simple group is solvable. We collect some information for the outer automorphism groups of simple groups of Lie type in the following two tables, where $q=p^f$ with $p$ prime. For the presentations of outer automorphism groups of classical simple groups, see \cite[1.7]{BHR-book}.

\begin{table}[h!]
\caption{}\label{tab10}
\centering
\begin{tabular}{|l|l|l|}
\hline
$L$ & $\Out(L)$ & $d$ \\
\hline
$\PSL_2(q)$ & $\Z_d\times\Z_f$ & $(2,q-1)$ \\
$\PSL_n(q)$, $n\geqslant3$ & $\Z_d{:}(\Z_2\times\Z_f)$ & $(n,q-1)$ \\
$\PSU_n(q)$, $n\geqslant3$ & $\Z_d{:}\Z_{2f}$ & $(n,q+1)$ \\
$\PSp_{2m}(q)$, $(m,p)\neq(2,2)$ & $\Z_d\times\Z_f$ & $(2,q-1)$ \\
$\PSp_4(q)$, $q$ even & $\Z_{2f}$ & $1$ \\
$\Omega_{2m+1}(q)$, $m\geqslant3$, $q$ odd & $\Z_2\times\Z_f$ & $1$ \\
$\POm_{2m}^-(q)$, $m\geqslant4$, $q^m\not\equiv3\pmod{4}$ & $\Z_d\times\Z_{2f}$ & $(2,q-1)$ \\
$\POm_{2m}^-(q)$, $m\geqslant4$, $q^m\equiv3\pmod{4}$ & $\D_8\times\Z_f$ & $4$ \\
$\POm_{2m}^+(q)$, $m\geqslant5$, $q^m\not\equiv1\pmod{4}$ & $\Z_2\times\Z_d\times\Z_f$ & $(2,q-1)$ \\
$\POm_{2m}^+(q)$, $m\geqslant5$, $q^m\equiv1\pmod{4}$ & $\D_8\times\Z_f$ & $4$ \\
$\POm_8^+(q)$ & $\Sy_d\times\Z_f$ & $2+(2,q-1)$ \\
\hline
\end{tabular}
\end{table}

\begin{table}[h!]
\caption{}\label{tab23}
\centering
\begin{tabular}{|l|l|l|}
\hline
$L$ & $|\Out(L)|$ & $d$ \\
\hline
$\G_2(q)$, $q\geqslant3$ & $(3,p)f$ & $1$ \\
$\F_4(q)$ & $(2,p)f$ & $1$ \\
$\E_6(q)$ & $2df$ & $(3,q-1)$ \\
$\E_7(q)$ & $df$ & $(2,q-1)$ \\
$\E_8(q)$ & $f$ & $1$ \\
$\,^2\B_2(q)$, $q=2^{2c+1}\geqslant2^3$ & $f$ & $1$ \\
$\,^2\G_2(q)$, $q=3^{2c+1}\geqslant3^3$ & $f$ & $1$ \\
$\,^2\F_4(q)$, $q=2^{2c+1}\geqslant2^3$ & $f$ & $1$ \\
$\,^2\F_4(2)'$ & $2$ & $1$ \\
$\,^3\D_4(q)$ & $3f$ & $1$ \\
$\,^2\E_6(q)$ & $2df$ & $(3,q+1)$ \\
\hline
\end{tabular}
\end{table}

Let $a$ and $m$ be positive integers. A prime number $r$ is called a \emph{primitive prime divisor} of the pair $(a,m)$ if $r$ divides $a^m-1$ but does not divide $a^\ell-1$ for any positive integer $\ell<m$. We will simply say that $r$ is a primitive prime divisor of $a^m-1$ when $a$ is prime.

\begin{lemma}\label{l6}
A prime number $r$ is a primitive prime divisor of $(a,m)$ if and only if $a$ has order $m$ in $\GF(r)^\times$. In particular, if $r$ is a primitive prime divisor of $(a,m)$, then $m\di r-1$ and $r>m$.
\end{lemma}

The following \emph{Zsigmondy's theorem} shows exactly when the primitive prime divisors exist.

\begin{theorem}\label{Zsigmondy}
\emph{(Zsigmondy, see \emph{\cite[Theorem IX.8.3]{blackburn1982finite}})} Let $a$ and $m$ be integers greater than $1$. Then $(a,m)$ has a primitive prime divisor except for $(a,m)=(2,6)$ or $(2^k-1,2)$ with some positive integer $k$.
\end{theorem}

Checking the orders of outer automorphism groups of finite simple groups and viewing Lemma \ref{l6}, one has the consequence below, see \cite[p.38 PROPOSITION B]{liebeck1990maximal}.

\begin{lemma}
Let $L$ be a simple group of Lie type over $\GF(q)$, where $q=p^f$ with $p$ prime and $m\geqslant3$. If $(q,m)\neq(2,6)$ or $(4,3)$, then no primitive prime divisor of $p^{fm}-1$ divides $|\Out(L)|$.
\end{lemma}

\subsection{Minimal index of proper subgroups}

A group $G$ is said to be \emph{perfect} if $G'=G$. A group $G$ is said to be \emph{quasisimple} if $G$ is perfect and $G/\Cen(G)$ is nonabelian simple.

\begin{lemma}\label{MinimalDegree}
Let $P(X)$ denote the smallest index of proper subgroups of an arbitrary group $X$. Then the following statements hold.
\begin{itemize}
\item[(a)] If $G$ is almost simple with socle $L$, then $|G|/|H|\geqslant P(L)$ for any core-free subgroup $H$ of $G$.
\item[(b)] If $G$ is quasisimple with center $Z$, then $P(G/Z)=P(G)$.
\item[(c)] If $G$ is a permutation group on $n$ points and $N$ is a normal subgroup of $G$, then $P(G/N)\leqslant n$.
\item[(d)] If $H$ and $K$ are subgroups of $G$ such that $|G|/|K|<P(H)$, then $H\leqslant K$.
\end{itemize}
\end{lemma}

\begin{proof}
Proof for parts~(a) and~(b) is fairly easy, so we omit it. To prove part~(c), use induction on $n$. When $n=1$, the conclusion holds trivially. Let $K\leqslant G$ be the stabilizer of an arbitrary point. Evidently, $|G|/|K|\leqslant n$, and $K$ is a permutation group on $n-1$ points. As $K/K\cap N\cong KN/N$, we see that $K/K\cap N$ is isomorphic to a subgroup $H$ of $G/N$. If $H\neq G/N$, then
$$
P(G/N)\leqslant\frac{|G/N|}{|H|}=\frac{|G||K\cap N|}{|N||K|}\leqslant\frac{|G|}{|K|}\leqslant n.
$$
If $H=G/N$, then $P(G/N)=P(K/K\cap N)\leqslant n-1$ by the inductive hypothesis. Consequently, part~(c) is true.

It remains to prove part~(d). Suppose on the contrary that $H\nleqslant K$. Then $H\cap K$ is a proper subgroup of $H$, and so $|H|/|H\cap K|\geqslant P(H)$. This causes a contradiction that
$$
P(H)>\frac{|G|}{|K|}\geqslant\frac{|HK|}{|K|}=\frac{|H|}{|H\cap K|}\geqslant P(H).
$$
Thereby we have $H\leqslant K$.
\end{proof}

%
%

A list of the smallest indices of proper subgroups of classical simple groups was obtained by Cooperstein \cite{cooperstein1978minimal}. In fact, the smallest index of proper subgroups of a classical simple group follows immediately from the classification of its maximal subgroups. This is cited in the following theorem and is referred to \cite[Theorem~5.2.2]{kleidman1990subgroup}, which also points out the two errors in Cooperstein's list.

\begin{theorem}\label{l7}
The smallest index $P(L)$ of proper subgroups of a classical simple group $L$ is as in \emph{Table~\ref{tab24}}.
\end{theorem}

\begin{table}[h!]
\caption{}\label{tab24}
\centering
\begin{tabular}{|l|l|}
\hline
$L$ & $P(L)$ \\
\hline
$\PSL_n(q)$, $(n,q)\neq(2,5)$, $(2,7)$, $(2,9)$, $(2,11)$, $(4,2)$ & $(q^n-1)/(q-1)$ \\
\hline
$\PSL_2(5)$, $\PSL_2(7)$, $\PSL_2(9)$, $\PSL_2(11)$, $\PSL_4(2)$ & $5$, $7$, $6$, $11$, $8$ \\
\hline
$\PSp_{2m}(q)$, $m\geqslant2$, $q>2$, $(m,q)\neq(2,3)$ & $(q^{2m}-1)/(q-1)$ \\
\hline
$\Sp_{2m}(2)$, $m\geqslant3$ & $2^{m-1}(2^m-1)$ \\
\hline
$\PSp_4(3)$ & $27$ \\
\hline
$\Omega_{2m+1}(q)$, $m\geqslant3$, $q\geqslant5$ odd & $(q^{2m}-1)/(q-1)$ \\
\hline
$\Omega_{2m+1}(3)$, $m\geqslant3$ & $3^m(3^m-1)/2$ \\
\hline
$\POm_{2m}^+(q)$, $m\geqslant4$, $q\geqslant3$ & $(q^m-1)(q^{m-1}+1)/(q-1)$ \\
\hline
$\POm_{2m}^+(2)$, $m\geqslant4$ & $2^{m-1}(2^m-1)$ \\
\hline
$\POm_{2m}^-(q)$, $m\geqslant4$ & $(q^m+1)(q^{m-1}-1)/(q-1)$ \\
\hline
$\PSU_3(q)$, $q\neq5$ & $q^3+1$ \\
\hline
$\PSU_3(5)$ & $50$ \\
\hline
$\PSU_4(q)$ & $(q+1)(q^3+1)$ \\
\hline
$\PSU_n(q)$, $n\geqslant5$, $(n,q)\neq(6m,2)$ & $\frac{(q^n-(-1)^n)(q^{n-1}-(-1)^{n-1})}{q^2-1}$ \\
\hline
$\PSU_n(2)$, $n\equiv0\pmod{6}$ & $2^{n-1}(2^n-1)/3$ \\
\hline
\end{tabular}
\end{table}

%

\section{Elementary facts concerning factorizations}

We first give several equivalent conditions for a factorization.

\begin{lemma}\label{p3}
Let $H,K$ be subgroups of $G$. Then the following are equivalent.
\begin{itemize}
\item[(a)] $G=HK$.
\item[(b)] $G=H^xK^y$ for any $x,y\in G$.
\item[(c)] $|H\cap K||G|=|H||K|$.
\item[(d)] $|G|\leqslant|H||K|/|H\cap K|$.
\item[(e)] $H$ acts transitively on $[G{:}K]$ by right multiplication.
\item[(f)] $K$ acts transitively on $[G{:}H]$ by right multiplication.
\end{itemize}
\end{lemma}

Due to part (b) of Lemma~\ref{p3}, we will consider conjugacy classes of subgroups when studying factorizations of a group. Given a group $G$ and its subgroups $H,K$, in order to inspect whether $G=HK$ holds we only need to compute the orders of $G$, $H$, $K$ and $H\cap K$ by part (c) or (d) of Lemma~\ref{p3}. This is usually easier than checking the equality $G=HK$ directly, see our \magma codes in APPENDIX B. Utilizing equivalent conditions (e) and (f) in Lemma~\ref{p3}, one can construct factorizations in terms of permutation groups.

\begin{example}
According to Lemma~\ref{p3}, each $k$-homogeneous permutation group $H$ of degree $n$ gives rise to a factorization $\Sy_n=H(\Sy_k\times\Sy_{n-k})$.
\begin{itemize}
\item[(a)] Each transitive permutation group $H$ of degree $n$ gives rise to a factorization $\Sy_n=H\Sy_{n-1}$. Since a group is transitive in its regular permutation representation, we see that each group is a factor of some symmetric group.
\item[(b)] Each $2$-homogeneous permutation group $H$ of degree $n$ gives rise to a factorization $\Sy_n=H(\Sy_2\times\Sy_{n-2})$. The solvable $2$-homogeneous group $H$ of degree $n\geqslant5$ only exists for prime power $n$, and is either a subgroup of $\AGaL_1(n)$ or one of the following (see \cite{kantor1972k} and \cite[Theorem~XII.7.3]{huppert1982finite}):
\[
\begin{array}{|l|l|}
\hline
H & n \\
\hline
5^2{:}\SL_2(3),\ 5^2{:}\Q_8.6,\ 5^2{:}\SL_2(3).4 & 5^2 \\
7^2{:}\Q_8.\Sy_3,\ 7^2{:}\SL_2(3).6 & 7^2 \\
11^2{:}\SL_2(3).5,\ 11^2{:}\SL_2(3).10 & 11^2 \\
23^2{:}\SL_2(3).22 & 23^2 \\
3^4{:}2^{1+4}.5,\ 3^4{:}2^{1+4}.\D_{10},\ 3^4{:}2^{1+4}.\AGL_1(5) & 3^4 \\
\hline
\end{array}
\]
\item[(c)] There are only three solvable $3$-homogeneous groups of degree $n\geqslant5$, namely, $\AGL_1(8)$ and $\AGaL_1(8)$ with $n=8$, and $\AGaL_1(32)$ with $n=32$. Each of them is a factor of $\Sy_n$ with the other factor being $\Sy_3\times\Sy_{n-3}$.
\end{itemize}
See the proof of Proposition~\ref{Alternating} for a more comprehensive treatment of these factorizations.
\end{example}

The following simple lemma will be used repeatedly in subsequent chapters.

\begin{lemma}\label{p4}
Let $H,K$ be subgroups of $G$ and $L$ be a normal subgroup of $G$. If $G=HK$, then we have the following divisibilities.
\begin{itemize}
\item[(a)] $|G|$ divides $|H||K|$.
\item[(b)] $|G|$ divides $|H\cap L||K||G/L|$.
\item[(c)] $|L|$ divides $|H\cap L||K|$.
\item[(d)] $|L|$ divides $|H\cap L||K\cap L||G/L|$.
\end{itemize}
\end{lemma}

\begin{proof}
It derives from $G=HK$ that $|H\cap K||G|=|H||K|$ and thus statement (a) holds. Then since $|H|=|H\cap L||HL/L|$ divides $|H\cap L||G/L|$, we conclude that $|G|$ divides $|H\cap L||K||G/L|$. Hence statement (b) holds, which is equivalent to (c). Since $|K|=|K\cap L||KL/L|$ divides $|K\cap L||G/L|$, we then further deduce statement (d).
\end{proof}

In the remainder of this section, we present some lemmas relating factorizations of a group and those of its subgroups.

\begin{lemma}\label{p9}
Let $H,K$ and $M$ be subgroups of $G$. If $G=HK$, then $M=(H\cap M)(K\cap M)$ if and only if $|HM||KM|\leqslant|G||(H\cap K)M|$. In particular, if $G=HK$ and $H\leqslant M$, then $M=H(K\cap M)$.
\end{lemma}

\begin{proof}
By Lemma~\ref{p3}, $M=(H\cap M)(K\cap M)$ if and only if
\begin{equation}\label{eq15}
|M|\leqslant|H\cap M||K\cap M|/|H\cap K\cap M|.
\end{equation}
Substituting $|H\cap M|=|H||M|/|HM|$, $|K\cap M|=|K||M|/|HK|$ and $|H\cap K\cap M|=|H\cap K||M|/|(H\cap K)M|$ into (\ref{eq15}), we obtain
$$
|HM||KM||H\cap K|\leqslant|H||K||(H\cap K)M|.
$$
Since $|H||K|/|H\cap K|=|G|$ in view of $G=HK$, the above inequality turns out to be $|HM||KM|\leqslant|G||(H\cap K)M|$. This proves the lemma.
\end{proof}

\begin{lemma}\label{p8}
Let $K,M$ be subgroups of $G$ and $H$ be a subgroup of $M$. If $M=H(K\cap M)$, then $G=HK$ if and only if $G=MK$.
\end{lemma}

\begin{proof}
If $G=HK$, then $H\leqslant M$ implies $G=MK$. If $G=MK$, then $G=(H(K\cap M))K=H((K\cap M)K)=HK$.
\end{proof}

The above two lemmas enable us to construct new factorizations from given ones.

\begin{example}\label{ExampleLinear}
Let $G=\PSL_4(3)$, $M=\PSp_4(3).2<G$ and $K=\Pa_1=3^3{:}\PSL_3(3)$. Then we have $G=MK$ and $K\cap M=3^3{:}(\Sy_4\times2)$ by \cite[3.1.1]{liebeck1990maximal}. Since the almost simple group $M=\PSp_4(3).2$ has a factorization $M=H(K\cap M)$ with $H=2^4{:}\AGL_1(5)$ (see row 11 of Table \ref{BothSolvable}), there holds $G=HK$ by Lemma \ref{p8}. This factorization is as described in row 6 of Table~\ref{tab1}.
\end{example}

\begin{example}
Let $G=\POm_8^+(3)$, $M=\Pa_4=3^6{:}\PSL_4(3)$ and $K=\N_1=\Omega_7(3)$. Then we have $G=MK$ and $K\cap M=3^{3+3}{:}\PSL_3(3)$ by \cite[5.1.15]{liebeck1990maximal}. Write $M=R{:}S$, where $R=\Z_3^6$ and $S=\PSL_4(3)$. As shown in Example \ref{ExampleLinear}, $S$ has a factorization $S=H_1K_1$ with $H_1=2^4{:}\AGL_1(5)$ and $K_1=3^3{:}\PSL_3(3)<K\cap M$. Let $H=R{:}H_1<M$. It follows that
$$
M=RS=RH_1K_1=HK_1=H(K\cap M),
$$
and thus $G=HK$ by Lemma \ref{p8}. This factorization is as described in row 27 of Table~\ref{tab1}.
\end{example}

%
%

\section{Maximal factorizations of almost simple groups}

The nontrivial maximal factorizations of almost simple groups are classicfied by Liebeck, Praeger and Saxl \cite{liebeck1990maximal}. According to \cite[THEOREM~A]{liebeck1990maximal}, any nontrivial maximal factorization $G=AB$ of almost simple group $G$ with socle $L$ classical of Lie type lies in TABLEs 1--4 of \cite{liebeck1990maximal}. In TABLE 1 of \cite{liebeck1990maximal}, the maximal subgroups $A$ and $B$ are given by some natural subgroups $X_A$ and $X_B$ of $A\cap L$ and $B\cap L$ respectively such that $A=\Nor_G(X_A)$ and $B=\Nor_G(X_B)$. In fact, the explicit group structures of $A\cap L$ and $B\cap L$ in TABLE 1 of \cite{liebeck1990maximal} can be read off from \cite{kleidman1990subgroup}, which gives the following lemma.

\begin{lemma}\label{l1}
Let $G$ be an almost simple group with socle $L$ classical of Lie type. If $G=AB$ is a nontrivial maximal factorization as described in \emph{TABLE 1} of \cite{liebeck1990maximal}, then letting $X_A$ and $X_B$ be as defined in \emph{TABLE 1} of \cite{liebeck1990maximal}, we have one of the following.
\begin{itemize}
\item[(a)] $A\cap L=X_A$ and $B\cap L=X_B$.
\item[(b)] $L=\PSL_n(q)$ with $n\geqslant4$ even, $B\cap L=X_B$, and $A\cap L=\PSp_n(q).a$ where $a=(2,q-1)(n/2,q-1)/(n,q-1)$.
\item[(c)] $L=\PSU_{2m}(q)$ with $m\geqslant2$, $A\cap L=X_A$, and $B\cap L=\PSp_{2m}(q).a$ where $a=(2,q-1)(m,q+1)/(2m,q+1)$.
\item[(d)] $L=\POm_{2m}^+(q)$ with $m\geqslant6$ even and $q>2$, $A\cap L=X_A$, and $B\cap L=(\PSp_2(q)\otimes\PSp_m(q)).a$ where $a=\gcd(2,m/2,q-1)$.
\end{itemize}
In particular, $|A\cap L|/|X_A|\leqslant2$ and $|B\cap L|/|X_B|\leqslant2$.
\end{lemma}

In light of Lemma \ref{l1}, we restate THEOREM~A of \cite{liebeck1990maximal} as follows.

\begin{theorem}\label{Maximal}
\emph{(Liebeck, Praeger and Saxl)} Let $G$ be an almost simple group with socle $L$ classical of Lie type not isomorphic to $\A_5$, $\A_6$ or $\A_8$. If $G=AB$ is a nontrivial maximal factorization of $G$, then interchanging $A$ and $B$ if necessary, the triple $(L,A\cap L,B\cap L)$ lies in \emph{Tables \ref{tabLinear}--\ref{tabOmegaPlus2}} in \emph{APPENDIX A}.
\end{theorem}

For a group $G$, the \emph{solvable radical} of $G$, denoted by $\Rad(G)$, is the product of all the solvable normal subgroups of $G$.

\begin{lemma}\label{p12}
If a group $G$ has precisely one (involving multiplicity) unsolvable composition factor, then $G/\Rad(G)$ is almost simple.
\end{lemma}

\begin{proof}
Let $R=\Rad(G)$ be the solvable radical of $G$. If $G/R$ has an abelian minimal normal subgroup $N$ say, then the full preimage of $N$ is solvable and normal in $G$, which is a contradiction since $R$ is the largest solvable normal subgroup of $G$. Thus each minimal normal subgroup of $G/R$ is nonabelian. As $G$ has only one unsolvable composition factor, so does $G/R$. Therefore, $G/R$ has only one minimal normal subgroup and the minimal normal subgroup is nonabelian simple, which shows that $G/R$ is almost simple.
\end{proof}

For a group $G$, let $G^{(\infty)}=\bigcap_{i=1}^\infty G^{(i)}$ be the first perfect group in the derived series of $G$. Obviously, $G$ is solvable if and only if $G^{(\infty)}=1$. In fact, $G^{(\infty)}$ is the smallest normal subgroup of $G$ such that $G/G^{(\infty)}$ is solvable.


The following proposition plays a fundamental role in our further analysis.

\begin{proposition}\label{Intersection}
Suppose $G$ is an almost simple group with socle classical of Lie type, and $G=AB$ with subgroups $A,B$ maximal and core-free in $G$. If $A$ has exactly one unsolvable composition factor and $R=\Rad(A)$, then $(A\cap B)R/R$ is core-free in $A/R$.
\end{proposition}

\begin{proof}
Let $S$ be the unique unsolvable composition factor of $A$. By Lemma \ref{p12}, $A/R$ is an almost simple group with socle $S$. Suppose that $(A\cap B)R/R$ contains $\Soc(A/R)=(A/R)^{(\infty)}=A^{(\infty)}R/R$. Then $(A\cap B)R\geqslant A^{(\infty)}R$. From $G=AB$ we deduce that $|G|/|B|=|A|/|A\cap B|$ divides $|A||R|/|(A\cap B)R|$. Hence $|G|/|B|$ divides
$$
\frac{|A||R|}{|A^{(\infty)}R|}=\frac{|A||A^{(\infty)}\cap R|}{|A^{(\infty)}|}=\frac{|A|}{|A^{(\infty)}/\Rad(A^{(\infty)})|}.
$$
Since $A^{(\infty)}/\Rad(A^{(\infty)})$ is also an almost simple group with socle $S$, this implies that $|G|/|B|$ divides $|A|/|S|$. However, inspecting the candidates in Tables \ref{tabLinear}--\ref{tabOmegaPlus2}, we conclude that the factorization $G=AB$ does not satisfy this divisibility condition. Thus $(A\cap B)R/R$ does not contain $\Soc(A/R)$, that is, $(A\cap B)R/R$ is core-free in $A/R$.
\end{proof}

In order to appeal the classification of maximal factorizations to investigate the general factorizations of an almost simple group $G$, say, we need to embed a nontrivial factorization $G=HK$ to the a maximal factorization $G=AB$. This can be easily accomplished by taking arbitrary maximal subgroups $A,B$ of $G$ containing $H,K$ respectively. However, such maximal subgroups $A$ and $B$ are not necessarily core-free.



\begin{lemma}\label{Embedding}
Suppose that $G$ is an almost simple group with socle $L$ and $G$ has a nontrivial factorization $G=HK$. Then the following statements hold.
\begin{itemize}
\item[(a)] $HL=KL=G$ if and only if for any maximal subgroups $A,B$ of $G$ containing $H,K$ respectively, $A,B$ are both core-free.
\item[(b)] There exist $L\trianglelefteq G^*\leqslant G$ and a factorization $G^*=H^*K^*$ of $G^*$ such that $H^*\cap L=H\cap L$, $K^*\cap L=K\cap L$ and $H^*L=K^*L=G^*$.
\end{itemize}
\end{lemma}

\begin{proof}
Assume that $HL=KL=G$. For any maximal subgroups $A$ of $G$ containing $H$, since $A\geqslant L$ will lead to a contradiction that $A\geqslant HL=G$, we know that $A$ is core-free in $G$. Similarly, any maximal subgroup $B$ of $G$ containing $K$ is core-free in $G$.

Conversely, assume that any maximal subgroups of $G$ containing $H,K$, respectively, are core-free in $G$. If $HL<G$, then the maximal subgroup of $G$ containing $HL$ (and thus containing $H$) is not core-free in $G$, contrary to the assumption. Hence $HL=G$, and similarly we have $KL=G$. Therefore, part~(a) is true.

For part~(b), take $G^*=HL\cap KL$, $H^*=H\cap G^*$ and $K^*=K\cap G^*$. Then $L\trianglelefteq G^*\leqslant G$, $H^*\cap L=H\cap L$ and $K^*\cap L=K\cap L$. It follows that $H^*$ and $K^*$ are both core-free in $G^*$ since $H$ and $K$ are both core-free in $G$. By \cite[Lemma~2(i)]{liebeck1996factorizations}, we have $G^*=H^*K^*$ and $H^*L=K^*L=G^*$. Thus $G^*=H^*K^*$ is a factorization satisfying part~(b).
\end{proof}

\begin{remark}\label{rmk1}
For a nontrivial factorization $G=HK$, if we are concerned with $H\cap\Soc(G)$ and $K\cap\Soc(G)$ instead of $H$ and $K$ (such as in proving Theorem \ref{SolvableFactor}(d)), then Lemma \ref{Embedding} allows us to assume that any maximal subgroups of $G$ containing $H,K$, respectively, are core-free in $G$.
\end{remark}


\chapter[Prime dimension]{The factorizations of linear and unitary groups of prime dimension}\label{p-dim}


We classify factorizations of almost simple linear and unitary groups of prime dimension in this chapter. It turns out that all these factorizations have at least one solvable factor unless the socle is $\PSL_3(4)$.

\section{Singer cycles}\label{sec4}

Let $n=ab\geqslant2$ and $V=\GF(q^n)$. Then $V$ may be viewed as an $n$-dimensional vector space over $\GF(q)$ and an $a$-dimensional vector space over $\GF(q^b)$. Let $g$ be a $\GF(q^b)$-linear transformation on $V$. This means that, by definition, $g$ is a bijection on $V$ satisfying
\[\mbox{$(u+v)^g=u^g+v^g$ for any $u,v\in V$ and}\]
\[\mbox{$(\lambda v)^g=\lambda(v^g)$ for any $\lambda\in\GF(q^b)$.}\]
Since $\GF(q)\subseteq\GF(q^b)$, we see from the above conditions that $g$ is also a $\GF(q)$-linear transformation on $V$. Therefore, $\GL_a(q^b)\leqslant\GL_n(q)$. In fact, for any $g\in\GL_a(q^b)$ and $\sigma\in\mathrm{Gal}(\GF(q^b)/\GF(q))$, it is direct to check by definition that $\sigma^{-1}g\sigma$ is still a $\GF(q^b)$-linear transformation. Hence we have
$$
\GL_a(q^b){:}\mathrm{Gal}(\GF(q^b)/\GF(q))=\GL_a(q^b){:}b\leqslant\GL_n(q).
$$

Taking $a=1$ and $b=n$ in the above argument, one obtains that $\GL_1(q^n)<\GL_1(q^n){:}n<\GL_n(q)$. We call the cyclic group $\GL_1(q^n)$ and its conjugates in $\GL_n(q)$ the \emph{Singer cycle} of $\GL_n(q)$, and call $\hat{~}\GL_1(q^n)$ and its conjugates in $\PGL_n(q)$ the \emph{Singer cycle} of $\PGL_n(q)$. Obviously, the group $\GL_1(q^n)$ consisting of all $\GF(q^n)$-linear transformations of $V$ is transitive on $V\setminus\{0\}$. It follows that Singer cycles are transitive on the $1$-spaces and $(n-1)$-spaces, respectively, of $V$. Accordingly, for $k=1$ or $n-1$ we obtain the factorizations
$$
\GL_n(q)=\GL_1(q^n)\Pa_k[\GL_n(q)]=(\GL_1(q^n){:}n)\Pa_k[\GL_n(q)]
$$
and
$$
\PGL_n(q)=\hat{~}\GL_1(q^n)\Pa_k[\PGL_n(q)]=(\hat{~}\GL_1(q^n){:}n)\Pa_k[\PGL_n(q)].
$$

It is worth mentioning that $\GaL_1(q^n)$ actually has a lot of subgroups which are transitive on $V\setminus\{0\}$, see \cite{Foulser1964,Foulser1969}.

We have seen in the above that there exist almost simple groups $G$ with socle $L=\PSL_n(q)$ and subgroups $H,K$ such that $G=HK$, $H\cap L\leqslant\hat{~}\GL_1(q^n){:}n$ and $K\cap L\leqslant\Pa_1$ or $\Pa_{n-1}$. In the rest of this section, we will show that if such a factorization holds then $K\cap L$ must contain a large normal subgroup of $\Pa_1$ or $\Pa_{n-1}$.

\begin{lemma}\label{l5}
Let $G$ be an almost simple group with socle $L=\PSL_n(q)$ and $(n,q)\neq(3,2)$, $(3,3)$ or $(4,2)$. If $G=HK$ for subgroups $H$ and $K$ of $G$ such that $H\cap L\leqslant\hat{~}\GL_1(q^n){:}n$ and $K\cap L\leqslant\Pa_1$ or $\Pa_{n-1}$, then $K<\PGaL_n(q)$.
\end{lemma}

\begin{proof}
Suppose $K\nleqslant\PGaL_n(q)$, and write $d=(n,q-1)$. Then $n\geqslant3$, $K$ involves the graph automorphism of $L$, and $K\cap L$ stabilizes a decomposition $V=V_1\oplus V_{n-1}$, where $V$ is an $n$-dimensional vector space over $\GF(q)$ and $V_1,V_{n-1}$ are $1$-space and $(n-1)$-space, respectively, in $V$. This implies that $|K\cap L|$ divides $|\GL_{n-1}(q)|/d$. Moreover, we conclude from $H\cap L\leqslant\hat{~}\GL_1(q^n){:}n$ that $|H\cap L|$ divides $n(q^n-1)/(d(q-1))$. Hence by Lemma \ref{p4} we obtain
\begin{equation}\label{eq19}
q^{n-1}\di2fn,
\end{equation}
and thereby $2^{n-1}\leqslant p^{n-1}\leqslant p^{f(n-1)}/f\leqslant2n$. This yields $n\leqslant4$, and so $n=3$ or $4$. However, we deduce $q=2$ from (\ref{eq19}) if $n=3$ or $4$, contrary to our assumption that $(n,q)\neq(3,2)$ or $(4,2)$. Thus $K\leqslant\PGaL_n(q)$, and further $K<\PGaL_n(q)$ since $K\neq\PGaL_n(q)$.
\end{proof}

\begin{lemma}\label{LowerLinear}
Let $G$ be an almost simple group with socle $L=\PSL_n(q)$. If $G=HK$ for subgroups $H,K$ of $G$ such that $H\cap L\leqslant\hat{~}\GL_1(q^n){:}n$ and $K\cap L\leqslant\Pa_1$ or $\Pa_{n-1}$, then one of the following holds.
\begin{itemize}
\item[(a)] $q^{n-1}{:}\SL_{n-1}(q)\trianglelefteq K\cap L$.
\item[(b)] $(n,q)\in\{(2,4),(3,2),(3,3),(3,4),(3,8)\}$ and $K$ is solvable.
\end{itemize}
\end{lemma}

\begin{proof}
Let $q=p^f$ with prime number $p$. We divide the discussion into two cases distinguishing $n=2$ or not. Note that if $n=2$, part~(a) turns out to be $\Z_p^f\trianglelefteq K\cap L$.

{\bf Case 1.} Assume $n=2$, and suppose $q\neq4$. Then $q\geqslant5$ and $K\cap L\leqslant\Z_p^f{:}\Z_{(q-1)/(2,q-1)}$.

First suppose $f\leqslant2$. Then $p>2$ as $q\geqslant5$. It derives from Lemma \ref{p4} that $q$ divides $|K\cap L|$. Hence $q{:}\SL_1(q)=\Z_p^f\trianglelefteq K\cap L$ as part~(a) of Lemma \ref{LowerLinear}.

Next suppose $f\geqslant3$. By Zsigmondy's theorem, $p^f-1$ has a primitive prime divisor $r$ except $(p,f)=(2,6)$. If $(p,f)=(2,6)$, then $2^4\cdot3\cdot7$ divides $|K\cap L|$ by Lemma \ref{p4}, but the only subgroups of $2^6{:}63$ in $\PSL_2(64)$ with order divisible by $2^4\cdot3\cdot7$ are $2^6{:}21$ and $2^6{:}63$, which leads to $2^6\trianglelefteq K\cap L$. If $(p,f)\neq(2,6)$, then Lemma \ref{p4} implies that $K\cap L$ has order divisible by $r$ and thus has an element of order $r$. Note that $\Z_r$ does not have any faithful representation over $\GF(p)$ of dimension less than $f$. We then have $\Z_p^f{:}\Z_r\trianglelefteq K\cap L$.

{\bf Case 2.} Assume $n\geqslant3$, and assume without loss of generality $K\cap L\leqslant\Pa_1$. Write $\calO=G/L$.
By Lemma~\ref{Embedding}, we may assume that there exist core-free maximal subgroups $A,B$ of $G$ containing $H,K$, respectively. Then $G=AB$ with $A=\hat{~}\GL_1(q^n){:}n.\calO$ and $B\leqslant\PGaL_n(q)$ by Lemma \ref{l5}. It follows from Theorem \ref{Maximal} that $B\cap L=\Pa_1$. Thus $B=\Pa_1.\calO$ and $\calO\leqslant\PGaL_n(q)/\PSL_n(q)=[(n,q-1)f]$. Since $|G|/|A|$ divides $|K|$, $|B|/|K|$ divides
$$
\frac{|A||B|}{|G|}=\frac{|A|(q-1)}{q^n-1}=\frac{|\hat{~}\GL_1(q^n)|(q-1)|\calO|}{q^n-1}=\frac{n|\calO|}{(n,q-1)}.
$$
This yields that $|B|/|K|$ divides $nf$. Suppose that $(n,q)\neq(3,2)$ or $(3,3)$. Then $B$ has a unique unsolvable composition factor $\PSL_{n-1}(q)$, and $B^{(\infty)}=(B\cap L)^{(\infty)}=q^{n-1}{:}\SL_{n-1}(q)$. Let $R=\Rad(B)$, $\overline{B}=B/R$ and $\overline{K}=KR/R$. It follows that $\overline{B}$ is an almost simple group with socle $\PSL_{n-1}(q)$ by Lemma \ref{p12}. Notice that $|\overline{B}|/|\overline{K}|$ divides $|B|/|K|$ and thus divides $nf$. Moreover, we conclude from Theorem~\ref{l7} that either each proper subgroup of $\PSL_{n-1}(q)$ has index greater than $nf$, or $(n,q)\in\{(3,4),(3,8),(3,9)\}$.

First assume that $\overline{K}$ is core-free in $\overline{B}$. Then the observation
$$
\frac{|\Soc(\overline{B})|}{|\overline{K}\cap\Soc(\overline{B})|}
=\frac{|\overline{K}\,\Soc(\overline{B})|}{|\overline{K}|}\leqslant\frac{|\overline{B}|}{|\overline{K}|}\leqslant nf
$$
implies that $(n,q)\in\{(3,4),(3,8),(3,9)\}$. If $(n,q)=(3,4)$ or $(3,8)$, then $\overline{K}$ is solvable since it is core-free in $\PGaL_2(q)$, which implies that $K$ is solvable as part~(b) asserts. If $(n,q)=(3,9)$, then computation in \magma \cite{bosma1997magma} shows that $3^4{:}\SL_2(9)\trianglelefteq K\cap L$, as part~(a) of Lemma \ref{LowerLinear}.

Next assume that $\overline{K}\geqslant\Soc(\overline{B})$. Since $\Soc(\overline{B})=\overline{B}^{(\infty)}$, this yields $KR\geqslant B^{(\infty)}$. As a consequence, $K\geqslant\SL_{n-1}(q)$, which implies $K\cap L\geqslant\SL_{n-1}(q)$. Note that $B\cap L=q^{n-1}{:}\hat{~}\GL_{n-1}(q)$ and $\SL_{n-1}(q)\leqslant\hat{~}\GL_{n-1}(q)\leqslant B^{(\infty)}$ acts irreducibly on $q^{n-1}$. We conclude that either $K\cap L\geqslant q^{n-1}{:}\SL_{n-1}(q)$ or $K\cap L\leqslant\hat{~}\GL_{n-1}(q)$. However, the latter causes $|A|_p|K\cap L|_p|\calO|_p<|G|_p$, contrary to Lemma \ref{p4}. Thus $K\cap L\geqslant q^{n-1}{:}\SL_{n-1}(q)$ as part~(a) asserts.
\end{proof}

\section{Linear groups of prime dimension}


%
%

Now we determine the factorizations of almost simple groups with socle $\PSL_n(q)$ for prime dimension $n$. We exclude the values of $(n,q)\in\{(2,4),(2,5),(2,9),(3,2)\}$ due to the isomorphisms $\PSL_2(4)\cong\PSL_2(5)\cong\A_5$, $\PSL_2(9)\cong\A_6$ and $\PSL_3(2)\cong\PSL_2(7)$.

\begin{theorem}\label{p-dimensionLinear}
Let $G$ be an almost simple group with socle $L=\PSL_n(q)$ and $n$ be a prime with $(n,q)\not\in\{(2,4),(2,5),(2,9),(3,2)\}$. If $G=HK$ is a nontrivial factorization, then interchanging $H$ and $K$ if necessary, one of the following holds.
\begin{itemize}
\item[(a)] $H\cap L\leqslant\hat{~}\GL_1(q^n){:}n$, and $q^{n-1}{:}\SL_{n-1}(q)\trianglelefteq K\cap L\leqslant\Pa_1$ or $\Pa_{n-1}$.
\item[(b)] $n=2$ with $q\in\{7,11,16,19,23,29,59\}$, or $n=3$ with $q\in\{3,4,8\}$, or $(n,q)=(5,2)$, and $(G,H,K)$ lies in \emph{Table \ref{tab2}}.
\end{itemize}
Conversely, for each $L$ there exists a factorization $G=HK$ satisfying part~\emph{(a)} with $\Soc(G)=L$, and each triple $(G,H,K)$ in \emph{Table \ref{tab2}} gives a factorization $G=HK$.
\end{theorem}

\begin{table}[htbp]
\caption{}\label{tab2}
\centering
\begin{tabular}{|l|l|l|l|}
\hline
row & $G$ & $H$ & $K$ \\
\hline
1 & $\PSL_2(7).\calO$ & $7{:}\calO$, $7{:}(3\times\calO)$ & $\Sy_4$ \\
\hline
2 & $\PSL_2(11).\calO$ & $11{:}(5\times\calO_1)$ & $\A_4.\calO_2$ \\
3 & $\PSL_2(11).\calO$ & $11{:}\calO$, $11{:}(5\times\calO)$ & $\A_5$ \\
\hline
4 & $\PGaL_2(16)$ & $17{:}8$ & $(\A_5\times 2).2$ \\
\hline
5 & $\PSL_2(19).\calO$ & $19{:}(9\times\calO$) & $\A_5$ \\
\hline
6 & $\PSL_2(23).\calO$ & $23{:}(11\times\calO)$ & $\Sy_4$ \\
\hline
7 & $\PSL_2(29)$ & $29{:}7$, $29{:}14$ & $\A_5$ \\
8 & $\PGL_2(29)$ & $29{:}28$ & $\A_5$ \\
\hline
9 & $\PSL_2(59).\calO$ & $59{:}(29\times\calO)$ & $\A_5$ \\
\hline
10 & $\PSL_3(2).2$ & $7{:}6$ & $8$ \\
\hline
11 & $\PSL_3(3).\calO$ & $13{:}(3\times\calO)$ & $3^2{:}\GaL_1(9)$ \\
\hline
12 & $\PSL_3(4).2$ & $\PGL_2(7)$ & $\M_{10}$ \\
13 & $\PSL_3(4).2^2$ & $\PGL_2(7)$ & $\M_{10}{:}2$ \\
14 & $\PSL_3(4).2^2$ & $\PGL_2(7)\times2$ & $\M_{10}$ \\
15 & $\PSL_3(4).(\Sy_3\times\calO)$ & $7{:}(3\times\calO).\Sy_3$ & $(2^4.(3\times\D_{10})).2$ \\
\hline
16 & $\PSL_3(8).(3\times\calO)$ & $73{:}(9\times\calO_1)$ & $2^{3+6}{:}7^2{:}(3\times\calO_2)$ \\
\hline
17 & $\PSL_5(2).\calO$ & $31{:}(5\times\calO)$ & $2^6{:}(\Sy_3\times\SL_3(2))$ \\
\hline
\end{tabular}

~\\
where $\calO\leqslant\Z_2$, and $\calO_1,\calO_2$ are subgroups of $\calO$ such that $\calO=\calO_1\calO_2$.
\end{table}

\begin{proof}
Suppose that $G=HK$ is a nontrivial factorization. By Lemma~\ref{Embedding}, there exist a group $G^*$ with socle $L$ and its factorization $G^*=H^*K^*$ such that $H^*\cap L=H\cap L$, $K^*\cap L=K\cap L$, and the maximal subgroups $A^*,B^*$ containing $H^*,K^*$ respectively are core-free in $G^*$. Now $G^*=A^*B^*$ is determined by Theorem \ref{Maximal}. Inspecting candidates there and interchanging $A^*$ and $B^*$ if necessary, we obtain the following possibilities as $n$ is prime.
\begin{itemize}
\item[(i)] $A^*\cap L=\hat{~}\GL_1(q^n){:}n$ and $B^*\cap L=\Pa_1$ or $\Pa_{n-1}$.
\item[(ii)] $n=2$ and $q\in\{7,11,16,19,23,29,59\}$ as in rows 5--8 of Table \ref{tabLinear}.
\item[(iii)] $L=\PSL_3(4)$, $A^*\cap L=\PSL_2(7)$ and $B^*\cap L=\A_6$.
\item[(iv)] $L=\PSL_5(2)$, $A^*\cap L=31.5$ and $B^*\cap L=\Pa_2$ or $\Pa_3$.
\end{itemize}
Let $A,B$ be maximal core-free subgroups (maximal among the core-free subgroups) of $G$ containing $H,K$ respectively.

{\bf Case 1.} Assume that (i) appears, which implies $H\cap L=H^*\cap L\leqslant\hat{~}\GL_1(q^n){:}n$ and $K\cap L=K^*\cap L\leqslant\Pa_1$ or $\Pa_{n-1}$. If $(n,q)\not\in\{(3,2),(3,3),(3,8)\}$, then Lemma~\ref{LowerLinear} already leads to part~(a). It then remains to treat $(n,q)=(3,3)$ and $(3,8)$, respectively.

First consider $(n,q)=(3,3)$. Since $|G|/|A|=144$, we deduce from the factorization $G=AK$ that $|K|$ must be divisible by $144$. Suppose that part~(a) fails, that is, $3^2{:}\SL_2(3)\nleqslant K\cap L$. Then simple computation in \magma \cite{bosma1997magma} shows $K=\AGaL_1(9)=3^2{:}\GaL_1(9)$ in order that $G=AK$. Now $|K|=144=|G|/|A|$, and it derives from $G=HK$ that $H=A$. Thus row~11 of Table \ref{tab2} occurs.

Next consider $(n,q)=(3,8)$, and suppose that part~(a) fails. Then by Lemma~\ref{LowerLinear}, $K$ is solvable. Besides, $H\leqslant A$ is solvable as well. Searching by \magma \cite{bosma1997magma} the factorizations of $G$ with two solvable factors, we obtain the factorizations in row~16 of Table \ref{tab2}.

{\bf Case 2.} Assume that (ii) appears. In this case, $A^*\cap L$ and $B^*\cap L$ are described in rows 5--8 of Table \ref{tabLinear} as follows:
\[
\begin{array}{|l|l|l|}
\hline
A^*\cap L & B^*\cap L & q \\
\hline
\Pa_1 & \A_5 & 11,19,29,59 \\
\Pa_1 & \Sy_4 & 7,23 \\
\Pa_1 & \A_4 & 11 \\
\D_{34} & \PSL_2(4) & 16 \\
\hline
\end{array}
\]
Arguing in the same vein as the above case leads to rows~1--9 of Table \ref{tab2}. We have also confirmed these factorizations with computation in \magma \cite{bosma1997magma}.

{\bf Case 3.} Assume that (iii) appears. Computation in \magma \cite{bosma1997magma} for this case gives the triple $(G,H,K)$ in rows~12--15 of Table \ref{tab2}. We remark that there are $H_2\cong\PGL_2(7)$, $K_2\cong\M_{10}{:}2$, $H_3\cong\PGL_2(7)\times2$ and $K_3\cong\M_{10}$ such that $\PSL_3(4){:}2^2=H_2K_2=H_3K_3$, but $\PSL_3(4){:}2^2\neq H_3K_2$. In fact, the subgroup $H_1$ of $H_3$ isomorphic to $\PGL_2(7)$ is not conjugate to $H_2$ in $\PSL_3(4){:}2^2$, and the subgroup $K_1$ of $K_2$ isomorphic to $\M_{10}$ is not conjugate to $K_3$ in $\PSL_3(4){:}2^2$.

{\bf Case 4.} Finally, assume that (iv) appears. In this case, $\Pa_2\cong B\leqslant L$ and $|G|/|A|=|B|$, refer to \cite{atlas}. Thus $H=A$ and $K=B$ as in row~17 of Table \ref{tab2}.

Conversely, since the Singer cycle $H$ of $G=\PGL_n(q)$ is transitive on the $1$-spaces and $(n-1)$-spaces respectively, it gives a factorization $G=HK$ satisfying part~(a), where $K=\Pa_1[\PGL_n(q)]$ and $\Pa_{n-1}[\PGL_n(q)]$ respectively. Moreover, computation in \magma \cite{bosma1997magma} verifies that each triple $(G,H,K)$ in Table \ref{tab2} gives a factorization $G=HK$. The proof is thus completed.
\end{proof}

As a consequence of Theorem \ref{p-dimensionLinear}, we have

\begin{corollary}
Let $G$ be an almost simple linear group of prime dimension. Then any nontrivial factorization of $G$ has at least one factor solvable, unless $\Soc(G)=\PSL_3(4)$ and the factorization $G=HK$ is described in rows~\emph{12--14} of \emph{Table \ref{tab2}}.
\end{corollary}

\section{Unitary groups of prime dimension}\label{tab5}

The factorizations of unitary groups of prime dimension are classified in the following theorem.

\begin{theorem}\label{p-dimensionUnitary}
Let $G$ be an almost simple group with socle $L=\PSU_n(q)$ for an odd prime $n$. If $G=HK$ is a nontrivial factorization, then interchanging $H$ and $K$ if necessary, $(G,H,K)$ lies in \emph{Table \ref{tab3}}. Conversely, each triple $(G,H,K)$ in \emph{Table \ref{tab3}} gives a factorization $G=HK$.
\end{theorem}

\begin{table}[htbp]
\caption{}\label{tab3}
\centering
\begin{tabular}{|l|l|l|l|l|}
\hline
row & $G$ & $H$ & $K$ & \\
\hline
1 & $\PSU_3(3).\calO$ & $3_+^{1+2}{:}8{:}\calO$ & $\PSL_2(7).\calO$ & $\calO\leqslant\Z_2$\\
2 & $\PSU_3(3).2$ & $3_+^{1+2}{:}8$ & $\PGL_2(7)$ & \\
\hline
3 & $\PSU_3(5).\calO$ & $5_+^{1+2}{:}8.\calO$ & $\A_7$ & $\calO\leqslant\Sy_3$\\
4 & $\PSU_3(5).2$ & $5_+^{1+2}{:}8$, $5_+^{1+2}{:}8{:}2$ & $\Sy_7$ & \\
5 & $\PSU_3(5).\Sy_3$ & $5_+^{1+2}{:}(3{:}8)$, $5_+^{1+2}{:}24{:}2$ & $\Sy_7$ & \\
\hline
6 & $\PSU_3(8).3^2.\calO$ & $57{:}9.\calO_1$ & $2^{3+6}{:}(63{:}3).\calO_2$ & $\calO_1\calO_2=\calO\leqslant\Z_2$ \\
\hline
\end{tabular}
\end{table}

\begin{proof}
Suppose $G=HK$ to be a nontrivial factorization. By Lemma~\ref{Embedding}, there exist a group $G^*$ with socle $L$ and its factorization $G^*=H^*K^*$ such that $H^*\cap L=H\cap L$, $K^*\cap L=K\cap L$, and the maximal subgroups $A^*,B^*$ containing $H^*,K^*$ respectively are core-free in $G^*$. Now $G^*=A^*B^*$ is determined by Theorem \ref{Maximal}, which shows that, interchanging $A^*$ and $B^*$ if necessary, the triple $(L,A^*\cap L,B^*\cap L)$ lies in rows 5--7 of Table \ref{tabUnitary} as follows:
\[
\begin{array}{|l|l|l|}
\hline
L & A^*\cap L & B^*\cap L \\
\hline
\PSU_3(3) & \PSL_2(7) & \Pa_1 \\
\PSU_3(5) & \A_7 & \Pa_1 \\
\PSU_3(8) & 19.3 & \Pa_1 \\
\hline
\end{array}
\]


Computation in \magma \cite{bosma1997magma} for these cases produces all nontrivial factorizations as listed in Table \ref{tab3}. We remark that there are two non-isomorphic groups of type $3_+^{1+2}{:}8$ for $H$ in row~2 of Table \ref{tab3}. Also, there are two non-isomorphic groups of type $5_+^{1+2}{:}(3{:}8)$ for $H$ in row~5 of Table \ref{tab3}, where one is $5_+^{1+2}{:}24$ and the other is not.
\end{proof}

As a consequence of Theorem \ref{p-dimensionUnitary}, we have

\begin{corollary}
Let $G$ be an almost simple unitary group of odd prime dimension. Then any nontrivial factorization of $G$ has at least one factor solvable.
\end{corollary}


\chapter{Non-classical groups}\label{Non-classical}


For non-classical almost simple groups, since the factorizations are classified \cite{hering1987factorizations,liebeck1990maximal,giudici2006factorisations}, we can read off those factorizations which have a solvable factor.

\section{The case that both factors are solvable}

We first classify factorizations of almost simple groups with both factors solvable based on Kazarin's result \cite{kazarin1986groups}, so that we can focus later on the factorizations with precisely one solvable factor.

\begin{proposition}\label{BothSolvable}
Let $G$ be an almost simple group with socle $L$. If $G=HK$ for solvable subgroups $H,K$ of $G$, then interchanging $H$ and $K$ if necessary, one of the following holds.
\begin{itemize}
\item[(a)] $L=\PSL_2(q)$, $H\cap L\leqslant\D_{2(q+1)/d}$ and $q\trianglelefteq K\cap L\leqslant q{:}((q-1)/d)$, where $q$ is a prime power and $d=(2,q-1)$.
\item[(b)] $L$ is one of the groups: $\PSL_2(7)\cong\PSL_3(2)$, $\PSL_2(11)$, $\PSL_3(3)$, $\PSL_3(4)$, $\PSL_3(8)$, $\PSU_3(8)$, $\PSU_4(2)\cong\PSp_4(3)$ and $\M_{11}$; moreover, $(G,H,K)$ lies in \emph{Table \ref{tab4}}.
\end{itemize}
Conversely, for each prime power $q$ there exists a factorization $G=HK$ satisfying part~\emph{(a)} with $\Soc(G)=L=\PSL_2(q)$, and each triple $(G,H,K)$ in \emph{Table \ref{tab4}} gives a factorization $G=HK$.
\end{proposition}

\begin{table}[htbp]
\caption{}\label{tab4}
\centering
\begin{tabular}{|l|l|l|l|}
\hline
row & $G$ & $H$ & $K$ \\
\hline
1 & $\PSL_2(7).\calO$ & $7{:}\calO$, $7{:}(3\times\calO)$ & $\Sy_4$ \\
2 & $\PSL_2(11).\calO$  & $11{:}(5\times\calO_1)$ & $\A_4.\calO_2$ \\
3 & $\PSL_2(23).\calO$ & $23{:}(11\times\calO)$ & $\Sy_4$ \\
4 & $\PSL_3(3).\calO$ & $13{:}\calO$, $13{:}(3\times\calO)$ & $3^2{:}2.\Sy_4$ \\
5 & $\PSL_3(3).\calO$ & $13{:}(3\times\calO)$ & $\AGaL_1(9)$ \\
6 & $\PSL_3(4).(\Sy_3\times\calO)$ & $7{:}(3\times\calO).\Sy_3$ & $2^4{:}(3\times\D_{10}).2$ \\
7 & $\PSL_3(8).(3\times\calO)$ & $73{:}(9\times\calO_1)$ & $2^{3+6}{:}7^2{:}(3\times\calO_2)$ \\
\hline
8 & $\PSU_3(8).3^2.\calO$ & $57{:}9.\calO_1$ & $2^{3+6}{:}(63{:}3).\calO_2$ \\
9 & $\PSU_4(2).\calO$ & $2^4{:}5$ & $3_+^{1+2}{:}2.(\A_4.\calO)$ \\
10 & $\PSU_4(2).\calO$ & $2^4{:}\D_{10}.\calO_1$ & $3_+^{1+2}{:}2.(\A_4.\calO_2)$ \\
11 & $\PSU_4(2).2$ & $2^4{:}5{:}4$ & $3_+^{1+2}{:}\Sy_3$, $3^3{:}(\Sy_3\times\calO)$, \\
 & & & $3^3{:}(\A_4\times2)$, $3^3{:}(\Sy_4\times\calO)$ \\
\hline
12 & $\M_{11}$ & $11{:}5$ & $\M_9.2$ \\
\hline
\end{tabular}

~\\
where $\calO\leqslant\Z_2$, and $\calO_1,\calO_2$ are subgroups of $\calO$ such that $\calO=\calO_1\calO_2$.
\end{table}

\begin{proof}
Suppose $G=HK$ for solvable subgroups $H,K$ of $G$. By the result of \cite{kazarin1986groups}, either $L=\PSL_2(q)$, or $L$ is one of the groups:
\[\mbox{$\PSU_3(8)$, $\PSU_4(2)\cong\PSp_4(3)$, $\PSL_4(2)$, $\M_{11}$, $\PSL_3(q)$ with $q=3,4,5,7,8$.}\]
For the latter case, computation in \magma \cite{bosma1997magma} excludes $\PSL_3(5)$ and $\PSL_3(7)$, and produces for the other groups all the factorizations $G=HK$ with $H,K$ solvable as in rows~4--12 of Table~\ref{tab4}. Next we assume the former case, namely, $L=\PSL_2(q)$.

If $L=\PSL_2(4)\cong\PSL_2(5)\cong\A_5$ or $L=\PSL_2(9)\cong\A_6$, it is easy to see that part~(a) holds. Thus we assume $L=\PSL_2(q)$ with $q\in\{4,5,9\}$. Appealing Theorem \ref{p-dimensionLinear}, we have the factorization $G=HK$ as described in either part~(a) or the first three rows of Table~\ref{tab4}. This completes the proof.
\end{proof}

\section{Exceptional groups of Lie type}

Consulting the classification \cite{hering1987factorizations} of factorizations of exceptional groups of Lie type, one obtains the following proposition.

\begin{proposition}\label{ExceptionalLie}
Suppose $G$ is an almost simple group with socle $L$ and $G=HK$ with core-free subgroups $H,K$ of $G$. If $H$ is solvable, then $L$ is not an exceptional group of Lie type.
\end{proposition}

\section{Alternating group socles}

In this section we study the case of alternating group socle.

\begin{proposition}\label{Alternating}
Suppose $n\geqslant5$ and $\Soc(G)=\A_n$ acting naturally on $\Omega_n=\{1,\dots,n\}$. If $G=HK$ for solvable subgroup $H$ and core-free unsolvable subgroup $K$ of $G$, then one of the following holds.
\begin{itemize}
\item[(a)] $\A_n\trianglelefteq G\leqslant\Sy_n$ with $n\geqslant6$, $H$ is transitive on $\Omega_n$, and $\A_{n-1}\trianglelefteq K\leqslant\Sy_{n-1}$.
\item[(b)] $\A_n\trianglelefteq G\leqslant\Sy_n$ with $n=p^f$ for some prime $p$, $H$ is $2$-homogeneous on $\Omega_n$, and $\A_{n-2}\trianglelefteq K\leqslant\Sy_{n-2}\times\Sy_2$; moreover, either $H\leqslant\AGaL_1(p^f)$ or $(H,n)$ lies in the table:
\[
\begin{array}{|l|l|}
\hline
H & n \\
\hline
5^2{:}\SL_2(3),\ 5^2{:}\Q_8.6,\ 5^2{:}\SL_2(3).4  & 5^2 \\
7^2{:}\Q_8.\Sy_3,\ 7^2{:}\SL_2(3).6 & 7^2 \\
11^2{:}\SL_2(3).5, \ 11^2{:}\SL_2(3).10 & 11^2 \\
23^2{:}\SL_2(3).22 & 23^2 \\
3^4{:}2^{1+4}.5,\ 3^4{:}2^{1+4}.\D_{10},\ 3^4{:}2^{1+4}.\AGL_1(5) & 3^4 \\
\hline
\end{array}
\]
\item[(c)] $\A_n\trianglelefteq G\leqslant\Sy_n$ with $n=8$ or $32$, $\A_{n-3}\trianglelefteq K\leqslant\Sy_{n-3}\times\Sy_3$, and $(H,n)=(\AGL_1(8),8)$, $(\AGaL_1(8),8)$ or $(\AGaL_1(32),32)$.
\item[(d)] $\A_6\trianglelefteq G\leqslant\Sy_6$, $H\leqslant\Sy_4\times\Sy_2$, and $K=\PSL_2(5)$ or $\PGL_2(5)$.
\item[(e)] $\A_6\trianglelefteq G\leqslant\Sy_6$, $H\leqslant\Sy_3\wr\Sy_2$, and $K=\PSL_2(5)$ or $\PGL_2(5)$.
\item[(f)] $n=6$ or $8$, and $(G,H,K)$ lies in \emph{Table \ref{tab6}}.
\end{itemize}
Conversely, for each $G$ in parts~\emph{(a)--(e)} there exists a factorization $G=HK$ as described, and each triple $(G,H,K)$ in \emph{Table \ref{tab6}} gives a factorization $G=HK$.
\end{proposition}

\begin{table}[htbp]
\caption{}\label{tab6}
\centering
\begin{tabular}{|l|l|l|l|}
\hline
row & $G$ & $H$ & $K$\\
\hline
1 & $\M_{10}$ & $3^2{:}\Q_8$ & $\PSL_2(5)$\\
2 & $\PGL_2(9)$ & $\AGL_1(9)$ & $\PSL_2(5)$\\
3 & $\PGaL_2(9)$ & $\AGaL_1(9)$ & $\PSL_2(5)$\\
4 & $\PGaL_2(9)$ & $\AGL_1(9)$, $3^2{:}\Q_8$, $\AGaL_1(9)$ & $\PGL_2(5)$\\
\hline
5 & $\A_8$ & $15$, $3\times\D_{10}$, $\GaL_1(16)$ & $\AGL_3(2)$\\
6 & $\Sy_8$ & $\D_{30}$, $\Sy_3\times5$, $\Sy_3\times\D_{10}$, & $\AGL_3(2)$\\
 & & $3\times\AGL_1(5)$, $\Sy_3\times\AGL_1(5)$ & \\
\hline
\end{tabular}
\end{table}

\begin{proof}
Since all core-free subgroups of $\Sy_5$ is solvable, we obtain $n\geqslant6$ by our assumption that $K$ is unsolvable. If $n=6$ and $G\nleqslant\Sy_6$, then computation in \magma \cite{bosma1997magma} leads to rows 1--4 of Table \ref{tab6}. Thus we assume $G\leqslant\Sy_n$ with $n\geqslant6$ in the rest of the proof. By THEOREM~D and Remark~2 after it in \cite{liebeck1990maximal}, one of the following cases appears.
\begin{itemize}
\item[(i)] $H$ is $k$-homogeneous on $\Omega_n$ and $\A_{n-k}\leqslant K\leqslant\Sy_{n-k}\times\Sy_k$ for some $k\in\{1,2,3,4,5\}$.
\item[(ii)] $\A_{n-k}\leqslant H\leqslant\Sy_{n-k}\times\Sy_k$ and $K$ is $k$-homogeneous on $\Omega_n$ for some $k\in\{1,2,3,4,5\}$.
\item[(iii)] $n=6$, $H\cap\A_6\leqslant\Sy_3\wr\Sy_2$ and $K\cap\A_6=\PSL_2(5)$ with $H,K$ both transitive on $\Omega_6$.
\item[(iv)] $n=8$, $H\geqslant\Z_{15}$ and $K=\AGL_3(2)$.
\end{itemize}

The $k$-homogeneous but not $k$-transitive permutation groups are determined by Kantor \cite{kantor1972k}. Besides, the $2$-transitive permutation groups are well-known, see for example \cite[Tables\,7.3-7.4]{cameron1999permutation}. Indeed, Huppert \cite{huppert1957zweifach} classified the solvable $2$-transitive permutation groups much earlier, see also \cite[Theorem XII.7.3]{huppert1982finite}. From these results, we conclude that any solvable $2$-homogeneous permutation group $H\leqslant\Sy_n$ is as described in part~(b). Moreover, for $k\geqslant3$ the only solvable $k$-homogeneous subgroups of $\Sy_n$ are $\AGL_1(8)$ with $(n,k)=(8,3)$, $\AGaL_1(8)$ with $(n,k)=(8,3)$ and $\AGaL_1(32)$ with $(n,k)=(32,3)$.

First suppose (i) appears. It is just part~(a) of Proposition \ref{Alternating} if $k=1$. If $k\geqslant2$, then the conclusion in the last paragraph leads to parts~(b) and (c).

Next suppose (ii) appears. Since $n\geqslant6$ and $\A_{n-k}\leqslant H$ is solvable, $k\neq1$. Assume $k=2$. We then have $n=6$ as $\A_{n-2}\leqslant H$ is solvable. Note that the only unsolvable $2$-homogeneous subgroups of $\Sy_6$ not containing $\A_6$ are $\PSL_2(5)$ and $\PGL_2(5)$ (see for example \cite[Table 2.1]{dixon1996permutation}). This corresponds to part~(d). Similarly, we obtain part~(e) if $k=3$. If $k\geqslant4$, $\A_{n-k}\leqslant H$ solvable implies $5\leqslant n\leqslant9$, but there are no $k$-homogeneous permutation groups for such degrees by \cite{kantor1972k} and \cite[Table 7.3 and Table 7.4]{cameron1999permutation}.

Finally, for (iii) and (iv), computation in \magma \cite{bosma1997magma} leads to part~(e) and rows 5--6 of Table \ref{tab6}. Thus the proof is completed.
\end{proof}

\section{Sporadic group socles}

Now we consider the sporadic almost simple groups.

\begin{proposition}\label{Sporadic}
Let $L=\Soc(G)$ be a sporadic simple group. If $G=HK$ for solvable subgroup $H$ and core-free unsolvable subgroup $K$ of $G$, then one of the following holds.
\begin{itemize}
\item[(a)] $\M_{12}\leqslant G\leqslant\M_{12}.2$, $H$ is transitive on $[G{:}\M_{11}]$, and $K=\M_{11}$.
\item[(b)] $G=\M_{24}$, $H$ is transitive on $[\M_{24}{:}\M_{23}]$, and $K=\M_{23}$.
\item[(c)] $(G,H,K)$ lies in \emph{Table \ref{tab8}}.
\end{itemize}
Conversely, each triple $(G,H,K)$ in parts~\emph{(a)--(c)} gives a factorization $G=HK$.
\end{proposition}

\begin{table}[htbp]
\caption{}\label{tab8}
\centering
\begin{tabular}{|l|l|l|l|}
\hline
row & $G$ & $H$ & $K$\\
\hline
1 & $\M_{11}$ & $11$, $11{:}5$ & $\M_{10}$\\
2 & $\M_{11}$ & $\M_9$, $\M_9.2$, $\AGL_1(9)$ & $\PSL_2(11)$\\
\hline
3 & $\M_{12}$ & $\M_9.2$, $\M_9.\Sy_3$ & $\PSL_2(11)$\\
4 & $\M_{12}.2$ & $\M_9.2$, $\M_9.\Sy_3$ & $\PGL_2(11)$\\
\hline
5 & $\M_{22}.2$ & $11{:}2$, $11{:}10$ & $\PSL_3(4){:}2$\\
\hline
6 & $\M_{23}$ & $23$ & $\M_{22}$\\
7 & $\M_{23}$ & $23{:}11$ & $\M_{22}$, $\PSL_3(4){:}2$, $2^4{:}\A_7$\\
\hline
8 &{$\J_2.2$} & $5^2{:}4$, $5^2{:}(4\times2)$, $5^2{:}12$, & {$\G_2(2)$}\\
 & & $5^2{:}\Q_{12}$, $5^2{:}(4\times\Sy_3)$ &\\
\hline
9 & $\HS$ & $5_+^{1+2}{:}8{:}2$ & $\M_{22}$\\
10 & $\HS.2$ & $5_+^{1+2}{:}(4\wr\Sy_2)$ & $\M_{22}$, $\M_{22}.2$\\
11 & $\HS.2$ & $5^2{:}4$, $5^2{:}(4\times2)$, $5^2{:}4^2$, $5_+^{1+2}{:}4$, & $\M_{22}.2$\\
 & & $5_+^{1+2}{:}(4\times2)$, $5_+^{1+2}{:}4^2$, $5_+^{1+2}{:}8{:}2$ &\\
\hline
12 & {$\He.2$} & $7_+^{1+2}{:}6$, $7_+^{1+2}{:}(6\times2)$, $7_+^{1+2}{:}(6\times3)$, & $\Sp_4(4).4$\\
 & & $7_+^{1+2}{:}(\Sy_3\times3)$, $7_+^{1+2}{:}(\Sy_3\times6)$ &\\
\hline
13 & $\Suz.2$ & $3^5{:}12$, $3^5{:}((11{:}5)\times2)$ & $\G_2(4).2$\\
\hline
\end{tabular}
\end{table}

\begin{proof}

Suppose that $G=HK$ is a nontrivial factorization with $H$ solvable and $K$ unsolvable. If $G=L$, then from \cite[Theorem 1.1]{giudici2006factorisations} we directly read off the triples $(G,H,K)$, as stated in the proposition. Now suppose $G\neq L$. Since this indicates $\Out(L)\neq1$, we have $L\neq\M_k$ for $k\in\{11,23,24\}$.

First assume $L=(H\cap L)(K\cap L)$. Then since $H\cap L$ is solvable, we deduce from \cite[Theorem 1.1]{giudici2006factorisations} that $L=\M_{12}$ or $\HS$. Consequently, $G=\M_{12}.2$ or $\HS.2$. Searching in \magma \cite{bosma1997magma} for factorizations in these two groups leads to part~(a) or one of rows~4, 10, 11 of Table \ref{tab8}.

Next assume $L\neq(H\cap L)(K\cap L)$. Then by \cite[Theorem 1.2]{giudici2006factorisations} and computation results of \magma \cite{bosma1997magma}, $(G,H,K)$ lies in one of rows 5, 8, 10--13 of Table \ref{tab8}.
\end{proof}

\begin{remark}
\begin{itemize}
\item[(i)] For $G=\J_2.2$ in row~8 of Table \ref{tab8}, there are two non-isomorphic groups of shape $5^2{:}4$ for $H$.
\item[(ii)] For $G=\HS.2$ in row~11 of Table \ref{tab8}, there are four non-isomorphic groups of shape $5^2{:}4$, two non-isomorphic groups of shape $5^2{:}(4\times2)$ and two non-isomorphic groups of shape $5_+^{1+2}{:}4$ for $H$.
\item[(iii)] For $G=\He.2$ in row~12 of Table \ref{tab8}, three non-isomorphic groups of shape $7_+^{1+2}{:}6$ for $H$.
\end{itemize}
\end{remark}


\chapter{Examples in classical groups}


We present in this chapter examples for infinite families of factorizations appearing in Table \ref{tab7}. Let $q=p^f$ throughout this chapter, where $p$ is a prime and $f$ is a positive integer.

\section{Examples in unitary groups}

Let $V$ be a vector space over $\GF(q^2)$ of dimension $2m\geqslant4$ equipped with a non-degenerate unitary form $\beta$. There is a basis $e_1,\dots,e_m$, $f_1,\dots,f_m$ of $V$ such that
$$
\beta(e_i,e_j)=\beta(f_i,f_j)=0,\quad\beta(e_i,f_j)=\delta_{i,j}
$$
for any $i,j\in\{1,\dots,m\}$ (see \cite[2.2.3]{liebeck1990maximal}). Let
$$
G=\GU(V,\beta)=\GU_{2m}(q)
$$
be the \emph{general unitary group} of dimension $2m$ over $\GF(q^2)$, and let $A$ be the stabilizer of the totally singular subspace $\langle e_1,\dots,e_m\rangle$ in $G$, called a \emph{parabolic subgroup}. Then
$$
A=\Pa_m[G]=q^{m^2}{:}\GL_m(q^2),
$$
see \cite[3.6.2]{wilson2009finite}. Fix an element $\mu$ of $\GF(q^2)$ such that $\mu+\mu^q\neq0$. Then $e_m+\mu f_m$ is a nonisotropic vector.
Let $K$ be the stabilizer of $e_m+\mu f_m$ in $G$. Then
$$
K=\GU_{2m-1}(q)\leqslant\N_1[G].
$$
We now construct a solvable subgroup $H$ of $A$ such that $G=HK$.

\begin{construction}\label{ConstructUnitary}
Let $R=\bfO_p(A)=q^{m^2}$ and $C=\GL_m(q^2)<A$. Take $S$ to be a Singer cycle of $C$ and $H=RS$.
\end{construction}

\begin{proposition}\label{ExampleUnitary}
In the above notation, $H=q^{m^2}{:}\GL_1(q^{2m})$ is a solvable subgroup of $\Pa_m[G]$, $H\cap K=R\cap K=q^{(m-1)^2}$, and
$$
\GU_{2m}(q)=G=HK=(q^{m^2}{:}\GL_1(q^{2m}))\GU_{2m-1}(q).
$$
\end{proposition}

\begin{proof}
Define linear maps $w_k(\lambda)$ for $1\leqslant k\leqslant m$ and $\lambda\in\GF(q^2)^*$, $x_{i,j}(\lambda)$ for $i\neq j$ and $\lambda\in\GF(q^2)$, $y_{i,j}(\lambda)$ for $1\leqslant i<j\leqslant m$ and $\lambda\in\GF(q^2)$, and $z_k(\lambda)$ for $1\leqslant k\leqslant m$, $\lambda\in\GF(q^2)$ and $\lambda+\lambda^q=0$ by
\begin{eqnarray*}
&w_k(\lambda):&e_k\mapsto\lambda e_k,\quad f_k\mapsto\lambda^{-q}f_k,\\
&x_{i,j}(\lambda):&e_j\mapsto e_j+\lambda e_i,\quad f_i\mapsto f_i-\lambda^qf_j,\\
&y_{i,j}(\lambda):&f_i\mapsto f_i+\lambda e_j,\quad f_j\mapsto f_j-\lambda^qe_i,\\
&z_k(\lambda):&f_k\mapsto f_k+\lambda e_k
\end{eqnarray*}
and fixing all the other basis vectors of $V$. By \cite[3.6.2]{wilson2009finite}, $A=R{:}C$ with $C$ denoting the group generated by all $w_k(\lambda)$ and $x_{i,j}(\lambda)$ and $R$ denoting the group generated by all $y_{i,j}(\lambda)$ and $z_k(\lambda)$.

Since $S$ is a Singer cycle of $C=\GL_m(q^2)$, the group $H=R{:}S=q^{m^2}{:}\GL_1(q^{2m})$ is solvable. It is obvious that $S$ is a Hall $p'$-subgroup of $H$. Let $M$ be a Hall $p'$-subgroup of $H\cap K$. Since $M$ is a $p'$-subgroup of $H$, $M\leqslant S^h$ for some $h\in R$. We then have $M\leqslant S^h\cap K$, and so $|H\cap K|_{p'}$ divides $|S^h\cap K|$. Similarly, $|H\cap K|_p$ divides $|R\cap K|$, and thus $|H\cap K|$ divides $|R\cap K||S^h\cap K|$.

First we calculate $|R\cap K|$. For all $1\leqslant i<j\leqslant m-1$, $1\leqslant k\leqslant m-1$ and $\lambda\in\GF(q)$, it is obvious that $y_{i,j}(\lambda)$ and $z_k(\lambda)$ both fix $e_m+\mu f_m$, that is to say, $y_{i,j}(\lambda)$ and $z_k(\lambda)$ are both in $K$. These $y_{i,j}(\lambda)$ and $z_k(\lambda)$ generate an elementary abelian group of order $q^{(m-1)^2}$. Now consider an element
$$
g=y_{1,m}(\lambda_1)\dots y_{m-1,m}(\lambda_{m-1})z_m(\lambda_m)
$$
in $R$. Notice that $g$ sends $e_m+\mu f_m$ to
$$
-\sum\limits_{i=1}^{m-1}\mu\lambda_i^qe_i+(1+\mu\lambda_m)e_m+\mu f_m.
$$
Then $g\in K$ if and only if $\lambda_1=\dots=\lambda_m=0$, that is, $g=1$. Thereby we conclude $|R\cap K|=q^{(m-1)^2}$.

Next we show $|S^h\cap K|=1$. Suppose on the contrary that $|S^h\cap K|$ is divisible by a prime number $r$. Then there exists a subgroup $X$ of order $|S^h\cap K|_r$ in $S$ such that $X^h\leqslant K$. Denote by $Y$ the subgroup of $A\cap K$ generated by
$$
\{w_k(\lambda):k\leqslant m-1,\lambda\in\GF(q^2)^*\}\cup\{x_{i,j}(\lambda):i\leqslant m-1,j\leqslant m-1,\lambda\in\GF(q^2)\}.
$$
Then $Y\cong\GL_{m-1}(q^2)$ and $Y$ fixes $e_m$. Since $|A\cap K|=q^{(2m-1)(m-1)}(q^{2m-2}-1)\dots(q^2-1)$ (see \cite[3.3.3]{liebeck1990maximal}), we know that $Y$ contains a Sylow $r$-subgroup of $A\cap K$. It follows that $X^h\leqslant Y^{h_1}$ for some $h_1\in A\cap K$. Hence $X\leqslant Y^{h_1h^{-1}}$, and so $X$ fixes $e_m^{h_1h^{-1}}$. This is a contradiction since $S$ acts regularly on the nonzero vectors of $\langle e_1,\dots,e_m\rangle$.

Now that $|H\cap K|$ divides $|R\cap K||S^h\cap K|=|R\cap K|=q^{(m-1)^2}$, we conclude $G=HK$ by Lemma~\ref{p3}. We also see that $H\cap K=R\cap K$ is a elementary abelian group of order $q^{(m-1)^2}$.
\end{proof}

The lemma below excludes certain factorizations of unitary groups, which will be useful in the proof of Theorem \ref{SolvableFactor}.

\begin{lemma}\label{l4}
In the above notation, let $\tau$ be the semilinear map on $V$ such that $(\lambda v)^\tau=\lambda^{p^e}v$ for any $\lambda\in\GF(q^2)$ and $v\in V$, where $e|2f$. Let $\Si=\SU(V,\beta){:}\langle\tau\rangle$ and $N\leqslant\Si$ be the stabilizer of $e_2+\mu f_2$. If $M$ is a maximal solvable subgroup of $\Si$ stabilizing $\langle e_1,e_2\rangle$, then $\Si\neq MN$.
\end{lemma}

\begin{proof}
Suppose $\Si=MN$. Let $L=\SU(V,\beta)$, $\Ga=\GU(V,\beta){:}\langle\tau\rangle$ and $Y\leqslant\Ga$ be the stabilizer of $e_2+\mu f_2$. Evidently, $M=X\cap\Si$ with $X$ a maximal solvable subgroup of $\Ga$ stabilizing $\langle e_1,e_2\rangle$. Note that the stabilizer of $\langle e_1,e_2\rangle$ in $\Ga$ is $A{:}\langle\tau\rangle=R{:}(C{:}\langle\tau\rangle)$. Then $R\leqslant X$, and $X=RQ$ with $Q$ maximal solvable in $C{:}\langle\tau\rangle$. Since $\Si=(X\cap\Si)N$, $|X|$ is divisible by a primitive prime divisor $r$ of $p^{4f}-1$, which implies that $|Q|$ is divisible by $r$.

Take $S$ and $H$ as in Construction~\ref{ConstructUnitary}. Viewing $C{:}\langle\tau\rangle\leqslant\GL_{4f}(p)$, we conclude that $S\leqslant Q$ and $|Q|=4f|S|/e$. As a consequence, $H\geqslant X$ and $|X|=4f|H|/e$. Moreover, $\Ga=XG=XGY=X(HK)Y=(XH)(KY)=XY$. It then derives from $\Si=(X\cap\Si)N=(X\cap\Si)(Y\cap\Si)$ and Lemma \ref{p9} that
\begin{equation}\label{eq18}
|X\Si||Y\Si|\leqslant|\Ga||(X\cap Y)\Ga|.
\end{equation}
It is easily seen $HL=KL=G$. Hence $H\Si=K\Si=\Ga$, and so $X\Si=Y\Si=\Ga$. From $|X|=4f|H|/e$ we deduce that $|X\cap Y|\leqslant4f|H\cap Y|/e$ and thus $|(X\cap Y)\Si|\leqslant4f|(H\cap Y)\Si|/e$. This in conjunction with $H\cap Y=H\cap K=R\cap K\leqslant L$ yields $|(X\cap Y)\Si|\leqslant4f|\Si|/e$. Therefore, (\ref{eq18}) implies that
$$
(q+1)|\Si|=|\Ga|\leqslant|(X\cap Y)\Si|\leqslant\frac{4f|\Si|}{e},
$$
which gives $(q,e)=(2,1)$, $(3,1)$, $(4,1)$ or $(8,1)$. However, computation in \magma \cite{bosma1997magma} shows that none of these is possible.
\end{proof}

\section{Examples in symplectic groups}

Let $p=2$ and $V$ be a vector space over $\GF(q)$ of dimension $2m\geqslant4$ equipped with a non-degenerate symplectic form $\beta$. There is a basis $e_1,\dots,e_m,f_1,\dots,f_m$ of $V$ such that
$$
\beta(e_i,e_j)=\beta(f_i,f_j)=0,\quad\beta(e_i,f_j)=\delta_{i,j}
$$
for any $i,j\in\{1,\dots,m\}$. Let
$$
G=\Sp(V,\beta)=\Sp_{2m}(q)
$$
be the \emph{symplectic group} of dimension $2m$ over $\GF(q)$, and let $A$ be the stabilizer of the totally singular subspace $\langle e_1,\dots,e_m\rangle$ in $G$. Then
$$
A=\Pa_m[G]=q^{m(m+1)/2}{:}\GL_m(q),
$$
see \cite[3.5.4]{wilson2009finite}. Let $Q$ be a quadratic form of minus type on $V$ associated with the form $\beta$. Fix an element $\sigma$ with $x^2+x+\sigma$ an irreducible quadratic over $\GF(q)$ such that
\[\begin{array}{l}
Q(e_1)=\dots=Q(e_{m-1})=Q(f_1)=\dots=Q(f_{m-1})=0, \\
Q(e_m)=1,\ Q(f_m)=\sigma
\end{array}\]
(see \cite[2.2.3]{liebeck1990maximal}). Take
$$
K=\GO^-(V,Q)=\GO^-_{2m}(q),
$$
the \emph{general orthogonal group of minus type} of dimension $2m$ over $\GF(q)$.


We construct a solvable subgroup $H$ of $A$ such that $G=HK$.

\begin{construction}\label{ConstructSymplectic}
Let $R=\bfO_2(A)=q^{m(m+1)/2}$ and $C=\GL_m(q)<A$. Take $S$ to be a Singer cycle of $C$ and $H=RS$.
\end{construction}

\begin{proposition}\label{ExampleSymplectic}
In the above notation, $H=q^{m(m+1)/2}{:}\GL_1(q^m)$ is a solvable subgroup of $\Pa_m[G]$, and
$$
\Sp_{2m}(q)=G=HK=(q^{m(m+1)/2}{:}\GL_1(q^m))\GO^-_{2m}(q).
$$
\end{proposition}

\begin{proof}
Define linear maps $w_k(\lambda)$ for $1\leqslant k\leqslant m$ and $\lambda\in\GF(q)^*$, $x_{i,j}(\lambda)$ for $i\neq j$ and $\lambda\in\GF(q)$, $y_{i,j}(\lambda)$ for $1\leqslant i<j\leqslant m$ and $\lambda\in\GF(q)$, and $z_k(\lambda)$ for $1\leqslant k\leqslant m$ and $\lambda\in\GF(q)$ by
\begin{eqnarray*}
&w_k(\lambda):&e_k\mapsto\lambda e_k,\quad f_k\mapsto\lambda^{-1}f_k,\\
&x_{i,j}(\lambda):&e_j\mapsto e_j+\lambda e_i,\quad f_i\mapsto f_i-\lambda f_j,\\
&y_{i,j}(\lambda):&f_i\mapsto f_i+\lambda e_j,\quad f_j\mapsto f_j+\lambda e_i,\\
&z_k(\lambda):&f_k\mapsto f_k+\lambda e_k
\end{eqnarray*}
and fixing all the other basis vectors of $V$. By \cite[3.5.4]{wilson2009finite}, $A=R{:}C$ with $C$ denoting the group generated by all $w_k(\lambda)$ and $x_{i,j}(\lambda)$ and $R$ denoting the group generated by all $y_{i,j}(\lambda)$ and $z_k(\lambda)$.

Since $S$ is a Singer cycle of $C=\GL_m(q)$, $H=R{:}S=q^{m(m+1)/2}{:}\GL_1(q^m)$ is solvable. Clearly, $S$ is a Hall $2'$-subgroup of $H$. Let $M$ be a Hall $2'$-subgroup of $H\cap K$. Since $M$ is a $2'$-subgroup of $H$, $M\leqslant S^h$ for some $h\in R$. We then have $M\leqslant S^h\cap K$, and so $|H\cap K|_{2'}$ divides $|S^h\cap K|$. Similarly, $|H\cap K|_2$ divides $|R\cap K|$, and thus $|H\cap K|$ divides $|R\cap K||S^h\cap K|$.

First we calculate $|R\cap K|$. For all $1\leqslant i<j\leqslant m-1$ and $\lambda\in\GF(q)$, it is obvious that $y_{i,j}(\lambda)$ fixes $Q(e_k)$ and $Q(f_k)$ for $k=1,\dots,m$ and thus $y_{i,j}(\lambda)\in K$. These $y_{i,j}(\lambda)$ generate an elementary abelian group of order $q^{(m-1)(m-2)/2}$. Now consider an element
$$
g=y_{1,m}(\lambda_1)\dots y_{m-1,m}(\lambda_{m-1})z_1(\mu_1)\dots z_m(\mu_m)
$$
in $R$. Notice that $g$ fixes $e_1,\dots,e_m$ and
\begin{eqnarray*}
&g:&f_i\mapsto f_i+\lambda_ie_m+\mu_ie_i,\quad i=1,\dots,m-1,\\
&&f_m\mapsto f_m+\sum\limits_{i=1}^{m-1}\lambda_ie_i+\mu_me_m.
\end{eqnarray*}
Then $g\in K$ if and only if
$$
Q(f_i+\lambda_ie_m+\mu_ie_i)=Q(f_i),\quad i=1,\dots,m-1
$$
and
$$
Q(f_m+\sum\limits_{i=1}^{m-1}\lambda_ie_i+\mu_me_m)=Q(f_m),
$$
which is equivalent to $\mu_i=-\lambda_i^2$ for $i=1,\dots,m-1$ and $\mu_m=0$ or $1$. Hence all the elements of shape $g$ in $K$ generate an elementary abelian group of order $2q^{m-1}$. Therefore, we conclude $|R\cap K|=2q^{m-1}\cdot q^{(m-1)(m-2)/2}=2q^{m(m-1)/2}$.

Next we show $|S^h\cap K|=1$. Suppose on the contrary that $|S^h\cap K|$ is divisible by an odd prime $r$. Then there exists a subgroup $X$ of order $r$ in $S$ such that $X^h\leqslant K$. Denote by $Y$ the subgroup of $A\cap K$ generated by
$$
\{w_k(\lambda):k\leqslant m-1,\lambda\in\GF(q)^*\}\cup\{x_{i,j}(\lambda):i\leqslant m-1,j\leqslant m-1,\lambda\in\GF(q)\}.
$$
Then $Y\cong\GL(m-1,q)$ and $Y$ fixes $e_m$. Since $|A\cap K|=2q^{m(m-1)}(q^{m-1}-1)\dots(q-1)$ (see \cite[3.2.4(a)]{liebeck1990maximal}), we know that $Y$ contains a Sylow $r$-subgroup of $A\cap K$. It follows that $X^h\leqslant Y^{h_1}$ for some $h_1\in A\cap K$. Hence $X\leqslant Y^{h_1h^{-1}}$, and so $X$ fixes $e_m^{h_1h^{-1}}$. This is a contradiction since $S$ acts regularly on the nonzero vectors of $\langle e_1,\dots,e_m\rangle$.

Now that $|H\cap K|$ divides $|R\cap K||S^h\cap K|=2q^{m(m-1)/2}$, we conclude $G=HK$ by Lemma~\ref{p3}.
\end{proof}

\section{Examples in orthogonal groups of odd dimension}

Let $p>2$ and $V$ be a vector space over $\GF(q)$ of dimension $2m+1\geqslant5$ equipped with a non-degenerate quadratic form $Q$ and the associated bilinear form $\beta$. There is a basis $e_1,\dots,e_m$, $f_1,\dots,f_m$, $d$ of $V$ such that
\[\begin{array}{l}
\beta(e_i,e_j)=\beta(f_i,f_j)=\beta(e_i,d)=\beta(f_i,d)=0,\quad\beta(e_i,f_j)=\delta_{i,j},\\
Q(e_i)=0,\ Q(f_j)=0,\ Q(d)=1
\end{array}\]
for any $i,j\in\{1,\dots,m\}$ (see \cite[2.2.3]{liebeck1990maximal}). Let
$$
G=\SO(V,Q)=\SO_{2m+1}(q)
$$
be the \emph{special orthogonal group} of dimension $2m+1$ over $\GF(q)$, and let $A=\Pa_m[G]$ be the stabilizer of the totally singular subspace $\langle e_1,\dots,e_m\rangle$ in $G$. Then
$$
A=\Pa_m[G]=(q^{m(m-1)/2}.q^m){:}\GL_m(q),
$$
where $q^{m(m-1)/2}.q^m$ is a special $p$-group with center of order $q^{m(m-1)/2}$, see \cite[3.7.4]{wilson2009finite}. Fix a non-square element $\mu$ in $\GF(q)$, and let $K$ be the stabilizer of the vector $e_m+\mu f_m$ in $G$. Then
$$
K=\SO^-_{2m}(q)\leqslant\N_1^-[G].
$$
We show that $A$ has a solvable subgroup $H$ such that $G=HK$.

\begin{construction}\label{ConstructOrthogonal1}
Let $R=\bfO_p(A)=q^{m(m-1)/2}.q^m$ and $C=\GL_m(q)<A$. Take $S$ to be a Singer cycle of $C$ and $H=RS$.
\end{construction}

\begin{proposition}\label{ExampleOrthogonal1}
In the above notation, $H=(q^{m(m-1)/2}.q^m){:}\GL_1(q^m)$ is a solvable subgroup of $\Pa_m[G]$, and
$$
\SO_{2m+1}(q)=G=HK=((q^{m(m-1)/2}.q^m){:}\GL_1(q^m))\SO^-_{2m}(q).
$$
\end{proposition}

\begin{proof}
Define linear maps $w_k(\lambda)$ for $1\leqslant k\leqslant m$ and $\lambda\in\GF(q)^*$ and $x_{i,j}(\lambda)$ for $i\neq j$ and $\lambda\in\GF(q)$ by
\begin{eqnarray*}
&w_k(\lambda):&e_k\mapsto\lambda e_k,\quad f_k\mapsto\lambda^{-1}f_k,\\
&x_{i,j}(\lambda):&e_j\mapsto e_j+\lambda e_i,\quad f_i\mapsto f_i-\lambda f_j
\end{eqnarray*}
and fixing all the other basis vectors of $V$. By \cite[3.7.4]{wilson2009finite}, $A=R{:}C$ with $C$ denoting the group generated by all $w_k(\lambda)$ and $x_{i,j}(\lambda)$ and $R$ denoting the kernel of the action of $A$ on $\langle e_1,\dots,e_m\rangle$.

Since $S$ is a Singer cycle of $C=\GL_m(q)$, $H=R{:}S=(q^{m(m-1)/2}.q^m){:}\GL_1(q^m)$ is solvable. It is obvious that $S$ is a Hall $p'$-subgroup of $H$. Let $M$ be a Hall $p'$-subgroup of $H\cap K$. Since $M$ is a $p'$-subgroup of $H$, $M\leqslant S^h$ for some $h\in R$. We then have $M\leqslant S^h\cap K$, and so $|H\cap K|_{p'}$ divides $|S^h\cap K|$. Similarly, $|H\cap K|_p$ divides $|R\cap K|$, and thus $|H\cap K|$ divides $|R\cap K||S^h\cap K|$.

First we calculate $|R\cap K|$. For any $g\in R\cap K$, since $g$ fixes both $e_m$ and $e_m+\mu f_m$, $g$ fixes $f_m$ as well, and so $g$ fixes
$$
W:=\langle e_m,f_m\rangle^\bot=\langle e_1,\dots,e_{m-1},f_1,\dots,f_{m-1},d\rangle.
$$
Thus we conclude that $R\cap K$ equals the pointwise stabilizer of $\langle e_1,\dots,e_{m-1}\rangle$ in $\SO(W,\beta)$. According to the previous paragraph, this is a special group of order $q^{m(m-1)/2}$, whence $|R\cap K|=q^{m(m-1)/2}$.

Next we show that $|S^h\cap K|=1$. Suppose on the contrary that $|S^h\cap K|$ is divisible by a prime number $r$. Then there exists a subgroup $X$ of order $|S^h\cap K|_r$ in $S$ such that $X^h\leqslant K$. Denote by $Y$ the subgroup of $A\cap K$ generated by
$$
\{w_k(\lambda):k\leqslant m-1,\lambda\in\GF(q)^*\}\cup\{x_{i,j}(\lambda):i\leqslant m-1,j\leqslant m-1,\lambda\in\GF(q)\}.
$$
Then $Y\cong\GL(m-1,q)$ and $Y$ fixes $e_m$. Since $|A\cap K|=q^{m(m-1)}(q^{m-1}-1)\dots(q-1)$ (see \cite[3.4.1]{liebeck1990maximal}), we know that $Y$ contains a Sylow $r$-subgroup of $A\cap K$. It follows that $X^h\leqslant Y^{h_1}$ for some $h_1\in A\cap K$. Hence $X\leqslant Y^{h_1h^{-1}}$, and so $X$ fixes $e_m^{h_1h^{-1}}$. This is a contradiction since $S$ acts regularly on the nonzero vectors of $\langle e_1,\dots,e_m\rangle$.

Now that $|H\cap K|$ divides $|R\cap K||S^h\cap K|=q^{m(m-1)/2}$, we conclude $G=HK$ by Lemma~\ref{p3}.
\end{proof}

\section{Examples in orthogonal groups of plus type}

Let $V$ be a vector space over $\GF(q)$ of dimension $2m\geqslant6$ equipped with a non-degenerate quadratic form $Q$ and the associated bilinear form $\beta$ such that the Witt index of $Q$ equals $m$. There is a basis $e_1,\dots,e_m$, $f_1,\dots,f_m$ of $V$ such that
$$
Q(e_i)=Q(f_i)=\beta(e_i,e_j)=\beta(f_i,f_j)=0,\quad\beta(e_i,f_j)=\delta_{i,j}
$$
for any $i,j\in\{1,\dots,m\}$ (see \cite[2.2.3]{liebeck1990maximal}). Let
$$
G=\SO(V,Q)=\SO^+_{2m}(q)
$$
be the \emph{special orthogonal group of plus type} of dimension $2m$ over $\GF(q)$, and let $A$ be the stabilizer of the totally singular subspace $\langle e_1,\dots,e_m\rangle$ in $G$. Then
$$
A=\Pa_m[G]=q^{m(m-1)/2}{:}\GL_m(q),
$$
see \cite[3.7.4]{wilson2009finite}. Let $K$ be the stabilizer of the vector $e_m+f_m$ in $G$. Then
$$
K=\SO_{2m-1}(q)\leqslant\N_1[G].
$$
We show in the following that $A$ has a solvable subgroup $H$ such that $G=HK$.

\begin{construction}\label{ConstructOrthogonal2}
Let $R=\bfO_p(A)=q^{m(m-1)/2}$ and $C=\GL_m(q)<A$. Take $S$ to be a Singer cycle of $C$ and $H=RS$.
\end{construction}

\begin{proposition}\label{ExampleOrthogonal2}
In the above notation, $H=q^{m(m-1)/2}{:}\GL_1(q^m)$ is a solvable subgroup of $\Pa_m[G]$, and
$$
\SO^+_{2m}(q)=G=HK=(q^{m(m-1)/2}{:}\GL_1(q^m))\SO_{2m-1}(q).
$$
\end{proposition}

\begin{proof}
Define linear maps $w_k(\lambda)$ for $1\leqslant k\leqslant m$ and $\lambda\in\GF(q)^*$, $x_{i,j}(\lambda)$ for $i\neq j$ and $\lambda\in\GF(q)$, and $y_{i,j}(\lambda)$ for $1\leqslant i<j\leqslant m$ and $\lambda\in\GF(q)$ by
\begin{eqnarray*}
&w_k(\lambda):&e_k\mapsto\lambda e_k,\quad f_k\mapsto\lambda^{-1}f_k,\\
&x_{i,j}(\lambda):&e_j\mapsto e_j+\lambda e_i,\quad f_i\mapsto f_i-\lambda f_j,\\
&y_{i,j}(\lambda):&f_i\mapsto f_i+\lambda e_j,\quad f_j\mapsto f_j-\lambda e_i
\end{eqnarray*}
and fixing all the other basis vectors of $V$. By \cite[3.7.4]{wilson2009finite}, $A=R{:}C$ with $C$ denoting the group generated by all $w_k(\lambda)$ and $x_{i,j}(\lambda)$ and $R$ denoting the group generated by all $y_{i,j}(\lambda)$.

Since $S$ is a Singer cycle of $C=\GL_m(q)$, $H=R{:}S=q^{m(m-1)/2}{:}\GL_1(q^m)$ is solvable. It is obvious that $S$ is a Hall $p'$-subgroup of $H$. Let $M$ be a Hall $p'$-subgroup of $H\cap K$. Since $M$ is a $p'$-subgroup of $H$, $M\leqslant S^h$ for some $h\in R$. We then have $M\leqslant S^h\cap K$, and so $|H\cap K|_{p'}$ divides $|S^h\cap K|$. Similarly, $|H\cap K|_p$ divides $|R\cap K|$, and thus $|H\cap K|$ divides $|R\cap K||S^h\cap K|$.

First we calculate $|R\cap K|$. For all $1\leqslant i<j\leqslant m-1$ and $\lambda\in\GF(q)$, it is obvious $y_{i,j}(\lambda)\in K$. These $y_{i,j}(\lambda)$ generate an elementary abelian group of order $q^{(m-1)(m-2)/2}$. Now consider an element
$$
g=y_{1,m}(\lambda_1)\dots y_{m-1,m}(\lambda_{m-1})
$$
in $R$. Notice that $g$ sends $e_m+f_m$ to
$$
-\sum\limits_{i=1}^{m-1}\lambda_ie_i+e_m+f_m.
$$
Then $g\in K$ if and only if $\lambda_1=\dots=\lambda_{m-1}=0$, which is equivalent to $g=1$. Therefore, we conclude $|R\cap K|=q^{(m-1)(m-2)/2}$.

Next we show $|S^h\cap K|=1$. Suppose on the contrary that $|S^h\cap K|$ is divisible by a prime number $r$. Then there exists a subgroup $X$ of order $|S^h\cap K|_r$ in $S$ such that $X^h\leqslant K$. Denote by $Y$ the subgroup of $A\cap K$ generated by
$$
\{w_k(\lambda):k\leqslant m-1,\lambda\in\GF(q)^*\}\cup\{x_{i,j}(\lambda):i\leqslant m-1,j\leqslant m-1,\lambda\in\GF(q)\}.
$$
Then $Y\cong\GL(m-1,q)$ and $Y$ fixes $e_m$. Since $|A\cap K|=q^{(m-1)^2}(q^{m-1}-1)\dots(q-1)$ (see \cite[3.6.1(a)]{liebeck1990maximal}), we know that $Y$ contains a Sylow $r$-subgroup of $A\cap K$. It follows that $X^h\leqslant Y^{h_1}$ for some $h_1\in A\cap K$. Hence $X\leqslant Y^{h_1h^{-1}}$, and so $X$ fixes $e_m^{h_1h^{-1}}$. This is a contradiction since $S$ acts regularly on the nonzero vectors of $\langle e_1,\dots,e_m\rangle$.

Now that $|H\cap K|$ divides $|R\cap K||S^h\cap K|=q^{(m-1)(m-2)/2}$, we conclude $G=HK$ by Lemma~\ref{p3}.
\end{proof}







\chapter{Reduction for classical groups}


Towards the proof of Theorem \ref{SolvableFactor}, we set up the strategy for the classical group case when precisely one factor is solvable, and establish some preconditioning results.

Throughout this chapter, let $G$ be an almost simple group with socle $L$ classical of Lie type, and $G=HK$ be a nontrivial factorization of $G$ with $H$ solvable. By the results in Chapter \ref{Non-classical}, we assume that $L$ is not isomorphic to $\A_5$, $\A_6$ or $\A_8$. To show that part~(d) of Theorem \ref{SolvableFactor} holds, we may further assume that the factorization $G=HK$ can be embedded into a nontrivial maximal factorization $G=AB$ by virtue of Lemma \ref{Embedding}. The candidates for the factorization $G=AB$ are listed in Tables \ref{tabLinear}--\ref{tabOmegaPlus2}.

\section{The case that $A$ is solvable}

First of all, we determine the case when $A$ is solvable. An inspection of Tables \ref{tabLinear}--\ref{tabOmegaPlus2} shows that this occurs for three classes of socle $L$:
\begin{itemize}
\item[(i)] $L=\PSL_n(q)$ with $n$ prime,
\item[(ii)] $L=\PSp_4(3)\cong\PSU_4(2)$,
\item[(iii)] $L=\PSU_3(3)$, $\PSU_3(5)$ or $\PSU_3(8)$.
\end{itemize}

The candidates in (i) and (iii) are treated respectively in Theorem~\ref{p-dimensionLinear} and Theorem~\ref{p-dimensionUnitary}. Thus we only need to deal with (ii). For an almost simple group $G$ with socle $\PSp_4(3)\cong\PSU_4(2)$, all the nontrivial factorizations of $G$ can be produced instantly by \magma \cite{bosma1997magma}. We state the computation result here.

\begin{proposition}\label{Small1}
Let $G$ be an almost simple group with socle $L=\PSp_4(3)\cong\PSU_4(2)$. Then the following four cases give all the nontrivial factorizations $G=HK$ with $H$ solvable.
\begin{itemize}
\item[(a)] Both $H$ and $K$ are solvable, and $(G,H,K)$ lies in rows~\emph{9--11} of \emph{Table \ref{tab4}}.
\item[(b)] $G=\PSp_4(3)$, $H\leqslant\Pa_k[\PSp_4(3)]$, and $(G,H,K,k)$ lies in rows~\emph{1--4} of \emph{Table \ref{tab9}}; moreover, $\Pa_k[\PSp_4(3)]$ is the only maximal subgroup of $G$ containing $H$.
\item[(c)] $G=\PSp_4(3).2$, and $L=(H\cap L)(K\cap L)$; in particular, $H=(H\cap L).\calO_1$ and $K=(K\cap L).\calO_2$, where $\calO_i\leqslant\Z_2$ for $i\in\{1,2\}$ and $\calO_1\calO_2=\Z_2$.
\item[(d)] $G=\PSp_4(3).2$, $L\neq(H\cap L)(K\cap L)$, $H\leqslant\Pa_k[\PSp_4(3).2]$, and $(G,H,K,k)$ lies in rows~\emph{5--6} of \emph{Table \ref{tab9}}; moreover, $\Pa_k[\PSp_4(3).2]$ is the only maximal subgroup of $G$ containing $H$.
\end{itemize}
\end{proposition}

\begin{table}[htbp]
\caption{}\label{tab9}
\centering
\begin{tabular}{|l|l|l|l|l|}
\hline
row & $G$ & $H$ & $K$ & $k$\\
\hline
1 & $\PSp_4(3)$ & $3_-^{1+2}$, $3_+^{1+2}$, $3_+^{1+2}{:}2$, $[3^4]$, $[3^4]{:}2$ & $2^4{:}\A_5$ & $1$, $2$\\
2 & $\PSp_4(3)$ & $3_+^{1+2}{:}4$ & $2^4{:}\A_5$ & $1$\\
3 & $\PSp_4(3)$ & $3_+^{1+2}{:}\Q_8$, $3_+^{1+2}{:}2.\A_4$ & $2^4{:}\A_5$, $\Sy_5$, $\A_6$, $\Sy_6$ & $1$\\
4 & $\PSp_4(3)$ & $3^3{:}\A_4$, $3^3{:}\Sy_4$ & $2^4{:}\A_5$ & $2$\\
\hline
5 & $\PSp_4(3).2$ & $3_+^{1+2}{:}8$ & $\Sy_5\times2$, $\A_6\times2$, $\Sy_6$, $\Sy_6\times2$ & $1$\\
6 & $\PSp_4(3).2$ & $3_+^{1+2}{:}\GaL_1(9)$, $3_+^{1+2}{:}2.\Sy_4$ & $\Sy_5$ & $1$\\
\hline
\end{tabular}
\end{table}

To sum up, we have the following proposition.

\begin{proposition}\label{p13}
Let $G$ be an almost simple group with socle $L$ classical of Lie type not isomorphic to $\A_5$, $\A_6$ or $\A_8$. If $G=HK$ is a nontrivial factorization such that $H$ is contained in some solvable maximal subgroup of $G$, then one of the following holds as in \emph{Theorem~\ref{SolvableFactor}}.
\begin{itemize}
\item[(a)] $L=\PSL_n(q)$ with $n$ prime, and $(G,H,K)$ lies in row~\emph{1} of \emph{Table~\ref{tab7}}, or rows~\emph{1--7} of \emph{Table~\ref{tab4}}, or rows~\emph{1--5}, \emph{9} of \emph{Table~\ref{tab1}}.
\item[(b)] $L=\PSp_4(3)\cong\PSU_4(2)$, and $(G,H,K)$ lies in rows~\emph{9--11} of \emph{Table~\ref{tab4}} or rows~\emph{10--11} of \emph{Table~\ref{tab1}}.
\item[(c)] $L=\PSU_3(q)$ with $q\in\{3,5,8\}$, and $(G,H,K)$ lies in row~\emph{8} of \emph{Table~\ref{tab4}} or rows~\emph{18--19} of \emph{Table~\ref{tab1}}.
\end{itemize}
\end{proposition}

\section{Inductive hypothesis}

The lemma below enables us to prove Theorem~\ref{SolvableFactor} by induction.

\begin{lemma}\label{l2}
Let $G=HB$ be a factorization and $H\leqslant A\leqslant G$. If $N$ is a normal subgroup of $A$ such that $A/N$ is almost simple, then either $\Soc(A/N)\trianglelefteq(A\cap B)N/N$ or $HN/N$ is a nontrivial factor of $A/N$.
\end{lemma}

\begin{proof}
Since $G=HB$ and $H\leqslant A$, we have $A=H(A\cap B)$ by Lemma \ref{p9}. Let $\overline{A}=A/N$, $\overline{H}=HN/N$ and $\overline{A\cap B}=(A\cap B)N/N$. Then it holds $\overline{A}=\overline{H}\,\overline{A\cap B}$, and the lemma follows.
\end{proof}

Combining Lemma \ref{l2} and Proposition \ref{Intersection}, we obtain the following lemma.

\begin{lemma}\label{l3}
Let $G$ be an almost simple classical group of Lie type, $G=HK$ is a nontrivial factorizations of $G$ with $H$ solvable and $HL=KL=G$. If $A$ is a maximal subgroup of $G$ containing $H$ such that $A$ has precisely one unsolvable composition factor, then $A/\Rad(A)$ is an almost simple group with a nontrivial solvable factor $H\Rad(A)/\Rad(A)$.
\end{lemma}

\begin{proof}
Take $B$ to be a maximal subgroup of $G$ containing $K$. By Lemma \ref{Embedding}, $B$ is core-free. Since $H\leqslant A$, we then obtain a nontrivial maximal factorization $G=AB$. Write $R=\Rad(A)$. Then $A/R$ is almost simple, and Proposition \ref{Intersection} asserts that $\Soc(A/R)\nleqslant(A\cap B)R/R$. Applying Lemma \ref{l2} with $N=R$ there, we know that $HR/R$ is a nontrivial factor of $A/R$. Besides, $HR/R$ is obviously solvable as $H$ is solvable. This completes the proof.
\end{proof}

Let $G$ be an almost simple group. We will use induction on the order of $G$ to finish the proof of Theorem \ref{SolvableFactor}. The base case for induction is settled by Propositions \ref{BothSolvable}--\ref{Sporadic} and \ref{p13}. Now we make the inductive hypothesis:

\begin{hypothesis}\label{Hypo}
For any almost simple group $G_1$ with order properly dividing $|G|$, if $G_1=H_1K_1$ for solvable subgroup $H_1$ and core-free subgroup $K_1$ of $G_1$, then the triple $(G_1,H_1,K_1)$ satisfies Theorem~\ref{SolvableFactor}.
\end{hypothesis}

For our convenience when applying the inductive hypothesis in subsequent chapters, we draw together some essential information in the next lemma under Hypothesis \ref{Hypo}.

\begin{lemma}\label{Reduction}
Suppose $G$ is an almost simple group satisfying Hypothesis~$\ref{Hypo}$ and $G=HK$ is a nontrivial factorization with $H$ solvable and $HL=KL=G$. If $A$ is a maximal subgroup of $G$ containing $H$ such that $A$ has the unique unsolvable composition factor $S$, then writing $R=\Rad(A)$ and $A/R=S.\calO$, we have the following statements.
\begin{itemize}
\item[(a)] $S$ is not an exceptional group of Lie type.
\item[(b)] $S\neq\POm_{2\ell}^-(q) $ for $\ell\geqslant4$.
\item[(c)] If $S=\PSL_\ell(q)$, then $H$ is as described in one of the following rows:
\[
\begin{array}{|l|l|l|}
\hline
\emph{row} & H\leqslant & \emph{remark}\\
\hline
1 & R.(((q^\ell-1)/((q-1)d)).\ell).\calO & d=(\ell,q-1) \\
2 & R.(q{:}((q-1)/d)).\calO & \ell=2,\ d=(2,q-1) \\
3 & R.(q^3{:}((q^3-1)/d).3).\calO & \ell=4,\ d=(4,q-1) \\
4 & R.(\Sy_2\times\Sy_3) & \ell=2,\ q=4 \\
5 & R.\Sy_4 & \ell=2,\ q\in\{5,7,11,23\} \\
6 & R.(\Sy_2\times\Sy_4) & \ell=2,\ q=9 \\
7 & R.(3^2{:}2.\Sy_4) & \ell=3,\ q=3 \\
8 & R.(2^4.(3\times\D_{10})).2 & \ell=3,\ q=4 \\
9 & R.(2^6{:}(56{:}7){:}6) & \ell=3,\ q=8 \\
10 & R.(\Sy_2\wr\Sy_4) & \ell=4,\ q=2 \\
11 & R.(\Sy_4\wr\Sy_2) & \ell=4,\ q=2 \\
12 & R.(2^4{:}5{:}4).\calO & \ell=4,\ q=3 \\
\hline
\end{array}
\]
\item[(d)] If $S=\PSp_{2\ell}(q)$ with $\ell\geqslant2$, then $H$ is as described in one of the following rows:
\[
\begin{array}{|l|l|l|}
\hline
\emph{row} & H\leqslant & \emph{remark}\\
\hline
1 & R.(q^{\ell(\ell+1)/2}{:}(q^\ell-1).\ell).\calO & q\ \emph{even} \\
2 & R.(q^{1+2}{:}((q^2-1)/2).2).\calO & \ell=2,\ q\ \emph{odd} \\
3 & R.(2^4{:}\AGL_1(5)) & \ell=2,\ q=3 \\
4 & R.(3^3{:}(\Sy_4\times 2)) & \ell=2,\ q=3 \\
5 & R.(3_+^{1+2}{:}2.\Sy_4) & \ell=2,\ q=3 \\
6 & R.(q^3{:}(q-1).\A_4).2 & \ell=2,\ q\in\{5,11\} \\
7 & R.(q^3{:}(q-1).\Sy_4).2 & \ell=2,\ q\in\{7,23\} \\
8 & R.(3_+^{1+2}{:}2.\Sy_4) & \ell=3,\ q=2 \\
9 & R.(3_+^{1+4}{:}2^{1+4}.\D_{10}).2 & \ell=3,\ q=3 \\
\hline
\end{array}
\]
\item[(e)] If $S=\PSU_{2\ell+1}(q)$ with $\ell\geqslant1$, then $\ell=1$ and $H$ is as described in one of the following rows:
\[
\begin{array}{|l|l|l|}
\hline
\emph{row} & H\leqslant & \emph{remark}\\
\hline
1 & R.(3_+^{1+2}{:}8{:}2) & q=3 \\
2 & R.(5_+^{1+2}{:}24{:}2) & q=5 \\
3 & R.(57{:}9.2) & q=8 \\
4 & R.(2^{3+6}{:}63{:}3).2 & q=8 \\
\hline
\end{array}
\]
\item[(f)] If $S=\PSU_{2\ell}(q)$ with $\ell\geqslant2$, then $H$ is as described in one of the following rows:
\[
\begin{array}{|l|l|l|}
\hline
\emph{row} & H\leqslant & \emph{remark}\\
\hline
1 & R.(q^{\ell^2}{:}((q^{2\ell}-1)/((q+1)d)).\ell).\calO & d=(2\ell,q+1) \\
2 & R.(3^3{:}(\Sy_4\times 2)) & \ell=2,\ q=2 \\
3 & R.(3_+^{1+2}{:}2.\Sy_4) & \ell=2,\ q=2 \\
4 & R.(3^4{:}\Sy_4).\calO & \ell=2,\ q=3 \\
5 & R.(3^4{:}3^2{:}4).\calO & \ell=2,\ q=3 \\
6 & R.(3_+^{1+4}.2.\Sy_4).\calO & \ell=2,\ q=3 \\
7 & R.(513{:}3).\calO & \ell=2,\ q=8 \\
\hline
\end{array}
\]
\item[(g)] If $S=\Omega_{2\ell+1}(q)$ with $\ell\geqslant3$ and $q$ odd, then $H$ is as described in one of the following rows:
\[
\begin{array}{|l|l|l|}
\hline
\emph{row} & H\leqslant & \emph{remark}\\
\hline
1 & R.((q^{\ell(\ell-1)/2}.q^\ell){:}((q^\ell-1)/2).\ell).\calO & \\
2 & R.(3^5{:}2^4.\AGL_1(5)).2 & \ell=3,\ q=3 \\
3 & R.(3^{6+4}{:}2^{1+4}.\AGL_1(5)).2 & \ell=4,\ q=3 \\
\hline
\end{array}
\]
\item[(h)] If $S=\POm_{2\ell}^+(q)$ with $\ell\geqslant4$, then $H$ is as described in one of the following rows:
\[
\begin{array}{|l|l|l|}
\hline
\emph{row} & H\leqslant & \emph{remark}\\
\hline
1 & R.(q^{\ell(\ell-1)/2}{:}((q^\ell-1)/d).\ell).\calO & d=(4,q^\ell-1) \\
2 & R.(3^6{:}2^4.\AGL_1(5)).\calO & \ell=4,\ q=3 \\
3 & R.([3^9].13.3).\calO & \ell=4,\ q=3 \\
\hline
\end{array}
\]
\end{itemize}
\end{lemma}
\begin{proof}
Let $\overline{A}=A/R$ and $\overline{H}=HR/R$. Then $\overline{A}$ is an almost simple group with socle $S$, $\calO=\overline{A}/S\leqslant\Out(S)$, and
$$
H\leqslant HR=R.\overline{H}\leqslant R.(\overline{H}\cap S).\calO.
$$
By Lemma~\ref{l3}, $\overline{H}$ is a nontrivial solvable factor of $\overline{A}$. By Hypothesis~\ref{Hypo}, the pair $(\overline{A},\overline{H})$ satisfies Theorem~\ref{SolvableFactor}. Thus the candidates for $(\overline{A},\overline{H})$ are from Propositions \ref{BothSolvable}, \ref{Alternating}, \ref{Sporadic} or Tables \ref{tab7}--\ref{tab1}. In particular, statements (a) and (b) of the lemma are both true. Next we prove statements (c)--(h) case by case.

{\bf Case 1.} Let $S=\PSL_\ell(q)$. If the solvable factor $\overline{H}$ of $\overline{A}$ is from Proposition \ref{BothSolvable}, then in view of $\PSL_2(4)\cong\PSL_2(5)\cong\A_5$, we see that one of rows 1--2, 4--5 and 7--9 appears.

Next, consider the solvable factors $\overline{H}$ from Proposition \ref{Alternating}, which occurs only for $(\ell,q)=(2,9)$ or $(4,2)$. For $(\ell,q)=(2,9)$, $S=\A_6$ and one of rows 1--2 and 6 appears. For $(\ell,q)=(4,2)$, $S=\A_8$ and one of rows 1, 3 and 10--11 appears.

Finally, the candidates from Tables \ref{tab7}--\ref{tab1} are just rows~1, 3 and 12 in statement (c). This completes the verification of statement (c).

{\bf Case 2.} Let $S=\PSp_{2\ell}(q)$ with $\ell\geqslant2$. Then $(\overline{A},\overline{H})$ either satisfies Proposition \ref{BothSolvable} or lies in Tables \ref{tab7}--\ref{tab1}. For the former, $(\ell,q)=(2,3)$ and one of rows 3--5 in statement (d) appears. For the latter, we have rows 1--2 and 4--9 in statement (d).

{\bf Case 3.} Assume that $S=\PSU_{2\ell+1}(q)$. Then $\ell=1$, and $(\overline{A},\overline{H})$ lies in rows 18--19 of Table~\ref{tab1} or row 8 of Table \ref{tab4}. This leads to the four rows in statement (e).

{\bf Case 4.} Let $S=\PSU_{2\ell}(q)$ with $\ell\geqslant2$. Then $(\overline{A},\overline{H})$ either satisfies Proposition \ref{BothSolvable} or is from Tables \ref{tab7}--\ref{tab1}. For the former, $(\ell,q)=(2,2)$ and one of rows 1--3 in statement (f) appears. For the latter, we have rows 1 and 4--7 in statement (f).

{\bf Case 5.} Let $S=\Omega_{2\ell+1}(q)$ with $\ell\geqslant3$. In this case, $(\overline{A},\overline{H})$ only arises in Tables \ref{tab7}--\ref{tab1}, from which we read off the three rows in statement (g).

{\bf Case 6.} Now let $S=\Omega^+_{2\ell}(q)$ with $\ell\geqslant4$. Then the only candidates in Tables \ref{tab7}--\ref{tab1} lead to the three rows in statement (h).
\end{proof}

\section{The case that $A$ has at least two unsolvable composition factors}

This section is devoted to the case when $A$ has at least two unsolvable composition factors. We can read off the candidates of maximal factorizations from Tables \ref{tabLinear}--\ref{tabOmegaPlus2}.

\begin{lemma}\label{l8}
Let $G$ be an almost simple group with socle $L$ classical of Lie type. If $G=AB$ is a nontrivial maximal factorization of $G$ such that $A$ has at least two unsolvable composition factors, then $A$ has exactly two unsolvable composition factors and $(L,A\cap L,B\cap L)$ lies in \emph{Table~\ref{tab11}}, where $a=\gcd(2,\ell,q-1)$.
\end{lemma}

\begin{table}[htbp]
\caption{}\label{tab11}
\centering
\begin{tabular}{|l|l|l|l|}
\hline
row & $L$ & $A\cap L$ & $B\cap L$\\
\hline
1 & $\Sp_{4\ell}(2^f)$, $f\ell\geqslant2$ & $\Sp_{2\ell}(2^f)\wr\Sy_2$ & $\GO_{4\ell}^-(2^f)$\\
2 & $\Sp_{4\ell}(4)$, $\ell\geqslant2$ & $\Sp_2(4)\times\Sp_{4\ell-2}(4)$ & $\Sp_{2\ell}(16).2$\\
3 & $\Sp_4(2^f)$, $f\geqslant2$ & $\GO_4^+(2^f)$ & $\Sp_2(4^f).2$\\
4 & $\Sp_4(2^f)$, $f\geqslant3$ odd & $\GO_4^+(2^f)$ & $\Sz(2^f)$\\
5 & $\Sp_6(2^f)$, $f\geqslant2$ & $\Sp_2(2^f)\times\Sp_4(2^f)$ & $\G_2(2^f)$\\
6 & $\Sp_4(4)$ & $\GO_4^+(4)$ & $\GO_4^-(4)$\\
\hline
7 & $\POm_{4\ell}^+(q)$, $\ell\geqslant3$, $q\geqslant4$ & $(\PSp_2(q)\times\PSp_{2\ell}(q)).a$ & $\N_1$\\
8 & $\POm_8^+(q)$, $q\geqslant5$ odd & $(\PSp_2(q)\times\PSp_4(q)).2$ & $\Omega_7(q)$\\
9 & $\Omega_8^+(2)$ & $(\SL_2(4)\times\SL_2(4)).2^2$ & $\Sp_6(2)$\\
10 & $\Omega_8^+(4)$ & $(\SL_2(16)\times\SL_2(16)).2^2$ & $\Sp_6(4)$\\
\hline
\end{tabular}
\end{table}

\begin{remark}
In row 6 of Table~\ref{tab11}, either $A,B$ are both $\mathcal{C}_8$ subgroups, or $A$ is in $\mathcal{C}_2$ and $B$ is in $\mathcal{C}_3$.
\end{remark}

Now we analyze the candidates in Table~\ref{tab11} under Hypothesis \ref{Hypo}.

\begin{proposition}\label{p1}
Let $G$ be an almost simple group satisfying \emph{Hypothesis \ref{Hypo}} with socle $L$ classical of Lie type. If $G=HK$ is a nontrivial factorization with $H$ solvable and $HL=KL=G$, and a maximal subgroup $A$ of $G$ containing $H$ has at least two unsolvable composition factors, then one of the following holds.
\begin{itemize}
\item[(a)] $L=\Sp_4(4)$, and $(L,H\cap L,K\cap L)$ lies in row~\emph{3} or \emph{4} of \emph{Table \ref{tab7}}.
\item[(b)] $L=\Omega_8^+(2)$ or $\Omega_8^+(4)$, and $(L,H\cap L,K\cap L)$ lies in row~\emph{9} of \emph{Table \ref{tab7}}.
\end{itemize}
\end{proposition}

\begin{proof}
Take $B$ to be a maximal subgroup of $G$ containing $K$. By Lemma \ref{Embedding}, $A,B$ are both core-free in $G$. Hence $G=AB$ is a nontrivial maximal factorization, and by Lemma \ref{l8}, $(L,A\cap L,B\cap L)$ lies in Table~\ref{tab11}. If $L=\Sp_4(4)$, then computation in \magma \cite{bosma1997magma} shows that $(L,H\cap L,K\cap L)$ lies in row 3 or 4 of Table \ref{tab7}, as described in part (a) of the proposition. For $L=\Omega_8^+(2)$ or $\Omega_8^+(4)$, computation in \magma \cite{bosma1997magma} shows that part (b) of the proposition holds. Thus we assume that $L\not\in\{\Sp_4(4),\Omega_8^+(2),\Omega_8^+(4)\}$, and in particular, none of rows 6 and 9--10 of Table \ref{tab11} appears.

{\bf Case 1.} Suppose $L=\PSp_4(q)$. Then it is seen in Table \ref{tab11} that $q=2^f\geqslant4$ and one of rows 1 (with $\ell=1$), 3 and 4 appears. Consequently, $A\cap L\cong\SL_2(q)\wr\Sy_2$, $B\cap L\cong\SL_2(q^2).2$ or $\Sz(q)$, and $G/L=\Z_e\leqslant\Z_f$.
Viewing
\[
|L|/|B\cap L|=
\left\{\begin{array}{ll}
q^2(q^2-1)/2, & \mbox{if $B\cap L\cong\SL_2(q^2).2$},\\
q^2(q^2-1)(q+1), & \mbox{if $B\cap L=\Sz(q)$},
\end{array}\right.
\]
we conclude that $|L|/|B\cap L|$ is divisible by $q^2(q^2-1)/2$, and so is $|H\cap L||G/L|$ by Lemma \ref{p4}.

Write $A\cap L=(T_1\times T_2).\Sy_2$, where $T_1\cong T_2\cong\SL_2(q)$, and let $X=(H\cap L)\cap(T_1\times T_2)$. Then $X$ is a solvable subgroup of $T_1\times T_2$, and $X$ has index at most $2$ in $H\cap L$. Thus $|X||G/L|$ is divisible by $q^2(q^2-1)/4$. Since $G/L=\Z_e\leqslant\Z_f$, it follows that $e|X|$ is divisible by $q^2(q^2-1)/4$.
Let $X_i=XT_i/T_i\lesssim T_{3-i}$ for $i=1,2$. Then $X_1,X_2$ are solvable, and $X\lesssim X_1\times X_2$.

Note that $A$ has a subgroup $T$ of index $2$ such that $T\cap L\cong\SL_2(q)\times\SL_2(q)$. If $H\leqslant T$, then $G=TB$ and $A=T(A\cap B)$, contrary to the facts that $A\cap B\leqslant T$ (see \cite[3.2.4(b) and 5.1.7(b)]{liebeck1990maximal}). Hence $H\nleqslant T$, and thus there exists an element $g\in H\setminus T$ interchanging $T_1$ and $T_2$ by conjugation. Since $X^g=(H\cap L)^g\cap(T_1\times T_2)^g=X$, $g$ induces an isomorphism between $X_1$ and $X_2$. As a consequence, $|X|$ divides $|X_1||X_2|=|X_1|^2$.

Any solvable subgroup of $\SL_2(2^f)$ has order dividing $q(q-1)$, $2(q-1)$, $2(q+1)$ or $12$, refer to \cite[Chapter~2, 8.27]{huppert1967}. Hence $|X|$ divides $(q(q-1))^2$, $(2(q-1))^2$, $(2(q+1))^2$ or $144$. Since $q^2(q^2-1)/4$ divides $e|X|$ and $e$ divides $f$, we deduce that $q^2(q^2-1)/4$ divides $fq^2(q-1)^2$, $4f(q-1)^2$, $4f(q+1)^2$ or $144f$.
If $q^2(q^2-1)/4$ divides $fq^2(q-1)^2$, or equivalently $q+1$ divides $4f(q-1)$, then observing $(q+1,4(q-1))=1$, we have $q+1\di f$, which is impossible. In the same vein, we see that the others are impossible too. This excludes the possibility of $L=\PSp_4(q)$, and especially, rows 3 and 4 of Table \ref{tab11}.


To complete the proof, we still need to exclude rows 1, 2, 5 and 7--8.

{\bf Case 2.} Suppose that $(L,A\cap L,B\cap L)$ lies in one of row 1 of Table~\ref{tab11}. From case 1 we see that $\ell\geqslant2$. Let $G/L=\Z_e\leqslant\Z_f$ and $T=(T_1\times T_2){:}e$, where $T_1\cong T_2\cong\Sp_{2\ell}(2^f)$. Observe that $A\cap B=(\GO_{2\ell}^-(2^f)\times\GO_{2\ell}^+(2^f)){:}e$ (see \cite[3.2.4(b)]{liebeck1990maximal}). Without loss of generality, suppose $A\cap B\cap T_1=\GO_{2\ell}^-(2^f)$. It derives from $G=HB$ that $A=H(A\cap B)$, and thus $T=(H\cap T)(A\cap B)$ since $A\cap B\leqslant T$. Now consider the factorization $\overline{T}=X_1\overline{A\cap B}$, where $\overline{T}=T/T_1=\Sp_{2\ell}(2^f).e$, $\overline{A\cap B}=(A\cap B)T_1/T_1=\GO_{2\ell}^+(2^f).e$ and $X_1=(H\cap T)T_1/T_1$ is solvable. By Hypothesis \ref{Hypo}, this factorization must satisfy the conclusion of Theorem \ref{SolvableFactor}. Hence we have $(\ell,f)=(2,1)$ or $(3,1)$.

Now $L=\Sp_{4\ell}(2)$ with $\ell\in\{2,3\}$. Replacing $T_1$ with $T_2$ in the above paragraph, we obtain the factorization $\Sp_{2\ell}(2)=X_2\GO_{2\ell}^-(2)$ along the same lines, where $X_2=(H\cap T)T_2/T_2$ is solvable. Moreover, since $B$ is not transitive on the $2\ell$-dimensional non-degenerate symplectic subspaces, we know that $G\neq TB$. Consequently, $H\nleqslant T$ as $G=HB$. It follows that each element in $H\setminus T$ induces an isomorphism between $X_1$ and $X_2$. However, computation in \magma \cite{bosma1997magma} shows that there do not exist two factorizations $\Sp_{2\ell}(2)=X_{(3-\varepsilon)/2}\GO_{2\ell}^\varepsilon(2)$, where $\varepsilon=\pm1$, such that $X_1$ and $X_2$ are both solvable and have the same order. This contradiction implies that row 1 of Table~\ref{tab11} cannot appear.

{\bf Case 3.} Suppose that $(L,A\cap L,B\cap L)$ lies in one of rows 2, 5, 7 and 8 of Table~\ref{tab11}. Then $A\cap L$ has a subgroup $\PSp_2(q)\times T$ of index at most $2$, where $T=\PSp_{2k}(q)$ with $k\geqslant2$ and $q\geqslant4$. We deduce from Zsigmondy's theorem that $q^{2k}-1$ has a primitive prime divisor $r$.
Write $N=\C_A(T)$, and note $T\vartriangleleft A$ and $N\cap T=1$. Then $N\vartriangleleft\Nor_A(T)=A$, and $T\cong NT/N\leqslant A/N\lesssim\Aut(T)$, which means that $A/N$ is an almost simple group with socle $\PSp_{2k}(q)$. By $N\cong NT/T\leqslant A/T$ we know that $|N|$ divides $2|\PSp_2(q)||\Out(L)|$. Consequently, $|N|$ is not divisible by $r$.

As the intersection $A\cap B$ is determined in \cite{liebeck1990maximal} (see 3.2.1(a), 5.23(b), 3.6.1(d) and 5.1.15 there), it follows readily that $(A\cap B)N/N$ does not contain $\PSp_{2k}(q)$.
Hence by Lemma~\ref{l2}, $HN/N$ is a nontrivial factor of $A/N$. Then since $HN/N$ is solvable, Hypothesis \ref{Hypo} implies that either $HN/N\cap\PSp_{2k}(q)\leqslant q^{k(k+1)/2}{:}(q^k-1).k$, or $q\in\{5,7,11,23\}$ and $HN/N\cap\PSp_{2k}(q)\leqslant q^3{:}(q-1).\Sy_4$. Thereby we conclude that $|HN/N|$ is not divisible by $r$.
Since $|H|=|HN/N||H\cap N|$ divides $|HN/N||N|$, this implies that $|H|$ is not divisible by $r$. However, the factorization $G=HB$ requires $|H|$ to be divisible by $|G|/|B|$ and thus by $r$, a contradiction.
\end{proof}


\chapter[Proof]{Proof of Theorem~\ref{SolvableFactor}}


We will complete the proof of Theorem~\ref{SolvableFactor} by induction on the order of the almost simple group. Thus throughout the first five sections of this chapter, we assume Hypothesis~\ref{Hypo}.

\section{Linear groups}

The main result of this section is the following proposition, which verifies Theorem \ref{SolvableFactor} for linear groups under Hypothesis~\ref{Hypo}.

\begin{proposition}\label{Linear}
Let $G$ be an almost simple group with socle $L=\PSL_n(q)$ not isomorphic to $\A_5$, $\A_6$ or $\A_8$. If $G=HK$ is a nontrivial factorization with $H$ solvable, $K$ unsolvable and $HL=KL=G$, then under \emph{Hypothesis \ref{Hypo}}, one of the following holds.
\begin{itemize}
\item[(a)] $H\cap L\leqslant\hat{~}\GL_1(q^n).n$, and $q^{n-1}{:}\SL_{n-1}(q)\trianglelefteq K\cap L\leqslant\Pa_1$ or $\Pa_{n-1}$.
\item[(b)] $L=\PSL_4(q)$, $H\cap L\leqslant q^3{:}((q^3-1)/(4,q-1)).3$, and $\PSp_4(q)\trianglelefteq K\cap L\leqslant\PSp_4(q).a$, where $a=(4,q+1)/(2,q-1)\leqslant2$.
\item[(c)] $L=\PSL_2(q)$ with $q\in\{11,16,19,29,59\}$ or $\PSL_4(q)$ with $q\in\{3,4\}$ or $\PSL_5(2)$, and $(L,H\cap L,K\cap L)$ is as described in \emph{Table \ref{tab13}}.
\end{itemize}
\end{proposition}

\begin{table}[htbp]
\caption{}\label{tab13}
\centering
\begin{tabular}{|l|l|l|l|}
\hline
row & $L$ & $H\cap L\leqslant$ & $K\cap L$\\
\hline
1 & $\PSL_2(11)$ & $11{:}5$ & $\A_5$\\
2 & $\PSL_2(16)$ & $\D_{34}$ & $\A_5$\\
3 & $\PSL_2(19)$ & $19{:}9$ & $\A_5$\\
4 & $\PSL_2(29)$ & $29{:}14$ & $\A_5$\\
5 & $\PSL_2(59)$ & $59{:}29$ & $\A_5$\\
6 & $\PSL_4(3)$ & $2^4{:}5{:}4$ & $\PSL_3(3)$, $3^3{:}\PSL_3(3)$\\
7 & $\PSL_4(3)$ & $3^3{:}26{:}3$ & $(4\times\PSL_2(9)){:}2$\\
8 & $\PSL_4(4)$ & $2^6{:}63{:}3$ & $(5\times\PSL_2(16)){:}2$\\
9 & $\PSL_5(2)$ & $31{:}5$ & $2^6{:}(\Sy_3\times\PSL_3(2))$\\
\hline
\end{tabular}
\end{table}

\begin{proof}
By Lemma \ref{Embedding}, we may take $A,B$ to be core-free maximal subgroups of $G$ containing $H,K$ respectively.
Then the maximal factorizations $G=AB$ are listed in Table \ref{tabLinear} (interchanging $A,B$ if necessary) by Theorem \ref{Maximal}. If $n$ is a prime, then the proposition holds by Theorem~\ref{p-dimensionLinear}. We thus assume that $n$ is a composite number. Under this assumption, $(L,A\cap L,B\cap L)$ lies in rows 1--4 of Table \ref{tabLinear}. For $L=\PSL_4(3)$ or $\PSL_4(4)$, computation in \magma \cite{bosma1997magma} shows that one of parts (a), (b) and (c) of Proposition \ref{Linear} holds. Thus assume $(n,q)\not\in\{(4,3),(4,4)\}$ for the rest of the proof.

{\bf Case 1.} Suppose that $B\cap L=\hat{~}\GL_a(q^b).b$ with $ab=n$ and $b$ prime, as in row 1 or row 4 of Table \ref{tabLinear} (with $A,B$ interchanged). Then $A\cap L=\Pa_1$, $\Pa_{n-1}$ or $\Stab(V_1\oplus V_{n-1})$. Notice that $A$ has the unique unsolvable composition factor $\PSL_{n-1}(q)$. Write $q=p^f$ with $p$ prime, $R=\Rad(A)$ and $A/R=\PSL_{n-1}(q).\calO$.

As $(n,q)\not\in\{(4,3),(4,4)\}$, we deduce from Lemma~\ref{Reduction}(c) that either $|H|_p$ divides $(n-1)|R|_p|\calO|_p$, or $(n,q)=(4,8)$ and $|H|$ divides $2^6\cdot56\cdot7\cdot6|R|$. If the latter occurs, then $|H|$ is not divisible by $73$, contrary to the factorization $G=HB$ since $a=b=2$ and $|G|/|B|$ is divisible by $73$. Hence $|H|_p$ divides $(n-1)|R|_p|\calO|_p$, and then as $|R||\calO|=|A|/|\PSL_{n-1}(q)|$, we have that $|H|_p$ divides $(n-1)fq^{n-1}$. Since $G=HB$, we know that $|H|_p|B\cap L|_p$ is divisible by $|L|_p$, and hence $|L|_p$ divides $(n-1)fq^{n-1}|B\cap L|_p$. Consequently,
$$
q^{n(n-1)/2}\di(n-1)fq^{n-1}\cdot q^{n(a-1)/2}b,
$$
that is, $q^{n(n-a)/2-n+1}\di b(n-1)f$. Since $(b,n-1)=1$ we conclude that either $q^{n(n-a)/2-n+1}\di bf$ or $q^{n(n-a)/2-n+1}\di (n-1)f$. Since $a\leqslant n/2$ and $b<n-1$, it follows that
\begin{equation}\label{eq16}
q^{n^2/4-n+1}\leqslant q^{n(n-a)/2-n+1}\leqslant\max(b,n-1)f=(n-1)f.
\end{equation}
This implies
$$
2^{n^2/4-n+1}\leqslant p^{n^2/4-n+1}\leqslant p^{f(n^2/4-n+1)}/f\leqslant n-1,
$$
which leads to $n=4$. However, substituting $n=4$ into (\ref{eq16}) gives $q\leqslant3f$, contradicting the assumption $(n,q)\not\in\{(4,2),(4,3)\}$.

{\bf Case 2.} Suppose that $B\cap L=\Pa_1$ or $\Pa_{n-1}$, as in row 1 or row 2 of Table \ref{tabLinear}.

Assume that row 2 of Table \ref{tabLinear} appears. Then $A\cap L=\PSp_n(q).c$ with $n\geqslant4$ even and $c=(2,q-1)(n/2,q-1)/(n,q-1)\leqslant2$. As $(n,q)\neq(4,3)$, Lemma~\ref{Reduction}(d) implies that either $(n,q)=(6,2)$ or $|H|$ is not divisible by any primitive prime divisor of $q^n-1$. Since the factorization $G=HK$ requires $|H|$ to be divisible by $|G|/|B|=(q^n-1)/(q-1)$, we have $(n,q)=(6,2)$ and thus $|H|$ is divisible by $(q^n-1)/(q-1)=63$. However, by Lemma~\ref{Reduction}(d), $|H|$ divides $2^6\cdot(2^3-1)\cdot3$ or $3^3\cdot2\cdot24$, which is a contradiction.

We thus have row 1 of Table \ref{tabLinear}, namely, $A\cap L=\hat{~}\GL_a(q^b).b$ with $ab=n$ and $b$ prime. Note that $A$ has the unique unsolvable composition factor $\PSL_a(q^b)$. By Lemma~\ref{Reduction}(c), $H\leqslant\PGaL_1((q^b)^a)=\PGaL_1(q^n)$. This by Lemma~\ref{LowerLinear} leads to part (a) of Proposition \ref{Linear}.

{\bf Case 3.} Suppose that $B\cap L=\PSp_n(q).c$ with $n\geqslant4$ even and $c=(2,q-1)(n/2,q-1)/(n,q-1)$, as in row 2 or row 3 of Table \ref{tabLinear} (with $A,B$ interchanged). Then
\[\mbox{$A\cap L=\Pa_1$, $\Pa_{n-1}$ or $\Stab(V_1\oplus V_{n-1})$.}\]
Notice that $A$ has the unique unsolvable composition factor $\PSL_{n-1}(q)$. Write $q=p^f$ with $p$ prime, $R=\Rad(A)$ and $A/R=\PSL_{n-1}(q).\calO$.

Assume $A\cap L=\Stab(V_1\oplus V_{n-1})\cong\hat{~}\GL_{n-1}(q)$. Then by Lemma~\ref{Reduction}(c), $|H|_p$ divides $2(n-1)f$. This implies that $|H|_p$ is not divisible by $|L|_p/|B\cap L|_p=q^{n(n-1)/2-n^2/4}$. Consequently, $|H|$ is not divisible by $|L|/|B\cap L|$, which is a contradiction to the factorization $G=HB$. Therefore, $A\cap L=\Pa_1$ or $\Pa_{n-1}$.

Assume $n\geqslant6$. By Zsigmondy's theorem, $(q,n-3)$ has a primitive prime divisor $r$ as $n\neq9$. Then Lemma~\ref{Reduction}(c) implies that $|H|$ is not divisible by $r$. Observe that $r$ does not divide $|B|$. This yields that $r$ does not divide $|H||B|$, contrary to the factorization $G=HB$ since $|G|$ is divisible by $r$.

We thus have $n=4$ and $\A\cap L\cong\Pa_1=q^3{:}(\GL_3(q)/(4,q-1))$. Hence $c=(2,q-1)^2/(4,q-1)=(4,q+1)/(2,q-1)$. By Lemma~\ref{Reduction}(c), either $H\leqslant R.(((q^3-1)/((q-1)(3,q-1))).3).\calO$, or $q=8$ and $|H|$ divides $2^6\cdot56\cdot7\cdot6|R|$. If the latter occurs, then $|H|$ is not divisible by $73$, contrary to the factorization $G=HB$ since $|L|/|B\cap L|=|\PSL_4(8)|/|\PSp_4(8)|$ is divisible by $73$. Therefore,
$$
H\leqslant R.\left(\frac{q^3-1}{(q-1)(3,q-1)}.3\right).\calO=\left(q^3.\frac{q^3-1}{(4,q-1)}.3\right).(G/L).
$$
Accordingly we have $H\cap L\leqslant q^3{:}((q^3-1)/(4,q-1)).3$, and this further yields by Lemma \ref{p4} that $q^3(q^4-1)(q^2-1)$ divides $6(4,q-1)f|K\cap L|$. Consequently,
$$
\frac{|\PSp_4(q).a|}{|K\cap L|}\leqslant\frac{6(4,q-1)f|\PSp_4(q).a|}{q^3(q^4-1)(q^2-1)}<\frac{q^4-1}{q-1}.
$$
Since $(q^4-1)/(q-1)$ equals the smallest index of proper subgroups of $\PSp_4(q)$ by Theorem~\ref{l7}, we obtain from Lemma \ref{MinimalDegree} that $\PSp_4(q)\trianglelefteq K\cap L$. Thus part (b) of Proposition \ref{Linear} follows.

{\bf Case 4.} Suppose that $B\cap L=\Stab(V_1\oplus V_{n-1})=\hat{~}\GL_{n-1}(q)$ with $n\geqslant4$ even, as in row 3 or row 4 of Table \ref{tabLinear}. Note that the factorization $G=HK$ requires $|H|$ to be divisible by $|L|/|B\cap L|=q^{n-1}(q^n-1)/(q-1)$.

Assume that row 3 of Table \ref{tabLinear} appears. Then $A\cap L=\PSp_n(q).c$ with $c=(2,q-1)(n/2,q-1)/(n,q-1)\leqslant2$. As $(n,q)\neq(4,3)$, Lemma~\ref{Reduction}(d) implies that either $(n,q)=(6,2)$ or $|H|$ is not divisible by any primitive prime divisor of $q^n-1$. Since $|H|$ is divisible by $(q^n-1)/(q-1)$, we have $(n,q)=(6,2)$ and thus $|H|$ is divisible by $(q^n-1)/(q-1)=63$. However, by Lemma~\ref{Reduction}(d), $|H|$ divides $2^6\cdot(2^3-1)\cdot3$ or $3^3\cdot2\cdot24$, which is a contradiction.

Now row 4 of Table \ref{tabLinear} appears, that is, $A\cap L=\hat{~}\GL_{n/2}(q^2).2$ with $q\in\{2,4\}$. As $(n,q)\not\in\{(4,2),(4,4)\}$, we have $n\geqslant6$. Since $A$ has the unique unsolvable composition factor $\PSL_{n/2}(q^2)$, Lemma~\ref{Reduction}(c) holds with $R=\Rad(A)$ and $A/R=\PSL_{n/2}(q^2).\calO$. To be more precise, row 1, 3 or 8 in Lemma~\ref{Reduction}(c) holds as $q\in\{2,4\}$ and $n\geqslant6$. Observe that $|R||\calO|=|A|/|\PSL_{n/2}(q^2)|$ divides $2(q^2-1)(n/2,q^2-1)$. If $q^n-1$ has a primitive prime divisor $r$, then the factorization $G=HK$ requires $|H|$ to be divisible by $q^{n-1}r$, but none of rows 1, 3 and 8 in Lemma~\ref{Reduction}(c) satisfies this. Thus $q^n-1$ does not have any primitive prime divisor, which is equivalent to $(n,q)=(6,2)$ by Zsigmondy's theorem. In this situation, row 1 or 8 in Lemma~\ref{Reduction}(c) appears, but neither of them allows $|H|$ to be divisible by $q^{n-1}(q^n-1)/(q-1)=2^5\cdot3^2\cdot7$, a contradiction.
\end{proof}

\section{Symplectic Groups}

In this section we verify Theorem \ref{SolvableFactor} for symplectic groups under Hypothesis \ref{Hypo}.

\begin{proposition}\label{Symplectic}
Let $G$ be an almost simple group with socle $L=\PSp_{2m}(q)$, where $m\geqslant2$. If $G=HK$ is a nontrivial factorization with $H$ solvable, $K$ unsolvable and $HL=KL=G$, then under \emph{Hypothesis \ref{Hypo}}, one of the following holds.
\begin{itemize}
\item[(a)] $q$ is even, $H\cap L\leqslant q^{m(m+1)/2}{:}(q^m-1).m<\Pa_m$, and $\Omega_{2m}^-(q)\trianglelefteq K\cap L\leq\GO_{2m}^-(q)$.
\item[(b)] $m=2$, $q$ is even, $H\cap L\leqslant q^3{:}(q^2-1).2<\Pa_1$, and $\Sp_2(q^2)\trianglelefteq K\cap L\leq\Sp_2(q^2).2$.
\item[(c)] $m=2$, $q$ is odd, $H\cap L\leqslant q^{1+2}{:}((q^2-1)/2).2<\Pa_1$ and $\PSp_2(q^2)\trianglelefteq K\cap L\leqslant\PSp_2(q^2).2$.
\item[(d)] $L=\PSp_4(q)$ with $q\in\{3,5,7,11,23\}$ or $\PSp_6(q)$ with $q=2$ or $3$, and $(L,H\cap L,K\cap L)$ is as described in \emph{Table \ref{tab14}}.
\end{itemize}
\end{proposition}

\begin{table}[htbp]
\caption{}\label{tab14}
\centering
\begin{tabular}{|l|l|l|l|}
\hline
row & $L$ & $H\cap L\leqslant$ & $K\cap L$\\
\hline
1 & $\PSp_4(3)$ & $3^3{:}\Sy_4$ & $2^4{:}\A_5$\\
2 & $\PSp_4(3)$ & $3_+^{1+2}{:}2.\A_4$ & $\A_5$, $2^4{:}\A_5$, $\Sy_5$, $\A_6$, $\Sy_6$\\
3 & $\PSp_4(5)$ & $5^3{:}4.\A_4$ & $\PSL_2(5^2)$, $\PSL_2(5^2){:}2$\\
4 & $\PSp_4(7)$ & $7^3{:}6.\Sy_4$ & $\PSL_2(7^2)$, $\PSL_2(7^2){:}2$\\
5 & $\PSp_4(11)$ & $11^3{:}10.\A_4$ & $\PSL_2(11^2)$, $\PSL_2(11^2){:}2$\\
6 & $\PSp_4(23)$ & $23^3{:}22.\Sy_4$ & $\PSL_2(23^2)$, $\PSL_2(23^2){:}2$\\
7 & $\Sp_6(2)$ & $3_+^{1+2}{:}2.\Sy_4$ & $\A_8$, $\Sy_8$\\
8 & $\PSp_6(3)$ & $3_+^{1+4}{:}2^{1+4}.\D_{10}$ & $\PSL_2(27){:}3$\\
\hline
\end{tabular}
\end{table}

Throughout this section, fix $G,L,H,K$ to be the groups in the condition of Proposition \ref{Symplectic}, and take $A,B$ to be core-free maximal subgroups of $G$ containing $H,K$ respectively (such $A$ and $B$ are existent by Lemma \ref{Embedding}). We shall first treat symplectic groups of dimension four in Section \ref{sec1} (see Lemma \ref{4-dimensionSymplectic}), and then treat the general case in Section \ref{sec3} by Lemmas~\ref{Symplectic1} and \ref{Symplectic2}, distinguishing the maximal factorization $G=AB$ in rows 1--12 or 13--16 of Table \ref{tabSymplectic}. These lemmas together prove Proposition \ref{Symplectic}.

\subsection{Symplectic groups of dimension four}\label{sec1}

The main result of this subsection is stated in the following lemma.

\begin{lemma}\label{4-dimensionSymplectic}
If $L=\PSp_4(q)$ with $q\geqslant3$, then one of the following holds.
\begin{itemize}
\item[(a)] $q$ is even, $H\cap L\leqslant q^3{:}(q^2-1).2<\Pa_2$, and $\Omega_4^-(q)\trianglelefteq K\cap L\leq\GO_4^-(q)$.
\item[(b)] $q$ is even, $H\cap L\leqslant q^3{:}(q^2-1).2<\Pa_1$, and $\Sp_2(q^2)\trianglelefteq K\cap L\leq\Sp_2(q^2).2$.
\item[(c)] $q$ is odd, $H\cap L\leqslant q^{1+2}{:}((q^2-1)/2).2<\Pa_1$ and $\PSp_2(q^2)\trianglelefteq K\cap L\leqslant\PSp_2(q^2).2$.
\item[(d)] $L=\PSp_4(q)$ with $q\in\{3,5,7,11,23\}$, and $(L,H\cap L,K\cap L)$ is as described in rows \emph{1--6} of \emph{Table \ref{tab14}}.
\end{itemize}
\end{lemma}

\begin{proof}
For $L=\PSp_4(3)$, the lemma holds as a consequence of Proposition \ref{Small1}. Thus we assume $q\geqslant4$ for the rest of our proof. By Proposition \ref{p1}, we further assume that $A\cap L$ has at most one unsolvable composition factor. Therefore, we only need to consider rows 1--11 of Table \ref{tabSymplectic} for the maximal factorization $G=AB$ by Lemma \ref{Maximal}. Moreover, $A\cap L\neq\Sz(q)$ as Lemma \ref{Reduction}(a) asserts, which rules out row 11 of Table \ref{tabSymplectic}. Hence we have the following candidates for the pair $(A\cap L,B\cap L)$.
\begin{itemize}
\item[(i)] $A\cap L\cong\PSL_2(q^2).2$ and $B\cap L=\Pa_1$ or $\Pa_2$ .
\item[(ii)] $A\cap L\cong\PSL_2(q^2).2$ and $B\cap L\cong\GO_4^-(q)$ or $\GO_4^+(q)$.
\item[(iii)] $A\cap L=\Pa_1$ and $B\cap L=\PSp_2(q^2).2$.
\item[(iv)] $q$ is even, $A\cap L=\Pa_2$ and $B\cap L=\GO_4^-(q)$.
\item[(v)] $q=4^f$ with $f\in\{1,2\}$, $A\cap L=\GO_4^-(4^f)$ and $B\cap L=\Sp_4(2^f)$.
\item[(vi)] $q=4^f$ with $f\in\{1,2\}$, $A\cap L=\Sp_4(2^f)$ and $B\cap L=\GO_4^-(4^f)$.
\end{itemize}

{\bf Case 1.} Suppose that $(A\cap L,B\cap L)$ is as described in (i). Let $R=\Rad(A)$ and $S=\Soc(A/R)$. Write $A/R=S.\calO$ and $q=p^f$ with $p$ prime. Then $S=\PSL_2(q^2)$, and we see from Lemma \ref{Reduction}(c) that $H\leqslant R.(((q^4-1)/((q^2-1)(2,q-1)).2).\calO$ or $R.(q^2{:}((q^2-1)/(2,q-1))).\calO$ as in row 1 or row 2 of the table there. In particular, $|H|$ divides $2(q^2+1)|R||\calO|/(2,q-1)$ or $q^2(q^2-1)|R||\calO|/(2,q-1)$. Since $|R||\calO|=|A|/|S|=2|A|/|A\cap L|$ divides $f(2,q-1)$, we thus obtain that $|H|$ divides $4f(q^2+1)$ or $2fq^2(q^2-1)$. Moreover, $|L|/|B\cap L|=(q^4-1)/(q-1)$ divides $|H|$ according to the factorization $G=HB$. Hence $(q^4-1)/(q-1)$ divides $4f(q^2+1)$ or $2fq^2(q^2-1)$, which is impossible.

{\bf Case 2.} Suppose that $(A\cap L,B\cap L)$ is as described in (ii). Let $R=\Rad(A)$ and $S=\Soc(A/R)$. Write $A/R=S.\calO$ and $q=p^f$ with $p$ prime. Then $S=\PSL_2(q^2)$, and we see from Lemma \ref{Reduction}(c) that $H\leqslant R.(((q^4-1)/((q^2-1)(2,q-1)).2).\calO$ or $R.(q^2{:}((q^2-1)/(2,q-1))).\calO$ as in row 1 or row 2 of the table there. In particular, $|H|$ divides $2(q^2+1)|R||\calO|/(2,q-1)$ or $q^2(q^2-1)|R||\calO|/(2,q-1)$. Since $|R||\calO|=|A|/|S|=2|A|/|A\cap L|=2|G|/|L|$ divides $2f(2,q-1)$, we thus obtain that $|H|$ divides $4f(q^2+1)$ or $2fq^2(q^2-1)$. Moreover, $|L|$ divides $|H||B\cap L|$ due to the factorization $G=HB$. Hence
\begin{equation}\label{eq14}
\mbox{$|L|/|B\cap L|$ divides $4f(q^2+1)$ or $2fq^2(q^2-1)$.}
\end{equation}
If $B\cap L\cong\GO_4^-(q)$, then $|L|/|B\cap L|=q^2(q^2-1)/2$. If $B\cap L\cong\GO_4^+(q)$, then $|L|/|B\cap L|=q^2(q^2+1)/2$. Thus by (\ref{eq14}), either
\[\mbox{$B\cap L\cong\GO_4^-(q)$ and $|H|$ divides $2fq^2(q^2-1)$, or}\]
\[\mbox{$q=4$ and $B\cap L=\GO_4^+(4)$.}\]
Computation in \magma \cite{bosma1997magma} shows that the latter gives no factorization $G=HB$ with $H$ solvable. Thus we have the former, which indicates $p=2$ since row 1 of Table \ref{tabSymplectic} is now excluded. Also, $H\cap L\leqslant\Pa_1[\PSL_2(q^2).2]$ since $H\leqslant R.(q^2{:}((q^2-1)/(2,q-1))).\calO$ as row 2 of the table in Lemma \ref{Reduction}(c). Combining the condition that $|H|$ divides $2fq^2(q^2-1)$ with the conclusion of the factorization $G=HK$ that $|L|$ divides $|H||K\cap L|$, we deduce that $|L|$ divides $2fq^2(q^2-1)|K\cap L|$. That is to say, $q^4(q^4-1)(q^2-1)$ divides $2fq^2(q^2-1)|K\cap L|$, or equivalently, $q^2(q^4-1)$ divides $2f|K\cap L|$. Since $K\cap L\leqslant B\cap L\cong\SL_2(q^2).2$, it then follows that $\SL_2(q^2)\lesssim K\cap L$. Note that $A\cap L$ lie in two possible Aschbacher classes of subgroups of $L$, namely, $\mathcal{C}_3$ and $\mathcal{C}_8$. We distinguish these two classes in the next two paragraphs.

Assume that $A\cap L=\Sp_2(q^2).2$ is a $\mathcal{C}_3$ subgroup of $L$. Then $B\cap L=\GO_4^-(q)$, and $\Pa_1[A\cap L]\leqslant\Pa_2[L]$. Hence $H\cap L\leqslant\Pa_2[L]$. Since $\Pa_2[L]=q^3{:}\GL_2(q)$, we have $H\cap L\leqslant q^3{:}M$ for some maximal solvable subgroup $M$ of $\GL_2(q)$. From the factorization $G=HB$ we deduce that $|L|$ divides $f|B\cap L||H\cap L|$, that is, $q^4(q^4-1)(q^2-1)$ divides $2fq^2(q^4-1)|H\cap L|$. Consequently, $q^2(q^2-1)$ divides $2f|H\cap L|$, which further yields that $(q^2-1)$ divides $2f|M|$. This implies that $M=(q^2-1){:}2$, and thus $H\cap L\leqslant q^3{:}(q^2-1).2$. Therefore, part (a) of Lemma \ref{4-dimensionSymplectic} appears.

Assume that $A\cap L=\GO_4^-(q)$ is a $\mathcal{C}_8$ subgroup of $L$. Then $B\cap L=\Sp_2(q^2).2$, and $\Pa_1[A\cap L]\leqslant\Pa_1[L]$. Hence $H\cap L\leqslant\Pa_1[L]$. Since $\Pa_1[L]=q^3{:}\GL_2(q)$, we have $H\cap L\leqslant q^3{:}M$ for some maximal solvable subgroup $M$ of $\GL_2(q)$. According to the factorization $G=HB$ we have that $|L|$ divides $f|B\cap L||H\cap L|$, that is, $q^4(q^4-1)(q^2-1)$ divides $2fq^2(q^4-1)|H\cap L|$. Consequently, $q^2(q^2-1)$ divides $2f|H\cap L|$, which yields that $(q^2-1)$ divides $2f|M|$. This implies that $M=(q^2-1){:}2$, and thus $H\cap L\leqslant q^3{:}(q^2-1).2$. Therefore, part (b) of Lemma \ref{4-dimensionSymplectic} appears.

{\bf Case 3.} Consider the pair $(A\cap L,B\cap L)$ as described in (iii). If $q$ is even, then since $H\cap L\leqslant\Pa_1[L]$, arguing as the second paragraph of Case 1 leads to part (b) of Lemma \ref{4-dimensionSymplectic}. Thus assume $q=p^f$ with odd prime $p$. Let $X$ be a maximal solvable subgroup of
$$
A\cap L=\Pa_1=q^{1+2}{:}((q-1)\times\Sp_2(q))/2=q^{1+2}{:}(q-1).\PSp_2(q)
$$
containing $H\cap L$. Then $X=q^{1+2}{:}((q-1)\times Y)/2$ for some maximal solvable subgroup $Y$ of $\Sp_2(q)$. By Lemma \ref{Reduction}(c), $Y/\Z_2\leqslant\PSp_2(q)$ lies in the following table:
\[
\begin{array}{|l|l|l|}
\hline
\text{row} & Y/\Z_2 & q\\
\hline
1 & \D_{q+1} & \text{odd} \\
2 & q{:}((q-1)/2) & \text{odd} \\
3 & \A_4 & 5,11 \\
4 & \Sy_4 & 7,23 \\
5 & \Sy_4 & 9 \\
\hline
\end{array}
\]
Moreover, the factorization $G=HB$ together with $H\cap L\leqslant X$ implies that $|L|$ divides $|X||B\cap L||\Out(L)|$, that is, $q+1$ divides $2qf|Y|$. Thereby checking the above table we conclude that only its rows 1 and 3--4 are possible. Notice that $|L|$ divides $|X||K\cap L||\Out(L)|$ due to the factorization $G=HK$. If row 1 appears, then $X=q^{1+2}{:}\Z_{(q^2-1)/2}.\Z_2$ and it follows that $q(q^4-1)$ divides $4f|K\cap L|$. This implies that $\PSp_2(q^2)\trianglelefteq K\cap L$, as part (c) of Lemma \ref{4-dimensionSymplectic}. If row 3 appears, then $X=q^{1+2}{:}(q-1).\A_4$ with $q\in\{5,11\}$ and it follows that $q(q^4-1)(q+1)$ divides $48f|K\cap L|$. This implies that $\PSp_2(q^2)\leqslant K\cap L$, and leads to part (d) of Lemma \ref{4-dimensionSymplectic}. Similarly, if row 4 appears, then $X=q^{1+2}{:}(q-1).\Sy_4$ with $q\in\{7,23\}$ and it follows that $q(q^4-1)(q+1)$ divides $96f|K\cap L|$. This implies that $\PSp_2(q^2)\leqslant K\cap L$, as part (d) of Lemma \ref{4-dimensionSymplectic}.

{\bf Case 4.} Consider the pair $(A\cap L,B\cap L)$ as described in (iv). Since $H\cap L\leqslant\Pa_2[L]$, arguing as the third paragraph of Case 1 leads to part (a) of Lemma \ref{4-dimensionSymplectic}.

{\bf Case 5.} Suppose that $(A\cap L,B\cap L)$ is as described in (v). If $f=1$, then $|L|/|B\cap L|$ is divisible by $5\cdot17$, but Lemma \ref{Reduction}(c) implies that $|H|$ is not divisible by $5\cdot17$, contrary to the factorization $G=HB$. If $f=2$, then $|L|/|B\cap L|$ is divisible by $17\cdot257$, but Lemma \ref{Reduction}(c) implies that $|H|$ is not divisible by $17\cdot257$, contrary to the factorization $G=HB$. Hence this case is not possible.

{\bf Case 6.} Finally, consider the pair $(A\cap L,B\cap L)$ as described in (vi). If $f=1$, then the factorization $G=HB$ requires  $|H\cap L|$ to be divisible by $15$, contrary to the fact that $A\cap L=\Sp_4(2)\cong\Sy_6$ does not possess any solvable subgroup of order divisible by $15$. If $f=2$, then $|L|/|B\cap L|$ is divisible by $17$, but Lemma \ref{Reduction}(d) implies that $|H|$ is not divisible by $17$, contradicting the factorization $G=HB$. Hence this case is not possible either. We thus complete the proof.
\end{proof}

\subsection{Symplectic groups of dimension at least six}\label{sec3}

We first embark on the infinite families for the maximal factorization $G=AB$, namely, rows 1--12 of Table \ref{tabSymplectic}.

\begin{lemma}\label{Symplectic3}
Let $m\geqslant3$, $q=p^f$ with $p$ prime and $(m,q)\neq(3,2)$. If the maximal factorization $G=AB$ lies in rows~\emph{1--12} of \emph{Table \ref{tabSymplectic}} (interchanging $A,B$ if necessary), then each primitive prime divisor of $p^{2fm}-1$ divides $|B\cap L|$.
\end{lemma}

\begin{proof}
By Lemma \ref{p4}, it suffices to show that any primitive prime divisor $r$ of $p^{2fm}-1$ does not divide $|H\cap L|$. Suppose to the contrary that $r$ divides $|H\cap L|$. Then $r$ divides $|A\cap L|$. Inspecting rows 1--12 of Table \ref{tabSymplectic}, we conclude that $A\cap L=\PSp_{2a}(q^b).b$ with $ab=m$ and $b$ prime or $\SO_{2m}^-(q)$ with $q$ even or $\G_2(q)$ with $m=3$. Notice that $\SO_6^-(q)\cong\PSU_4(q).2$ if $q$ is even. Since $r$ divides $|H|$, we deduce from Lemma~\ref{Reduction}(a), (b) and (f) that either $A\cap L=\PSp_{2a}(q^b).b$ with $ab=m$ and $b$ prime or $L=\Sp_6(8)$ and $A\cap L=\SO_6^-(8)$.

Suppose that $L=\Sp_6(8)$ and $A\cap L=\SO_6^-(8)$. Then one sees from Table \ref{tabSymplectic} that $B\cap L=\Sp_2(8^3).3$, $\Pa_3$ or $\G_2(8)$. For these candidates of $B\cap L$, it holds that $|L|/|B\cap L|$ is divisible by $13$. However, Lemma~\ref{Reduction}(f) shows that $|H|$ is not divisible by $13$, contrary to the factorization $G=HB$. Therefore, $A\cap L=\PSp_{2a}(q^b).b$ with $ab=m$ and $b$ prime.

Write $R=\Rad(A)$ and $A/R=\PSp_{2a}(q^b).\calO$. Since $r$ divides $|H|$, we deduce from Lemma \ref{Reduction}(c) (with $\ell=2$ there) and (d) that $a=1$, $b=m$ is prime and $H\leqslant R.\D_{2(q^m+1)/(2,q-1)}.\calO$. This together with the equality $|R||\calO|=|A|/|\PSp_2(q^m)|$ yields that $|H|$ divides
$$
\frac{2(q^m+1)|A|}{(2,q-1)|\PSp_2(q^m)|}=\frac{2m(q^m+1)|A|}{(2,q-1)|A\cap L|}=\frac{2m(q^m+1)|G|}{(2,q-1)|L|},
$$
whence $|H|$ divides $2fm(q^m+1)$. Then from the factorization $G=HB$ we obtain that
\begin{equation}\label{eq9}
\mbox{$|L|/|B\cap L|$ divides $2fm(q^m+1)$.}
\end{equation}
As seen from Table \ref{tabSymplectic}, the possibilities for $B\cap L$ are:
\[\mbox{$\Pa_1$,\quad$\SO_{2m}^+(q)$ with $q$ even,\quad$\SO_{2m}^-(q)$ with $q$ even.}\]
We proceed by these three cases for $B\cap L$.

{\bf Case 1}: $B\cap L=\Pa_1$. Then $(q^{2m}-1)/(q-1)$ divides $2fm(q^m+1)$ by (\ref{eq9}), that is, the divisibility in
\begin{equation}\label{eq10}
q^m-1\di2fm(q-1).
\end{equation}
It follows by (\ref{eq10}) that $(m,q)\neq(3,4)$. Hence $p^{fm}-1$ has a primitive prime divisor $s$ by Zsigmondy's theorem as $m$ is prime. However, since $2fm(q-1)$ is not divisible by $s$, (\ref{eq10}) does not hold, which is a contradiction.

{\bf Case 2}: $B\cap L=\SO_{2m}^+(q)$ with $p=2$. Then $q^m(q^m+1)/2$ divides $2fm(q^m+1)$ by (\ref{eq9}), that is, $2^{fm}\di4fm$. This forces $fm=4$, contrary to the condition that $m\geqslant3$ is prime.

{\bf Case 3}: $B\cap L=\SO_{2m}^-(q)$ with $p=2$. Then $q^m(q^m-1)/2$ divides $2fm(q^m+1)$ by (\ref{eq9}), that is, $q^m(q^m-1)\di4fm(q^m+1)$. Since $(q^m,q^m+1)=(q^m-1,q^m+1)=1$, this implies $q^m(q^m-1)\di4fm$, which is impossible.
\end{proof}

\begin{lemma}\label{Symplectic4}
Let $m\geqslant3$, $q=2^f$ and $(m,q)\neq(3,2)$. If the maximal factorization $G=AB$ lies in rows~\emph{1--12} of \emph{Table \ref{tabSymplectic}} (interchanging $A,B$ if necessary) with $B\cap L=\SO_{2m}^-(q)$, then $H\leqslant\Pa_m[G]$.
\end{lemma}

\begin{proof}
By Proposition \ref{p1}, the maximal factor $A$ has at most one unsolvable composition factor. Consequently, $A\cap L\neq\Sp_m\wr\Sy_2$. Then in rows 1--12 of Table \ref{tabSymplectic}, the possibilities for $A\cap L$ are:
\[\mbox{$\Sp_{2a}(q^b).b$ with $ab=m$ and $b$ prime,\quad$\Pa_m$,}\]
\[\mbox{$\GO_{2m}^+(q)$ with $q\in\{2,4\}$,\quad$\Sp_{2m}(q^{1/2})$ with $q\in\{4,16\}$.}\]
We proceed by these four cases for $A\cap L$.

{\bf Case 1}: $A\cap L=\Sp_{2a}(q^b).b$ with $ab=m$ and $b$ prime. In this case, $A=\Sp_{2a}(q^b).\Z_{be}$, where $\Z_e=G/L\leqslant\Z_f$. By Lemma \ref{l2}, $H$ is a solvable nontrivial factor of $A$.

Assume $H\nleqslant\Pa_a[A]$, a parabolic subgroup of $A$. By Hypothesis \ref{Hypo}, we see from Theorem \ref{SolvableFactor} that either $a=1$ and $|H|$ divides $2(q^m+1)me$, or $a=2$ and $H\leqslant\Pa_1[A]$. For the latter, the observation $\Pa_1[A]\leqslant\Pa_{m/2}[G]$ yields that $H\leqslant\Pa_{m/2}[G]$, and thus the factorization $G=\Pa_{m/2}[G]B$ arises, contrary to Theorem \ref{Maximal}. Hence we have the former, and in particular $|H|$ divides $2(q^m+1)mf$. Then the factorization $G=HB$ implies that $|L|/|B\cap L|=q^m(q^m-1)/2$ divides $2(q^m+1)mf$, that is, $q^m(q^m-1)\di4fm(q^m+1)$. Since $(q^m,q^m+1)=(q^m-1,q^m+1)=1$, this derives $fq^m(q^m-1)\di4fm$, which is impossible.

Therefore, $H\leqslant\Pa_a[A]$, then the observation $\Pa_a[A]\leqslant\Pa_m[G]$ implies that $H\leqslant\Pa_m[G]$, as the lemma asserts.

{\bf Case 2}: $A\cap L=\Pa_m$. Then $H\leqslant A=\Pa_m[G]$ as stated in the lemma.

{\bf Case 3}: $A\cap L=\GO_{2m}^+(q)$ with $q\in\{2,4\}$. In this case, $A=\GO_{2m}^+(q).\Z_e$, where $\Z_e\leqslant\Z_f$. By Lemma \ref{l2}, $H$ is a solvable nontrivial factor of $A$.

Assume $H\nleqslant\Pa_m[A]$. Viewing $\GO_6^+(q)=\PSL_4(q).2$, we conclude from Hypothesis \ref{Hypo} and Theorem \ref{SolvableFactor} that one of the following holds.
\begin{itemize}
\item[(i)] $m=4$ and $H\leqslant\Pa_1[A]$.
\item[(ii)] $(m,q)=(4,2)$ and $H\leqslant\Sy_9$.
\item[(iii)] $(m,q)=(3,4)$ and $|H|$ divides $4(4^4-1)|\Out(\PSL_4(4))|=2^4\cdot3\cdot5\cdot17$.
\end{itemize}
If (i) holds, then the observation $\Pa_1[A]\leqslant\Pa_1[G]$ yields that $H\leqslant\Pa_1[G]$, and thus the factorization $G=\Pa_1[G]B$ arises, contrary to Theorem \ref{Maximal}. For (ii), computation in \magma \cite{bosma1997magma} shows that it gives no factorization $G=HB$ with $H$ solvable. Thus we have (iii), and in particular $|H|$ is not divisible by $7$. However, the factorization $G=HB$ requires $|H|$ to be divisible by $|L|/|B\cap L|$, and thus $|H|$ is divisible by $7$, a contradiction.

Consequently, $H\leqslant\Pa_m[A]$, and it follows that $H\leqslant\Pa_m[G]$, as the lemma states.

{\bf Case 4}: $A\cap L=\Sp_{2m}(q^{1/2})$ with $q\in\{4,16\}$. If $(m,q)=(3,4)$, then Hypothesis \ref{Hypo} in conjunction with Theorem \ref{SolvableFactor} implies that $|H|$ is not divisible by $63$, contrary to the condition that $2^5\cdot63=|L|/|B\cap L|$ divides $|H|$. Thus $(m,q)\neq(3,4)$, and it follows that $2^{fm}-1$ has a primitive prime divisor $r$. Since $r$ divides $|G|/|B|$, $r$ should also divide $|H|$ as the factorization $G=HB$ requires. However, by Hypothesis \ref{Hypo} and Theorem \ref{SolvableFactor}, $r$ does not divide $|H|$, which is a contradiction.
\end{proof}

Now we finish the analysis for maximal factorizations $G=AB$ from rows 1--12 of Table \ref{tabSymplectic}.

\begin{lemma}\label{Symplectic1}
If $m\geqslant3$ and the maximal factorization $G=AB$ lies in rows~\emph{1--12} of \emph{Table \ref{tabSymplectic}} ($A,B$ may be interchanged), then one of the following holds.
\begin{itemize}
\item[(a)] $q$ is even, $H\cap L\leqslant q^{m(m+1)/2}{:}(q^m-1).m<\Pa_m$, and $\Omega_{2m}^-(q)\trianglelefteq K\cap L\leq\GO_{2m}^-(q)$.
\item[(b)] $G=\PSp_6(2)$, $H=3_+^{1+2}{:}2.M$ with $M\leqslant\Sy_4$, and $\A_8\trianglelefteq K\leqslant\GO_6^+(2)\cong\Sy_8$; moreover, $(M,A,K)$ lies in the following table:
\[
\begin{array}{|l|l|l|}
\hline
M & A & K\\
\hline
2^2,\ 4 & \GO_6^-(2),\ \G_2(2) & \Sy_8 \\
\D_8 & \GO_6^-(2),\ \G_2(2) & \A_8,\ \Sy_8 \\
\A_4 & \GO_6^-(2) & \Sy_8 \\
\Sy_4 & \GO_6^-(2) & \A_8,\ \Sy_8 \\
\hline
\end{array}
\]
\item[(c)] $G=\PSp_6(3).2$, $A=\Pa_1[G]$, $K=B=\PSL_2(27){:}6$, $H=3_+^{1+4}{:}2^{1+4}.\AGL_1(5)$ and $H\cap L=3_+^{1+4}{:}2^{1+4}.\D_{10}$.
\end{itemize}
\end{lemma}

\begin{proof}
Let $q=p^f$ with $p$ prime. For $(m,q)=(3,2)$, computation in \magma \cite{bosma1997magma} verifies the lemma directly. Thus we assume $(m,q)\neq(3,2)$ henceforth, and it follows that $p^{2fm}-1$ has a primitive prime divisor $r$. By Lemma \ref{Symplectic3}, $r$ divides $|B\cap L|$. Then according to Table \ref{tabSymplectic}, the possibilities for $B\cap L$ are:
\[\mbox{$\PSp_{2a}(q^b).b$ with $ab=m$ and $b$ prime,}\]
\[\mbox{$\GO_{2m}^-(q)$ with $q$ even,\quad$\G_2(q)$ with $m=3$ and $q$ even.}\]
We proceed by these three cases for $B\cap L$.

{\bf Case 1.} Suppose that $B\cap L=\PSp_{2a}(q^b).b$, where $ab=m$ and $b$ is prime. By Proposition \ref{p1}, $A$ has at most one unsolvable composition factor. Then inspecting Table \ref{tabSymplectic}, we obtain the following candidates for $A\cap L$:
\[\mbox{$\Pa_1$,\quad$\GO_{2m}^+(q)$ with $q$ even,\quad$\GO_{2m}^-(q)$ with $q$ even,\quad$\N_2$ with $b=2$ and $q=2$.}\]
If $(m,q)\neq(4,2)$, then $p^{f(2m-2)}-1$ has a primitive prime divisor $s$ by Zsigmondy's theorem. Let $s=3^3\cdot7$ if $(m,q)=(4,2)$. Since $s$ divides $|L|/|B\cap L|$, the factorization $G=HB$ requires $|H|$ to be divisible by $s$.

First assume that $A\cap L=\Pa_1$ or $A\cap L=\N_2$ with $b=q=2$. Then $A$ has the unique unsolvable composition factor $\PSp_{2m-2}(q)$. Since $s$ divides $|H|$, we conclude from Lemma \ref{Reduction}(d) that $(m,q)=(3,3)$, $A\cap L=\Pa_1$ and $H\leqslant\Rad(A).(2^4{:}\AGL_1(5))$. In this situation, computation in \magma \cite{bosma1997magma} shows that part (c) of Lemma~\ref{Symplectic1} appears.

Next assume that $A\cap L=\GO_{2m}^+(q)$ with $p=2$. If $m\geqslant4$, then there is a contradiction that $|H|$ is not divisible by $s$ by Lemma~\ref{Reduction}(h). Thus $m=3$, which indicates that $A$ has the unique unsolvable composition factor $\PSL_4(q)$. Since $|H|$ is divisible by $s$, we deduce from Lemma \ref{Reduction}(c) that $|H|$ divides $8f(q^4-1)/(q-1)$. Therefore, the factorization $G=HB$ implies that $|L|$ divides $|H||\PSp_2(q^3).3|$ and thus $8f(q^2+1)(q+1)|\PSp_2(q^3).3|$. This turns out to be that $q^6(q^2-1)(q-1)\di24f$, which is impossible.

Now assume that $A\cap L=\GO_{2m}^-(q)$ with $p=2$. By Lemma~\ref{Reduction}(b), $m=3$ and thus $A$ has the unique unsolvable composition factor $\PSU_4(q)$. It then derives from Lemma~\ref{Reduction}(f) that either $|H|$ divides $4fq^4(q^4-1)/(q+1)$ or $q=8$ and $|H|$ is not divisible by $13$. The latter is contrary to the factorization $G=HB$ since $|L|/|B\cap L|$ is divisible by $13$. Consequently, $|H|$ divides $4fq^4(q^4-1)/(q+1)$. Thereby the factorization $G=HB$ implies that $|L|$ divides $|H||\PSp_2(q^3).3|$ and thus $4fq^4(q^2+1)(q-1)|\PSp_2(q^3).3|$. This turns out to be that $q^2(q^2-1)(q+1)\di12f$, which is impossible.

{\bf Case 2.} Suppose that $B\cap L=\GO_{2m}^-(q)$ with $p=2$. By Lemma \ref{Symplectic4} we have $H\leqslant\Pa_m[G]$. Let $X=\Pa_m[G]$ be a maximal subgroup of $G$ containing $H$. Then $X=q^{m(m+1)/2}{:}\GL_m(q).\Z_e$ and $B=\GO_{2m}^-(q).\Z_e$, where $\Z_e\leqslant\Z_f$. If $(m,q)\neq(3,4)$ or $(6,2)$, then $2^{fm}-1$ has a primitive prime divisor $s$ by Zsigmondy's theorem. Let $s=7$ if $(m,q)=(3,4)$ or $(6,2)$. Since $s$ divides $|L|/|B\cap L|$, we conclude by Lemma \ref{p4} that $s$ also divides $|H|$. This together with Lemma \ref{Reduction}(c) implies that
\begin{equation}\label{eq17}
H\leqslant(q^{m(m+1)/2}{:}(q^m-1).m).\Z_e<X,
\end{equation}
and thus $H\cap L\leqslant q^{m(m+1)/2}{:}(q^m-1).m<\Pa_m$.

It derives from the factorization $G=HK$ that $B=(H\cap B)K$. Note that $H\cap B$ is solvable and $K$ is unsolvable. Then by Hypothesis \ref{Hypo} and Theorem \ref{SolvableFactor}, either $B^{(\infty)}\trianglelefteq K$, or $m=3$ and $B$ has the unique unsolvable composition factor $\PSU_4(q)$ with $|K|$ not divisible by any primitive prime divisor of $2^{4f}-1$. In light of (\ref{eq17}), the latter indicates that $|H||K|$ is not divisible by any primitive prime divisor of $2^{4f}-1$, contradicting the factorization $G=HK$. Thus we have $B^{(\infty)}\trianglelefteq K$, which leads to $\Omega_{2m}^-(q)\trianglelefteq K\cap L$. Therefore, part (a) of Lemma \ref{Symplectic1} appears.

{\bf Case 3.} Suppose that $m=3$ and $B\cap L=\G_2(q)$ with $p=2$. In this case, we have $f>1$ by our assumption that $(m,q)\neq(3,2)$. By Proposition \ref{p1}, $A$ has at most one unsolvable composition factor. Then inspecting Table \ref{tabSymplectic}, we obtain the following candidates for $A\cap L$:
\[\mbox{$\GO_6^+(q)$,\quad$\GO_{2m}^-(q)$,\quad$\Pa_1$.}\]
From the factorization $G=HB$ we know that $|L|/|B\cap L|$ divides $|H|$, that is,
\begin{equation}\label{eq11}
q^3(q^4-1)\di|H|.
\end{equation}
By Zsigmondy's theorem, $q^4-1$ has a primitive prime divisor $s$. Then $s$ divides $|H|$ as a consequence of (\ref{eq11}). Accordingly, we exclude the candidate $A\cap L=\Pa_1$ by Lemma \ref{Reduction}(d) since in this situation $A$ has the unique unsolvable composition factor $\PSp_4(q)$.

Assume $A\cap L=\GO_6^+(q)$. Then $A$ has the unique unsolvable composition factor $\PSL_4(q)$. Since $s$ divides $|H|$, we deduce from Lemma \ref{Reduction}(c) that $|H|_2$ divides $8f$. This together with (\ref{eq11}) yields that $q^3\di8f$, which is impossible.

Now assume $A\cap L=\GO_6^-(q)$. Note that $A=\Omega_6^-(q).\Z_{2f/e}=\SU_4(q).\Z_{2f/e}$, where $e\di2f$. Let $V$ be a vector space over $\GF(q^2)$ of dimension $4$ equipped with a non-degenerate unitary form $\beta$, and $u_1,u_2,u_3,u_4$ be a orthonormal basis of $V$. Let $\tau$ be the semilinear transformation of $V$ fixing the basis vectors $u_1,u_2,u_3,u_4$ such that $(\lambda v)^\tau=\lambda^{p^e}v$ for any $v\in V$ and $\lambda\in\GF(q^2)$. Then we can write $A=S{:}T$ with $S=\SU(V,\beta)$ and $T=\langle\tau\rangle=\Z_{2f/e}$ such that $A\cap B=(S\cap B){:}T$ and $S\cap B=\SU_3(q)$ is the stabilizer of $u_1$ in $S$ by \cite[5.2.3(b)]{liebeck1990maximal}. It derives from $G=HB$ that $A=H(A\cap B)$ and thus $A=H(S\cap B)T$. Denote the stabilizer of $u_1$ in $A$ by $N$. Take $e_1,e_2,f_1,f_2$ to be a basis of $V$ and $\mu\in\GF(q^2)$ such that $e_2+\mu f_2=u_1$ and
$$
\beta(e_i,e_j)=\beta(f_i,f_j)=0,\quad\beta(e_i,f_j)=\delta_{i,j}
$$
for any $i,j\in\{1,2\}$. Then $A=HN$ since $(S\cap B)T\subseteq N$. By Hypothesis \ref{Hypo}, this factorization should satisfy Theorem \ref{SolvableFactor}, that is, $H^x\leqslant M$ with some $x\in A$ and maximal solvable subgroup $M$ of $A$ stabilizing $\langle e_1,e_2\rangle$. However, this gives $A=H^xN=MN$, contrary to Lemma \ref{l4}.
\end{proof}

For rows 13--16 of Table \ref{tabSymplectic}, since the case $L=\PSp_4(3)$ has been treated in Lemma \ref{4-dimensionSymplectic}, we only need to consider rows 14--16.

\begin{lemma}\label{Symplectic2}
The maximal factorization $G=AB$ does not lie in rows~\emph{14--16} of \emph{Table \ref{tabSymplectic}} ($A,B$ may be interchanged).
\end{lemma}

\begin{proof}
By Lemma \ref{Reduction}(b), $A\cap L\neq\GO_8^-(2)$. We now deal with the remaining candidates in rows 14--16 of Table \ref{tabSymplectic}.

{\bf Case 1.} $L=\PSp_6(3)$, $A\cap L=\PSL_2(13)$ and $B\cap L=\Pa_1$. Since $|L|/|B\cap L|=(3^6-1)/(3-1)$, we deduce from $G=HB$ that $|H|$ is divisible by $7\cdot13$. However, Lemma \ref{Reduction}(c) shows that $|H|$ is not divisible by $7\cdot13$, a contradiction.

{\bf Case 2.} $L=\PSp_6(3)$, $A\cap L=\Pa_1$ and $B\cap L=\PSL_2(13)$. In order that $B$ is a maximal subgroup of $G$ such that $B\cap L=\PSL_2(13)$, we have $G=L$ (see \cite{atlas}). Consequently, $A=A\cap L=\Pa_1=3_+^{1+4}{:}\Sp_4(3)$.
Since $|L|/|B\cap L|$ is divisible by $2^7\cdot5$, we deduce from $G=HB$ that $|H|$ is divisible by $2^7\cdot5$. However, $\Sp_4(3)$ has no solvable subgroup of order divisible by $2^7\cdot5$, which indicates that $A$ has no solvable subgroup of order divisible by $2^7\cdot5$. Hence this case is not possible either.

{\bf Case 3.} $L=\Sp_8(2)$, $A\cap L=\Sy_{10}$ and $B\cap L=\GO_8^-(2)$. In this case, $G=L$ and $|G|/|B|=120$. Hence $|H|$ is divisible by $120$ according to the factorization $G=HB$. Applying Hypothesis \ref{Hypo} to the factorization $A=H(H\cap B)$, we view from Proposition~\ref{Alternating} that $H$ is a transitive subgroup of $\Sy_{10}$. However, $\Sy_{10}$ does not have a solvable transitive subgroup of order divisible by $120$, which is a contradiction.

{\bf Case 4.} $L=\Sp_8(2)$, $A\cap L=\PSL_2(17)$ and $B\cap L=\GO_8^+(2)$. In this case, $G=L$ and $A\cap B=\D_{18}$ (see \cite[5.1.9]{liebeck1990maximal}). Since $|G|/|B|$ is divisible by $2^3\cdot17$, we deduce from $G=HB$ that $|H|$ is divisible by $2^3\cdot17$. Hence $H$ must be $\Z_{17}{:}\Z_8$ as a solvable subgroup of $\PSL_2(17)$. However, this leads to $H\cap(A\cap B)=2$ and thus $|H||A\cap B|<|H\cap(A\cap B)||A|$, contradicting the factorization $A=H(A\cap B)$.

{\bf Case 5.} $L=\Sp_8(2)$, $A\cap L=\GO_8^+(2)$ and $B\cap L=\PSL_2(17)$. From the factorization $G=HB$ we deduce that $|H|$ is divisible by $|G|/|B|=2^{12}\cdot3^3\cdot5^2\cdot7$, but the only proper subgroup of $\GO_8^+(2)$ with order divisible by $2^{12}\cdot3^3\cdot5^2\cdot7$ is $\Omega_8^+(2)$, which is unsolvable. Hence this case is impossible too.
\end{proof}

\section{Unitary Groups}

This section is devoted to the proof of the proposition below, which confirms Theorem~\ref{SolvableFactor}
for unitary groups under Hypothesis \ref{Hypo}.

\begin{proposition}\label{Unitary}
Let $G$ be an almost simple group with socle $L=\PSU_n(q)$, where $n\geqslant3$. If $G=HK$ is a nontrivial factorization with $H$ solvable, $K$ unsolvable and $HL=KL=G$, then under \emph{Hypothesis \ref{Hypo}}, one of the following holds.
\begin{itemize}
\item[(a)] $n=2m$, $H\cap L\leqslant q^{m^2}{:}((q^{2m}-1)/((q+1)(n,q+1))).m<\Pa_m$, and $\SU_{2m-1}(q)\trianglelefteq K\cap L\leqslant\N_1$.
\item[(b)] $L=\PSU_3(q)$ with $q=\in\{3,5\}$ or $\PSU_4(q)$ with $q\in\{2,3,8\}$, and $(L,H\cap L,K\cap L)$ is as described in \emph{Table \ref{tab15}}.
\end{itemize}
\end{proposition}

\begin{table}[htbp]
\caption{}\label{tab15}
\centering
\begin{tabular}{|l|l|l|l|}
\hline
row & $L$ & $H\cap L\leqslant$ & $K\cap L$\\
\hline
1 & $\PSU_4(2)$ & $3^3{:}\Sy_4$ & $2^4{:}\A_5$\\
2 & $\PSU_4(2)$ & $3_+^{1+2}{:}2.\A_4$ & $\A_5$, $2^4{:}\A_5$, $\Sy_5$, $\A_6$, $\Sy_6$\\
3 & $\PSU_3(3)$ & $3_+^{1+2}{:}8$ & $\PSL_2(7)$\\
4 & $\PSU_3(5)$ & $5_+^{1+2}{:}8$ & $\A_7$\\
\hline
5 & $\PSU_4(3)$ & $3^4{:}\D_{10}$, $3^4{:}\Sy_4$, & $\PSL_3(4)$\\
 & & $3^4{:}3^2{:}4$, $3_+^{1+4}.2.\Sy_4$ & \\
\hline
6 & $\PSU_4(8)$ & $513{:}3$ & $2^{12}{:}\SL_2(64).7$\\
\hline
\end{tabular}
\end{table}

We start the proof of the proposition by taking $A,B$ to be core-free maximal subgroups of $G$ containing $H,K$, respectively (such $A$ and $B$ are existent by Lemma \ref{Embedding}). The maximal factorizations $G=AB$ are listed in Table \ref{tabUnitary} by Theorem \ref{Maximal}. If $n$ is prime, then the proposition holds clearly by Theorem~\ref{p-dimensionUnitary}. We thus assume that $n$ is a composite number. For $L=\PSU_4(2)\cong\PSp_4(3)$, it is seen from Proposition \ref{Small1} that Proposition \ref{Unitary} is true. For $L=\PSU_4(3)$, computation in \magma \cite{bosma1997magma} shows that part (a) or (b) of Proposition \ref{Unitary} holds. Thus assume $(n,q)\not\in\{(4,2),(4,3)\}$ in the following.

Under the above assumptions, we only need to consider rows 1--4 and 10--12 of Table \ref{tabUnitary}. After analysis for these rows below, the proposition will follow by Lemmas \ref{Unitary2} and \ref{Unitary3}.

\begin{lemma}\label{Unitary1}
If $(n,q)\not\in\{(4,2),(4,3)\}$ and the maximal factorization $G=AB$ lies in rows~\emph{1--4} of \emph{Table \ref{tabUnitary}} (interchanging $A,B$ if necessary), then either $H\cap L\leqslant\Pa_m$ and $B\cap L=\N_1$ or $(L,H\cap L,K\cap L)$ is as described in row~\emph{6} of \emph{Table \ref{tab15}}.
\end{lemma}

\begin{proof}
As in rows 1--4 of Table \ref{tabUnitary}, $n=2m\geqslant4$ and either $A\cap L$ or $B\cap L$ is $\N_1$.

{\bf Case 1.} Assume that $A\cap L=\N_1$. Then $A$ has the unique unsolvable composition factor $\PSU_{2m-1}(q)$. By Lemma \ref{Reduction}(e), we have $m=2$ and $q\in\{5,8\}$ since $(n,q)\neq(4,3)$. Consequently, neither row 3 nor 4 of Table \ref{tabUnitary} appears. Now $B\cap L=\Pa_2$ or $\PSp_4(q).a$ with $a\leqslant2$ as in row 1 or row 2 of Table \ref{tabUnitary}. One checks directly for $q\in\{5,8\}$ that $r=(q^2-q+1)/3$ is a prime number dividing $|L|/|B\cap L|$. It follows then from the factorization $G=HB$ that $r$ divides $|H|$, which leads to $q=8$ and $H\cap L\leqslant\GU_1(8^3){:}3=513{:}3$ by Lemma \ref{Reduction}(e). In in turn implies that $|L|$ divides $513\cdot3|K\cap L||\Out(L)|$, that is, $2^{17}\cdot3^2\cdot5\cdot7^2\cdot13$ divides $|K\cap L|$. Thereby we obtain $K\cap L=\Pa_2=2^{12}{:}\SL_2(64).7$ as in row~6 of Table \ref{tab15}.

{\bf Case 2.} Assume that $B\cap L=\N_1=\hat{~}\GU_{2m-1}(q)$ and one of the following four cases appears.
\begin{itemize}
\item[(i)] $A\cap L=\Pa_m$.
\item[(ii)] $A\cap L=\PSp_{2m}(q).a$ with $a=(2,q-1)(m,q+1)/(n,q+1)\leqslant2$.
\item[(iii)] $q=2$ and $A\cap L=\hat{~}\SL_m(4).2$.
\item[(iv)] $q=4$ and $A\cap L=\hat{~}\SL_m(16).3.2$.
\end{itemize}
Let $q=p^f$ with $p$ prime. For case (i), the lemma already holds. By Zsigmondy's theorem, $p^{fn}-1$ has a primitive prime divisor $r$ if $(n,q)\neq(6,2)$. Let $r=7$ if $(n,q)=(6,2)$. Since $G=HB$ and $r$ divides $|L|/|B\cap L|$, it follows that $r$ divides $|H|$ by Lemma \ref{p4}.

First suppose that case (ii) appears. Then from Lemma \ref{Reduction}(d) we conclude that $(n,q)=(6,2)$ and $H\leqslant\Pa_3[A]$. Note that $\Pa_3[\PSp_6(2)]\leqslant\Pa_3[\PSU_6(2)]$. We thereby obtain $H\cap L\leqslant\Pa_3[\PSU_6(2)]$ as the lemma asserts.

Next suppose that case (iii) or (iv) appears. Then $p=2$ and $f\in\{1,2\}$. Write $R=\Rad(A)$ and $A/R=\PSL_m(q^2).\calO$. Noticing that $r$ divides $|H|$, we conclude from Lemma \ref{Reduction}(c) that
$$
H\leqslant R.\left(\frac{q^{2m}-1}{(q^2-1)(m,q^2-1)}.m\right).\calO.
$$
In particular, $|H|_2$ divides $m(|R||\calO|)_2$. This together with the equality $|R||\calO|=|A|/|\PSL_m(q^2)|$ implies that $|H|_2$ divides $2m(|A|/|A\cap L|)_2=2m|G/L|_2$ and thus divides $4fm$. However, the factorization $G=HB$ indicates that $|H|_2$ is divisible by $(|G|/|B|)_2$. Hence $q^{2m-1}=(|G|/|B|)_2$ divides $4fm$, which is impossible as $(n,q)\neq(4,2)$.
\end{proof}

\begin{lemma}\label{Unitary2}
If $(n,q)\not\in\{(4,2),(4,3)\}$ and the maximal factorization $G=AB$ lies in rows~\emph{1--4} of \emph{Table \ref{tabUnitary}} (interchanging $A,B$ if necessary), then either part~\emph{(a)} of \emph{Proposition \ref{Unitary}} holds or $(L,H\cap L,K\cap L)$ is as described in row~\emph{6} of \emph{Table \ref{tab15}}.
\end{lemma}

\begin{proof}
By Lemma \ref{Unitary1}, we may assume that $A\cap L=\Pa_m$ with $n=2m$, and $B\cap L=\N_1$. Let $d=(n,q+1)$ and $q=p^f$ with $p$ prime. Note that $\Pa_m=q^{m^2}{:}(\SL_m(q^2).(q-1)/d)$ and $\N_1=\GU_{n-1}(q)/d$. Then $A$ has the unique unsolvable composition factor $\PSL_m(q^2)$. By Zsigmondy's theorem, $(q^2,m)$ has a primitive prime divisor $r$. It derives from the factorization $G=HB$ that $r$ divides $|H|$ since $r$ divides $|L|/|B\cap L|$. Hence by Lemma~\ref{Reduction}(c),
$$
H\leqslant\Rad(A).\left(\frac{q^{2m}-1}{(q^2-1)(m,q^2-1)}.m\right).\calO=\left(q^{m^2}.\frac{q^{2m}-1}{(q+1)d}.m\right).(G/L)
$$
with $A/\Rad(A)=\PSL_m(q^2).\calO$. Accordingly, we have $H\cap L\leqslant q^{m^2}{:}((q^{2m}-1)/((q+1)d)).m$.

Suppose $m=2$. Then the conclusion $H\cap L\leqslant q^4{:}((q^4-1)/((q+1)d)).2$ above implies that $|L|$ divides
$$
\frac{2q^4(q^4-1)|K\cap L||\Out(L)|}{(q+1)d}=4fq^4(q^2-1)(q-1)|K\cap L|.
$$
due to the factorization $G=HK$. It follows that $q^2(q^3+1)(q^2-1)(q+1)$ divides $4df|K\cap L|$, and thus
$$
\frac{|B\cap L|}{|K\cap L|}\leqslant\frac{4df|B\cap L|}{q^2(q^3+1)(q^2-1)(q+1)}=4fq.
$$
Note that each proper subgroup of $\SU_3(q)$ has index greater than $4fq$ by Theorem~\ref{l7}. We thereby obtain $\SU_3(q)\trianglelefteq K\cap L$ from Lemma \ref{MinimalDegree}, whence part (a) of Proposition \ref{Unitary} appears.

Now suppose $m\geqslant3$. Let $R=\Rad(B)$, $\overline{B}=B/R$, $\overline{H\cap B}=(H\cap B)R/R$ and $\overline{K}=KR/R$. By Lemma \ref{p12}, $\overline{B}$ is almost simple with socle $\PSU_{2m-1}(q)$. From $G=HK$ we deduce $B=(H\cap B)K$ and further $\overline{B}=\overline{H\cap B}\,\overline{K}$. Note that $\overline{H\cap B}$ is solvable, and $\overline{B}$ does not have any solvable nontrivial factor by Hypothesis \ref{Hypo}. We then have $\PSU_{2m-1}(q)\trianglelefteq\overline{K}$. Therefore, $K\cap L$ contains $(B\cap L)^{(\infty)}=\SU_{2m-1}(q)$, and so $\SU_{2m-1}(q)\trianglelefteq K\cap L\leqslant B\cap L=\N_1$ as part (a) of Proposition \ref{Unitary}. This completes the proof.
\end{proof}

\begin{lemma}\label{Unitary3}
The maximal factorization $G=AB$ does not lie in rows~\emph{10--12} of \emph{Table \ref{tabUnitary}} (interchanging $A,B$ if necessary).
\end{lemma}

\begin{proof}
For the factorization $G=AB$ in these rows, $A$ has a unique unsolvable composition factor $S$. By Lemma \ref{l3} and Hypothesis \ref{Hypo}, $S\neq\PSU_5(2)$, $\J_3$, $\PSU_7(2)$ or $\PSU_{11}(2)$. Thus we have the following three cases for $(L,A\cap L,B\cap L)$.
\begin{itemize}
\item[(i)] $L=\PSU_6(2)$, $A\cap L=\PSU_4(3).2$ and $B\cap L=\N_1$.
\item[(ii)] $L=\PSU_6(2)$, $A\cap L=\M_{22}$ and $B\cap L=\N_1$.
\item[(iii)] $L=\PSU_{12}(2)$, $A\cap L=\Suz$ and $B\cap L=\N_1$.
\end{itemize}

First assume that case (i) appears. Then by Lemma \ref{Reduction}(f), $|H|$ is not divisible by $7$. However, $|G|/|B|$ is divisible by $7$, which is a contradiction to the factorization $G=HB$.

Next assume the case (ii) appears. Then by Lemma \ref{l3} and Hypothesis \ref{Hypo}, $|H|$ is not divisible by $7$ as row 5 of Table \ref{tab8} suggests. However, $|G|/|B|$ is divisible by $7$, contrary to the factorization $G=HB$.

Finally assume the case (iii) appears. Then again by Lemma \ref{l3} and Hypothesis \ref{Hypo}, $|H|$ is not divisible by $7$ as row 13 of Table \ref{tab8} suggests. However, $|G|/|B|$ is divisible by $7$, contrary to the factorization $G=HB$.
\end{proof}

\section{Orthogonal groups of odd dimension}

In this section we verify Theorem \ref{SolvableFactor} for orthogonal groups of odd dimension under Hypothesis \ref{Hypo}. First of all, we compute in \magma \cite{bosma1997magma} all the nontrivial factorizations of $G$ with $\Soc(G)=\Omega_7(3)$, and list them in the following proposition.

\begin{proposition}\label{Small2}
Let $G$ be an almost simple group with socle $L=\Omega_7(3)$. Then the following three cases give all the nontrivial factorizations $G=HK$ with $H$ solvable.
\begin{itemize}
\item[(a)] $G=\Omega_7(3)$, $H<\Pa_k[\Omega_7(3)]$, and $(G,H,K,k)$ lies in rows~\emph{1--2} of \emph{Table \ref{tab12}}; moreover, $\Pa_k[\Omega_7(3)]$ is the only maximal subgroup of $G$ containing $H$.
\item[(b)] $G=\SO_7(3)$, and $L=(H\cap L)(K\cap L)$; in particular, $H=(H\cap L).\calO_1$ and $K=(K\cap L).\calO_2$, where $\calO_i\leqslant\Z_2$ for $i\in\{1,2\}$ and $\calO_1\calO_2=\Z_2$.
\item[(c)] $G=\SO_7(3)$, $L\neq(H\cap L)(K\cap L)$, $H<\Pa_3[\SO_7(3)]$, $K<\N_1^-[\SO_7(3)]$ and $(G,H,K,k)$ lies in row~\emph{3} of \emph{Table \ref{tab12}}; moreover, $\Pa_3[\SO_7(3)]$ is the only maximal subgroup of $G$ containing $H$.
\end{itemize}
\end{proposition}

\begin{table}[htbp]
\caption{}\label{tab12}
\centering
\begin{tabular}{|l|l|l|l|l|}
\hline
row & $G$ & $H$ & $K$ & $k$\\
\hline
1 & $\Omega_7(3)$ & $3^5{:}2^4.5$, $3^5{:}2^4.\D_{10}$, $3^5{:}2^4.\AGL_1(5)$ & $\G_2(3)$ & $1$\\
2 & $\Omega_7(3)$ & $3^{3+3}{:}13$, $3^{3+3}{:}13{:}3$ & $\N_1^-$, $\Sp_6(2)$ & $3$\\
\hline
3 & $\SO_7(3)$ & $3^{3+3}{:}13{:}2$, $3^{3+3}{:}13{:}6$ & $\Omega_6^-(3).2$ & $3$\\
\hline
\end{tabular}
\end{table}

Now we state the main result of this section.

\begin{proposition}\label{Omega}
Let $G$ be an almost simple group with socle $L=\Omega_{2m+1}(q)$ with $m\geqslant3$ and $q$ odd. If $G=HK$ is a nontrivial factorization with $H$ solvable, $K$ unsolvable and $HL=KL=G$, then under \emph{Hypothesis \ref{Hypo}}, one of the following holds.
\begin{itemize}
\item[(a)] $H\cap L\leqslant(q^{m(m-1)/2}.q^m){:}((q^m-1)/2).m<\Pa_m$, and $\Omega_{2m}^-(q)\trianglelefteq K\cap L\leqslant\N_1^-$.
\item[(b)] $L=\Omega_7(3)$ or $\Omega_9(3)$, and $(L,H\cap L,K\cap L)$ is as described in \emph{Table \ref{tab16}}.
\end{itemize}
\end{proposition}

\begin{table}[htbp]
\caption{}\label{tab16}
\centering
\begin{tabular}{|l|l|l|l|}
\hline
row & $L$ & $H\cap L\leqslant$ & $K\cap L$\\
\hline
1 & $\Omega_7(3)$ & $3^5{:}2^4.\AGL_1(5)$ & $\G_2(3)$\\
2 & $\Omega_7(3)$ & $3^{3+3}{:}13{:}3$ & $\Sp_6(2)$\\
3 & $\Omega_9(3)$ & $3^{6+4}{:}2^{1+4}.\AGL_1(5)$ & $\Omega^-_8(3)$, $\Omega^-_8(3).2$\\
\hline
\end{tabular}
\end{table}

\begin{proof}
Let $q=p^f$ with odd prime $p$. By Lemma \ref{Embedding}, we may take $A,B$ to be core-free maximal subgroups of $G$ containing $H,K$ respectively. For $L=\Omega_7(3)$, the proposition follows immediately from Proposition \ref{Small2}. Thus we assume $(m,q)\neq(3,3)$ for the rest of the proof. Under this assumption, the maximal factorizations $G=AB$ lies in rows 1--5 of Table \ref{tabOmega} (interchanging $A,B$ if necessary) by Theorem \ref{Maximal}. Further by Lemma \ref{Reduction}, the possibilities for $A\cap L$ are:
\[\mbox{$\N_1^-$ with $m=3$,\quad$\Pa_m$,\quad$\Pa_1$ with $m=3$,\quad$\N_1^+$ with $m=3$,}\]
\[\mbox{$\N_2^-$ with $m=3$,\quad$\N_2^+$ with $m=3$,\quad$\PSp_6(3).a$ with $(m,q)=(6,3)$ and $a\leqslant2$.}\]

{\bf Case 1.} Suppose that $A\cap L=\N_1^-$ with $m=3$. Then $A$ has the unique unsolvable composition factor $\PSU_4(q)$, and $B\cap L=\Pa_m$ or $\G_2(q)$ seen from Table \ref{tabOmega}. By Zsigmondy'e theorem, $p^{6f}-1$ has a primitive prime divisor $r$, and we conclude from Lemma \ref{Reduction}(f) that $r$ does not divide $|H|$. It follows that $r$ divides $|B\cap L|$ according to the factorization $G=HB$. This implies that $B\cap L\neq\Pa_m$, and thus $B\cap L=\G_2(q)$. Consequently, we have $\SO_7(q)\nleqslant G$ by \cite[Proposition 5.7.2]{BHR-book}, since $B$ is a maximal subgroup of $G$. Hence $A=\SU_4(q).\Z_{2f/e}$ with $e\di2f$.

Let $V$ be a vector space over $\GF(q^2)$ of dimension $4$ equipped with a non-degenerate unitary form $\beta$, and $u_1,u_2,u_3,u_4$ be a orthonormal basis of $V$. Let $\tau$ be the semilinear transformation of $V$ fixing the basis vectors $u_1,u_2,u_3,u_4$ such that $(\lambda v)^\tau=\lambda^{p^e}v$ for any $v\in V$ and $\lambda\in\GF(q^2)$. Then we can write $A=S{:}T$ with $S=\SU(V,\beta)$ and $T=\langle\tau\rangle=\Z_{2f/e}$ such that $A\cap B=(S\cap B){:}T$ and $S\cap B=\SU_3(q)$ is the stabilizer of $u_1$ in $S$ by \cite[5.1.14]{liebeck1990maximal}. It derives from $G=HB$ that $A=H(A\cap B)$ and thus $A=H(S\cap B)T$. Denote the stabilizer of $u_1$ in $A$ by $N$. Take $e_1,e_2,f_1,f_2$ to be a basis of $V$ and $\mu\in\GF(q^2)$ such that $e_2+\mu f_2=u_1$ and
$$
\beta(e_i,e_j)=\beta(f_i,f_j)=0,\quad\beta(e_i,f_j)=\delta_{i,j}
$$
for any $i,j\in\{1,2\}$. Then $A=HN$ since $(S\cap B)T\subseteq N$. By Hypothesis \ref{Hypo}, this factorization should satisfy Theorem \ref{SolvableFactor}, that is, $H^x\leqslant M$ with some $x\in A$ and maximal solvable subgroup $M$ of $A$ stabilizing $\langle e_1,e_2\rangle$. However, this gives $A=H^xN=MN$, contrary to Lemma \ref{l4}.

{\bf Case 2.} Suppose that $A\cap L=\Pa_m$. Then it is seen from Table \ref{tabOmega} that $B\cap L=\N_1^-$. Noticing $\Pa_m=(q^{m(m-1)/2}.q^m){:}\SL_m(q).((q-1)/2)$, we conclude that $A$ has the unique unsolvable composition factor $\PSL_m(q)$. By Zsigmondy's theorem, $p^{2fm}-1$ has a primitive prime divisor $r$, and $p^{fm}-1$ has a primitive prime divisor $s$. Note that $s$ divides $|H|$ as $s$ divides $|L|/|B\cap L|$. Then by Lemma \ref{Reduction}(c), either
$$
H\leqslant R.\left(\frac{q^m-1}{(q-1)(m,q-1)}.m\right).\calO=\left((q^{m(m-1)/2}.q^m){:}\frac{q^m-1}{2}.m\right).(G/L),
$$
or $(m,q)=(4,3)$ and $H\leqslant R.(2^4{:}5{:}4).\calO$, where $R,\calO$ are defined in Lemma~\ref{Reduction} with $A/R=\PSL_m(q).\calO$. Thus either
\[\mbox{$H\cap L\leqslant (q^{m(m-1)/2}.q^m){:}((q^m-1)/2).m$, or}\]
\[\mbox{$L=\Omega_9(3)$ and $H\cap L\leqslant3^{6+4}{:}2^{1+4}.\AGL_1(5)$.}\]
As a consequence, $r$ does not divide $|H|$. This in conjunction with the factorization $G=HK$ yields that $r$ divides $|K|$ since $r$ divides $|L|$.

Let $R=\Rad(B)$, $\overline{B}=B/R$, $\overline{H\cap B}=(H\cap B)R/R$ and $\overline{K}=KR/R$. By Lemma \ref{p12}, $\overline{B}$ is almost simple with socle $\POm_{2m}^-(q)$. From $G=HK$ we deduce $B=(H\cap B)K$ and further $\overline{B}=\overline{H\cap B}\,\overline{K}$. Moreover, $\overline{H\cap B}$ is solvable and $r$ divides $|\overline{K}|$ since $r$ divides $|K|$. By Hypothesis \ref{Hypo}, $\overline{B}$ does not have any solvable nontrivial factor of order divisible by $r$. We then conclude that $\POm_{2m}^-(q)\trianglelefteq\overline{K}$. Hence $K\cap L$ contains $(B\cap L)^{(\infty)}=\Omega_{2m}^-(q)$, and so $\Omega_{2m}^-(q)\trianglelefteq K\cap L\leqslant B\cap L=\N_1^-$.

To sum up, we have shown in this case that either part (a) of Proposition \ref{Omega} holds, or the triple $(L,H\cap L,K\cap L)$ is described in row 3 of Table \ref{tab16}.

{\bf Case 3.} Suppose that $A\cap L\in\{\Pa_1,\N_1^+,\N_2^-,\N_2^+\}$ with $m=3$. In this case, we have $B\cap L=\G_2(q)$ from Table \ref{tabOmega}. By Zsigmondy's theorem, $p^{4f}-1$ has a primitive prime divisor $r$, and $r$ divides $|H|$ since $r$ divides $|L|/|B\cap L|$. Thereby we conclude from Lemma~\ref{Reduction}(c) and (d) that $A\cap L=\N_1^+$ and $H\leqslant R.(((q^2+1)(q+1)/(4,q-1)).4).\calO$ with $R=\Rad(A)$ and $A/R=\PSL_4(q).\calO$. This implies that $|H|_p$ divides
$$
(|R||\calO|)_p=(|A|/|\PSL_4(q)|)_p=(|A|/|A\cap L|)_p=|G/L|_p
$$
and thus divides $f$. As the factorization $G=HB$ requires $|H|_p$ to be divisible by $|L|_p/|B\cap L|_p=q^3$, we then obtain a contradiction that $q^3\di f$.

{\bf Case 4.} Finally suppose that $\PSp_6(3).a$ with $(m,q)=(6,3)$ and $a\leqslant2$. Then we have $B\cap L=\N_1^-$ as Table \ref{tabOmega} shows. It follows that $|L|/|B\cap L|$ is divisible by $7$, and thereby the factorization $G=HB$ forces $|H|$ to be divisible by $7$. However, we view from Lemma~\ref{Reduction}(d) that $|H|$ is not divisible by $7$. Thus this case is impossible too. The proof is completed.
\end{proof}

\section{Orthogonal groups of even dimension}

The main result of this section is the proposition below, which confirms Theorem~\ref{SolvableFactor}
for orthogonal groups of even dimension under Hypothesis \ref{Hypo}.

\begin{proposition}\label{OmegaPlus}
Let $G$ be an almost simple group with socle $L=\POm_{2m}^\varepsilon(q)$ with $m\geqslant4$ and $\varepsilon=\pm$. If $G=HK$ is a nontrivial factorization with $H$ solvable, $K$ unsolvable and $HL=KL=G$, then under \emph{Hypothesis \ref{Hypo}}, $\varepsilon=+$ and one of the following holds.
\begin{itemize}
\item[(a)] $m\geqslant5$, $H\cap L\leqslant q^{m(m-1)/2}{:}((q^m-1)/(4,q^m-1)).m<\Pa_m$ or $\Pa_{m-1}$, and $\Omega_{2m-1}(q)\trianglelefteq K\cap L\leqslant\N_1$.
\item[(b)] $m=4$, $H\cap L\leqslant q^6{:}((q^4-1)/(4,q^4-1)).4<\Pa_1$ or $\Pa_3$ or $\Pa_4$, and $K\cap L=\Omega_7(q)$.
\item[(c)] $L=\Omega_8^+(2)$ or $\POm_8^+(3)$, and $(L,H\cap L,K\cap L)$ is as described in \emph{Table \ref{tab17}}.
\end{itemize}
\end{proposition}

\begin{table}[htbp]
\caption{}\label{tab17}
\centering
\begin{tabular}{|l|l|l|l|}
\hline
row & $L$ & $H\cap L\leqslant$ & $K\cap L$\\
\hline
1 & $\Omega_8^+(2)$ & $2^2{:}15.4$ & $\Sp_6(2)$\\
2 & $\Omega_8^+(2)$ & $2^6{:}15.4$ & $\A_9$\\
3 & $\POm_8^+(3)$ & $3^6{:}2^4.\AGL_1(5)$& $\Omega_7(3)$\\
4 & $\POm_8^+(3)$ & $3^6{:}(3^3{:}13{:}3)$, $3^{3+6}{:}13.3$ & $\Omega_8^+(2)$\\
\hline
\end{tabular}
\end{table}

We treat the orthogonal groups of minus type, plus type of dimension $8$, and plus type of dimension at least $10$, respectively in the following three subsections. The above proposition will follow by Lemmas\ref{OmegaMinus}, \ref{OmegaPlus1} and \ref{OmegaPlus2}. Throughout this section, fix $G,L,H,K$ to be the groups in the condition of Proposition \ref{Symplectic}, and take $A,B$ to be core-free maximal subgroups of $G$ containing $H,K$ respectively (such $A$ and $B$ exist by Lemma \ref{Embedding}).

\subsection{Orthogonal groups of minus type}

First, we exclude the possibility of orthogonal groups of minus type.

\begin{lemma}\label{OmegaMinus}
Let $G$ be an almost simple group with socle $L$. If $G=HK$ is a nontrivial factorization with $H$ solvable, $K$ unsolvable and $HL=KL=G$, then under \emph{Hypothesis \ref{Hypo}} we have $L\neq\POm_{2m}^-(q)$ for $m\geqslant4$.
\end{lemma}

\begin{proof}
Suppose on the contrary that $L=\POm_{2m}^-(q)$ with $m\geqslant4$. By Lemma \ref{Embedding}, we may take $A,B$ to be core-free maximal subgroups of $G$ containing $H,K$ respectively. By Theorem \ref{Maximal}, the maximal factorizations $G=AB$ lies in Table \ref{tabOmegaMinus} (interchanging $A,B$ if necessary). If $m$ is odd, then we have by Lemma \ref{Reduction}(b) and (e) that $A\cap L\neq\Pa_1$, $\N_2^+$ or $\hat{~}\GU_m(q)$. Hence the candidates for $A\cap L$ are:
\[\mbox{$\N_1$,\quad$\Omega_m^-(q^2).2$ with $m$ even and $q\in\{2,4\}$,\quad$\A_{12}$ with $(m,q)=(5,2)$.}\]
Let $q=p^f$ for prime number $p$. We proceed by the above three cases for $A\cap L$.

{\bf Case 1.} Suppose $A\cap L=\N_1$. In this case, as Table \ref{tabOmegaMinus} shows, either $B\cap L=\hat{~}\GU_m(q)$ with $q$ odd, or $B\cap L=\Omega_m^-(q^2).2$ with $m$ even and $q\in\{2,4\}$. If $(m,q)\neq(4,2)$, then $p^{f(2m-2)}-1$ has a primitive prime divisor $r$ by Zsigmondy's theorem. If $(m,q)=(4,2)$, then let $r=63$. Then $r$ divides $|L|/|B\cap L|$, and thus $r$ divides $|H|$ due to the factorization $G=HK$. However, since $A$ has the unique unsolvable composition factor $\POm_{2m-1}^-(q)$, we deduce from Lemma \ref{Reduction}(d) and (g) that $|H|$ is not divisible by $r$, a contradiction.

{\bf Case 2.} Suppose that $A\cap L=\Omega_m^-(q^2).2$ with $m$ even and $q\in\{2,4\}$. Then $p=2$ and it is seen from Table \ref{tabOmegaMinus} that $B\cap L=\N_1$ and $G=\Aut(L)$. By Zsigmondy's theorem, $2^{2fm}-1$ has a primitive prime divisor $r$. Since $r$ divides $|L|/|B\cap L|$, we know that $r$ divides $|H|$ according to the factorization $G=HB$. It follows that $m=4$ and $H\leqslant R.\D_{2(q^4+1)}.\calO$ by Lemma \ref{Reduction}(b), (c) and (f), where $R,\calO$ are defined in Lemma \ref{Reduction} with $A/R=\PSL_2(q^4).\calO$. In particular, $|H|$ divides
$$
2(q^4+1)|R||\calO|=\frac{2(q^4+1)|A|}{|\PSL_2(q^4)|}=\frac{4(q^4+1)|A|}{|A\cap L|}=4(q^4+1)|\Out(L)|=8f(q^4+1),
$$
which implies that $|G|/|B|$ divides $8f(q^4+1)$ due to the factorization $G=HB$. As a consequence, $q^3=(|G|/|B|)_2$ divides $8f$. This restricts $q=2$, whence $L=\POm_8^-(2)$. However, computation in \magma \cite{bosma1997magma} shows that $G=\SO_8^-(2)$ allows no factorization $G=HB$ with $H$ solvable. Thus this case is impossible too.

{\bf Case 3.} Finally suppose that $A\cap L=\A_{12}$ with $(m,q)=(5,2)$. Then $B\cap L=\Pa_1$, $\A_{12}\trianglelefteq A\leqslant\Sy_{12}$, and $A\cap B=(\Sy_4\times\Sy_8)\cap A$ (see \cite[5.2.16]{liebeck1990maximal}). This indicates that the factorization $A=H(A\cap B)$ does not satisfy Theorem \ref{SolvableFactor}, contrary to Hypothesis \ref{Hypo}. This completes the proof.
\end{proof}

\subsection{Orthogonal groups of dimension eight}

We find out in \magma \cite{bosma1997magma} all the nontrivial factorizations of $G$ with $\Soc(G)=\POm_8^+(2)$ in the next proposition.

\begin{proposition}\label{Small3}
Let $G$ be an almost simple group with socle $L=\Omega_8^+(2)$. Then the following five cases give all the nontrivial factorizations $G=HK$ with $H$ solvable.
\begin{itemize}
\item[(a)] $G=\Omega_8^+(2)$, $H$ is contained in $\Pa_k$ for some $k\in\{1,3,4\}$, and $(H,K)$ lies in row~\emph{5} of \emph{Table \ref{tab18}}; moreover $H$ is not contained in any maximal subgroup of $G$ other than $\Pa_1$, $\Pa_3$ and $\Pa_4$.
\item[(b)] $G=\Omega_8^+(2)$, $H$ is contained in $\A_9$, and $(H,K)$ lies in row~\emph{1} of \emph{Table \ref{tab18}}; moreover $H$ is not contained in any maximal subgroup of $G$ other than $\A_9$.
\item[(c)] $G=\Omega_8^+(2)$, $H$ is contained in $(3\times\PSU_4(2)){:}2$ and $\Pa_k$ simultaneously for some $k\in\{1,3,4\}$, and $(H,K)$ lies in rows~\emph{2--4} of \emph{Table \ref{tab18}}; moreover $H$ is not contained in any maximal subgroup of $G$ other than $\Pa_1$, $\Pa_3$, $\Pa_4$ and $(3\times\PSU_4(2)){:}2$.
\item[(d)] $G=\Omega_8^+(2)$, $H$ is contained in $(\PSL_2(4)\times\PSL_2(4)).2^2$ and $\Pa_k$ simultaneously for some $k\in\{1,3,4\}$, and $(H,K)$ lies in row~\emph{1} of \emph{Table \ref{tab18}}; moreover $H$ is not contained in any maximal subgroup of $G$ other than $\Pa_1$, $\Pa_3$, $\Pa_4$ and $(\PSL_2(4)\times\PSL_2(4)).2^2$.
\item[(e)] $G\neq\Omega_8^+(2)$, and $L=(H\cap L)(K\cap L)$.
\end{itemize}
\end{proposition}

\begin{table}[htbp]
\caption{}\label{tab18}
\centering
\begin{tabular}{|l|l|l|l|}
\hline
row & maximal subgroups of $G$ containing $H$ & $H$ & $K$\\
\hline
1 & $\Pa_1$, $\Pa_3$, $\Pa_4$, $\A_9$, $(\PSL_2(4)\times\PSL_2(4)).2^2$ & $2^2{:}15.4$ & $\Sp_6(2)$\\
2 & $\Pa_1$, $\Pa_3$, $\Pa_4$, $(3\times\PSU_4(2)){:}2$ & $2^4{:}15$ & $\Sp_6(2)$\\
3 & $\Pa_1$, $\Pa_3$, $\Pa_4$, $(3\times\PSU_4(2)){:}2$ & $2^4{:}15.2$ & $\Sp_6(2)$\\
4 & $\Pa_1$, $\Pa_3$, $\Pa_4$, $(3\times\PSU_4(2)){:}2$ & $2^4{:}15.4$ & $\Sp_6(2)$, $\A_9$\\
5 & $\Pa_1$, $\Pa_3$, $\Pa_4$ & $2^6{:}15$, $2^6{:}15.2$, $2^6{:}15.4$ & $\Sp_6(2)$, $\A_9$\\
\hline
\end{tabular}
\end{table}

For the rest of this subsection, let $d=(2,q-1)$ and $L=\POm_8^+(q)$ with $q=p^f$ for prime number $p$. We aim to show that part (b) or (c) of Proposition \ref{OmegaPlus} appears in this situation.

\begin{lemma}\label{OmegaPlus4}
If $q\geqslant3$ and $B\cap L=\Omega_7(q)$, then $H$ stabilizes a totally isotropic $k$-space with $k=1$ or $4$.
\end{lemma}

\begin{proof}
By Proposition \ref{p1}, we may assume that $A$ has at most one unsolvable composition factor. Then in view of Lemma \ref{Reduction}(b), we see from Table \ref{tabOmegaPlus2} that the possibilities for $A\cap L$ are:
\[\mbox{$\Pa_1$,\quad$\Pa_3$,\quad$\Pa_4$,\quad$\Omega_7(q)$,\quad$\hat{~}((q+1)/d\times\Omega_6^-(q)).2^d$,}\]
\[\mbox{$\hat{~}((q-1)/d\times\Omega_6^+(q)).2^d$,\quad$\Omega_8^+(2)$ with $q=3$.}\]
If $A\cap L=\Pa_1$  or $\Pa_3$ or $\Pa_4$, then the lemma holds already. We deal with the remaining candidates one by one below. By Zsigmondy's theorem, $p^{4f}-1$ has a primitive prime divisor $r$, and $r$ divides $|H|$ as $r$ divides $|G|/|B|$.

{\bf Case 1.} Suppose $A\cap L=\Omega_7(q)$. Since $r$ divides $|H|$, it derives from Lemma \ref{Reduction}(d) and (g) that $q=3$ and $H\leqslant\Rad(A).(3^5{:}2^4.\AGL_1(5)).2$. Note that $L$ has a graph automorphism of order $3$ (see \cite[(15.1)]{aschbacher1984maximal}) which permutes $\{\Pa_1,\Pa_3,\Pa_4\}$. We may assume without loss of generality that $A\cap L=\N_1$. Then by Proposition \ref{Small2}, we have $H<\Pa_1[A]$. Consequently, $H$ stabilizes a totally isotropic $1$-space since $\Pa_1[A]\leqslant\Pa_1[G]$.

{\bf Case 2.} Suppose $A\cap L=\hat{~}((q+1)/d\times\Omega_6^-(q)).2^d$. Note that $L$ has a graph automorphism of order $3$ which sends $\N_2^-$ to a $\mathcal{C}_3$ subgroup $\hat{~}\GU_4(q).2$ of $L$ and permutes $\{\Pa_1,\Pa_3,\Pa_4\}$ \cite[(15.1)]{aschbacher1984maximal}. We may assume without loss of generality that $A\cap L=\hat{~}\GU_4(q).2\in\mathcal{C}_3$. Since $r$ divides $|H|$, we deduce from Lemma \ref{l3} and Hypothesis \ref{Hypo} that $H\cap L\leqslant\Pa_2[\hat{~}\GU_4(q).2]$. A totally singular unitary $2$-space over $\GF(q^2)$ is also a totally singular orthogonal $4$-space over $\GF(q)$. Therefore, $H$ stabilizes a totally isotropic $4$-space.

{\bf Case 3.} Suppose $A\cap L=\hat{~}((q-1)/d\times\Omega_6^+(q)).2^d$. In this case, $A$ has the unique unsolvable composition factor $\PSL_4(q)$. Write $R=\Rad(A)$ and $A/R=\PSL_4(q).\calO$. Since $r$ divides $|H|$, we deduce from Lemma~\ref{Reduction}(c) that either
$$
H\leqslant R.\left(\frac{q^4-1}{(q-1)(4,q-1)}.4\right).\calO,\mbox{ or $q=3$ and $H\leqslant R.(2^4{:}5{:}4).\calO$.}
$$
Consequently, $|H|_p$ divides
$$
4(|R||\calO|)_p=\frac{4|A|_p}{|\PSL_4(q)|_p}=\frac{4|A|_p(2^d)_p}{|A\cap L|_p}=4|G/L|_p(2^d)_p,
$$
and thus divides $48f$. According to the factorization $G=HB$ we know that $|H|_p$ is divisible by $|L|_p/|B\cap L|_p=q^3$. Hence we obtain $q^3\di48f$, which is impossible as $q\geqslant3$.

{\bf Case 4.} Suppose that $A\cap L=\Omega_8^+(2)$ with $q=3$. In this case, $|L|/|B\cap L|$ is divisible by $27$. Thus the factorization $G=HB$ requires $|H|$ to be divisible by $27$. However, Lemma \ref{Reduction}(h) implies that $|H|$ is not divisible by $27$, a contradiction. This completes the proof.
\end{proof}

\begin{lemma}\label{OmegaPlus1}
If $L=\POm_8^+(q)$, then one of the following holds.
\begin{itemize}
\item[(a)] $m=4$, $H\cap L\leqslant q^6{:}((q^4-1)/(4,q^4-1)).4<\Pa_1$ or $\Pa_3$ or $\Pa_4$, and $K\cap L=\Omega_7(q)$.
\item[(b)] $L=\Omega_8^+(2)$ or $\POm_8^+(3)$, and $(L,H\cap L,K\cap L)$ is as described in \emph{Table \ref{tab17}}.
\end{itemize}
\end{lemma}

\begin{proof}
If $L=\Omega_8^+(2)$, the lemma follows directly by Proposition \ref{Small3}. Thus we assume $q\geqslant3$ for the rest of the proof. As a consequence, $p^{6f}-1$ has a primitive prime divisor $r$.

Assume that $r$ divides $|H\cap L|$. Then $r$ divides $|A\cap L|$, and inspecting Table \ref{tabOmegaPlus2} we obtain the candidates for $A\cap L$:
\begin{equation}\label{eq13}
\mbox{$\Omega_7(q)$,\quad$\hat{~}((q+1)/d\times\Omega_6^-(q)).2^d$,\quad$\Omega_8^+(2)$ with $q=3$,\quad$2^6{:}\A_8$ with $q=3$.}
\end{equation}
Since $r$ divides $|H|$, we deduce from Lemma~\ref{Reduction}(d) and (f)--(h) that either $A\cap L=\hat{~}((q+1)/d\times\Omega_6^-(q)).2^d$ with $q=8$ or $A\cap L=2^6{:}\A_8$ with $q=3$. Suppose that the former occurs. Then $B\cap L=\Omega_7(8)$, $\Pa_1$, $\Pa_3$ or $\Pa_4$, and hence $|L|/|B\cap L|$ is divisible by $13$. However, Lemma~\ref{Reduction}(f) implies that $|H|$ is not divisible by $13$ since $r$ divides $|H|$. This contradicts the factorization $G=HB$ as it requires $|L|/|B\cap L|$ to divide $|H|$ by Lemma~\ref{p4}.

Now $A\cap L=2^6{:}\A_8$ with $q=3$, and it follows that $B\cap L=\Pa_1$, $\Pa_3$ or $\Pa_4$, as in Table \ref{tabOmegaPlus2}. Therefore, $|L|/|B\cap L|$ is divisible by $35$, which indicates that $|H|$ is divisible by $35$ due to the factorization $G=HB$. Note that $A/\Rad(A)$ is an almost simple group with socle $\A_8$, and $H\Rad(A)/\Rad(A)$ is a nontrivial solvable factor of $A/\Rad(A)$ by Lemma \ref{l3}. This is contrary to Hypothesis \ref{Hypo} as Proposition \ref{Alternating} implies that $A/\Rad(A)$ has no nontrivial solvable factor of order divisible by $35$.

Consequently, $r$ does not divide $|H\cap L|$. Thus the factorization $G=HB$ forces $r$ to divide $|B\cap L|$. We thereby obtain the candidates for $B\cap L$ as in (\ref{eq13}). By Zsigmondy's theorem, $p^{4f}-1$ has a primitive prime divisor $s$.

{\bf Case 1.} Suppose $B\cap L=\Omega_7(q)$. Then by Lemma \ref{OmegaPlus4}, we may assume $A\cap L=\Pa_k$ with $k\in\{1,3,4\}$. Note that $\Pa_k=q^6{:}(\SL_4(q).(q-1)/d)/d$. In particular, $A$ has the unique unsolvable composition factor $\PSL_4(q)$. It derives from the factorization $G=HB$ that $s$ divides $|H|$ since $s$ divides $|L|/|B\cap L|$. Hence by Lemma~\ref{Reduction}(c), either
$$
H\leqslant R.\left(\frac{q^4-1}{(q-1)(4,q-1)}.4\right).\calO=\left(q^6.\frac{q^4-1}{(4,q^4-1)}.4\right).(G/L),
$$
or $q=3$ and $H\leqslant R.(2^4{:}5{:}4).\calO=(3^6.2^4.\AGL_1(5)).(G/L)$, where $R,\calO$ are defined in Lemma \ref{Reduction} with $A/R=\PSL_4(q).\calO$. Accordingly, either
$$
H\cap L\leqslant q^6{:}\frac{q^4-1}{(4,q^4-1)}.4,\mbox{ or $q=3$ and $H\cap L\leqslant3^6{:}2^4.\AGL_1(5)$.}
$$

Now $|H\cap L|$ divides $2^4q^6(q^4-1)/(4,q^4-1)$, and we deduce from the factorization $G=HB$ that $|L|$ divides $2^4q^6(q^4-1)|K\cap L||G/L|/(4,q^4-1)$ by Lemma \ref{p4}. Since $|G/L|$ divides $2f(4,q^4-1)$, this implies that $|L|$ divides $2^5fq^6(q^4-1)|K\cap L|$, that is,
$$
\mbox{$q^6(q^6-1)(q^4-1)(q^2-1)$ divides $2^5fd|K\cap L|$.}
$$
Hence we conclude $K\cap L=\Omega_7(q)$ (see \cite[Tables~8.28--8.29 and 8.39--8.40]{BHR-book}). Therefore, either part (a) of the lemma holds, or $(L,H\cap L,K\cap L)$ is as described in row 3 of Table \ref{tab17}.

{\bf Case 2.} Suppose $B\cap L=\hat{~}((q+1)/d\times\Omega_6^-(q)).2^d$. By Theorem \ref{Maximal}, we have the candidates for $A\cap L$:
\[\mbox{$\Omega_7(q)$,\quad$\Pa_1$,\quad$\Pa_3$,\quad$\Pa_4$,\quad$(3\times\Omega_6^+(4)).2$ with $q=4$}.\]
Let $t$ be a primitive prime divisor of $p^{3f}-1$. According to the factorization $G=HB$, the order $|H|$ is divisible by $|L|/|B\cap L|$, and thus divisible by $st$. However, since $A$ has socle $\Omega_7(q)$ or $\PSL_4(q)$, we deduce from Lemma~\ref{Reduction}(c), (d) and (g) that $st$ does not divide $|H|$, which is a contradiction.

{\bf Case 3.} Suppose that $B\cap L=\Omega_8^+(2)$ with $q=3$. It is seen in Table \ref{tabOmegaPlus2} that the candidates for $A\cap L$ are $\Pa_1$, $\Pa_3$, $\Pa_4$, $\Pa_{13}$, $\Pa_{14}$ and $\Pa_{34}$. From the factorization $G=HB$ we know that $|H|$ is divisible by $|L|/|B\cap L|$ and thus divisible by $13$.

Assume $A\cap L=\Pa_1$, $\Pa_3$ or $\Pa_4$. Then $A\cap L=3^6{:}\PSL_4(3)$, which implies that $A$ has the unique unsolvable composition factor $\PSL_4(3)$. Since $|H|$ is divisible by $13$, we conclude from Lemma \ref{Reduction}(c) that $H\cap L\leqslant3^6{:}(3^3{:}13{:}3)$, as in row 4 of Table \ref{tab17}.

Next assume $A\cap L=\Pa_{13}$, $\Pa_{14}$ or $\Pa_{34}$. Then $A\cap L=3^{3+6}{:}\PSL_3(3)$, which shows that $A$ has the unique unsolvable composition factor $\PSL_3(3)$. Since $|H|$ is divisible by $13$, we conclude from Lemma \ref{Reduction}(c) that $H\cap L\leqslant3^{3+6}{:}13.3$, as in row 4 of Table \ref{tab17}.

{\bf Case 4.} Suppose that $B\cap L=2^6{:}\A_8$ with $q=3$. Then seen from Table \ref{tabOmegaPlus2}, $A\cap L=\Pa_k=3^6{:}\PSL_4(3)$ with $k\in\{1,3,4\}$. In view of the factorization $G=HB$, we know that $|H|$ is divisible by $|L|/|B\cap L|$, and thus divisible by $5\cdot13$. However, as $A$ has the unique unsolvable composition factor $\PSL_4(3)$, Lemma \ref{Reduction}(c) shows that $|H|$ is not divisible by $5\cdot13$, which is a contradiction. The proof is thus finished.
\end{proof}

\subsection{Orthogonal groups of even dimension at least ten}

In this subsection, let $L=\POm_{2m}^+(q)$ with $m\geqslant5$ and $q=p^f$ for prime number $p$. We aim to show that part (a) of Proposition \ref{OmegaPlus} holds for such $L$.

\begin{lemma}\label{OmegaPlus3}
If $B\cap L=\N_1$, then $H$ stabilizes a totally singular $m$-space.
\end{lemma}

\begin{proof}
Due to Hypothesis \ref{Hypo}, the conclusion of Theorem \ref{SolvableFactor} excludes the possibility of the case $A\cap L=\Co_1$. Thus we see from Table \ref{tabOmegaPlus1} that there are six cases for $A\cap L$.
\begin{itemize}
\item[(i)] $A\cap L=\Pa_m$ or $\Pa_{m-1}$.
\item[(ii)] $m$ is even, and $A\cap L=\hat{~}\GU_m(q).2$.
\item[(iii)] $m$ is even, $q>2$, and $A\cap L=(\PSp_2(q)\otimes\PSp_m(q)).a$ with $a\leqslant2$.
\item[(iv)] $A\cap L=\hat{~}\GL_m(q).2$.
\item[(v)] $m$ is even, $q=2$ or $4$, and $A\cap L=\Omega_m^+(q^2).2^2$.
\item[(vi)] $m=8$, and $A\cap L=\Omega_9(q).a$ with $a\leqslant2$.
\end{itemize}
If $A\cap L=\Pa_m$ or $\Pa_{m-1}$, then the lemma holds already. Now we deal with (ii)--(vi).

{\bf Case 1.} Suppose that $A\cap L=\hat{~}\GU_m(q).2$ with $m=2\ell$, as in (ii). By Hypothesis \ref{Hypo}, we see from Theorem \ref{SolvableFactor} that $H\leqslant\Pa_\ell[A]$. Note that a totally singular unitary $\ell$-space over $\GF(q^2)$ is also a totally singular orthogonal $m$-space over $\GF(q)$. We then conclude that $H\leqslant\Pa_\ell[A]$ stabilizes a totally singular $m$-space.

{\bf Case 2.} Suppose that (iii) appears. As Proposition \ref{p1} indicates that $A$ has at most one unsolvable composition factor, we have $q=3$ and $A\cap L=(\PSp_2(3)\times\PSp_m(3)).a$ with $a\leqslant2$. By Zsigmondy's theorem, $3^m-1$ has a primitive prime divisor $r$. From the factorization $G=HB$ we conclude that $r$ divides $|H|$ since $r$ divides $|L|/|B\cap L|$. However, Lemma \ref{Reduction}(d) implies that $|H|$ is not divisible by $r$, which is a contradiction.

{\bf Case 3.} Suppose $A\cap L=\hat{~}\GL_m(q).2$, as in (iv). By Lemma~\ref{Reduction}(c), we have $H\leqslant R.(((q^m-1)/((q-1)(m,q-1))).m).\calO$, where $R$ and $\calO$ are defined in Lemma \ref{Reduction} with $A/R=\PSL_m(q).\calO$. It follows that $|H|_p$ divides
$$
m(|R||\calO|)_p=\frac{m|A|_p}{|\PSL_m(q)|_p}=\frac{m(2|A|)_p}{|A\cap L|_p}=m(2|G/L|)_p
$$
and thus divides $4fm$. This implies that $q^{m-1}=(|L|/|B\cap L|)_p$ divides $4fm$ due to the factorization $G=HB$. As a consequence, we obtain
$$
2^{m-1}\leqslant p^{m-1}\leqslant p^{f(m-1)}/f\leqslant4m,
$$
which is impossible as $m\geqslant5$.

{\bf Case 4.} Suppose that (v) appears, that is, $A\cap L=\Omega^+_m(q^2).2^2$ with $m=2\ell$ and $q\in\{2,4\}$. By Zsigmondy's theorem, $2^{fm}-1$ has primitive prime divisor $r$ if $(m,q)\neq(6,2)$. Set $r=7$ if $(m,q)=(6,2)$. Since $r$ divides $|G|/|B|$, $r$ also divides $|H|$ as the factorization $G=HB$ requires. Due to Hypothesis \ref{Hypo}, it derives from the conclusion of Theorem \ref{SolvableFactor} that $H\leqslant\Pa_\ell[A]$ or $\Pa_{\ell-1}[A]$. Note that a totally singular orthogonal $\ell$-space over $\GF(q^2)$ is also a totally singular orthogonal $m$-space over $\GF(q)$. We then conclude have $H$ stabilizes a totally singular $m$-space.

{\bf Case 5.}  Assume that $L=\POm_{16}^+(q)$ and $A\cap L=\Omega_9(q).a$ with $a\leqslant2$, as in (vi). By Zsigmondy's theorem, $p^{8f}-1$ has a primitive prime divisor $r$. According to the factorization $G=HB$ we know that $r$ divides $|H|$ since $r$ divides $|L|/|B\cap L|$. However, Lemma \ref{Reduction}(g) implies that $|H|$ is not divisible by $r$, which is a contradiction.
\end{proof}

\begin{lemma}\label{OmegaPlus2}
If $L=\POm_{2m}^+(q)$ with $m\geqslant5$, then $H\cap L\leqslant q^{m(m-1)/2}{:}((q^m-1)/(4,q^m-1)).m<\Pa_m$ or $\Pa_{m-1}$, and $\Omega_{2m-1}(q)\trianglelefteq K\cap L\leqslant\N_1$.
\end{lemma}

\begin{proof}
By Zsigmondy's theorem, $p^{f(2m-2)}-1$ has a primitive prime divisor $r$. Assume that $r$ divides $|H\cap L|$. Then $r$ divides $|A\cap L|$, and inspecting Table \ref{tabOmegaPlus1} we obtain the candidates for $A\cap L$:
\begin{equation}\label{eq12}
\mbox{$\N_1$,\quad$\hat{~}\GU_m(q).2$ with $m$ even,\quad$\N_2^-$.}
\end{equation}
Since $r$ divides $|H|$, we conclude from Lemma~\ref{Reduction}(b), (d) and (f) that none of them is possible. Therefore, $r$ does not divide $|H\cap L|$. Hence the factorization $G=HB$ forces $r$ to divide $|B\cap L|$ by Lemma \ref{p4}. We thus obtain the candidates for $B\cap L$ as in (\ref{eq12}).

{\bf Case 1.} Suppose $B\cap L=\N_1$. Then by Lemma \ref{OmegaPlus3}, we may assume that $A\cap L=\Pa_k$ with $k\in\{m,m-1\}$. Let $d=(4,q^m-1)/(2,q-1)$. Then $\Cen(\Omega_{2m}^+(q))=\Z_d$ and $\Pa_k=q^{m(m-1)/2}{:}(\SL_m(q).(q-1)/(2,q-1))/d$ (see \cite[Proposition 4.1.20]{kleidman1990subgroup}). In particular, $A$ has the unique unsolvable composition factor $\PSL_m(q)$. Hence we deduce from Lemma \ref{Reduction}(c) that
$$
H\leqslant\Rad(A).\left(\frac{q^m-1}{(q-1)(m,q-1)}.m\right).\calO=\left(q^{m(m-1)/2}.\frac{q^m-1}{(4,q^m-1)}.m\right).(G/L),
$$
with $A/\Rad(A)=\PSL_m(q).\calO$. Consequently,
\begin{equation}\label{eq8}
H\cap L\leqslant q^{m(m-1)/2}{:}\frac{q^m-1}{(4,q^m-1)}.m.
\end{equation}

Let $R=\Rad(B)$, $\overline{B}=B/R$, $\overline{K}=KR/R$ and $\overline{H\cap B}=(H\cap B)R/R$. From $G=HK$ we deduce $B=(H\cap B)K$ and $\overline{B}=\overline{H\cap B}{\,}\overline{K}$. Note that $\overline{B}$ is almost simple with socle $\Omega_{2m-1}(q)$. Since $\overline{H\cap B}$ is solvable, it then derives from Hypothesis \ref{Hypo} that either $\overline{K}\cap\Soc(\overline{B})\leqslant\Omega_{2m-2}^-(q).2$ or $\overline{K}\trianglerighteq\Omega_{2m-1}(q)$.

Assume that $\overline{K}\cap\Soc(\overline{B})\leqslant\Omega_{2m-2}^-(q).2$. Since $|\overline{B}|/|\overline{K}|$ divides $|B|/|K|$ and $|G|/|K|$ divides $|H\cap L||\Out(L)|$, we deduce that $(|G|/|B|)(|\overline{B}|/|\overline{K}|)$ divides $|H\cap L||\Out(L)|$. This together with (\ref{eq8}) leads to
$$
\frac{q^{2(m-1)}(q^m-1)(q^{m-1}-1)}{2(2,q-1)}\di2fmq^{m(m-1)/2}(q^m-1),
$$
which gives
\begin{equation}\label{eq20}
q^{m-1}-1\di4fm(2,q-1).
\end{equation}
As a consequence,
$$
\frac{p^{4f}-1}{5f}\leqslant\frac{p^{f(m-1)}-1}{fm}\leqslant4(2,q-1),
$$
which forces $q=2$. Substituting this into the above inequality, we obtain that $2^{m-1}-1\leqslant4m$ and thus $m=5$. However, the pair $(m,q)=(5,2)$ does not satisfy (\ref{eq20}), a contradiction.

Therefore, we have $\overline{K}\trianglerighteq\Omega_{2m-1}(q)$. It follows that $K\cap L$ contains $(B\cap L)^{(\infty)}=\Omega_{2m-1}(q)$, and thus $\Omega_{2m-1}(q)\trianglelefteq K\cap L\leqslant B\cap L=\N_1$, as the lemma states.

{\bf Case 2.} Suppose that $B\cap L=\hat{~}\GU_m(q).2$ with $m$ even. As listed in Table \ref{tabOmegaPlus1}, there are three candidates for $A\cap L$:
\[\mbox{$\N_1$,\quad$\Pa_1$,\quad$\N_2^+$ with $q=4$}.\]
By Zsigmondy's theorem, $p^{f(2m-4)}-1$ has a primitive prime divisor $s$. According to the factorization $G=HB$, the order $|H|$ is divisible by $|L|/|B\cap L|$, and thus divisible by $s$. However, since $A$ has socle $\Omega_{2m-1}(q)$ or $\POm_{2m-2}^+(q)$, we deduce from Lemma~\ref{Reduction}(d), (g) and (h) that $s$ does not divide $|H|$, which is a contradiction.

{\bf Case 3.} Finally suppose $B\cap L=\N_2^-$. Then by Theorem \ref{Maximal}, we obtain all the candidates for $A\cap L$:
\[\mbox{$\Pa_m$,\quad$\Pa_{m-1}$,\quad$\hat{~}\GL_m(q).2$ with $q\in\{2,4\}$}.\]
By Zsigmondy's theorem, $p^{f(m-1)}-1$ has a primitive prime divisor $s$ if $(m,q)\neq(7,2)$. Set $s=7$ if $(m,q)=(7,2)$. From the factorization $G=HB$ we know that $|L|/|B\cap L|$ divides $|H|$, whence $s$ divides $|H|$. However, since $A$ has the unique unsolvable composition factor $\PSL_m(q)$, we conclude from Lemma \ref{Reduction}(c) that $s$ does not divide $|H|$, a contradiction. The proof is thus completed.
\end{proof}

\section{Completion of the proof}

We are now able to complete the proof of Theorem~\ref{SolvableFactor} by summarizing the results obtained in previous sections.

\begin{lemma}\label{Complete}
Let $G$ be an almost simple group with socle $L$. Then each nontrivial factorization $G=HK$ with $H$ solvable satisfies \emph{Theorem~\ref{SolvableFactor}}.
\end{lemma}

\begin{proof}
Proposition~\ref{ExceptionalLie} shows that $L$ cannot be an exceptional group of Lie type. If $K$ is also solvable or $L$ is an alternating group or a sporadic simple group, then Propositions~\ref{BothSolvable}, \ref{Alternating} and \ref{Sporadic}, respectively, describe the triple $(G,H,K)$.

Now assume that $K$ is unsolvable and $L$ is a classical group of Lie type not isomorphic to any alternating group. To prove that part~(d) of Theorem \ref{SolvableFactor} holds, we may embed $G=HK$ into a nontrivial maximal factorization $G=AB$ (this means that we can find core-free maximal subgroups $A,B$ containing $H,K$ respectively, see Remark \ref{rmk1}). If $A$ is solvable, then part~(d) of Theorem \ref{SolvableFactor} follows by Proposition \ref{p13}. Under Hypothesis~\ref{Hypo}, Propositions~\ref{Linear}, \ref{Symplectic}, \ref{Unitary}, \ref{Omega} and~\ref{OmegaPlus} show that the triple $(G,H,K)$ lies in Table~\ref{tab7} or Table~\ref{tab1}. Therefore, we conclude by induction that any nontrivial factorization $G=HK$ with $H$ solvable satisfies Theorem~\ref{SolvableFactor}.
\end{proof}

To finish the proof of Theorem \ref{SolvableFactor}, it remains to verify that for each socle $L$ listed in Table~\ref{tab7} or Table~\ref{tab1}, there exist factorizations as described.

\begin{lemma}\label{Existence}
For each $L$ in \emph{Table \ref{tab7}} and \emph{Table~\ref{tab1}}, there exist group $G$ and subgroups $H,K$ as described such that $\Soc(G)=L$ and $G=HK$.
\end{lemma}

\begin{proof}
To see the existence for row 1 of Table \ref{tab7}, let $G=\PGL_n(q)$, $L=\Soc(G)$, $H$ be a Singer cycle and $K$ be the stabilizer of a $1$-space in $G$. Then $H\cap L=\hat{~}\GL_1(q^n)=(q^n-1)/(n,q-1)$, $K\cap L=\Pa_1$, and $G=HK$. Moreover, take $\tau$ to be a graph automorphism of $L$ of order $2$ such that $K^\tau\cap L=\Pa_{n-1}$, and then we have $\Soc(G^\tau)=L$ and $G^\tau=H^\tau K^\tau$.

Take $m=3$ in Proposition \ref{ExampleOrthogonal2}. Then by the isomorphisms $\POm_6^+(q)\cong\PSL_4(q)$ and $\POm_5(q)\cong\PSp_4(q)$, we see that row 2 of Table \ref{tab7} arises.

Next let $G,H,K$ be defined as in Proposition \ref{ExampleSymplectic}. Then $G=\Soc(G)=\Sp_{2m}(q)$ with $q$ even, $H$ is solvable, $H=q^{m(m+1)/2}{:}(q^m-1)\leqslant\Pa_m$, $K=\GO_{2m}^-(q)$, and $G=HK$ by Proposition \ref{ExampleSymplectic}. Hence row 3 of Table \ref{tab7} arises. If $m=2$, then further take $\tau$ to be a graph automorphism of $G$ of order $2$, under which we have $G^\tau=H^\tau K^\tau$, $H^\tau\leqslant(\Pa_2)^\tau=\Pa_1$ and $K^\tau=\Sp_2(q^2).2$ by \cite[(14.1)]{aschbacher1984maximal}. This shows that the row 4 of Table \ref{tab7} arises.

Take $m=2$ in Proposition \ref{ExampleOrthogonal1}. Then by the isomorphisms $\POm_5(q)\cong\PSp_4(q)$ and $\POm_4^-(q)\cong\PSp_2(q^2)$, we see that row 5 of Table \ref{tab7} arises.

For row 6 of Table \ref{tab7}, let $G,H,K$ be defined as in Proposition \ref{ExampleUnitary}, $Z=\Cen(G)$, $\overline{G}=G/Z$, $L=\Soc(\overline{G})$, $\overline{H}=HZ/Z$ and $\overline{K}=KZ/Z$. Then $\overline{G}=\PGU_{2m}(q)$, $\overline{H}\cap L=q^{m^2}{:}(q^{2m}-1)/((q+1)(2m,q+1))<\Pa_m$, $\overline{K}\cap L=\N_1$, and $\overline{G}=\overline{H}{\,}\overline{K}$ holds since $G=HK$ by Proposition \ref{ExampleUnitary}.

For row 7 of Table \ref{tab7}, let $G,H,K$ be defined as in Proposition \ref{ExampleOrthogonal1} and $L=\Soc(G)$. Then $G=\SO_{2m+1}(q)$, $H\cap L=(q^{m(m-1)/2}.q^m){:}(q^m-1)/2<\Pa_m$, $K\cap L=\N_1^-$, and $G=HK$ by Proposition \ref{ExampleOrthogonal1}.

Now let $G,A,H,K$ be defined as in Proposition \ref{ExampleOrthogonal2}, $Z=\Cen(G)$, $\overline{G}=G/Z$, $L=\Soc(\overline{G})$, $\overline{H}=HZ/Z$, $\overline{A}=AZ/Z$ and $\overline{K}=KZ/Z$. Then $\overline{G}=\PSO_{2m}^+(q)$, $\overline{A}\cap L$ is one of $\Pa_m$ and $\Pa_{m-1}$, say $\overline{A}\cap L=\Pa_m$, $\overline{H}\cap L=q^{m(m-1)/2}{:}(q^m-1)/(4,q^m-1)$, $\overline{K}\cap L=\N_1$, and $\overline{G}=\overline{H}{\,}\overline{K}$ by Proposition \ref{ExampleOrthogonal2}. Moreover, take $\sigma$ to be an automorphism of $L$ such that $\overline{A}^\sigma\cap L=\Pa_{m-1}$. We have $\overline{G}=\overline{H}^\sigma\overline{K}^\sigma$. Hence row 8 of Table \ref{tab7} arises. If $m=4$, then further take $\tau$ to be a graph automorphism of $L$ of order $3$ such that $\overline{A}^\tau\cap L=\Pa_1$, under which we have $\Soc(\overline{G}^\tau)=L$ and $\overline{G}^\tau=\overline{H}^\tau\overline{K}^\tau$. Consequently, row 9 of Table \ref{tab7} arises.

Finally, computation in \magma \cite{bosma1997magma} shows that for each row in Table~\ref{tab1}, there exists factorization $G=HK$ with $\Soc(G)=L$ and $(H\cap L,K\cap L)$ as described. This completes the proof.
\end{proof}


\chapter[Cayley graphs of solvable groups]{$s$-Arc transitive Cayley graphs of solvable groups}


This chapter is devoted to proving Theorem~\ref{CayleyGraph}, and is organized as follows. Section~\ref{Cay-pre} collects preliminary results that are needed in the ensuing arguments. In particular, Lemma~\ref{reduction-qp} reduces the proof of Theorem~\ref{CayleyGraph} to the quasiprimitive case. Section~\ref{pty-Gorenstein} presents a result regarding the outer automorphism group of simple groups, which generalizes a well-known result of Gorenstein and plays an important role in the proof of Theorem~\ref{CayleyGraph}. Section~\ref{qp-case} further reduces the quasiprimitive case to the affine case and the almost simple case. Then the final section completes the proof Theorem~\ref{CayleyGraph} by citing the classification of the affine case in \cite{IP} and treating the almost simple case based on Theorem~\ref{SolvableFactor}.

\section{Preliminaries}\label{Cay-pre}

Throughout this section, let $\Ga=(V,E)$ be a connected $G$-arc-transitive graph and $\{\alpha,\beta\}$ be an edge of $\Ga$. Denote by $\Ga(\alpha)$ the set of neighbors of $\alpha$ in $\Ga$, and $G_\alpha^{\Ga(\alpha)}$ the permutation group on $\Ga(\alpha)$ induced by $G_\alpha$.

\subsection{Normal subgroups and normal quotients}

As mentioned in the Introduction chapter, for a normal subgroup $N$ of $G$ which is intransitive on $V$, we have a normal quotient graph $\Ga_N=(V_N,E_N)$ of $\Ga$, where $V_N$ the set of $N$-orbits on $V$ and $E_N$ the set of $N$-orbits on $E$. We shall see that the normal quotient $\Ga_N$ inherits certain properties of $\Ga$.

The first lemma is a well-known result, see for example \cite[Lemma~1.6]{Praeger1985}.

\begin{lemma}\label{normal-quotient}
Let $\Ga=(V,E)$ be a connected $G$-arc-transitive graph such that $G_\alpha^{\Ga(\alpha)}$ is primitive, $N$ be a normal subgroup of $G$, and $\overline{G}=G/N$. If $|V_N|\geqslant3$, then the following statements hold.
\begin{itemize}
\item[(a)] $N$ is semiregular on $V$ and $\overline{G}$ is faithful on $V_N$.
\item[(b)] For any $\alpha\in V$ and $B\in V_N$, $G_\alpha\cong\overline{G}_B$.
\end{itemize}
\end{lemma}

The next lemma shows that the $s$-arc-transitivity is inherited by normal quotients.

\begin{lemma}\label{normal-q}
\emph{(Praeger \cite{Praeger})} Assume that $\Ga$ is a connected non-bipartite $(G,s)$-arc-transitive graph with $s\geqslant2$. Then there exists a normal subgroup $N$ of $G$ such that $G/N$ is a quasiprimitive permutation group on $V_N$ of type almost simple (AS), affine (HA), twisted wreath product (TW) or product action (PA) and $\Ga_N$ is $(G/N,s)$-arc transitive.
\end{lemma}

Although a normal quotient of a Cayley graph is not necessarily a Cayley graph, we have the following observation.

\begin{lemma}\label{Cay-q}
If $\Ga$ is a Cayley graph of a group $R$ and $N$ is a normal subgroup of $G$. Then $RN/N$ is transitive on $V_N$.
\end{lemma}

This reduces the proof of Theorem~\ref{CayleyGraph} to the quasiprimitive case.

\begin{lemma}\label{reduction-qp}
\emph{Theorem~\ref{CayleyGraph}} holds if it holds for the case where $X=G$ is vertex-quasiprimitive.
\end{lemma}

For vertex-transitive normal subgroups of $G$, we have the following consequences.

\begin{lemma}\label{l10}
Let $K$ be a vertex-transitive normal subgroup of $G$. Then the following statements hold.
\begin{itemize}
\item[(a)] $G_\alpha/K_\alpha\cong G/K$.
\item[(b)] If $K$ is arc-transitive, then $G_{\alpha\beta}/K_{\alpha\beta}\cong G/K$.
\item[(c)] If $G_\alpha^{\Ga(\alpha)}$ is primitive and $K_\alpha\neq1$, then $K$ is arc-transitive.
\end{itemize}
\end{lemma}

\begin{proof}
Since $K$ is vertex-transitive, we have $G=KG_\alpha$, and thus
$$
G_\alpha/K_\alpha=G_\alpha/G_\alpha\cap K\cong G_\alpha K/K=G/K,
$$
as part~(a) states. If $K$ is arc-transitive, then $G=KG_{\alpha\beta}$, and so
$$
G_{\alpha\beta}/K_{\alpha\beta}=G_{\alpha\beta}/G_{\alpha\beta}\cap K\cong G_{\alpha\beta}K/K=G/K,
$$
proving part~(b). Now suppose that $G_\alpha^{\Ga(\alpha)}$ is primitive and $K_\alpha\neq1$. Since $\Ga$ is connected, we see that if $K_\alpha^{\Ga(\alpha)}=1$ then $K_\alpha=1$. Thus by our assumption, $K_\alpha^{\Ga(\alpha)}\neq1$. It follows that $K_\alpha^{\Ga(\alpha)}$ is transitive since $K_\alpha^{\Ga(\alpha)}$ is normal in $G_\alpha^{\Ga(\alpha)}$ and $G_\alpha^{\Ga(\alpha)}$ is primitive. This implies that $K$ is arc-transitive.
\end{proof}

\subsection{Vertex and arc stabilizers}

Now suppose further that $\Ga$ is $(G,2)$-arc-transitive. Denote by $G_\alpha^{[1]}$ the kernel of the action induced by $G_\alpha$ on $\Ga(\alpha)$, and $G_{\alpha\beta}^{[1]}:=G_\alpha^{[1]}\cap G_\beta^{[1]}$. Noting that $G_\alpha^{[1]}\trianglelefteq G_{\alpha\beta}$ and the kernel of the action of $G_\alpha^{[1]}$ on $\Ga(\beta)$ is equal to $G_{\alpha\beta}^{[1]}$, we have the observation as follows.

\begin{lemma}\label{abeq}
$G_\alpha^{[1]}/G_{\alpha\beta}^{[1]}\cong(G_\alpha^{[1]})^{\Ga(\beta)}\trianglelefteq G_{\alpha\beta}^{\Ga(\beta)}\cong G_{\alpha\beta}^{\Ga(\alpha)}$.
\end{lemma}

The next theorem is a fundamental result in the study of $2$-arc-transitive graphs, see~\cite[\S4]{Weiss}.

\begin{theorem}\label{DoubleStar}
Let $\Ga$ be a connected $(G,2)$-arc-transitive graph, and $\{\alpha,\beta\}$ be an edge of $\Ga$. Then the following statements hold.
\begin{itemize}
\item[(a)] $G_{\alpha\beta}^{[1]}$ is a $p$-group for some prime $p$.
\item[(b)] If $G_{\alpha\beta}^{[1]}$ is a nontrivial $p$-group with prime $p$, then $G_\alpha^{\Ga(\alpha)}\trianglerighteq\PSL_d(q)$ and $|\Ga(\alpha)|=(q^d-1)/(q-1)$ for some $p$-power $q$ and integer $d\geqslant2$.
\end{itemize}
\end{theorem}

Utilizing Theorem~\ref{DoubleStar} and results of~\cite{Li-Seress-Song}, we establish the following lemma.

\begin{lemma}\label{normal-stab}
Let $\Ga$ be $(G,2)$-arc-transitive, and $K$ be a vertex-transitive normal subgroup of $G$ such that $K_\alpha\neq1$ is imprimitive on $\Ga(\alpha)$. Then $G_\alpha^{\Ga(\alpha)}$ is an affine $2$-transitive permutation group of degree $p^d$, where $p$ is prime and $d\geqslant2$, and the following hold.
\begin{itemize}
\item[(a)] $G_{\alpha\beta}^{[1]}=1$.
\item[(b)] $K_{\alpha\beta}\cong\Z_k\times\Z_m$ with $k\di m$ and $\Z_m=K_{\alpha\beta}^{\Ga(\alpha)}\leqslant\GL_1(p^f)$ for some proper divisor $f$ of $d$.
\item[(c)] $K_\alpha\cong\Z_k\times(\Z_p^d{:}\Z_m)$.
\item[(d)] $G_\alpha^{[1]}\leqslant\Z_m$. In particular, $|G_\alpha^{[1]}|$ divides $p^f-1$.
\end{itemize}
\end{lemma}

\begin{proof}
By \cite[Corollary~1.2 and Lemma~3.3]{Li-Seress-Song}, $G_\alpha^{\Ga(\alpha)}$ is an affine $2$-transitive permutation group of degree $p^d$ with $p$ prime, and $K_{\alpha\beta}^{\Ga(\alpha)}\leqslant\GL_1(p^f)$, where $f$ is a proper divisor of $d$.
Consequently, $K_{\alpha\beta}^{\Ga(\alpha)}=\Z_m$ for some $m\di(p^f-1)$. Applying Theorem~\ref{DoubleStar} we have $G_{\alpha\beta}^{[1]}=1$, as part~(a) asserts. This together with Lemma~\ref{abeq} implies that $K_\alpha^{[1]}=K_\alpha^{[1]}/K_{\alpha\beta}^{[1]}$ is isomorphic to a subgroup of $K_{\alpha\beta}^{\Ga(\alpha)}$, so that $K_\alpha^{[1]}=\Z_k$ for some $k\di m$.

Now that $K_{\alpha\beta}/K_\alpha^{[1]}\cong K_{\alpha\beta}^{\Ga(\alpha)}$ is cyclic, $K_\alpha^{[1]}$ contains the commutator subgroup $K_{\alpha\beta}'$ of $K_{\alpha\beta}$. For the same reason, $K_\beta^{[1]}\geqslant K_{\alpha\beta}'$. Therefore, $K_{\alpha\beta}'\leqslant K_\alpha^{[1]}\cap K_\beta^{[1]}=K_{\alpha\beta}^{[1]}=1$, and hence $K_{\alpha\beta}$ is abelian. Since $K_\alpha^{[1]}=\Z_k$ and $K_{\alpha\beta}/K_\alpha^{[1]}=\Z_m$ are both cyclic, we conclude that $K_{\alpha\beta}=K_\alpha^{[1]}\times\Z_m$, proving part~(b). It follows that $K_\alpha=\bfO_p(K_\alpha){:}K_{\alpha\beta}=K_\alpha^{[1]}\times(\bfO_p(K_\alpha){:}\Z_m)$, as in part~(c).

Finally, let $X=KG_\alpha^{[1]}$. Then $X_\alpha^{[1]}=G_\alpha^{[1]}$, $X_\alpha^{\Ga(\alpha)}=K_\alpha^{\Ga(\alpha)}$, and $X_{\alpha\beta}^{\Ga(\alpha)}=K_{\alpha\beta}^{\Ga(\alpha)}$.
Viewing $X_{\alpha\beta}^{[1]}\leqslant G_{\alpha\beta}^{[1]}=1$, we deduce that
$G_\alpha^{[1]}=X_\alpha^{[1]}\cong(X_\alpha^{[1]})^{\Ga(\beta)}\lhd X_{\alpha\beta}^{\Ga(\beta)}=K_{\alpha\beta}^{\Ga(\beta)}$,
which leads to part~(d).
\end{proof}

\subsection{Coset graph construction}

As $\Ga$ is $G$-arc-transitive, there exists $g\in G$ interchanging $\alpha$ and $\beta$. Consequently,
$$
G_{\alpha\beta}^g=(G_\alpha\cap G_\beta)^g=G_\alpha^g\cap G_\beta^g=G_\beta\cap G_\alpha=G_{\alpha\beta},
$$
namely, $g$ normalizes the arc stabilizer $G_{\alpha\beta}$. Because $G_\alpha$ is transitive on $\Ga(\alpha)$, so the valency of $\Ga$ equals
$$
|\Ga(\alpha)|=|G_\alpha|/|G_{\alpha\beta}|=|G_\alpha|/|G_\alpha\cap G_\beta|=|G_\alpha|/|G_\alpha\cap G_\alpha^g|.
$$
Moreover, $\langle G_\alpha,g\rangle=G$ since $\Ga$ is connected. Thus we have the following lemma.

\begin{lemma}\label{CosetGraph}
There exists $g\in\Nor_G(G_{\alpha\beta})$ such that
\begin{itemize}
\item[(a)] $g^2\in G_{\alpha\beta}$ and $\langle G_\alpha,\Nor_G(G_{\alpha\beta})\rangle=\langle G_\alpha,g\rangle=G$;
\item[(b)] the valency of $\Ga$ equals $|G_\alpha|/|G_\alpha\cap G_\alpha^g|$;
\item[(c)] $\Ga$ is $(G,2)$-arc-transitive if and only if $G_\alpha$ is $2$-transitive on $[G_\alpha{:}G_\alpha\cap G_\alpha^g]$.
\end{itemize}
\end{lemma}

Conversely, suppose that $K$ is a core-free subgroup of $G$ and $g$ is an element in $G\setminus K$ with $g^2\in K$. Then $G$ acts faithfully on $[G{:}K]$ by right multiplication, so that $G$ can be regarded as a transitive permutation group on $[G{:}K]$. Denote the points $K$ and $Kg$ in $[G{:}K]$ by $\alpha$ and $\beta$, respectively. Then $G_\alpha=K$, $G_\beta=K^g$, and $(\alpha,\beta)^g=(\beta,\alpha)$. The graph with vertex set $[G{:}K]$ and edge set $\{\alpha,\beta\}^G$ is $G$-arc-transitive, where two vertices $Kx,Ky$ are adjacent if and only if $yx^{-1}\in KgK$. Such a graph is called a \emph{coset graph}, denoted by $\Cos(G,K,KgK)$. It is steadily seen that $\Cos(G,K,KgK)$ is connected if and only if $\langle K,g\rangle=G$.

\begin{remark}
Replacing $g$ by some power of $g$, we may assume that $g$ is a $2$-element. This accelerates the search for $(G,2)$-arc-transitive graphs for given $G$ and $G_\alpha$, see the \magma codes in Appendix\,B.
\end{remark}

We introduce the Hoffman-Singlton graph and the Higman-Sims graph in terms of coset graphs.

\begin{example}\label{Hoffman-S}
Let $G=\PSiU_3(5)$, $K$ be a subgroup of $G$ isomorphic to $\Sy_7$ and $M$ a subgroup of $K$ isomorphic to $\Sy_6$. The normalizer $\Nor_G(M)$ of $M$ in $G$ is an extension of $M$ by $\Z_2$. Take $g$ to be any element of $\Nor_G(M)\setminus M$. Then $g^2\in M$ and $\langle K,g\rangle=G$, whence $\Cos(G,K,KgK)$ is connected. The graph $\Cos(G,K,KgK)$ is the well-known \emph{Hoffman-Singlton graph} \cite{Hoffman1960}, which is non-bipartite and $3$-transitive. Moreover, the full automorphism group of $\Cos(G,K,KgK)$ is isomorphic to $G$ and has a subgroup $5_+^{1+2}{:}8{:}2$ transitive on vertices.
\end{example}

\begin{example}\label{Higman-S}
Let $G=\HS$, $K$ be a subgroup of $G$ isomorphic to $\M_{22}$ and $M$ a subgroup of $K$ isomorphic to $\PSL_3(4)$. The normalizer $\Nor_G(M)$ of $M$ in $G$ is an extension of $M$ by $\Z_2$. Take $g$ to be any element of $\Nor_G(M)\setminus M$. Then $g^2\in M$ and $\langle K,g\rangle=G$. Thus $\Cos(G,K,KgK)$ is a connected $(G,2)$-arc-transitive graph. This is the \emph{Higman-Sims graph}, see \cite{atlas}. The full automorphism group of $\Cos(G,K,KgK)$ is isomorphic to $G.2$ and has a subgroup $5_+^{1+2}{:}8{:}2$ transitive on vertices. Note that the valency of $\Cos(G,K,KgK)$ is $|K|/|M|=22$. Hence $\Cos(G,K,KgK)$ is not $3$-arc-transitive since $|M|$ is not divisible by $3^2\cdot7^2=(22-1)^2$. Moreover, $\Cos(G,K,KgK)$ is non-bipartite since $G$ does not have a subgroup of index two.
\end{example}

\subsection{Other facts}

The following lemma is a well-known result in elementary number theory, which is a consequence of so-called Legendre's formula.

\begin{lemma}\label{divisors-n!}
For any positive integer $n$ and prime $p$, the $p$-part $(n!)_p<p^{n/(p-1)}$.
\end{lemma}

We also need a theorem of Dixon on the maximal order of solvable permutation groups of given degree.

\begin{theorem}\label{l13}
\emph{(Dixon \cite[Theorem~3]{Dixon1967})} For any solvable subgroup $R$ of $\Sy_n$, the order $|R|\leqslant24^{(n-1)/3}$.
\end{theorem}

The next lemma is a consequence of~\cite[Corollary~5]{LPS}.

\begin{lemma}\label{l14}
Let $T$ ba a nonabelian simple group. Then for any solvable subgroup $R$ of $T$, there exists a prime divisor $r$ of $|T|$ not dividing $|R|$.
\end{lemma}

\section{A property of finite simple groups}\label{pty-Gorenstein}

%

A well-known theorem of Gorenstein \cite[Theorem~1.53]{gorenstein} says that for a nonabelian simple group $T$ of order divisible by a prime $r$, if $|T|$ is not divisible by $r^2$ then $|\Out(T)|$ is not divisible by $r$. We need an improved version of this theorem, which will be established in this section. The proof will invoke the following elementary number theory result. Recall that for any positive integer $f$, we denote the $r$-part of $f$ by $f_r$ and denote $f_{r'}=f/f_r$.

\begin{lemma}\label{r-part}
Let $f$ be a positive integer, $r$ be a prime number, and
$$
r_0=
\begin{cases}
r,\quad\text{if $r>2$}\\
4,\quad\text{if $r=2$}.
\end{cases}
$$
Then for any integer $t>1$, the following statements hold.
\begin{itemize}
\item[(a)] $r\di t^f-1 $ if and only if $r\di t^{f_{r'}}-1$.
\item[(b)] If $t\equiv1\pmod{r_0}$, then $(t^f-1)_r=f_r(t-1)_r$.
\item[(c)] If $r_0\di t^f-1$, then $(t^f-1)_r\geqslant r_0f_r$.
\end{itemize}
\end{lemma}

\begin{proof}
Let $f_r=r^m$, where $m\geqslant0$. As $r-1$ divides $r^m-1$, we derive $t^f=t^{f_{r'}(r^m-1)}t^{f_{r'}}\equiv t^{f_{r'}}\pmod{r}$ by Fermat's little theorem. Then part~(a) follows immediately.

Suppose that $t\equiv1\pmod{r_0}$. Then writing $u=t-1$ we have
$$
\frac{t^r-1}{t-1}=\frac{(u+1)^r-1}{u}=r+\frac{r(r-1)u}{2}+\sum\limits_{k=3}^r{r\choose k}u^{k-1}.
$$
Since $r_0\di u$, the above equality implies that $(t^r-1)/(t-1)\equiv r\pmod{r^2}$. As a consequence, $(t^r-1)_r=r(t-1)_r$. Applying this repeatedly, we obtain
$$
(t^f-1)_r=(t^{f_{r'}r^m}-1)_r=r(t^{f_{r'}r^{m-1}}-1)_r=\dots=r^m(t^{f_{r'}}-1)_r.
$$
In the meanwhile, the condition $t\equiv1\pmod{r}$ implies that
$$
\frac{t^{f_{r'}}-1}{t-1}=1+t+t^2+\dots+t^{f_{r'}-1}\equiv f_{r'}\pmod{r}
$$
and thus $(t^{f_{r'}}-1)/(t-1)$ is not divisible by $r$. Hence $(t^f-1)_r=r^m(t^{f_{r'}}-1)_r=r^m(t-1)_r$ as part~(b) asserts.

Next suppose that $t>1$ is an integer with $r_0\di t^f-1$. It follows from part~(a) that $r\di t^{f_{r'}}-1$. If $r>2$, then by part~(b) (replacing $t$ with $t^{f_{r'}}$ there), one has $(t^f-1)_r=f_r(t^{f_{r'}}-1)_r\geqslant rf_r$. If $r=2$ and $f_r=1$, then part~(c) holds trivially. If $r=2$ and $i_r\geqslant2$, then $t$ is odd and so $t^2\equiv1\pmod{8}$, whence we conclude from part~(b) that
$$
(t^f-1)_r=(t^{2f_{2'}\cdot2^{m-1}}-1)_2=2^{m-1}(t^{2f_{2'}}-1)_2\geqslant2^{m-1}\cdot8=2^{m+2}=r_0f_r.
$$
This proves part~(c).
\end{proof}

Now we give the improved version of Gorenstein's theorem mentioned above.

\begin{theorem}\label{CommonDivisor}
Let $T$ be a simple group, and $r$ be a common prime divisor of $|T|$ and $|\Out(T)|$. Then $|T|_r\geqslant r|\Out(T)|_r$, and further, for $r=2$ or $3$, $|T|_r=r|\Out(T)|_r$ if and only if one of the following occurs, where $p$ is a prime.
\begin{itemize}
\item[(a)] $r=2$, and $T\cong\PSL_2(p^f)$ with $p\equiv\pm3\pmod{8}$.
\item[(b)] $r=3$, and $T\cong\PSL_2(p^f)$ with $p\equiv\pm2$ or $\pm4\pmod{9}$ and $f\equiv0\pmod{3}$.
\item[(c)] $r=3$, and $T\cong\PSL_3(p^f)$, where either $p\equiv2$ or $5\pmod{9}$ and $f\equiv2$, $3$ or $4\pmod{6}$, or $p\equiv4$ or $7\pmod{9}$ and $f\not\equiv0\pmod{3}$.
\item[(d)] $r=3$, and $T\cong\PSU_3(p^f)$, where either $p\equiv2$ or $5\pmod{9}$ and $f\equiv0$, $1$ or $5\pmod{6}$, or $p\equiv4$ or $7\pmod{9}$ and $f\equiv0\pmod{3}$.
\end{itemize}
\end{theorem}

\begin{proof}
Assume that $T$ is alternating or sporadic. Then either $|\Out(T)|\leqslant2$, or $|\Out(T)|=4$ and $T=\A_6$. Hence $r=2$, and the inequality $|T|_r\geqslant r|\Out(T)|_r$ holds steadily. Further, $|T|_2=2|\Out(T)|_2$ if and only if $T=\A_5\cong\PSL_2(5)$ or $T=\A_6\cong\PSL_2(9)$, as in part~(a) of the lemma.

In the remainder of the proof, assume that $T$ is a simple group of Lie type defined over a field of characteristic $p$, and $T$ is not isomorphic to $\A_5$, $\A_6$, $\A_8$ or $\A_9$. If $r=p$, then either $T=\PSL_2(p^f)$, or $|T|_p\geqslant q^3>p|\Out(T)|_p$. Suppose that $T=\PSL_2(p^f)$. Then $p^f\geqslant8$. For $p=2$ or $3$, $|T|_p=p^f>pf\geqslant pf_p=p|\Out(T)|_p$. For $p\geqslant5$, $|T|_p=p^f\geqslant pf\geqslant pf_p=p|\Out(T)|_p$. Thus assume that $r\neq p$ hereafter.

{\bf Case 1.} Suppose that $r=2$. Then, in particular, $T\neq\,^2\B_2(q)$ or $\,^2\F_4(q)$. An inspection of $|T|$ and $|\Out(T)|$ for simple groups $T$ of Lie type shows that $|\Out(T)|_2\leqslant2d_2f_2$, and one of the following holds.
\begin{itemize}
\item[(i)] $|T|$ is divisible by $d(q^2-1)$.
\item[(ii)] $T=\PSL_2(q)$, $|T|_2=(q^2-1)_2/2$, and $|\Out(T)|_2=2f_2$.
\item[(iii)] $T=\,^2\G_2(q)$ with $q=3^{2c+1}\geqslant3^3$, $|T|_2=(q^3+1)_2(q-1)_2$, and $|\Out(T)|_2=1$.
\end{itemize}
Since $4\di q^2-1$, we conclude from Lemma \ref{r-part}(c) that
\begin{equation}\label{eq7}
(q^2-1)_2=(p^{2f}-1)_2\geqslant4(2f)_2=8f_2.
\end{equation}
If (i) occurs, then the above equality implies
$$
|T|_2\geqslant d_2(q^2-1)_2\geqslant8d_2f_2>2|\Out(T)|_2.
$$
If (iii) occurs, then apparently $|T|_2>2=2|\Out(T)|_2$.

Now assume that (ii) occurs. Then by (\ref{eq7}),
$$
|T|_2=(q^2-1)_2/2\geqslant4f_2=2|\Out(T)|_2.
$$
Moreover, $|T|_2=2|\Out(T)|_2$ if and only if
\begin{equation}\label{eq6}
(q^2-1)_2=8f_2.
\end{equation}
Note that $(q^2-1)_2=((p^2)^f-1)_2=f_2(p^2-1)_2$ by Lemma \ref{r-part}(b). We see that (\ref{eq6}) is equivalent to $(p^2-1)_2=8$, which is further equivalent to the condition in part~(a) of the lemma.

{\bf Case 2.} Suppose that $r=3$. Then $T\neq\,^2\G_2(q)$ and $p^2\equiv1\pmod{3}$. An inspection of $|T|$ and $|\Out(T)|$ divides simple groups $T$ of Lie type into the following six classes.
\begin{itemize}
\item[(i)] $|T|$ is divisible by $3d(q^2-1)$, and $|\Out(T)|_3=d_3f_3$.
\item[(ii)] $T=\PSL_2(q)$, $|T|_3=(q^2-1)_3$, and $|\Out(T)|_3=f_3$.
\item[(iii)] $T=\PSL_3(q)$, $|T|_3=(q^3-1)_3(q^2-1)_3/d_3$, and $|\Out(T)|_3=d_3f_3$.
\item[(iv)] $T=\PSU_3(q)$, $|T|_3=(q^3+1)_3(q^2-1)_3/d_3$, and $|\Out(T)|_3=d_3f_3$.
\item[(v)] $T=\POm_8^+(q)$, $|T|_3=(q^6-1)_3(q^4-1)_3^2(q^2-1)_3$, and $|\Out(T)|_3=3f_3$.
\item[(vi)] $T=\,^3\D_4(q)$, $|T|_3=(q^8+q^4+1)_3(q^6-1)_3(q^2-1)_3$, and $|\Out(T)|_3=3f_3$.
\end{itemize}
By Lemma \ref{r-part}(b) we have $(q^2-1)_3=f_3(p^2-1)_3$. If (i), (v) or (vi) occurs, then
$$
|T|_3/|\Out(T)|_3\geqslant3(q^2-1)_3/f_3=3(p^2-1)_3>3.
$$

Assume that~(ii) occurs. Then $|T|_3=(q^2-1)_3=f_3(p^2-1)_3=(p^2-1)_3|\Out(T)|_3$. Hence $|T|_3\geqslant3|\Out(T)|_3$, and $|T|_3=3|\Out(T)|_3$ if and only if $(p^2-1)_3=3$, which is equivalent to $p\equiv\pm2$ or $\pm4\pmod{9}$. This together with the condition that $3$ divides $|\Out(T)|$ leads to part~(b) of the lemma.

Next assume that (iii) appears. Then $|T|_3/|\Out(T)|_3=(q^3-1)_3(p^2-1)_3/d_3^2$. If $q\equiv2\pmod{3}$, then $|T|_3/|\Out(T)|_3=(p^2-1)_3\geqslant3$, and $|T|_3/|\Out(T)|_3=3$ is equivalent to $(p^2-1)_3=3$, which occurs exactly when $p\equiv2$ or $5\pmod{9}$ and $f$ odd. If $q\equiv1\pmod{3}$, then
$$
\frac{|T|_3}{|\Out(T)|_3}=\frac{(q^3-1)_3(p^2-1)_3}{9}=\frac{(q^6-1)_3(p^2-1)_3}{9}=\frac{3f_3(p^2-1)_3^2}{9}\geqslant3,
$$
and $|T|_3/|\Out(T)|_3=3$ is equivalent to $f_3(p^2-1)_3^2=9$, which occurs exactly when either $p\equiv2$ or $5\pmod{9}$ and $f\equiv\pm2\pmod{6}$, or $p\equiv4$ or $7\pmod{9}$ and $f\not\equiv0\pmod{3}$. To sum up, $|T|_3=3|\Out(T)|_3$ if and only if either $p\equiv2$ or $5\pmod{9}$ and $f\not\equiv0\pmod{6}$, or $p\equiv4$ or $7\pmod{9}$ and $f\not\equiv0\pmod{3}$. This together with the condition that $3$ divides $|\Out(T)|$ leads to part~(c) of the lemma.

Now assume that (iv) appears. Then $|T|_3/|\Out(T)|_3=(q^3+1)_3(p^2-1)_3/d_3^2$. If $q\equiv1\pmod{3}$, then $|T|_3/|\Out(T)|_3=(p^2-1)_3\geqslant3$, and $|T|_3/|\Out(T)|_3=3$ is equivalent to $(p^2-1)_3=3$, which occurs exactly when $p\equiv2$ or $5\pmod{9}$ and $f\equiv\pm1\pmod{6}$. If $q\equiv2\pmod{3}$, then
$$
\frac{|T|_3}{|\Out(T)|_3}=\frac{(q^3+1)_3(p^2-1)_3}{9}=\frac{(q^6-1)_3(p^2-1)_3}{9}=\frac{3f_3(p^2-1)_3^2}{9}\geqslant3,
$$
and $|T|_3/|\Out(T)|_3=3$ is equivalent to $f_3(p^2-1)_3^2=9$, which occurs exactly when either $p\equiv2$ or $5\pmod{9}$ and $f$ is even, or $p\equiv4$ or $7\pmod{9}$. To sum up, $|T|_3=3|\Out(T)|_3$ if and only if either $p\equiv2$ or $5\pmod{9}$ and $f\not\equiv3\pmod{6}$, or $p\equiv4$ or $7\pmod{9}$. This together with the condition that $3$ divides $|\Out(T)|$ leads to part~(d) of the lemma.

{\bf Case 3.} Assume that $r\geqslant5$. Then an inspection of $|T|$ and $|\Out(T)|$ for simple groups $T$ of Lie type shows that $|T|_p\geqslant q$ and $|\Out(T)|_r=d_rf_r$.

Suppose that $d_r>1$. Then $T=\PSL_n(q)$ or $\PSU_n(q)$, and $n\geqslant5$. First assume that $T=\PSL_n(q)$. Then $r\di q-1=p^f-1$, and hence we conclude from Lemma~\ref{r-part}(c) that $(q-1)_r=(p^f-1)_r\geqslant rf_r$. This together with the observation that $|T|$ is divisible by $d(q-1)$ yields $|T|_r\geqslant d_r(q-1)_r\geqslant rd_rf_r=r|\Out(T)|_r$. Now assume that $T=\PSU_n(q)$. Then $r\di q+1$, whence $r$ divides $q^2-1$ but not $q-1$. It follows that $(q+1)_r=(q^2-1)_r$, and appealing Lemma~\ref{r-part}(c) we obtain $(q+1)_r=(p^{2f}-1)_r\geqslant rf_r$. This together with the observation that $|T|$ is divisible by $d(q+1)$ yields $|T|_r\geqslant d_r(q+1)_r\geqslant rd_rf_r=r|\Out(T)|_r$.

Suppose that $d_r=1$. As $r$ divides $|T|$, an inspection of $|T|$ and $|\Out(T)|$ for simple groups $T$ of Lie type shows that one the following happens.
\begin{itemize}
\item[(i)] $r\di q^i-1$ for some $i\geqslant1$ with $|T|_r\geqslant(q^i-1)_r$.
\item[(ii)] $r\di q^i+1$ for some $i\geqslant2$ with $|T|_r\geqslant(q^i+1)_r$.
\item[(iii)] $r\di q^8+q^4+1$ with $|T|_r\geqslant(q^8+q^4+1)_r$.
\end{itemize}
First assume that (i) appears. Then we conclude from Lemma~\ref{r-part}(c) that $(q^i-1)_r=(p^{if}-1)_r\geqslant rf_r$, which implies
$|T|_r\geqslant(q^i-1)_r\geqslant rf_r=r|\Out(T)|_r$. Next assume that (ii) appears. Then $r$ divides $q^{2i}-1$ but not $q^i-1$. It follows that $(q^i+1)_r=(q^{2i}-1)_r$, and appealing Lemma~\ref{r-part}(c) we obtain $(q^i+1)_r=(p^{2if}-1)_r\geqslant rf_r$. This leads to $|T|_r\geqslant(q^i+1)_r\geqslant rf_r=r|\Out(T)|_r$. Finally assume that (iii) appears. If $r\di q^4-1$, then $q^8+q^4+1\equiv3\pmod{r}$, contrary to the condition that $r\di q^8+q^4+1$. Consequently, $r$ does not divide $q^4-1$, and so $(q^8+q^4+1)_r=(q^{12}-1)_r$. Now appealing Lemma~\ref{r-part}(c) we have $(q^8+q^4+1)_r=(p^{12f}-1)_r\geqslant rf_r$, which implies
$$
|T|_r\geqslant(q^8+q^4+1)_r\geqslant rf_r=r|\Out(T)|_r.
$$
The proof is thus completed.
\end{proof}

\section{Reduction to affine and almost simple groups}\label{qp-case}

In~\cite[Theorem~2]{Praeger} Praeger proved that the quasiprimitive groups acting $2$-arc-transitively on a connected graph can be divided into four different types, which were later called HA, AS, TW and PA~\cite[Theorem~5.1]{Praeger1997}, see Lemma~\ref{normal-q}. In this section, we determine the quasiprimitive types for those containing a vertex-transitive solvable subgroup. It will be shown that only types HA and AS can occur.

We fix some notation throughout this section. Let $\Ga=(V,E)$ be a connected $(G,2)$-arc-transitive graph with $G$ quasiprimitive on $V$, and $R$ be a solvable vertex-transitive subgroup of $G$. Take an edge $\{\alpha,\beta\}\in E$ and denote $N=\Soc(G)$. Then $N=T_1\times\dots\times T_\ell$, where $\ell\geqslant2$ and $T_1\cong\dots\cong T_\ell\cong T$ for some nonabelian simple group $T$. Let $\calT=\{T_1,\dots,T_\ell\}$, and $K$ be the kernel of $G$ acting by conjugation on $\calT$. Then $N\trianglelefteq K\leqslant\Aut(T_1)\times\dots\times\Aut(T_\ell)$.

\begin{lemma}\label{X=soluble-pro}
Assume that $N_\alpha$ is solvable. Then either $T=\PSL_2(q)$ with $q\geqslant5$, or $T=\PSU_3(8)$, $\PSU_4(2)$, $\PSL_4(2)$, $\M_{11}$, or $\PSL_3(q)$ with $q\in\{3,4,5,7,8\}$.
\end{lemma}

\begin{proof}
Let $X=NR$. Then $X=NX_\alpha$, and
$$
X_\alpha/N_\alpha\cong NX_\alpha/N=X/N=NR/N\cong R/(R\cap N),
$$
whence $X_\alpha/N_\alpha$ is solvable. This together with the assumption that $N_\alpha$ is solvable implies that $X_\alpha$ is solvable. Hence $X$ is a product of two solvable groups $R$ and $X_\alpha$, and then by \cite{kazarin1986groups} we conclude that $T$ is a one of the groups listed in the lemma.
\end{proof}

One will see that for TW or PA type $G$, the assumption of Lemma~\ref{X=soluble-pro} is satisfied, so that the simple groups listed in Lemma~\ref{X=soluble-pro} are all the possibilities for $T$. For later use, we give some properties for these groups in the next lemma.

\begin{lemma}\label{l16}
Let $T$ be a simple group listed in \emph{Lemma~\ref{X=soluble-pro}}, $r$ be the largest prime divisor of $|T|$, and $m$ be the smallest index of solvable subgroups in $T$. Then $r\leqslant m$, and the following statements hold.
\begin{itemize}
\item[(a)] If $m<3|\Out(T)|$, then $T=\PSL_2(5)$ or $\PSL_2(9)$.
\item[(b)] If $m\leqslant60$, then either $(T,m)=(\PSL_2(q),q+1)$ with prime power $8\leqslant q\leqslant59$, or $(T,m)$ lies in \emph{Table~\ref{tab20}}.
\end{itemize}
\end{lemma}

\begin{table}[htbp]
\caption{}\label{tab20}
\centering
\begin{tabular}{|l|llllll|}
\hline
$T$&$\PSL_2(5)$&$\PSL_2(7)$&$\PSL_3(3)$&$\PSL_4(2)$&$\PSU_4(2)$&$\M_{11}$\\
\hline
$m$&$5$&$7$&$13$&$35$&$40$&$55$\\
\hline
\end{tabular}
\end{table}

\begin{proof}
First, for $T=\PSU_3(8)$, $\PSU_4(2)$, $\PSL_4(2)$, $\M_{11}$, or $\PSL_3(q)$ with $q\in\{3,4,5,7,8\}$, the value of $m$ can be derived from~\cite{atlas}, verifying the lemma directly. Next assume that $T=\PSL_2(q)$ with prime power $q\geqslant5$. For $q=5,7,9,11$, we see from~\cite{atlas} that $m=5,7,10,12$, respectively, satisfying the conclusion of the lemma. Thus assume that $q=8$ or $q\geqslant13$ in the following. Then according to Theorem~\ref{l7}, the smallest index of proper subgroups of $T$ is $q+1$. As a consequence, $m\geqslant q+1$. At the meanwhile, $T=\PSL_2(q)$ has a solvable subgroup $q{:}(q-1)/(2,q-1)$ of index $q+1$. Thereby we obtain $m=q+1\geqslant3|\Out(T)|$. Moreover, $r\di q(q+1)(q-1)$ implies that $r\di q$, $q+1$ or $q-1$, and hence $r\leqslant q+1=m$. This proves the lemma.
\end{proof}

\begin{lemma}\label{l11}
$G$ is not of type TW.
\end{lemma}

\begin{proof}
Suppose that $G$ is quasiprimitive of TW type. Then $N_\alpha=1$, and $|V|=|N|=|T|^\ell$. As a consequence, $T$ is a simple group listed in Lemma~\ref{X=soluble-pro}. Let $m$ be the smallest index of solvable subgroups of $\Aut(T)$, and $\overline{R}=RK/K$. Then $\overline{R}\leqslant G/K\lesssim\Sy_\ell$, and $|R\cap K||\overline{R}|=|R|$ is divisible by $|T|^\ell$ since $R$ is transitive on $V$. Let $R_i$ be the projection of $R\cap K$ into $\Aut(T_i)$, where $1\leqslant i\leqslant\ell$. Then $R\cap K\lesssim R_1\times\dots\times R_\ell$. By Lemma~\ref{l16}, either $m\geqslant3|\Out(T)|$, or $T=\PSL_2(5)$ or $\PSL_2(9)$.

First assume that $m\geqslant3|\Out(T)|$. For $1\leqslant i\leqslant\ell$, since $|R_i|m\leqslant|\Aut(T)|$, we have $|R_i|\leqslant|T|/3$. Now $|R\cap K|\leqslant|R_1|\cdots|R_\ell|\leqslant|T|^\ell/3^\ell$, while $|R\cap K||\overline{R}|\geqslant|T|^\ell$. Thus $|\overline{R}|\geqslant3^\ell$, which violates Theorem~\ref{l13} on the order of the solvable subgroup $\overline{R}$ in $\Sy_\ell$.

Next assume that $T=\PSL_2(5)$. Then any solvable subgroup of $\Aut(T)$ has order dividing $24$ or $20$. Hence $|R_1|\cdots|R_\ell|$ divides $24^x20^y$ for some nonnegative integers $x$ and $y$ with $x+y=\ell$. Viewing that $|R\cap K|$ divides $|R_1|\cdots|R_\ell|$ and $|\overline{R}|$ divides $\ell!$, we obtain $60^\ell\di24^x20^y\ell!$ since $60^\ell$ divides $|R|=|R\cap K||\overline{R}|$. In particular, $3^\ell\di3^x(\ell!)_3$ and $5^\ell\di5^y(\ell!)_5$. This implies $\ell<x+\ell/2$ and $\ell<y+\ell/4$ by Lemma~\ref{divisors-n!}, which leads to a contradiction that $2\ell<(x+\ell/2)+(y+\ell/4)=x+y+3\ell/4=7\ell/4$.

Finally assume that $T=\PSL_2(9)$. Then any solvable subgroup of $\Aut(T)$ has order dividing $288$ or $40$, and so $|R\cap K|$ divides $288^x40^y$ for some nonnegative integers $x$ and $y$ with $x+y=\ell$. It follows that $360^\ell\di288^x40^y\ell!$. In particular, $3^{2\ell}\di3^{2x}(\ell!)_3$ and $5^\ell\di5^y(\ell!)_5$. This implies $2\ell<2x+\ell/2$ and $\ell<y+\ell/4$ by Lemma~\ref{divisors-n!}, which leads to a contradiction that $2\ell<x+y+\ell/2=3\ell/2$. The proof is thus completed.
\end{proof}

In the next few lemmas we exclude quasiprimitive type PA for $G$. Recall that $K\leqslant\Aut(T_1)\times\dots\times\Aut(T_\ell)$. We call a subgroup of $K$ \emph{diagonal} if it isomorphic to its projection into $\Aut(T_i)$ for each $1\leqslant i\leqslant\ell$.

\begin{lemma}\label{K-imprim}
If $G$ is of type PA, then $K_\alpha^{\Ga(\alpha)}$ is transitive and imprimitive.
\end{lemma}

\begin{proof}
As $G$ is quasiprimitive of type PA, we have $N_\alpha\neq1$. It follows from Lemma~\ref{l10} that $N_\alpha^{\Ga(\alpha)}$ is transitive, and so is $K_\alpha^{\Ga(\alpha)}$. For $1\leqslant i\leqslant\ell$, let $R_i$ be the projection of $R\cap N$ into $T_i$ and $M_i=\C_K(T_i)$. Then $M_i\lhd K$, and $K/M_i$ is almost simple with socle $T$. Note that $M_i\times T_i\geqslant N$ is transitive on $V$ while $M_i$ is not. We infer that $T_i$ transitively permutes the orbits of $M_i$, and so $M_i$ has more than two orbits.

Suppose for a contradiction that $K_\alpha^{\Ga(\alpha)}$ is primitive. Then $M_i$ is semiregular on $V$ by Lemma~\ref{normal-quotient}. Hence $K_\alpha\cap M_i=1$ for each $1\leqslant i\leqslant\ell$, which means that $K_\alpha$ is diagonal. Without loss of generality, assume $K_\alpha=\{(x,\dots,x)\mid x\in P\}$ for some $P\leqslant\Aut(T)$. Then $K_\alpha\cong P$ and $N_\alpha=K_\alpha\cap N=K_\alpha\cap(T_1\times\dots\times T_\ell)=\{(x,\dots,x)\mid x\in P\cap T\}$. Consequently, $|V|=|N|/|N_\alpha|$ is divisible by $|T|^{\ell-1}$, and $K_\alpha/N_\alpha\cong P/P\cap T\cong PT/T\leqslant\Out(T)$. Since $N$ is transitive on $V$, we derive that
$$
K/N=K_\alpha N/N\cong K_\alpha/K_\alpha\cap N=K_\alpha/N_\alpha\lesssim\Out(T).
$$
Hence $|R\cap K|/|R\cap N|$ divides $|\Out(T)|$, and $|R\cap K||RK/K|=|R|$ is divisible by $|T|^{\ell-1}$ due to the vertex-transitivity of $R$. Then as $|R\cap N|$ divides $|R_1\times\dots\times R_\ell|$,
\begin{equation}\label{eq27}
\mbox{$|T|^{\ell-1}$ divides $|R_1\times\dots\times R_\ell||\Out(T)||RK/K|$.}
\end{equation}

Consider an arbitrary orbit of $R$ by conjugation on $\calT$, say $\{T_1,\dots,T_k\}$, where $1\leqslant k\leqslant\ell$. Clearly, $R_1\cong\dots\cong R_k$. Let $\overline{N}=T_1\times\dots\times T_k$, $L=M_1\cap\dots\cap M_k$, $\overline{K}=KL/L$, $\overline{V}$ be the orbits of $L$ on $V$, $v\in\overline{V}$, and $\overline{R}$ be the permutation group on $\overline{V}$ induced by $R$. Then $\overline{N},\overline{K}$ are transitive permutation groups on $\overline{V}$, and since $K_\alpha$ is diagonal, $\overline{K}_v$ is diagonal too. Moreover, the projection of $\overline{R}\cap\overline{N}$ into $T_i$ is still $R_i$ for $1\leqslant i\leqslant k$. Noticing $\overline{R}\,\overline{K}/\overline{K}\lesssim\Sy_k$, then along the same lines of the previous paragraph, we derive that
\begin{equation}\label{eq5}
\mbox{$|T|^{k-1}$ divides $|R_1\times\dots\times R_k||\Out(T)||\Sy_k|=|R_1|^k|\Out(T)|k!$}.
\end{equation}
Assume that $k\geqslant1$. Since $R_1$ is a solvable subgroup of $T_1$, we know from Lemma~\ref{l14} that there exists a prime $r$ dividing $|T|$ but not $|R_1|$. By Theorem~\ref{CommonDivisor}, $|T|_r\geqslant r|\Out(T)|_r$, whence $r|T|_r^{k-2}$ divides $(k!)_r$ by (\ref{eq5}). Accordingly,
\begin{equation}\label{eq4}
r|T|_r^{k-2}\leqslant(k!)_r<r^{k/(r-1)}
\end{equation}
by Lemma~\ref{divisors-n!}. In particular, $r^{k-1}<r^{k/(r-1)}$, which forces $r=2$. As $|T|$ is divisible by $4$, (\ref{eq4}) implies that $2^{2k-3}<2^k$, and so $k\leqslant2$. If $k=2$, then $|T|_2$ divides $2|\Out(T)|_2$ because of (\ref{eq5}), and thereby we have $T=\PSL_2(p^f)$ with prime $p\equiv\pm3\pmod{8}$ by Theorem~\ref{CommonDivisor}. However, in this situation (\ref{eq5}) requires that $|\PSL_2(p^f)|=p^f(p^{2f}-1)/2$ divides $2|R_1|^2|\Out(T)|=4f|R_1|^2$, which is not possible for a solvable subgroup $R_1$ of $\PSL_2(p^f)$. Therefore, $k=1$. It follows that $\overline{K}=K/M_1$ is an almost simple group with socle $T$, and $\overline{R}\cap\overline{N}=\overline{R}\cap\soc(\overline{K})\cong R_1$.

The conclusion of the previous paragraph implies that $R$ normalizes $T_i$ for each $1\leqslant i\leqslant\ell$, whence $R\leqslant K$. For $1\leqslant i\leqslant\ell$ denote by $\alpha_i$ the orbit of $M_i$ on $V$ containing $\alpha$. Then by the transitivity of $R$ we have a factorization $K/M_i=(RM_i/M_i)(K/M_i)_{\alpha_i}$. Moreover, $K/M_i$ is an almost simple group with socle $T$, $(RM_i/M_i)\cap\soc(K/M_i)\cong R_i$, and $(K/M_i)_{\alpha_i}\cong K_\alpha$ by Lemma~\ref{normal-quotient}. Let $Y=RT_1$. Then $T_1$ is normal in $Y$ as $T_1$ is normal in $K$. Since $(T_1)_\alpha=1$ by Lemma~\ref{normal-quotient}, we have $Y_\alpha\cap T_1=1$, and thus
$$
Y_\alpha=Y_\alpha/Y_\alpha\cap T_1\cong Y_\alpha T_1/T_1\leqslant Y/T_1=RT_1/T_1\cong R/R\cap T_1
$$
is solvable. Now $Y=RY_\alpha$ is a product of two solvable subgroups, so by~\cite{kazarin1986groups}, $T\cong T_1$ is one of the groups: $\PSL_2(q)$ with $q\geqslant5$, $\PSU_3(8)$, $\PSU_4(2)$, $\PSL_4(2)$, $\M_{11}$, and $\PSL_3(q)$ with $q\in\{3,4,5,7,8\}$. Applying Theorem~\ref{SolvableFactor} to the factorizations $K/M_i=(RM_i/M_i)(K/M_i)_{\alpha_i}$ for $1\leqslant i\leqslant\ell$, where $\soc(K/M_1)\cong\dots\cong\soc(K/M_\ell)\cong T$ is one of the above groups and $(K/M_1)_{\alpha_1}\cong\dots\cong(K/M_\ell)_{\alpha_\ell}$, we conclude that there exists a prime divisor $r$ of $|T|$ dividing none of $|(RM_i/M_i)\cap\soc(K/M_i)|$ for $1\leqslant i\leqslant\ell$. Consequently, $r$ does not divide $|R_1\times\dots\times R_k|$. Thus (\ref{eq27}) implies that $|T|_r^{\ell-1}$ divides $|\Out(T)|_r$, which is not possible by Theorem~\ref{CommonDivisor}. This contradiction completes the proof.
\end{proof}

\begin{lemma}\label{Restriction}
Let $\overline{R}=RK/K$. If $G$ is of type PA, then $G_\alpha^{\Ga(\alpha)}$ is an affine $2$-transitive permutation group of degree $p^d$, where $p$ is prime and $d\geqslant2$, and the following statements hold.
\begin{itemize}
\item[(a)] $\overline{R}$ is isomorphic to a solvable subgroup of $G_{\alpha\beta}/K_{\alpha\beta}$, and $|T|^\ell/|R_1\times\dots R_\ell|$ divides $p^d|K_{\alpha\beta}||\overline{R}|$.
\item[(b)] $|T|^\ell/|R_1\times\dots R_\ell|$ divides $|K_{\alpha\beta}^{\Ga(\alpha)}||G_\alpha^{\Ga(\alpha)}|$.
\item[(c)] $K_\alpha$ is solvable, and $K_{\alpha\beta}$ is diagonal.
\item[(d)] $T$ is one of the groups listed in \emph{Lemma~\ref{X=soluble-pro}}.
\end{itemize}
\end{lemma}

\begin{proof}
Suppose that $G$ has type PA. Then $K_\alpha^{\Ga(\alpha)}$ is imprimitive by Lemma~\ref{K-imprim}. We thereby see from Lemma~\ref{normal-stab} that $G_\alpha^{\Ga(\alpha)}$ is an affine $2$-transitive permutation group of degree $p^d$, where $p$ is prime and $d\geqslant2$, and $K_\alpha$ is solvable.

Since $N_\alpha\neq1$ and $K_\alpha\neq1$, Lemma~\ref{l10} shows that $G/N\cong G_{\alpha\beta}/N_{\alpha\beta}$ and $G/K\cong G_{\alpha\beta}/K_{\alpha\beta}$. Consequently, $\overline{R}$ is isomorphic to a solvable subgroup of $G_{\alpha\beta}/K_{\alpha\beta}$, and $|K/N|=|K_{\alpha\beta}/N_{\alpha\beta}|$. Because $N$ is arc-transitive by Lemma~\ref{l10} and $R$ is transitive on $V$, so $|R|$ is divisible by $|V|=|N|/|N_\alpha|=|T|^\ell/(p^d|N_{\alpha\beta}|)$. Then as $|R|=|R\cap K||\overline{R}|$ divides $|K/N||R\cap N||\overline{R}|=|K_{\alpha\beta}/N_{\alpha\beta}||R\cap N||\overline{R}|$, we deduce that
$|T|^\ell/|R_1\times\dots R_\ell|$ divides $p^d|K_{\alpha\beta}||\overline{R}|$. Hence part~(a) holds.

Viewing $\overline{R}\lesssim G_{\alpha\beta}/K_{\alpha\beta}$, we derive from part~(a) that $|T|^\ell/|R_1\times\dots R_\ell|$ divides $p^d|G_{\alpha\beta}|=|G_\alpha|=|G_\alpha^{[1]}||G_\alpha^{\Ga(\alpha)}|$. Thus part~(b) is true since $|G_\alpha^{[1]}|$ divides $|K_{\alpha\beta}^{\Ga(\alpha)}|$ by Lemma~\ref{normal-stab}.

To prove that $K_{\alpha\beta}$ is diagonal, take an arbitrary $i\in\{1,\dots,\ell\}$ and let $M=\prod_{j\not=i}K_j$. If $M_\alpha^{\Ga(\alpha)}$ is transitive, then the neighbors of $\alpha$ are in the same orbit of $M$ and so $\Ga$ will be bipartite. Hence $M_\alpha^{\Ga(\alpha)}$ is an intransitive normal subgroup of the Frobenius group $K_\alpha^{\Ga(\alpha)}$. This forces $M_{\alpha\beta}^{\Ga(\alpha)}=1$, which means $M_{\alpha\beta}=M_\alpha^{[1]}$. Similarly, $M_{\alpha\beta}=M_\beta^{[1]}$. It follows that $K_{\alpha\beta}\cap M=M_{\alpha\beta}=M_\alpha^{[1]}\cap M_\beta^{[1]}=M_{\alpha\beta}^{[1]}=1$. Therefore, $K_{\alpha\beta}$ is diagonal, proving part~(c).

Finally, as $K_\alpha$ is solvable we know that $N_\alpha$ is solvable too. Then Lemma~\ref{X=soluble-pro} shows that $T$ is one of the groups listed there.
\end{proof}

\begin{lemma}\label{l9}
$G$ is not of type PA.
\end{lemma}

\begin{proof}
Let $r$ be the largest prime divisor of $|T|$, $m$ be the smallest index of solvable subgroups in $T$, and $\overline{R}=RK/K$. Suppose for a contradiction that $G$ has type PA. Then from Lemma~\ref{Restriction} we know that $G_\alpha^{\Ga(\alpha)}$ is an affine $2$-transitive permutation group of degree $p^d$, where $p$ is prime and $d\geqslant2$. The affine $2$-transitive permutation groups were classified by Hering~\cite{Hering1985}, see also~\cite{Liebeck-rank3}. By virtue of this classification, we will exclude the candidates of $2$-transitive permutation groups for $G_\alpha^{\Ga(\alpha)}$ case by case. According to Lemma~\ref{Restriction}(d), $T$ satisfies Lemma~\ref{X=soluble-pro}, and so the integers $r\leqslant m$ satisfy Lemma~\ref{l16}. Noticing $N_\alpha^{\Ga(\alpha)}\geqslant\Soc(G_\alpha^{\Ga(\alpha)})=\Z_p^d$ as $N_\alpha^{\Ga(\alpha)}\vartriangleleft G_\alpha^{\Ga(\alpha)}$, we see that $|T|^\ell=|N|$ is divisible by $p$. Hence $|T|$ is divisible by $p$, and so $m\geqslant r\geqslant\max\{5,p\}$. Then the observation $m^\ell\leqslant|T|^\ell/|R_1\times\dots R_\ell|$ together with Lemma~\ref{Restriction}(b) yields
\begin{equation}\label{eq2}
\max\{5,p\}^\ell\leqslant m^\ell\leqslant|K_{\alpha\beta}^{\Ga(\alpha)}||G_\alpha^{\Ga(\alpha)}|.
\end{equation}
For any group $X$, let $P(X)$ denote the smallest index of a proper subgroup of $X$. Note that $G/K\cong G_{\alpha\beta}/K_{\alpha\beta}$ by Lemma~\ref{l10}, and $G/K$ is isomorphic to a transitive subgroup of $\Sy_\ell$ since $N$ is the unique minimal normal subgroup of $G$.

{\bf Case 1.} Suppose that $\SL_n(q)\leqslant G_{\alpha\beta}^{\Ga(\alpha)}\leqslant\GaL_n(q)$, where $n\geqslant2$ and $q=p^f$. Then $K_{\alpha\beta}^{\Ga(\alpha)}\leqslant\Z_{q-1}$, and $\PSL_n(q)$ is a section of $G_{\alpha\beta}/K_{\alpha\beta}$. It follows from Lemma~\ref{MinimalDegree}(c) that $\ell\geqslant P(\PSL_n(q))$ due to $G_{\alpha\beta}/K_{\alpha\beta}\lesssim\Sy_\ell$. Thereby (\ref{eq2}) leads to
\begin{equation}\label{eq3}
\max\{5,p\}^{P(\PSL_n(q))}\leqslant(q-1)|\AGaL_n(q)|.
\end{equation}

Assume that $(n,q)\neq(4,2)$ or $(2,q)$ with $q\leqslant11$. Then we have $P(\PSL_n(q))=(q^n-1)/(q-1)$ according to Table~\ref{tab24}, and hence (\ref{eq3}) implies
\begin{equation}\label{eq1}
5^{(q^n-1)/(q-1)}\leqslant|K_{\alpha\beta}^{\Ga(\alpha)}||G_\alpha^{\Ga(\alpha)}|\leqslant(q-1)|\AGaL_n(q)|<q^{n^2+n+2}.
\end{equation}
If $q=2$, then $n\geqslant3$ and the above inequality turns out to be $5^{2^n-1}<2^{n^2+n+2}$, not possible. It derives from (\ref{eq1}) that $5^{q^{n-1}}<5^{(q^n-1)/(q-1)}<q^{n^2+n+2}$. Hence $q^{n-1}\ln5<(n^2+n+2)\ln q<(n^2+n+2)q/2$, and so $q^{n-2}\ln25<n^2+n+2$. This indicates $n\leqslant3$ as $q\geqslant3$. Further, if $n=3$, then (\ref{eq1}) turns out to be $5^{q^2+q+1}<q^{14}$, which is not possible. We therefore have $n=2$. However, (\ref{eq1}) then turns out to be $5^{q+1}<q^8$, violating our assumption that $q>11$.

Consequently, we only need to deal with the candidates in the table below.
\[
\begin{array}{|l|lllllllll|}\hline
n&2&2&2&2&2&2&2&2&4 \\
\hline
q&2&3&4&5&7&8&9&11&2 \\
\hline
P(\PSL_n(q))&2&3&5&5&7&9&6&11&8 \\
\hline
\end{array}
\]
For $(n,q)=(2,2)$, $(2,7)$, $(2,11)$ or $(4,2)$, the inequality (\ref{eq3}) does not hold, a contradiction. Thus $n=2$ and $q\in\{3,4,5,8,9\}$. Recall from (\ref{eq2}) that
\begin{equation}\label{eq33}
m^\ell\leqslant|K_{\alpha\beta}^{\Ga(\alpha)}||G_\alpha^{\Ga(\alpha)}|\leqslant(q-1)|\AGaL_2(q)|=q^3(q^2-1)(q-1)^2f.
\end{equation}

Assume that $q=3$. Then $K_{\alpha\beta}^{\Ga(\alpha)}\leqslant\Z_2$, and
$$
|\Sy_\ell|\geqslant|G_{\alpha\beta}/K_{\alpha\beta}|\geqslant|G_{\alpha\beta}^{\Ga(\alpha)}/K_{\alpha\beta}^{\Ga(\alpha)}|
\geqslant|\SL_2(3)|/2\geqslant12.
$$
This implies $\ell\geqslant4$, and so (\ref{eq33}) requires $m^4\leqslant2^5\cdot3^3$. Accordingly we obtain $m\leqslant5$, and thus $T=\A_5$ by Lemma~\ref{l16}. Observe that any solvable subgroup of $T$ has index divisible by $5$ or $6$. We infer from Lemma~\ref{Restriction}(b) that $5^x6^y\di2^5\cdot3^3$ for some nonnegative integers $x$ and $y$ with $x+y=\ell$, which is impossible as $\ell\geqslant4$.

Assume that $q=4$. Then $\ell\geqslant P(\PSL_2(4))=5$, and (\ref{eq33}) requires $m^5\leqslant2^6\cdot3^3\cdot5$. Hence $m\leqslant7$, and so $T=\A_5$ or $\PSL_2(7)$ by Lemma~\ref{l16}. If $T=\A_5$, then any solvable subgroup of $T$ has index divisible by $5$ or $6$, and we infer from Lemma~\ref{Restriction}(b) that $5^x6^y\di2^6\cdot3^3\cdot5$ for some nonnegative integers $x$ and $y$ with $x+y=\ell\geqslant5$, a contradiction. Thus $T\neq\A_5$. Similarly, we derive that $T\neq\PSL_2(7)$.

Assume that $q=5$. Then $\ell\geqslant P(\PSL_2(5))=5$, and (\ref{eq33}) requires $m^5\leqslant2^7\cdot3\cdot5^3$. Hence $m\leqslant8$, and thus Lemma~\ref{l16} implies $T=\A_5$ as $|T|$ is divisible by $p=5$. Now we derive from Lemma~\ref{Restriction}(b) that $5^x6^y\di2^7\cdot3\cdot5^3$ for some nonnegative integers $x$ and $y$ with $x+y=\ell\geqslant5$, a contradiction.

Assume that $q=8$. Then $\ell\geqslant P(\PSL_2(4))=9$, and (\ref{eq33}) requires $m^5\leqslant2^9\cdot3^2\cdot7^3$. Hence $m\leqslant5$, and so $T=\A_5$ by Lemma~\ref{l16}. Now we derive from Lemma~\ref{Restriction}(b) that $5^x6^y\di2^9\cdot3^2\cdot7^3$ for some nonnegative integers $x$ and $y$ with $x+y=\ell\geqslant9$, a contradiction.

Assume that $q=9$. Then $\ell\geqslant P(\PSL_2(9))=6$, and (\ref{eq33}) requires $m^6\leqslant2^{11}\cdot3^6\cdot5$. Hence $m\leqslant13$, and thus Lemma~\ref{l16} implies that $T=\A_5$, $\PSL_2(7)$, $\PSL_2(8)$, $\A_6$, $\PSL_2(11)$ or $\PSL_3(3)$. If $T=\PSL_2(7)$, then any solvable subgroup of $T$ has index divisible by $7$ or $8$, and thereby we derive from Lemma~\ref{Restriction}(b) that $7^x8^y\di2^{11}\cdot3^6\cdot5$ for some nonnegative integers $x$ and $y$ with $x+y=\ell\geqslant6$, not possible. Similarly we exclude the candidates $\PSL_2(8)$, $\A_6$, $\PSL_2(11)$ and $\PSL_3(3)$ for $T$. It follows that $T=\A_5$, and so $|T|^\ell/|R_1\times\dots R_\ell|$ is divisible by $5^x6^y$ for some nonnegative integers $x$ and $y$ with $x+y=\ell$. According to Lemma~\ref{Restriction}(c), $K_{\alpha\beta}\lesssim\Aut(T)=\Sy_5$. This together with Lemma~\ref{normal-stab}(b) and $K_{\alpha\beta}^{\Ga(\alpha)}\leqslant\Z_8$ implies that $|K_{\alpha\beta}|$ divides $4$. Hence $5^x6^y$ divides $2^2\cdot3^4|\overline{R}|$ by Lemma~\ref{Restriction}(a). Note that Lemma~\ref{Restriction}(b) implies $5^x6^y\di2^{11}\cdot3^6\cdot5$, which yields $x\leqslant1$ and $y\leqslant6$. Therefore, $(x,y)=(1,6)$, $(1,5)$ or $(0,6)$ as $\ell=x+y\geqslant6$. However, none of these values for $(x,y)$ allows $5^x6^y$ to divide $2^2\cdot3^4|\overline{R}|$ for some solvable $\overline{R}\lesssim\Sy_\ell=\Sy_{x+y}$, a contradiction.

{\bf Case 2.} Suppose that $G_{\alpha\beta}^{\Ga(\alpha)}\trianglerighteq\Sp_{2n}(q)'$, where $n\geqslant2$ and $q=p^f$. Then $G_{\alpha\beta}^{\Ga(\alpha)}\leqslant(\Z_{q-1}\circ\Sp_{2n}(q)).\Z_f$, $K_{\alpha\beta}^{\Ga(\alpha)}\leqslant\Z_{q-1}$, and $\PSp_{2n}(q)'$ is a section of $G_{\alpha\beta}/K_{\alpha\beta}$. It follows from Lemma~\ref{MinimalDegree}(c) that $\ell\geqslant P(\PSp_{2n}(q)')$ due to $G_{\alpha\beta}/K_{\alpha\beta}\lesssim\Sy_\ell$. Let $q=p^f$. Then (\ref{eq2}) yields
\begin{equation}\label{eq24}
\max\{5,p\}^{P(\PSp_{2n}(q)')}\leqslant q^{2n}|\Sp_{2n}(q)|(q-1)^2f.
\end{equation}
Noticing that $f\leqslant q$, we obtain
\begin{equation}\label{eq23}
\max\{5,p\}^{P(\PSp_{2n}(q)')}\leqslant q^{2n+1}|\Sp_{2n}(q)|(q-1)^2<q^{2n^2+3n+3}.
\end{equation}
Moreover, by Theorem~\ref{l7}, one of the following appears.
\begin{itemize}
\item[(i)] $q>2$, $(n,q)\neq(2,3)$ and $P(\PSp_{2n}(q)')=(q^{2n}-1)/(q-1)$.
\item[(ii)] $q=2$, $n\geqslant3$ and $P(\PSp_{2n}(q)')=2^{n-1}(2^n-1)$.
\item[(iii)] $(n,q)=(2,3)$ and $P(\PSp_{2n}(q)')=27$.
\item[(iv)] $(n,q)=(2,2)$ and $P(\PSp_{2n}(q)')=6$.
\end{itemize}
The possibility for~(iii) and~(iv) is directly ruled out by (\ref{eq24}). If~(i) occurs, then (\ref{eq23}) implies that $q^{2n-1}<(q^{2n}-1)/(q-1)<(2n^2+3n+3)f$, and thus $3^{2n-1}\leqslant q^{2n-1}/f<2n^2+3n+3$, not possible. If~(ii) occurs, then (\ref{eq23}) implies that
$$
2^{2^n(2^n-1)}<5^{2^{n-1}(2^n-1)}<2^{2n^2+3n+3},
$$
and thus $2^n(2^n-1)<2n^2+3n+3$, not possible either.

{\bf Case 3.} Suppose that $G_{\alpha\beta}^{\Ga(\alpha)}\trianglerighteq\G_2(q)'$, where $q=2^f$. Then $p^d=q^6$, $G_{\alpha\beta}^{\Ga(\alpha)}\leqslant(\Z_{q-1}\circ\G_2(q)).\Z_f$, $K_{\alpha\beta}^{\Ga(\alpha)}\leqslant\Z_{q-1}$, and $\G_2(q)'$ is a section of $G_{\alpha\beta}/K_{\alpha\beta}$. It follows from Lemma~\ref{MinimalDegree}(c) that $\ell\geqslant P(\G_2(q)')$ due to $G_{\alpha\beta}/K_{\alpha\beta}\lesssim\Sy_\ell$. Then (\ref{eq2}) yields
\begin{equation}\label{eq25}
5^{P(\G_2(q)')}\leqslant q^6|\G_2(q)|(q-1)^2f<q^{23}.
\end{equation}
For $q=2$ or $4$, $P(\G_2(q)')=28$ or $416$ by~\cite{atlas}, contrary to (\ref{eq25}). Thus $q\geqslant8$. Consulting the classification of maximal subgroups of $\G_2(2^f)'=\G_2(2^f)$ in~\cite{Cooperstein1981}, one immediately concludes that $P(\G_2(q)')=(q^6-1)/(q-1)$. Thereby (\ref{eq25}) implies that
$$
2^{q^5}<5^{q^5}<5^{(q^6-1)/(q-1)}<q^{23}=2^{23f},
$$
and so $8^5/3\leqslant q^5/f<23$, a contradiction.

{\bf Case 4.} Suppose that $G_\alpha^{\Ga(\alpha)}\leqslant\AGaL_1(p^d)$. Let $H=G_{\alpha\beta}^{\Ga(\alpha)}\cap\GL_1(p^d)$, and $M$ be the full preimage of $H$ under the natural homomorphism from $G_{\alpha\beta}$ to $G_{\alpha\beta}^{\Ga(\alpha)}$. Note by Lemma~\ref{normal-stab} that $G_{\alpha\beta}^{[1]}=1$, $K_{\alpha\beta}^{\Ga(\alpha)}\leqslant\GL_1(p^f)$ for some proper divisor $f$ of $d$, and $|G_\alpha^{[1]}|$ divides $p^f-1$. From (\ref{eq2}) we deduce
\begin{equation}\label{eq31}
\max\{5,p\}^\ell\leqslant|K_{\alpha\beta}^{\Ga(\alpha)}||G_\alpha^{\Ga(\alpha)}|\leqslant(p^f-1)p^d(p^d-1)d.
\end{equation}

As $M/M\cap G_\alpha^{[1]}\cong M^{\Ga(\alpha)}=H$ is cyclic, we have $M'\leqslant G_\alpha^{[1]}$. Similarly, $M/M\cap G_\beta^{[1]}\cong M^{\Ga(\beta)}\cong H$ yields $M'\leqslant G_\beta^{[1]}$. Thus $M'\leqslant G_\alpha^{[1]}\cap G_\beta^{[1]}=G_{\alpha\beta}^{[1]}=1$, which means that $M$ is abelian. Consequently, $M$ has a subgroup $L$ isomorphic to $M^{\Ga(\beta)}\cong H$. Denote $\overline{G}=G_{\alpha\beta}/K_{\alpha\beta}$, $\overline{M}=MK_{\alpha\beta}/K_{\alpha\beta}$ and $\overline{L}=LK_{\alpha\beta}/K_{\alpha\beta}$. Clearly, $\overline{L}$ is a cyclic normal subgroup of $\overline{M}$. Since $\overline{M}$ is a normal subgroup of $\overline{G}$ and $\overline{G}\cong G/K$ is a transitive permutation group on $\calT$, the orbits of $\overline{M}$ on $\calT$, say $\Delta_1,\dots,\Delta_k$, are of equal size. For $1\leqslant i\leqslant k$, let $M_i$ and $L_i$ be the restriction of $\overline{M}$ and $\overline{L}$, respectively, on $\Delta_i$. Then $\overline{M}=M_1\times\dots\times M_k$, $\overline{L}=L_1\times\dots\times L_k$, and $L_i$ is a cyclic normal subgroup of $M_i$ for each $1\leqslant i\leqslant k$. From the transitivity of $M_i$ on $\Delta_i$ we deduce that $L_i$ is semiregular on $\Delta_i$, and thus $|L_i|$ divides $|\Delta_i|=|\Delta_1|$ for each $1\leqslant i\leqslant k$. Noticing that $|L_1|,\dots,|L_k|$ are pairwise coprime as $\overline{L}$ is cyclic, we conclude that $|\overline{L}|=|L_1|\cdots|L_k|$ divides $|\Delta_1|$. In particular, $|\overline{L}|$ divides $\ell$. Because $|\overline{M}/\overline{L}|$ divides $|M/L|=|M\cap G_\alpha^{[1]}|$ and $|G_\alpha^{[1]}|$ divides $p^f-1$, so $|\overline{M}/\overline{L}|$ divides $p^f-1$. Moreover, $|\overline{G}/\overline{M}|$ divides $d$, since $|\overline{G}/\overline{M}|$ divides $|G_{\alpha\beta}/M|=|G_{\alpha\beta}^{\Ga(\alpha)}/H|$ and $|G_{\alpha\beta}^{\Ga(\alpha)}/H|$ divides $|\GaL_1(p^d)/|\GL_1(p^d)|=d$. It follows that $|\overline{G}|$ divides $(p^f-1)d|\overline{L}|$ and thus divides $(p^f-1)d\ell$. Since $\overline{G}$ is divisible by $|G_{\alpha\beta}^{\Ga(\alpha)}/K_{\alpha\beta}^{\Ga(\alpha)}|$ while $|G_{\alpha\beta}^{\Ga(\alpha)}/K_{\alpha\beta}^{\Ga(\alpha)}|$ is divisible by $(p^d-1)/(p^f-1)$, we thereby obtain \begin{equation}\label{eq29}
(p^d-1)/(p^f-1)\di(p^f-1)d\ell.
\end{equation}
Observe that the greatest common divisor of $(p^d-1)/(p^f-1)$ and $p^f-1$ equals $(d/f,p^f-1)$ due to $(p^d-1)/(p^f-1)\equiv d/f\pmod{p^f-1}$. Then (\ref{eq29}) is equivalent to $(p^d-1)/(p^f-1)\di(d/f,p^f-1)d\ell$, which implies that
\begin{equation}\label{eq30}
\ell\geqslant\frac{(p^d-1)f}{(p^f-1)d^2}\geqslant\frac{(p^d-1)d/2}{(p^{d/2}-1)d^2}=\frac{p^{d/2}+1}{2d}.
\end{equation}

Suppose $p=2$. Combining (\ref{eq31}) and (\ref{eq30}) one obtains
\begin{equation}\label{eq21}
5^\frac{2^{d/2}+1}{2d}\leqslant(2^{d/2}-1)2^d(2^d-1)d.
\end{equation}
As a consequence, $2^{(2^{d/2}+1)/d}<5^{(2^{d/2}+1)/(2d)}<2^{3d}$, or equivalently, $2^{d/2}+1<3d^2$. This implies that $d\leqslant20$. However, $d=20$ does not satisfy (\ref{eq29}), whence $d\leqslant19$. If $d$ is prime, then $f=1$ and (\ref{eq29}) yields $2^d-1\di\ell$, which in conjunction with (\ref{eq31}) gives $5^{2^d-1}\leqslant2^d(2^d-1)d$, a contradiction. If $d$ is a power of $2$, then (\ref{eq29}) yields $2^{d/2}+1\di\ell$, and so (\ref{eq31}) leads to $5^{2^{d/2}+1}\leqslant(2^{d/2}-1)2^d(2^d-1)d$, not possible. If $d=9$, then $73\di\ell$ by (\ref{eq29}), not satisfying (\ref{eq31}). If $d=10$, then by (\ref{eq29}), either $f\leqslant2$ and $31\di\ell$ or $f=5$ and $33\di\ell$, still not satisfying (\ref{eq31}). The same argument excludes $d\in\{12,14,15,18\}$. Therefore, $d=6$. If $f\leqslant2$, then $7\di\ell$ by (\ref{eq29}), not satisfying (\ref{eq31}). Thus $f=3$ as a proper divisor of $d$. Since $G_{\alpha\beta}^{\Ga(\alpha)}/K_{\alpha\beta}^{\Ga(\alpha)}$ is divisible by $(2^6-1)/(2^3-1)=9$ and $G_{\alpha\beta}^{\Ga(\alpha)}/K_{\alpha\beta}^{\Ga(\alpha)}$ is a section of $\Sy_\ell$, we conclude $\ell\geqslant6$. Then (\ref{eq2}) requires
$$
m^6\leqslant m^\ell\leqslant|K_{\alpha\beta}^{\Ga(\alpha)}||G_\alpha^{\Ga(\alpha)}|\leqslant|\GL_1(2^3)||\AGaL_6(2)|=2^7\cdot3^3\cdot7^2,
$$
which means $m\leqslant7$. It follows from Lemma~\ref{l16} that $T=\A_5$ or $\PSL_2(7)$. If $T=\A_5$, then any solvable subgroup of $T$ has index divisible by $5$ or $6$, and thereby we derive from Lemma~\ref{Restriction}(b) that $5^x6^y\di2^7\cdot3^3\cdot7^2$ for some nonnegative integers $x$ and $y$ with $x+y=\ell\geqslant6$, not possible. Hence $T\neq\A_5$. Similarly $T\neq\PSL_2(7)$, a contradiction.

Suppose $p=3$. Combining (\ref{eq31}) and (\ref{eq30}) one obtains
$$
5^\frac{3^{d/2}+1}{2d}\leqslant(3^{d/2}-1)3^d(3^d-1)d.
$$
As a consequence, $5^{(3^{d/2}+1)/(2d)}<3^{3d}$, and thus $5^{3^{d/2}+1}<3^{6d^2}<5^{5d^2}$. This is equivalent to $3^{d/2}+1<5d^2$, whence $d\leqslant11$. If $d=2$, then $f=1$, and (\ref{eq31}) yields $\ell\leqslant3$ while $|G_{\alpha\beta}^{\Ga(\alpha)}/K_{\alpha\beta}^{\Ga(\alpha)}|$ is divisible by $(3^2-1)/(3-1)=4$, contrary to the fact that $G_{\alpha\beta}^{\Ga(\alpha)}/K_{\alpha\beta}^{\Ga(\alpha)}$ is a section of $\Sy_\ell$. If $d=3$, then $13\di\ell$ by (\ref{eq29}), not satisfying (\ref{eq31}). The same argument excludes $d\in\{5,6,7,8,9,10,11\}$. Therefore, $d=4$, and so $K_{\alpha\beta}^{\Ga(\alpha)}\leqslant\GL_1(3^2)$ as $f\leqslant2$. By (\ref{eq29}) we have $5\di\ell$. Then since (\ref{eq2}) requires
$$
m^5\leqslant m^\ell\leqslant|K_{\alpha\beta}^{\Ga(\alpha)}||G_\alpha^{\Ga(\alpha)}|\leqslant|\GL_1(3^2)||\AGaL_4(3)|=2^9\cdot3^4\cdot5,
$$
we obtain $m\leqslant11$. It follows from Lemma~\ref{l16} that $T=\A_5$, $\PSL_2(7)$, $\PSL_2(8)$ or $\A_6$. If $T=\PSL_2(7)$, then any solvable subgroup of $T$ has index divisible by $7$ or $8$, and thereby we derive from Lemma~\ref{Restriction}(b) that $7^x8^y\di2^9\cdot3^4\cdot5$ for some nonnegative integers $x$ and $y$ with $x+y=\ell\geqslant5$, which is impossible. Hence $T\neq\PSL_2(7)$. Similarly, $T\neq\PSL_2(8)$ or $\A_6$. Now $T=\A_5$, and thus $|T|^\ell/|R_1\times\dots R_\ell|$ is divisible by $5^x6^y$ for some nonnegative integers $x$ and $y$ with $x+y=\ell$. Thereby Lemma~\ref{Restriction}(b) implies that $5^x6^y\di2^9\cdot3^4\cdot5$, which leads to $(x,y)=(1,4)$ as $x+y=\ell\geqslant5$. In view of $K_{\alpha\beta}^{\Ga(\alpha)}\leqslant\Z_8$, we derive from Lemma~\ref{normal-stab}(b) that $K_{\alpha\beta}$ is a $2$-group. Accordingly, $5^x3^y$ divides $3^4|\overline{R}|$ by Lemma~\ref{Restriction}(a), whence $|\overline{R}|$ is divisible by $5$. It follows that $\overline{R}$ is transitive on $\calT$. However, this indicates that $R_1\cong\dots\cong R_5$ and thus $(|\A_5|/|R_1|)^5$ divides $2^9\cdot3^4\cdot5$, not possible.

Thus far we have known that $p\geqslant5$. Then (\ref{eq31}) together with (\ref{eq30}) implies
\begin{equation}\label{eq32}
p^\frac{p^{d/2}+1}{2d}\leqslant(p^{d/2}-1)p^d(p^d-1)d<p^{5d/2}d,
\end{equation}
or equivalently, $p^{d/2}+1<5d^2+2d\log_pd$. As a consequence, $5^{d/2}+1<5d^2+2d\log_5d$, which restricts $d\leqslant6$. If $d=6$, then (\ref{eq32}) forces $p=5$ and so (\ref{eq29}) yields $21\di\ell$, not satisfying (\ref{eq31}). If $d=5$, then (\ref{eq29}) implies that $(p^5-1)/(p-1)\di25\ell$ and so $\ell>25p^4$, which leads to $p^{p^4/25}<5p^{11}\leqslant p^{12}$ by (\ref{eq31}), a contradiction. If $d=3$, then (\ref{eq29}) implies that $(p^2+p+1)/(9,p^2+p+1)\di\ell$, which leads to $p^{(p^2+p+1)/(9,p^2+p+1)}<3p^7<p^8$ by (\ref{eq31}), not possible. If $d=4$, then it follows from (\ref{eq32}) that $p=5$ or $7$. If $d=4$ and $p=5$, then (\ref{eq29}) yields $13\di\ell$, violating (\ref{eq31}). If $d=4$ and $p=7$, then (\ref{eq29}) yields $25\di\ell$, again violating (\ref{eq31}). Consequently, $d\neq4$, and thereby we have $d=2$. Now (\ref{eq32}) gives $p\in\{5,7,11,13,17,19\}$. If $p=13$ or $17$, then (\ref{eq29}) yields $7\di\ell$ or $9\di\ell$, respectively, not satisfying (\ref{eq31}). We next rule out the possibilities for $p\in\{5,7,11,19\}$.

Assume $p=5$. Then (\ref{eq29}) yields $3\di\ell$, and so (\ref{eq2}) requires
$$
m^3\leqslant m^\ell\leqslant|K_{\alpha\beta}^{\Ga(\alpha)}||G_\alpha^{\Ga(\alpha)}|\leqslant|\GL_1(5)||\AGaL_2(5)|=2^7\cdot3\cdot5^3,
$$
giving $m\leqslant36$. Moreover, $|T|$ is divisible by $p=5$. Thus it follows from Lemma~\ref{l16} that $T=\A_8$ or $\PSL_2(q)$ with $q\in\{5,9,11,16,19,25,29,31\}$. If $T=\A_8$, then any solvable subgroup of $T$ has index divisible by $3$ or $7$, and thereby we derive from Lemma~\ref{Restriction}(b) that $3^x7^y\di2^7\cdot3\cdot5^3$ for some nonnegative integers $x$ and $y$ with $x+y=\ell\geqslant3$, not possible. Hence $T\neq\A_8$. Similarly, $T\neq\PSL_2(q)$ for $q\in\{11,16,25,29,31\}$. Therefore, $T=\A_5$, $\A_6$ or $\PSL_2(19)$. We derive a contradiction below under the assumption that $T=\A_5$, while a similar contradiction can be derived for $T=\A_6$ or $\PSL_2(19)$. Note that $|T|^\ell/|R_1\times\dots R_\ell|$ is divisible by $5^x6^y$ for some nonnegative integers $x$ and $y$ with $x+y=\ell$, and thereby Lemma~\ref{Restriction}(b) implies $5^x6^y\di2^9\cdot3^4\cdot5$. This shows that $x\leqslant1$ and $y\leqslant4$, which leads to $(x,y)=(3,0)$ or $(2,1)$ since $3\di(x+y)$. In view of $K_{\alpha\beta}^{\Ga(\alpha)}\leqslant\Z_4$, we derive from Lemma~\ref{normal-stab}(b) that $K_{\alpha\beta}$ is a $2$-group, and so $5^x3^y$ divides $5^2|\overline{R}|$ by Lemma~\ref{Restriction}(a). If $(x,y)=(3,0)$, then $5$ divides $|\overline{R}|$, contrary to the condition $\overline{R}\lesssim\Sy_\ell=\Sy_3$. Consequently, $(x,y)=(2,1)$, and thus $3$ divides $|\overline{R}|$. It follows that $\overline{R}$ is transitive on $\calT$. However, this indicates that $R_1\cong R_2\cong R_3$ and so $(|\A_5|/|R_1|)_2^3$ divides $5^2|\overline{R}|$, again contrary to $\overline{R}\lesssim\Sy_3$.

Assume $p=7$. Since $G_{\alpha\beta}^{\Ga(\alpha)}/K_{\alpha\beta}^{\Ga(\alpha)}$ is divisible by $(7^2-1)/(7-1)=8$ and $G_{\alpha\beta}^{\Ga(\alpha)}/K_{\alpha\beta}^{\Ga(\alpha)}$ is a section of $\Sy_\ell$, we conclude $\ell\geqslant4$. Then (\ref{eq2}) requires
$$
m^4\leqslant m^\ell\leqslant|K_{\alpha\beta}^{\Ga(\alpha)}||G_\alpha^{\Ga(\alpha)}|\leqslant|\GL_1(7)||\AGaL_2(7)|=2^6\cdot3^2\cdot7^2,
$$
which means $m\leqslant12$. Moreover, $|T|$ is divisible by $p=7$. Thus it follows from Lemma~\ref{l16} that $T=\PSL_2(7)$ or $\PSL_2(8)$. If $T=\PSL_2(8)$, then any solvable subgroup of $T$ has index divisible by $9$ or $28$, and thereby we derive from Lemma~\ref{Restriction}(b) that $9^x28^y\di2^6\cdot3^2\cdot7^2$ for some nonnegative integers $x$ and $y$ with $x+y=\ell\geqslant4$, not possible. Now $T=\PSL_2(7)$, and so $|T|^\ell/|R_1\times\dots R_\ell|$ is divisible by $7^x8^y$ for some nonnegative integers $x$ and $y$ with $x+y=\ell$. Then Lemma~\ref{Restriction}(b) implies that $7^x8^y\di2^6\cdot3^2\cdot7^2$, which leads to $(x,y)=(2,2)$ as $x+y=\ell\geqslant5$. In view of $K_{\alpha\beta}^{\Ga(\alpha)}\leqslant\Z_6$, we derive from Lemma~\ref{normal-stab}(b) that $|K_{\alpha\beta}|$ divides $36$. Consequently, $7^x8^y$ divides $2^2\cdot3^2\cdot7^2|\overline{R}|$ by Lemma~\ref{Restriction}(a). However, this indicates that $|\overline{R}|$ is divisible by $2^4$, contrary to the condition $\overline{R}\lesssim\Sy_\ell=\Sy_4$.

Assume $p=11$. Since $G_{\alpha\beta}^{\Ga(\alpha)}/K_{\alpha\beta}^{\Ga(\alpha)}$ is divisible by $(11^2-1)/(11-1)=12$ and $G_{\alpha\beta}^{\Ga(\alpha)}/K_{\alpha\beta}^{\Ga(\alpha)}$ is a section of $\Sy_\ell$, we have $\ell\geqslant4$. Moreover,(\ref{eq29}) indicates $3\di\ell$. Thereby we conclude that $\ell\geqslant6$. Then (\ref{eq2}) requires
$$
m^6\leqslant m^\ell\leqslant|K_{\alpha\beta}^{\Ga(\alpha)}||G_\alpha^{\Ga(\alpha)}|\leqslant|\GL_1(11)||\AGaL_2(11)|=2^5\cdot3\cdot5^3\cdot11^3,
$$
which means $m\leqslant15$. Also, $|T|$ is divisible by $p=11$. Thus it follows from Lemma~\ref{l16} that $T=\PSL_2(11)$. Now any solvable subgroup of $T$ has index divisible by $11$ or $12$, and thereby we derive from Lemma~\ref{Restriction}(b) that $11^x12^y\di2^5\cdot3\cdot5^3\cdot11^3$ for some nonnegative integers $x$ and $y$ with $x+y=\ell\geqslant6$, not possible.

Assume $p=19$. Then (\ref{eq29}) yields $5\di\ell$, and so (\ref{eq2}) requires
$$
m^5\leqslant m^\ell\leqslant|K_{\alpha\beta}^{\Ga(\alpha)}||G_\alpha^{\Ga(\alpha)}|\leqslant|\GL_1(19)||\AGaL_2(19)|=2^5\cdot3^4\cdot5\cdot19^2,
$$
giving $m\leqslant21$. Moreover, $|T|$ is divisible by $p=19$. Thus it follows from Lemma~\ref{l16} that $T=\PSL_2(19)$. Note that any solvable subgroup of $T$ has index divisible by $19$ or $20$. Thereby we derive from Lemma~\ref{Restriction}(b) that $19^x20^y\di2^5\cdot3^4\cdot5\cdot19^2$ for some nonnegative integers $x$ and $y$ with $x+y=\ell\geqslant5$, not possible.

{\bf Case 5.} Suppose that $G_\alpha^{\Ga(\alpha)}$ is not one of the $2$-transitive permutation groups in the previous cases. Then $(p^d,G_{\alpha\beta}^{\Ga(\alpha)},K_{\alpha\beta}^{\Ga(\alpha)})$ lies in the table below.
\[
\begin{array}{|l|l|l|l|}
\hline
\text{row} &p^d & G_{\alpha\beta}^{\Ga(\alpha)} & K_{\alpha\beta}^{\Ga(\alpha)}\leqslant \\
\hline
1 &3^6 & \SL_2(13) & 2 \\
2 &2^4 & \A_7 & 1 \\
3 &3^4 & 2^{1+4}.5\leqslant G_{\alpha\beta}^{\Ga(\alpha)}\leqslant2^{1+4}.\Sy_5 & 2 \\
\hline
4 &5^2 & \SL_2(3) & 2 \\
5 &5^2 & \Q_8.6,\ \SL_2(3).4 & 4 \\
6 &7^2 & \Q_8.\Sy_3,\ \SL_2(3).6 & 6 \\
7 &11^2 & \SL_2(3).5,\ \SL_2(3).10 & 10 \\
8 &23^2 & \SL_2(3).22 & 22 \\
\hline
9 &11^ 2 & \SL_2(5), 5\times\SL_2(5) & 10 \\
10 &19^ 2 & 9\times\SL_2(5) & 18 \\
11 &29^ 2 & 7\times\SL_2(5), 28.\PSL_2(5) & 28 \\
12 &59^ 2 & 29\times\SL_2(5) & 58 \\
\hline
\end{array}
\]

Assume that row~1 appears. Then it follows from Lemma~\ref{MinimalDegree}(c) that $\ell\geqslant P(\PSL_2(13))=14$ since $\PSL_2(13)$ is a section of $G_{\alpha\beta}/K_{\alpha\beta}\lesssim\Sy_\ell$. Thereby (\ref{eq2}) yields $5^{14}\leqslant2\cdot3^6|\SL_2(13)|$, a contradiction. Similar argument rules out row~2.

Assume that row~3 appears. Then $G_{\alpha\beta}^{\Ga(\alpha)}/K_{\alpha\beta}^{\Ga(\alpha)}$ has a section $2^4{:}5$, and so $\Sy_\ell$ has a section $2^4{:}5$. This implies that $\ell\geqslant10$. It follows from (\ref{eq2}) that $5^{10}\leqslant2\cdot3^4|2^{1+4}.\Sy_5|$, a contradiction.

Next consider rows~4--8. As $|\Sy_\ell|\geqslant|G_{\alpha\beta}/K_{\alpha\beta}|\geqslant|G_{\alpha\beta}^{\Ga(\alpha)}/K_{\alpha\beta}^{\Ga(\alpha)}|
\geqslant12$, we have $\ell\geqslant4$. Note that $|K_{\alpha\beta}^{\Ga(\alpha)}||G_\alpha^{\Ga(\alpha)}|$ divides $24p^2(p-1)^2$. Then (\ref{eq2}) yields $m^4\leqslant24p^2(p-1)^2$, which gives an upper bound for $m$ in the table below. This together with the condition that $p$ divides $|T|$ restricts the possibilities for $T$ by Lemma~\ref{l16}, also shown in the table.
\[
\begin{array}{|l|l|l|l|l|}
\hline
p&5&7&11&23\\
\hline
m\leqslant&9&14&23&49\\
\hline
T&\A_5&\PSL_2(7),\PSL_2(8),\PSL_2(13)&\PSL_2(11)&\PSL_2(23)\\
\hline
\end{array}
\]
For $p=5$, since any solvable subgroup of $T=\A_5$ has index divisible by $5$ or $6$, we derive from Lemma~\ref{Restriction}(b) that $5^x6^y\di24\cdot5^2\cdot4^2$ for some nonnegative integers $x$ and $y$ with $x+y=\ell\geqslant4$, a contradiction. For $p=7$, if $T=\PSL_2(7)$, then since any solvable subgroup of $T$ has index divisible by $7$ or $8$, Lemma~\ref{Restriction}(b) yields that $7^x8^y\di24\cdot7^2\cdot6^2$ for some nonnegative integers $x$ and $y$ with $x+y=\ell\geqslant4$, not possible. Similarly, $T$ is not $\PSL_2(8)$ or $\PSL_2(13)$ either for $p=7$. For $p=11$, since any solvable subgroup of $T=\PSL_2(11)$ has index divisible by $11$ or $12$, it follows from Lemma~\ref{Restriction}(b) that $11^x12^y\di24\cdot11^2\cdot10^2$ for some nonnegative integers $x$ and $y$ with $x+y=\ell\geqslant4$, a contradiction. For $p=23$, any solvable subgroup of $T=\PSL_2(23)$ has index divisible by $23$ or $24$, and thereby it follows from Lemma~\ref{Restriction}(b) that $11^x12^y\di24\cdot23^2\cdot22^2$ for some nonnegative integers $x$ and $y$ with $x+y=\ell\geqslant4$, again a contradiction.

Finally, consider rows~9--12. Note that $|K_{\alpha\beta}^{\Ga(\alpha)}||G_\alpha^{\Ga(\alpha)}|$ divides $60p^2(p-1)^2$, and $\ell\geqslant P(\PSL_2(5))=5$ since $\PSL_2(5)$ is a section of $G_{\alpha\beta}/K_{\alpha\beta}\lesssim\Sy_\ell$. It derives from
(\ref{eq2}) that $m^5\leqslant60p^2(p-1)^2$, which gives $m\leqslant14$, $23$, $33$ or $58$ corresponding to $p=11$, $19$, $29$ or $59$, respectively. We conclude that $T=\PSL_2(p)$ by Lemma~\ref{l16} as $|T|$ is divisible by $p$. Then any solvable subgroup of $T$ has index divisible by $p$ or $p+1$, and so Lemma~\ref{Restriction}(b) implies that
$p^x(p-1)^y\di60p^2(p-1)^2$ for some nonnegative integers $x$ and $y$ with $x+y=\ell\geqslant5$, not possible.
\end{proof}

We summarize the outcomes of this section in the following proposition.

\begin{proposition}\label{reduction-AS}
Let $\Ga=(V,E)$ be a connected $(G,2)$-arc-transitive graph, and $R$ be a solvable vertex-transitive subgroup of $G$. If $G$ is quasiprimitive on $V$, then $G$ is either an almost simple group or an affine group.
\end{proposition}

\begin{proof}
We have seen in Lemmas~\ref{l11} and~\ref{l9} that the $G$ is not of type TW or PA. Then by~\cite[Theorem~2]{Praeger}, the quasiprimitive type of $G$ must be HA or AS.
\end{proof}

\section[Proof]{Proof of Theorem~\ref{CayleyGraph}}

We shall complete the proof of Theorem~\ref{CayleyGraph} in this section. Since quasiprimitive $2$-arc-transitive graphs of affine type are classified in \cite{IP}, we only need to treat the almost simple case by Proposition~\ref{reduction-AS}.
For convenience, we make a hypothesis as follows.

\begin{hypothesis}\label{GraphHypo}
Let $\Ga=(V,E)$ be a connected $(G,s)$-arc-transitive graph, where $s\geqslant2$ and the valency of $\Ga$ is at least three. Suppose that $G$ is almost simple with socle $L$, and $G$ is quasiprimitive on $V$ containing a solvable vertex-transitive subgroup $R$. Let $\{\alpha,\beta\}\in E$.
\end{hypothesis}

By the quasiprimitivity of $G$ we immediately derive from the above hypothesis that $L$ is transitive on $V$ and $\Ga$ is non-bipartite. Furthermore, by Lemmas~\ref{l10} and~\ref{CosetGraph} we have the consequences below.

\begin{lemma}\label{l17}
Under \emph{Hypothesis~\ref{GraphHypo}}, the following statements hold.
\begin{itemize}
\item[(a)] $G$ has a factorization $G=RG_\alpha$ with a solvable factor $R$ and a core-free factor $G_\alpha$.
\item[(b)] Either $G=L$ or $G_\alpha\nleqslant L$.
\item[(c)] There exists a $2$-element $g$ in $\Nor_G(G_{\alpha\beta})$ such that
$$
\langle G_\alpha,\Nor_G(G_{\alpha\beta})\rangle=\langle G_\alpha,g\rangle=G
$$
and $G_\alpha$ is $2$-transitive on $[G_\alpha{:}G_\alpha\cap G_\alpha^g]$.
\end{itemize}
\end{lemma}

The candidates for $(G,R,G_\alpha)$ in the factorization $G=RG_\alpha$ are categorized by Theorem~\ref{SolvableFactor} into parts~(a)--(d) there, and we will analyze them separately. Above all, we consider the candidates $(G,R,G_\alpha)=(G,H,K)$ as in part~(a) of Theorem~\ref{SolvableFactor}, which is the case where both $H$ and $K$ are solvable.

\begin{lemma}\label{SolvableStab}
Under \emph{Hypothesis~\ref{GraphHypo}}, if $G_\alpha$ is solvable, then $L=\PSL_2(q)$ for some prime power $q$, so that $\Ga$ is classified in \cite{HNP}.
\end{lemma}


\begin{proof}
Suppose that $L\neq\PSL_2(q)$ for any prime power $q$. Then by Proposition~\ref{BothSolvable}, $(G,R,G_\alpha)=(G,H,K)$ or $(G,K,H)$ for some triple $(G,H,K)$ in rows~4--12 of Table~\ref{tab4}.

Assume that $L=\PSL_3(3)$ as in rows~4--5 of Table~\ref{tab4}. Then since $G_\alpha$ is $2$-transitive on $\Ga(\alpha)$, we infer that $G_\alpha=3^2{:}2.\Sy_4$ or $\AGaL_1(9)$ and in particular $G_\alpha<L$. This yields $G=L$ by Lemma~\ref{l17}(b), and so $|V|=|G|/|G_\alpha|=13$ or $39$. As a consequence, the valency of $\Ga$ is even, which restricts $G_\alpha=3^2{:}2.\Sy_4$ and $G_{\alpha\beta}=3^2{:}2.\Sy_3$. However, this causes $\Nor_G(G_{\alpha\beta})=G_{\alpha\beta}<G_\alpha$ and then $\langle G_\alpha,\Nor_G(G_{\alpha\beta})=G_\alpha$, contrary to Lemma~\ref{l17}(c).

For $L=\PSL_3(4)$ as in row~6 of Table~\ref{tab4}, since the quasiprimitivity of $G$ on $V$ requires that there is no subgroup of index two in $G$ containing $G_\alpha$, the possibilities for $(G,G_\alpha)$ are $(L.\Sy_3,7{:}3.\Sy_3)$, $(L.(\Sy_3\times2),7{:}6.\Sy_3)$ and $(L.\Sy_3,2^4{:}(3\times\D_{10}).2)$. Searching in \magma \cite{bosma1997magma} for these candidates gives no such connected $(G,2)$-arc-transitive graph.

In the following we exclude rows~7--12 of Table~\ref{tab4}. If $\Ga$ has valency at least five, then since $G_\alpha$ is solvable, we have $G_{\alpha\beta}^{[1]}=1$ by Theorem~\ref{DoubleStar} and thus $G_\alpha^{[1]}\trianglelefteq G_{\alpha\beta}^{\Ga(\beta)}\cong G_{\alpha\beta}^{\Ga(\alpha)}$ by Lemma~\ref{abeq}. In particular, if $\Ga$ has valency at least five, then
\begin{equation}\label{eq34}
\text{$|G_\alpha|=|G_\alpha^{[1]}||G_\alpha^{\Ga(\alpha)}|$ divides
$|G_{\alpha\beta}^{\Ga(\alpha)}||G_\alpha^{\Ga(\alpha)}|=|G_\alpha^{\Ga(\alpha)}|^2/|\Ga(\alpha)|$}.
\end{equation}

Let $L=\PSL_3(8)$ as in row~7 of Table~\ref{tab4}. Since $G_\alpha$ is $2$-transitive on $\Ga(\alpha)$, we have $G_\alpha=2^{3+6}{:}7^2{:}3$ or $2^{3+6}{:}7^2{:}6$. If $G_\alpha=2^{3+6}{:}7^2{:}6$, then $|\Ga(\alpha)|=7$ and $G=\PSL_3(8).6$, which leads to a contradiction that both $|\Ga(\alpha)|$ and $|V|=|G|/|G_\alpha|$ are odd. Thus $G_\alpha=2^{3+6}{:}7^2{:}3$, and so $G_\alpha^{\Ga(\alpha)}=2^3{:}7{:}3$. However, (\ref{eq34}) requires $|G_\alpha|$ to divide $|G_\alpha^{\Ga(\alpha)}|^2$, a contradiction.

Let $L=\PSU_3(8)$ as in row~8 of Table~\ref{tab4}. Since $G_\alpha$ is $2$-transitive on $\Ga(\alpha)$ and $G$ does not have a subgroup of index two containing $G_\alpha$, we have $(G,G_\alpha)=(\PSU_3(8).3^2.\calO,2^{3+6}{:}(63{:}3).\calO)$ with $\calO\leqslant\Z_2$. As a consequence, $|V|=|G|/|G_\alpha|=513$ is odd, and so $|\Ga(\alpha)|$ is even. Hence $|\Ga(\alpha)|=2^6$ and $G_\alpha^{\Ga(\alpha)}=2^6{:}(63{:}3).\calO$, but this does not satisfy (\ref{eq34}), a contradiction.

Assume next that $L=\PSU_4(2)$ and $G_\alpha=H$ or $K$ as in rows~9--11 of Table~\ref{tab4}. For each candidate of $G_\alpha$, let $X$ be the maximal subgroup of $G$ containing $G_\alpha$, where $X\cap L=2^4{:}\A_5$, $3_+^{1+2}{:}2.\A_4$ or $3^3{:}\Sy_4$. Then computation in \magma \cite{bosma1997magma} verifies that for any subgroup $M$ of $G_\alpha$ such that $G_\alpha$ acts $2$-transitively on $[G_\alpha{:}M]$, one has $\Nor_G(M)\leqslant X$. This violates Lemma~\ref{l17}(c).

For $G=\M_{11}$ as in row~12 Table~\ref{tab4}, $G_\alpha=\M_9.2=3^2{:}\Q_8.2$ and $G_{\alpha\beta}=\Q_8.2$, which makes $|V|=|G|/|G_\alpha|$ and $|\Ga(\alpha)|=|G_\alpha|/|G_{\alpha\beta}|$ both odd, a contradiction. This proves the lemma.
\end{proof}

Next we treat the candidates for $(G,R,G_\alpha)=(G,H,K)$ as described in part~(b) of Theorem~\ref{SolvableFactor}.

\begin{lemma}\label{AlternatingGraph}
Under \emph{Hypothesis~\ref{GraphHypo}}, if $L=\A_n$ with $n\geqslant7$, then $\Ga=\K_n$.
\end{lemma}

\begin{proof}
The triple $(G,R,G_\alpha)=(G,H,K)$ is classified in Proposition~\ref{Alternating}, which shows that one of the following occurs.
\begin{itemize}
\item[(i)] $\A_n\trianglelefteq G\leqslant\Sy_n$ with $n\geqslant6$, and $\A_{n-1}\trianglelefteq G_\alpha\leqslant\Sy_{n-1}$.
\item[(ii)] $\A_n\trianglelefteq G\leqslant\Sy_n$ with $n=p^f$ for some prime $p$, and $\A_{n-2}\trianglelefteq G_\alpha\leqslant\Sy_{n-2}\times\Sy_2$.
\item[(iii)] $\A_n\trianglelefteq G\leqslant\Sy_n$ with $n=8$ or $32$, and $\A_{n-3}\trianglelefteq G_\alpha\leqslant\Sy_{n-3}\times\Sy_3$.
\item[(iv)] $(G,R,G_\alpha)=(G,H,K)$ in rows~5--6 of Table~\ref{tab6}.
\end{itemize}

First assume that~(i) occurs. Then viewing Lemma~\ref{l17}(b) we have $(G,G_\alpha)=(\Sy_n,\Sy_{n-1})$ or $(\A_n,\A_{n-1})$. It follows that $G$ is $2$-transitive on $[G{:}G_\alpha]$, and hence $\Ga=\K_n$.

Next assume that~(ii) occurs. For $n=7$, $8$ and $9$, respectively, computation in \magma \cite{bosma1997magma} shows that there is no such connected $(G,2)$-arc-transitive graph. Thus $n\geqslant11$.

Suppose $G=\A_n$. Then either $G_\alpha=\A_{n-2}$ or $G_\alpha=(\Sy_{n-2}\times\Sy_2)\cap\A_n\cong\Sy_{n-2}$. Assume that $G_\alpha=\A_{n-2}$. Because $G_\alpha$ acts $2$-transitively on $[G_\alpha{:}G_{\alpha\beta}]$, so we have $G_{\alpha\beta}=\A_{n-3}$. It follows that $\Nor_G(G_{\alpha\beta})=(\Sy_{n-3}\times\Sy_3)\cap\A_n$. Thus for any $2$-element $g\in\Nor_G(G_{\alpha\beta})\setminus G_\alpha$, we have $\langle G_\alpha,g\rangle=\langle\A_{n-2},(1,2)(n-2,n-1)\rangle$, $\langle\A_{n-2},(1,2)(n-2,n)\rangle$ or $\langle\A_{n-2},(1,2)(n-1,n)\rangle$, and hence $\langle G_\alpha,g\rangle\cong\A_{n-1}$ or $\Sy_{n-2}$. This is contrary to Lemma~\ref{l17}(c). Now $G_\alpha=(\Sy_{n-2}\times\Sy_2)\cap\A_n=\langle\A_{n-2},(1,2)(n-1,n)\rangle$. Since $G_\alpha$ acts $2$-transitively on $[G_\alpha{:}G_{\alpha\beta}]$, we may assume $G_{\alpha\beta}=\langle\A_{n-3},(1,2)(n-1,n)\rangle$ without loss of generality. It follows that $\Nor_G(G_{\alpha\beta})=G_{\alpha\beta}<G_\alpha$, still contrary to Lemma~\ref{l17}(c).

Suppose $G=\Sy_n$. Then $G_\alpha=\Sy_{n-2}$, $\A_{n-2}\times\Sy_2$ or $\Sy_{n-2}\times\Sy_2$ by Lemma~\ref{l17}(b). This implies that $L=\A_{n}$ is $2$-arc-transitive on $\Ga$, which is not possible as shown in the previous paragraph.

Now assume that~(iii) appears. If $n=8$, then searching in \magma \cite{bosma1997magma} shows that no such connected $(G,2)$-arc-transitive graph arises. Therefore, $n=32$. According to Proposition~\ref{Alternating}, $R=\AGaL_1(32)$, which implies that $|G_\alpha|$ is divisible by $3|\Sy_{29}|$ since $|G_\alpha|$ is divisible by $|G|/|R|$. Hence $\A_{29}\times\Z_3\trianglelefteq G_\alpha$, and so $\Z_3\leqslant G_\alpha^{[1]}$, $\A_{29}\leqslant G_\alpha^{\Ga(\alpha)}\leqslant\Sy_{29}$ and $\A_{28}\leqslant G_{\alpha\beta}^{\Ga(\alpha)}\leqslant\Sy_{28}$. By Theorem~\ref{DoubleStar}(b) we conclude $G_{\alpha\beta}^{[1]}=1$, and it then follows from Lemma~\ref{abeq} that $G_\alpha^{[1]}\cong(G_\alpha^{[1]})^{\Ga(\beta)}\trianglelefteq G_{\alpha\beta}^{\Ga(\alpha)}$, not possible.

Finally, for case~(iv), computation in \magma \cite{bosma1997magma} shows that there exists no such connected $(G,2)$-arc-transitive graph. This completes the proof.
\end{proof}

The case where $(G,R,G_\alpha)=(G,H,K)$ as in part~(c) of Theorem~\ref{SolvableFactor} is dealt with in the next lemma.

\begin{lemma}\label{SporadicGraph}
Under \emph{Hypothesis~\ref{GraphHypo}}, if $L$ is a sporadic simple group and $G_\alpha$ is unsolvable, then either $\Ga$ is a complete graph, or $\Ga$ is the Higman-Sims graph and $(G,G_\alpha)=(\HS,\M_{22})$ or $(\HS.2,\M_{22}.2)$.
\end{lemma}


\begin{proof}
By Proposition~\ref{Sporadic}, either $G$ acts $2$-transitively on $[G{:}K]$, or $(G,K)$ lies in rows~3--4 and~7--13 of Table~\ref{tab8}. For the former, $\Ga$ is a complete graph. Since $G_\alpha$ is $2$-transitive on $\Ga(\alpha)$, it is not possible for $G_\alpha=\Sp_4(4).4$ or $\G_2(4).2$ as in row~12 or~13, respectively, of Table~\ref{tab8}. Moreover, computation in \magma \cite{bosma1997magma} shows that there exists no non-bipartite connected $(G,2)$-arc-transitive graph for rows~3--4 of Table~\ref{tab8}. Thus we only need to consider rows~7--11 of Table~\ref{tab8}.

Let $G=\M_{23}$ as in row~7 of Table~\ref{tab8}. If $G_\alpha=\M_{22}$, then $G$ acts $2$-transitively on $[G{:}G_\alpha]$ and so $\Ga=\K_{23}$. If $G_\alpha=\PSL_3(4){:}2$, then $G_{\alpha\beta}=2^4{:}\Sy_5$ and $\Nor_G(G_{\alpha\beta})=G_{\alpha\beta}$, contradicting Lemma~\ref{l17}(c). If $G_\alpha=2^4{:}\A_7$, then $G_{\alpha\beta}=\A_7$, $2^4{:}\A_6$ or $2^4{:}\GL_3(2)$ and $\Nor_G(G_{\alpha\beta})=G_{\alpha\beta}$, again contradicting Lemma~\ref{l17}(c).

If $G=\J_2.2$ as in row~8 of Table~\ref{tab8}, then $G_\alpha=\G_2(2)$ and $G_{\alpha\beta}=3_+^{1+2}{:}8{:}2$ as $G_\alpha$ is $2$-transitive on $\Ga(\alpha)$, but $\Nor_G(G_{\alpha\beta})=G_{\alpha\beta}$, violating the connectivity of $\Ga$.

Now let $L=\HS$ as in rows~9--11 of Table~\ref{tab8}. Viewing Lemma~\ref{l17}(b), we have $(G,G_\alpha)=(\HS.\calO,\M_{22}.\calO)$ with $\calO\leqslant\Z_2$. From the $2$-transitivity of $G_\alpha$ on $\Ga(\alpha)$ we derive that $G_{\alpha\beta}=\PSL_3(4).\calO$. Hence $\Ga$ has order $|G|/|G_\alpha|=100$ and valency $|G_\alpha|/|G_{\alpha\beta}|=22$. This is the well-known Higman-Sims graph, see Example~\ref{Higman-S}.
\end{proof}

Finally we embark on the candidates for $(G,R,G_\alpha)=(G,H,K)$ as in part~(d) of Theorem~\ref{SolvableFactor}.

\begin{lemma}\label{ClassicalGraph}
Under \emph{Hypothesis~\ref{GraphHypo}}, if $G$ is a classical group of dimension greater than two and $G_\alpha$ is unsolvable, then $\Ga$ is the Hoffman-Singlton graph and $(G,G_\alpha)=(\PSU_3(5),\A_7)$ or $(\PSiU_3(5),\Sy_7)$.
\end{lemma}

\begin{proof}
According to part~(d) of Theorem~\ref{SolvableFactor}, $(G,R,G_\alpha)=(G,H,K)$ with $(L,H\cap L,K\cap L)$ described in Tables~\ref{tab7}--\ref{tab1}. Since $G_\alpha$ acts $2$-transitively on $\Ga(\alpha)$, inspecting the candidates in Tables~\ref{tab7}--\ref{tab1} we conclude that one of the following occurs.
\begin{itemize}
\item[(i)] $L=\PSL_n(q)$, where $n\geqslant3$ and $(n,q)\neq(3,2)$ or $(3,3)$, and $L_\alpha\trianglerighteq q^{n-1}{:}\SL_{n-1}(q)$.
\item[(ii)] $L=\PSp_4(q)$, and $\PSL_2(q^2)\trianglelefteq L_\alpha\leqslant\PSL_2(q^2).2$ as in row~4 or 5 of Table~\ref{tab7}.
\item[(iii)] $L=\PSU_4(q)$, and $\SU_3(q)\trianglelefteq L_\alpha$ as in row~6 of Table~\ref{tab7}.
\item[(iv)] $(L,H\cap L,K\cap L)$ lies in rows~6--21, 23 and~25--26 of Table~\ref{tab1}.
\end{itemize}

{\bf Case 1.} Suppose that (i) appears. Then either $L_\alpha^{\Ga(\alpha)}\trianglerighteq q^{n-1}{:}\SL_{n-1}(q)$ is affine with $|\Ga(\alpha)|=q^{n-1}$, or $L_\alpha^{\Ga(\alpha)}\trianglerighteq\PSL_{n-1}(q)$ is almost simple with $|\Ga(\alpha)|=(q^{n-1}-1)/(q-1)$. Accordingly, $G_{\alpha\beta}$ is the maximal subgroup of index $q^{n-1}$ or $(q^{n-1}-1)/(q-1)$ in $G_\alpha$. If $G\leqslant\PGaL_n(q)$, then $\Nor_G(G_{\alpha\beta})\leqslant\Pa_1[G]$, which leads to $\langle G_\alpha,\Nor_G(G_{\alpha\beta})\rangle\leqslant\Pa_1[G]$, contrary to Lemma~\ref{CosetGraph}(c). Consequently, $G\nleqslant\PGaL_n(q)$, and so $G\cap\PGaL_n(q)$ has index $2$ in $G$. Moreover, $G\cap\PGaL_n(q)$ contains $G_\alpha$ by Lemma~\ref{l5}. This implies that $G\cap\PGaL_n(q)$ is not transitive on $V$, which is impossible as $G$ is quasiprimitive on $V$.

{\bf Case 2.} Suppose that (ii) appears. In light of the graph automorphism of $L$ for even $q$, we may assume without loss of generality that $L_\alpha$ is a $\mathcal{C}_3$ subgroup of $L$. Denote the stabilizer of a totally isotropic $1$ or $2$-space in $L$ by $\Pa_1$ or $\Pa_2$, respectively. Since $L_\alpha\trianglelefteq G_\alpha$ and $G_\alpha$ is $2$-transitive on $\Ga(\alpha)$, we deduce from the classification of $2$-transitive permutation groups that $L_\alpha$ is $2$-transitive on $\Ga(\alpha)$ and
$$
L_{\alpha\beta}=L_\alpha\cap\Pa_2=q^2{:}\frac{q^2-1}{(2,q-1)}\text{ or }q^2{:}\frac{q^2-1}{(2,q-1)}.2.
$$
Therefore, $\Ga$ is $(L,2)$-arc-transitive, and then by Lemma~\ref{CosetGraph} there exists $g\in L$ such that $g$ normalizes $L_{\alpha\beta}$, $g^2\in L_{\alpha\beta}$ and $\langle L_\alpha,g\rangle=L$. Let $X=\PSL_2(q^2).2$ be the maximal subgroup of $L$ containing $L_\alpha$, and $N=\langle g\rangle L_{\alpha\beta}=L_{\alpha\beta}\langle g\rangle$. Since $\bfO_p(L_{\alpha\beta})\neq1$, we know that $\bfO_p(N)\neq1$ and hence $N$ is contained in the parabolic subgroup $\Pa_1$ or $\Pa_2$ of $L$ by Borel-Tits theorem. Note that the intersection of $\Pa_1$ with the $\mathcal{C}_3$ subgroup $X$ has order $|X||\Pa_1|/|L|=2q^2(q-1)/(2,q-1)$ as the factorization $L=X\Pa_1$ holds. If $N\leqslant\Pa_1$, then $L_{\alpha\beta}\leqslant\Pa_1$, which leads to a contradiction that
$$
2q^2(q-1)\geqslant|\Pa_1\cap X|\geqslant|\Pa_1\cap L_\alpha|\geqslant|L_{\alpha\beta}|\geqslant q^2(q^2-1).
$$
Consequently, $N\leqslant\Pa_2$, from which we conclude
$$
N=q^2{:}\frac{q^2-1}{(2,q-1)}.2=X\cap\Pa_2
$$
as $|N|/|L_{\alpha\beta}|=2$. It follows that $N\leqslant X$, and then $\langle L_\alpha,g\rangle\leqslant\langle L_\alpha,N\rangle\leqslant X\neq L$, contrary to the connectivity of $\Ga$.

{\bf Case 3.} Suppose that (iii) appears. In this case, $Y\trianglelefteq L_\alpha\leqslant X$, where $Y=\SU_3(q)$ and $X=\GU_3(q)/d$ with $d=(4,q+1)$. Denote the stabilizer of a totally isotropic $1$ or $2$-space in $L$ by $\Pa_1$ or $\Pa_2$, respectively. Since $L_\alpha\trianglelefteq G_\alpha$ and $G_\alpha$ is $2$-transitive on $\Ga(\alpha)$, we deduce from the classification of $2$-transitive permutation groups that $L_\alpha$ is $2$-transitive on $\Ga(\alpha)$ and $L_{\alpha\beta}\cap Y=q^{1+2}{:}(q^2-1)$. Therefore, $\Ga$ is $(L,2)$-arc-transitive, and then by Lemma~\ref{CosetGraph} there exists $g\in L$ such that $g$ normalizes $L_{\alpha\beta}$, $g^2\in L_{\alpha\beta}$ and $\langle L_\alpha,g\rangle=L$. Let $N=\langle g\rangle L_{\alpha\beta}=L_{\alpha\beta}\langle g\rangle$. Since $\bfO_p(L_{\alpha\beta})\neq1$, we know that $\bfO_p(N)\neq1$ and hence $N$ is contained in the parabolic subgroup $\Pa_1$ or $\Pa_2$ of $L$ by Borel-Tits theorem. We conclude that
$$
N\leqslant q^{1+2}{:}(q^2-1).\frac{q+1}{d}\leqslant X
$$
as $|N|/|L_{\alpha\beta}|=2$, but this implies $\langle L_\alpha,g\rangle\leqslant\langle L_\alpha,N\rangle\leqslant X\neq L$, contrary to the connectivity of $\Ga$.

{\bf Case 4.} Suppose that (iv) appears. For the candidates of $(L,H\cap L,K\cap L)$ in rows~6--11 of Table~\ref{tab1}, searching in \magma \cite{bosma1997magma} gives no such connected $(G,2)$-arc-transitive graph.

Let $L=\PSp_4(p)$, where $p\in\{5,7,11,23\}$, as in rows~12--15 of Table~\ref{tab1}. Then $L_\alpha=\PSL_2(p^2){:}\calO$ with $\calO\leqslant\Z_2$, and since $G_\alpha^{\Ga(\alpha)}$ is $2$-transitive, $L_\alpha^{\Ga(\alpha)}$ is also $2$-transitive. Therefore, $\Ga$ is $(L,2)$-arc-transitive. By the $2$-transitivity of $L_\alpha$ on $\Ga(\alpha)$ we derive that $L_{\alpha\beta}=(p^2{:}(p^2-1)/2).\calO$ and then $\Nor_L(L_{\alpha\beta})<\PSL_2(p^2){:}2$. However, this implies that $\langle L_\alpha,\Nor_L(L_{\alpha\beta})\rangle\leqslant\PSL_2(p^2){:}2<L$, contrary to the connectivity of $\Ga$.

Assume that $L=\Sp_6(2)$ as in row~16 of Table~\ref{tab1}. Then $G=L$, and $G_\alpha=\A_8$ or $\Sy_8$. If $G_\alpha=\Sy_8$, then $G$ is $2$-transitive but not $3$-transitive on $[G{:}G_\alpha]$, a contradiction. Thus $G_\alpha=\A_8$, and we have $G_{\alpha\beta}=2^3{:}\GL_3(2)$ or $\A_7$ since $G_\alpha$ acts $2$-transitively on $\Ga(\alpha)$. If $G_{\alpha\beta}=2^3{:}\GL_3(2)$, then $\Nor_G(G_{\alpha\beta})=G_{\alpha\beta}<G_\alpha$, violating the connectivity of $\Ga$. If $G_{\alpha\beta}=\A_7$, then $\Sy_7\cong\Nor_G(G_{\alpha\beta})<G_\alpha$, still violating the connectivity of $\Ga$.

Let $L=\PSp_6(3)$ as in row~17 of Table~\ref{tab1}. Then $L_\alpha=\PSL_2(27).3$ is also $2$-transitive on $\Ga(\alpha)$, and hence $\Ga$ is $(L,2)$-arc-transitive. Due to the $2$-transitivity of $L_\alpha^{\Ga(\alpha)}$ we deduce $G_{\alpha\beta}=3^3{:}13{:}3$. However, it follows that $\Nor_L(L_{\alpha\beta})=L_{\alpha\beta}<L_\alpha$, contrary to the connectivity of $\Ga$. The same arguments rules out rows~18 and~20 of Table~\ref{tab1}.

Now let $L=\PSU_3(5)$ as in row~19 of Table~\ref{tab1}. Then $L_\alpha=\A_7$ is also $2$-transitive on $\Ga(\alpha)$, and hence $L_{\alpha\beta}=\PSL_2(7)$ or $\A_6$. If $L_{\alpha\beta}=\PSL_2(7)$, then $\Nor_L(L_{\alpha\beta})=L_{\alpha\beta}<L_\alpha$, not possible. Consequently, $L_{\alpha\beta}=\A_6$, and $\Ga$ has order $|L|/|L_\alpha|=50$ and valency $|L_\alpha|/|L_{\alpha\beta}|=7$. This is the well-known Hoffman-Singleton graph introduced in Example~\ref{Hoffman-S}, where we see that $\Aut(\Ga)=\PSiU_3(5)$ and so $(G,G_\alpha)=(\PSU_3(5),\A_7)$ or $(\PSiU_3(5),\Sy_7)$.

Let $L=\PSU_4(8)$ as in row~21 of Table~\ref{tab1}. Then $L_\alpha=2^{12}{:}\SL_2(64).7$, and $|V|=|L|/|L_\alpha|=4617$ is odd. Since $G_\alpha$ is $2$-transitive on $\Ga(\alpha)$, it follows that $|\Ga(\alpha)|=65$ is odd, a contradiction. (We remark that the subgroup $2^{12}{:}\SL_2(64)$ of $L_\alpha$ is not isomorphic to $\ASL_2(64)$ and does not have any $2$-transitive affine permutation representation.) The same argument excludes row~23 of Table~\ref{tab1}.

Assume finally that $L=\Omega_8^+(2)$ as in rows~25--26 of Table~\ref{tab1}, where $L_\alpha=\Sp_6(2)$ or $\A_9$. Then $L_\alpha$ is also $2$-transitive on $\Ga(\alpha)$. If $L_\alpha=\Sp_6(2)$, then $L_{\alpha\beta}=\Sy_8$ or $\SO_6^-(2)$, but $\Nor_L(L_{\alpha\beta})=L_{\alpha\beta}<L_\alpha$, violating the connectivity of $\Ga$. If $L_\alpha=\A_9$, then $L_{\alpha\beta}=\A_8$, but $\Nor_L(L_{\alpha\beta})=L_{\alpha\beta}<L_\alpha$, still violating the connectivity of $\Ga$. This completes the proof.
\end{proof}


%

\vskip0.1in
{\bf Proof of Theorem~\ref{CayleyGraph}:}
Let $\Si$ be a non-bipartite connected $(X,s)$-transitive graph and $R$ be a solvable vertex-transitive subgroup of $X$, where $s\geqslant2$ and the valency of $\Si$ is at least three. By virtue of Lemma~\ref{reduction-qp} we may assume that $X=G$ is vertex-quasiprimitive and $\Si=\Ga$. By Proposition~\ref{reduction-AS}, $G$ is affine or almost simple.

If $G$ is affine, then $\Ga$ is not $(G,3)$-arc-transitive by~\cite[Proposition~2.3]{Li05}, and so $s=2$ as in part~(b) of Theorem~\ref{CayleyGraph}.

Now assume that $G$ is almost simple. Then $(G,\Ga)$ satisfies Hypothesis~\ref{GraphHypo}. If $\PSL_2(q)\leqslant G\leqslant\PGaL_2(q)$ for some prime power $q\geqslant4$, then $\Ga$ is classified in~\cite{HNP}, from which we conclude that $s=3$ if and only if $\Ga$ is the Petersen graph. Thus assume that $\Soc(G)\neq\PSL_2(q)$. Then Lemmas~\ref{SolvableStab}--\ref{ClassicalGraph} show that $(G,\Ga)$ satisfies part~(a), (d) or (e) of Theorem~\ref{CayleyGraph}. Note that the complete graph is not $3$-arc-transitive, the Hoffman-Singlton graph is $3$-transitive by Example~\ref{Hoffman-S} and the Higman-Sims graph is not $3$-arc-transitive by Example~\ref{Higman-S}. This completes the proof of Theorem~\ref{CayleyGraph}.
\qed

\appendix


\chapter[Maximal factorizations of almost simple classical groups]{Tables for nontrivial maximal factorizations of almost simple classical groups}


\begin{table}[htbp]
\caption{$L=\PSL_n(q)$, $n\geqslant2$}\label{tabLinear}
\centering
\begin{tabular}{|l|l|l|l|}
\hline
row & $A\cap L$ & $B\cap L$ & remark\\
\hline
1 & $\hat{~}\GL_a(q^b).b$ & $\Pa_1$, $\Pa_{n-1}$ & $ab=n$, $b$ prime \\
\hline
2 & $\PSp_n(q).a$ & $\Pa_1$, $\Pa_{n-1}$ & $n\geqslant4$ even, \\
 & & & $a=(2,q-1)(n/2,q-1)/(n,q-1)$ \\
\hline
3 & $\PSp_n(q).a$ & $\Stab(V_1\oplus V_{n-1})$ & $n\geqslant4$ even, $G\nleqslant\PGaL_n(q)$, \\
 & & & $a=(2,q-1)(n/2,q-1)/(n,q-1)$ \\
\hline
4 & $\hat{~}\GL_{n/2}(q^2).2$ & $\Stab(V_1\oplus V_{n-1})$ & $n\geqslant4$ even, $q\in\{2,4\}$, $G\nleqslant\PGaL_n(q)$\\
\hline
5 & $\Pa_1$ & $\A_5$ & $n=2$, $q\in\{11,19,29,59\}$ \\
\hline
6 & $\Pa_1$ & $\Sy_4$ & $n=2$, $q\in\{7,23\}$ \\
\hline
7 & $\Pa_1$ & $\A_4$ & $G=\PGL_2(11)$ \\
\hline
8 & $\D_{34}$ & $\PSL_2(4)$ & $G=\PGaL_2(16)$ \\
\hline
9 & $\PSL_2(7)$ & $\A_6$ & $n=3$, $q=4$, $G=L.2$ \\
\hline
10 & $31.5$ & $\Pa_2$, $\Pa_3$ & $n=5$, $q=2$ \\
\hline
\end{tabular}
\end{table}

\begin{table}[htbp]
\caption{$L=\PSp_{2m}(q)$, $m\geqslant2$}\label{tabSymplectic}
\centering
\begin{tabular}{|l|l|l|l|}
\hline
row & $A\cap L$ & $B\cap L$ & remark\\
\hline
1 & $\PSp_{2a}(q^b).b$ & $\Pa_1$ & $ab=m$, $b$ prime \\
\hline
2 & $\Sp_{2a}(q^b).b$ & $\GO_{2m}^+(q)$, $\GO_{2m}^-(q)$ & $q$ even, $ab=m$, $b$ prime \\
\hline
3 & $\GO_{2m}^-(q)$ & $\Pa_m$ & $q$ even \\
\hline
4 & $\GO_{2m}^-(q)$ & $\Sp_m(q)\wr\Sy_2$ & $m$ even, $q$ even \\
\hline
5 & $\Sp_m(4).2$ & $\N_2$ & $m\geqslant4$ even, $q=2$ \\
\hline
6 & $\GO_{2m}^-(2)$ & $\GO_{2m}^+(2)$ & $q=2$ \\
\hline
7 & $\GO_{2m}^-(4)$ & $\GO_{2m}^+(4)$ & $q=4$, $G=L.2$, two classes \\
 & & & of factorizations if $m=2$ \\
\hline
8 & $\Sp_m(16).2$ & $\N_2$ & $m\geqslant4$ even, $q=4$, $G=L.2$ \\
\hline
9 & $\GO_{2m}^-(4)$ & $\Sp_{2m}(2)$ & $q=4$, $G=L.2$ \\
\hline
10 & $\GO_{2m}^-(16)$ & $\Sp_{2m}(4)$ & $q=16$, $G=L.4$ \\
\hline
11 & $\Sz(q)$ & $\GO_4^+(q)$ & $m=2$, $q=2^f$, $f\geqslant3$ odd, \\
 & & & two classes of factorizations \\
\hline
12 & $\G_2(q)$ & $\GO_6^+(q)$, $\GO_6^-(q)$, $\Pa_1$, $\N_2$ & $m=3$, $q$ even \\
\hline
13 & $2^4.\A_5$ & $\Pa_1$, $\Pa_2$ & $m=2$, $q=3$ \\
\hline
14 & $\PSL_2(13)$ & $\Pa_1$ & $m=3$, $q=3$ \\
\hline
15 & $\GO_8^-(2)$ & $\Sy_{10}$ & $m=4$, $q=2$ \\
\hline
16 & $\PSL_2(17)$ & $\GO_8^+(2)$ & $m=4$, $q=2$ \\
\hline
\end{tabular}
\end{table}

\begin{table}[htbp]
\caption{$L=\PSU_n(q)$, $n\geqslant3$}\label{tabUnitary}
\centering
\begin{tabular}{|l|l|l|l|}
\hline
row & $A\cap L$ & $B\cap L$ & remark\\
\hline
1 & $\N_1$ & $\Pa_m$ & $n=2m$ \\
\hline
2 & $\N_1$ & $\PSp_{2m}(q).a$ & $n=2m$, $a=(2,q-1)(m,q+1)/(n,q+1)$ \\
\hline
3 & $\N_1$ & $\hat{~}\SL_m(4).2$ & $n=2m\geqslant6$, $q=2$ \\
\hline
4 & $\N_1$ & $\hat{~}\SL_m(16).3.2$ & $n=2m$, $q=4$, $G\geqslant L.4$ \\
\hline
5 & $\PSL_2(7)$ & $\Pa_1$ & $n=3$, $q=3$ \\
\hline
6 & $\A_7$ & $\Pa_1$ & $n=3$, $q=5$ \\
\hline
7 & $19.3$ & $\Pa_1$ & $n=3$, $q=8$, $G\geqslant L.3^2$ \\
\hline
8 & $3^3.\Sy_4$ & $\Pa_2$ & $n=4$, $q=2$ \\
\hline
9 & $\PSL_3(4)$ & $\Pa_1$, $\PSp_4(3)$, $\Pa_2$ & $n=4$, $q=3$ \\
\hline
10 & $\N_1$ & $\PSU_4(3).2$, $\M_{22}$ & $n=6$, $q=2$ \\
\hline
11 & $\J_3$ & $\Pa_1$ & $n=9$, $q=2$ \\
\hline
12 & $\Suz$ & $\N_1$ & $n=12$, $q=2$ \\
\hline
\end{tabular}
\end{table}

\begin{table}[htbp]
\caption{$L=\Omega_{2m+1}(q)$, $m\geqslant3$, $q$ odd}\label{tabOmega}
\centering
\begin{tabular}{|l|l|l|l|}
\hline
row & $A\cap L$ & $B\cap L$ & remark\\
\hline
1 & $\N_1^-$ & $\Pa_m$ & \\
\hline
2 & $\G_2(q)$ & $\Pa_1$, $\N_1^+$, $\N_1^-$, $\N_2^-$ & $m=3$ \\
\hline
3 & $\G_2(q)$ & $\N_2^+$ & $m=3$, $q>3$ \\
\hline
4 & $\PSp_6(q).a$ & $\N_1^-$ & $m=6$, $q=3^f$, $a\leqslant2$ \\
\hline
5 & $\F_4(q)$ & $\N_1^-$ & $m=12$, $q=3^f$ \\
\hline
6 & $\G_2(3)$ & $\Sy_9$, $\Sp_6(2)$ & $m=3$, $q=3$ \\
\hline
7 & $\N_1^+$ & $\Sy_9$, $\Sp_6(2)$ & $m=3$, $q=3$ \\
\hline
8 & $\Pa_3$ & $\Sy_9$, $\Sp_6(2)$, $2^6.\A_7$ & $m=3$, $q=3$ \\
\hline
\end{tabular}
\end{table}

\begin{table}[htbp]
\caption{$L=\POm_{2m}^-(q)$, $m\geqslant4$}\label{tabOmegaMinus}
\centering
\begin{tabular}{|l|l|l|l|}
\hline
row & $A\cap L$ & $B\cap L$ & remark\\
\hline
1 & $\Pa_1$ & $\hat{~}\GU_m(q)$ & $m$ odd \\
\hline
2 & $\N_1$ & $\hat{~}\GU_m(q)$ & $m$ odd \\
\hline
3 & $\Pa_1$ & $\Omega_m^-(q^2).2$ & $m$ even, $q\in\{2,4\}$, $G=\Aut(L)$ \\
\hline
4 & $\N_2^+$ & $\GU_m(4)$ & $m$ odd, $q=4$, $G=\Aut(L)$ \\
\hline
5 & $\A_{12}$ & $\Pa_1$ & $m=5$, $q=2$ \\
\hline
\end{tabular}
\end{table}

\begin{table}[htbp]
\caption{$L=\POm_{2m}^+(q)$, $m\geqslant5$}\label{tabOmegaPlus1}
\centering
\begin{tabular}{|l|l|l|l|}
\hline
row & $A\cap L$ & $B\cap L$ & remark\\
\hline
1 & $\N_1$ & $\Pa_m$, $\Pa_{m-1}$ & \\
\hline
2 & $\N_1$ & $\hat{~}\GU_m(q).2$ & $m$ even \\
\hline
3 & $\N_1$ & $(\PSp_2(q)\otimes\PSp_m(q)).a$ & $m$ even, $q>2$, $a=\gcd(2,m/2,q-1)$ \\
\hline
4 & $\N_2^-$ & $\Pa_m$, $\Pa_{m-1}$ & \\
\hline
5 & $\Pa_1$ & $\hat{~}\GU_m(q).2$ & $m$ even \\
\hline
6 & $\N_1$ & $\hat{~}\GL_m(q).2$ & $G\geqslant L.2$ if $m$ odd \\
\hline
7 & $\N_1$ & $\Omega_m^+(4).2^2$ & $m$ even, $q=2$ \\
\hline
8 & $\N_1$ & $\Omega_m^+(16).2^2$ & $m$ even, $q=4$, $G\geqslant L.2$ \\
\hline
9 & $\N_2^-$ & $\hat{~}\GL_m(2).2$ & $q=2$, $G\geqslant L.2$ if $m$ odd \\
\hline
10 & $\N_2^-$ & $\hat{~}\GL_m(4).2$ & $q=4$, $G\geqslant L.2$, $G\neq\GO_{2m}^+(4)$ \\
\hline
11 & $\N_2^+$ & $\hat{~}\GU_m(4).2$ & $m$ even, $q=4$, $G=L.2$ \\
\hline
12 & $\Omega_9(q).a$ & $\N_1$ & $m=8$, $a\leqslant2$ \\
\hline
13 & $\Co_1$ & $\N_1$ & $m=12$, $q=2$ \\
\hline
\end{tabular}
\end{table}

\begin{table}[htbp]
\caption{$L=\POm_8^+(q)$}\label{tabOmegaPlus2}
\centering
\begin{tabular}{|l|l|l|l|}
\hline
row & $A\cap L$ & $B\cap L$ & remark\\
\hline
1 & $\Omega_7(q)$ & $\Omega_7(q)$ & $A=B^\tau$ for some triality $\tau$ \\
\hline
2 & $\Omega_7(q)$ & $\Pa_1$, $\Pa_3$, $\Pa_4$ & \\
\hline
3 & $\Omega_7(q)$ & $\hat{~}((q+1)/d\times\Omega_6^-(q)).2^d$ & $d=(2,q-1)$, $B$ in $\mathcal{C}_1$ or $\mathcal{C}_3$ \\
\hline
4 & $\Omega_7(q)$ & $\hat{~}((q-1)/d\times\Omega_6^+(q)).2^d$ & $q>2$, $G=L.2$ if $q=3$, \\
 & & & $d=(2,q-1)$, $B$ in $\mathcal{C}_1$ or $\mathcal{C}_2$ \\
\hline
5 & $\Omega_7(q)$ & $(\PSp_2(q)\otimes\PSp_4(q)).2$ & $q$ odd, $B$ in $\mathcal{C}_1$ or $\mathcal{C}_4$ \\
\hline
6 & $\Omega_7(q)$ & $\Omega_8^-(q^{1/2})$ & $q$ square, $B$ in $\mathcal{C}_5$ or $\mathcal{S}$ \\
\hline
7 & $\Pa_1$, $\Pa_3$, $\Pa_4$ & $\hat{~}((q+1)/d\times\Omega_6^-(q)).2^d$ & $d=(2,q-1)$, $B$ in $\mathcal{C}_1$ or $\mathcal{C}_3$ \\
\hline
8 & $\A_9$ & $\Omega_7(2)$, $\Pa_1$, $\Pa_3$, $\Pa_4$, & $q=2$ \\
 & & $(3\times\Omega_6^-(2)).2$ & \\
\hline
9 & $\Omega_7(2)$ & $(\PSL_2(4)\times\PSL_2(4)).2^2$ & $q=2$, $B$ in $\mathcal{C}_2$ or $\mathcal{C}_3$ \\
\hline
10 & $\Omega_8^+(2)$ & $\Omega_7(3)$ & $q=3$ \\
\hline
11 & $\Omega_8^+(2)$ & $\Pa_1$, $\Pa_3$, $\Pa_4$ & $q=3$ \\
\hline
12 & $\Omega_8^+(2)$ & $\Pa_{13}$, $\Pa_{14}$, $\Pa_{34}$ & $q=3$, $G\geqslant L.2$ \\
\hline
13 & $2^6.\A_8$ & $\Pa_1$, $\Pa_3$, $\Pa_4$ & $q=3$, $G\geqslant L.2$ \\
\hline
14 & $\Omega_7(4)$ & $(\PSL_2(16)\times\PSL_2(16)).2^2$ & $q=4$, $G\geqslant L.2$ \\
\hline
15 & $(5\times\Omega_6^-(4)).2$ & $(3\times\Omega_6^+(4)).2$ & $q=4$, $G\geqslant L.2$ \\
\hline
\end{tabular}
\end{table}


\chapter[\magma codes]{\magma codes}



\begin{verbatim}
/*
\end{verbatim}
  Input : a finite group G\\
  Output: sequence of pairs [H,K] such that G=HK with H and K both core-free
\begin{verbatim}
*/
_Factorizations:=function(G)
 list:=[];
 S:=[i:i in Subgroups(G)|#Core(G,i`subgroup) eq 1];
 T:=MaximalSubgroups(G);
 while #T gt 0 do
  K:=T[1]`subgroup;
  core:=Core(G,K);
  if #core eq 1 then
   for i in S do
    H:=i`subgroup;
    if #H*#K eq #G*#(H meet K) then
     Append(~list,[H,K]);
     T:=T cat MaximalSubgroups(K);
    end if;
   end for;
  else
   T:=T cat MaximalSubgroups(K);
  end if;
  Remove(~T,1);
 end while;
 return list;
end function;
\end{verbatim}

\begin{verbatim}


/*
\end{verbatim}
  Input : a finite group G\\
  Output: sequence of pairs [H,K] such that G=HK with H solvable and K core-free
\begin{verbatim}
*/
_SolvableFactorizations:=function(G)
 list:=[];
 S:=SolvableSubgroups(G);
 T:=MaximalSubgroups(G);
 while #T gt 0 do
  K:=T[1]`subgroup;
  core:=Core(G,K);
  if #core eq 1 then
   for i in S do
    H:=i`subgroup;
    if #H*#K eq #G*#(H meet K) then
     Append(~list,[H,K]);
     T:=T cat MaximalSubgroups(K);
    end if;
   end for;
  else
   T:=T cat MaximalSubgroups(K);
  end if;
  Remove(~T,1);
 end while;
 return list;
end function;
\end{verbatim}

\begin{verbatim}


/*
\end{verbatim}
  Input : a finite group G\\
  Output: the maximal solvable subgroups as well as some other solvable subgroups of G
\begin{verbatim}
*/
_MaximalSolvableCandidates:=function(G)
 list:=[];
 temp:=MaximalSubgroups(G);
 while #temp gt 0 do
  if IsSolvable(temp[1]`subgroup) eq true then
   Append(~list,temp[1]`subgroup);
  else
   temp:=temp cat MaximalSubgroups(temp[1]`subgroup);
  end if;
  Remove(~temp,1);
 end while;
 return list;
end function;
\end{verbatim}

\begin{verbatim}


/*
\end{verbatim}
  Input : a finite group G\\
  Output: sequence of pairs [H,K] such that G=HK with H maximal solvable and K core-free
\begin{verbatim}
*/
_MaximalSolvableFactorizations1:=function(G)
 list:=[];
 S:=_MaximalSolvableCandidates(G);
 T:=MaximalSubgroups(G);
 while #T gt 0 do
  K:=T[1]`subgroup;
  core:=Core(G,K);
  if #core eq 1 then
   for H in S do
    if #H*#K eq #G*#(H meet K) then
     Append(~list,[H,K]);
     T:=T cat MaximalSubgroups(K);
    end if;
   end for;
  else
   T:=T cat MaximalSubgroups(K);
  end if;
  Remove(~T,1);
 end while;
 return list;
end function;
\end{verbatim}

\begin{verbatim}


/*
\end{verbatim}
  Input : a finite group G and a subgroup A\\
  Output: sequence of pairs [H,K] such that G=HK with H maximal solvable in A and K core-free
\begin{verbatim}
*/
_MaximalSolvableFactorizations2:=function(G,A)
 list:=[];
 S:=_MaximalSolvableCandidates(A);
 T:=MaximalSubgroups(G);
 while #T gt 0 do
  K:=T[1]`subgroup;
  core:=Core(G,K);
  if #core eq 1 then
   for H in S do
    if #H*#K eq #G*#(H meet K) then
     Append(~list,[H,K]);
     T:=T cat MaximalSubgroups(K);
    end if;
   end for;
  else
   T:=T cat MaximalSubgroups(K);
  end if;
  Remove(~T,1);
 end while;
 return list;
end function;
\end{verbatim}

\begin{verbatim}


/*
\end{verbatim}
  Input : a finite group G\\
  Output: the core-free maximal subgroups of G
\begin{verbatim}
*/
_CoreFreeMaximalSubgroups:=function(G)
 L:=MaximalSubgroups(G);
 return [i:i in L|#Core(G,i`subgroup) eq 1];
end function;
\end{verbatim}

\begin{verbatim}


/*
\end{verbatim}
  Input : a finite group G\\
  Output: sequence of pairs [H,K] such that G=HK with H maximal solvable and K core-free maximal
\begin{verbatim}
*/
_MaximalFactorizations1:=function(G)
 list:=[];
 S:=_MaximalSolvableCandidates(G);
 T:=_CoreFreeMaximalSubgroups(G);
 for H in S do
  for i in T do
   K:=i`subgroup;
   if #H*#K eq #G*#(H meet K) then
    Append(~list,[H,K]);
   end if;
  end for;
 end for;
 return list;
end function;
\end{verbatim}

\begin{verbatim}


/*
\end{verbatim}
  Input : a finite group G and a subgroup A\\
  Output: sequence of pairs [H,K] such that G=HK with H maximal solvable in A and K core-free maximal
\begin{verbatim}
*/
_MaximalFactorizations2:=function(G,A)
 list:=[];
 S:=_MaximalSolvableCandidates(A);
 T:=_CoreFreeMaximalSubgroups(G);
 for H in S do
  for i in T do
   K:=i`subgroup;
   if #H*#K eq #G*#(H meet K) then
    Append(~list,[H,K]);
   end if;
  end for;
 end for;
 return list;
end function;
\end{verbatim}

\begin{verbatim}


/*
\end{verbatim}
  Input : a finite group G and a subgroup K\\
  Output: sequence of conjugacy classes of subgroups M of K such that K=$\mathrm{G}_\alpha$ and M=$\mathrm{G}_{\alpha\beta}$ for an edge $\{\alpha,\beta\}$ of some (G,2)-arc-transitive graph
\begin{verbatim}
*/
_TwoArcTransitive:=function(G,K)
 list:=[];
 I:=[i:i in MaximalSubgroups(K)|Transitivity(CosetImage(K,i`subgroup
 )) gt 1 and sub<G|K,Normalizer(G,i`subgroup)> eq G];
 for i in I do
  M:=i`subgroup;
  N:=Normaliser(G,M);
  P:=Sylow(N,2);
  R,_:=RightTransversal(N,Normaliser(N,P));
  s:=0;
  for x in R do
   if exists(g){g:g in P^x|K^g meet K eq M and sub<G|K,g> eq G
   and g^2 in K} then
    s:=s+1;
   end if;
  end for;
  if s gt 0 then
   Append(~list,M);
  end if;
 end for;
 return list;
end function;
\end{verbatim}

\end{document}